\documentclass[
10pt, 
a4paper, 
oneside, 
headinclude,footinclude, 
]{scrartcl}

\usepackage[
nochapters, 
pdfspacing, 
dottedtoc 
]{classicthesis} 
\usepackage{amsopn}
\usepackage[T1]{fontenc} 
\usepackage{multirow}
\usepackage[utf8]{inputenc} 
\usepackage{hhline}
\usepackage{graphicx} 
\graphicspath{{Figures/}} 
\usepackage{indentfirst}
\usepackage{enumitem} 
\usepackage{savesym}
\usepackage{mathrsfs}
\usepackage{geometry}
\usepackage{hyperref}\hypersetup{colorlinks=true, citecolor=blue}
\usepackage{enumitem}
\usepackage{mathtools}
\usepackage[dvipsnames]{xcolor}
\usepackage{float}
\usepackage{longtable}
\usepackage{dirtytalk}

\usepackage[
    backend=bibtex,
    style=numeric,
    natbib=false,
    url=false, 
    doi=false,
    eprint=false,
    maxnames=50
]{biblatex}

\usepackage{xcolor}
\usepackage[skins]{tcolorbox}
\usepackage{cjhebrew}
\usepackage{esint}
\usepackage{subfig} 

\usepackage{amsmath,amssymb,amsthm,amsfonts} 

\usepackage{varioref} 

\usepackage{thmtools}
\usepackage{thm-restate}

\usepackage{hyperref}
\newcommand{\vertiii}[1]{{\left\vert\kern-0.25ex\left\vert\kern-0.25ex\left\vert #1 
    \right\vert\kern-0.25ex\right\vert\kern-0.25ex\right\vert}}

\theoremstyle{plain}
\newtheorem{teorema}{Theorem}[section]
\newtheorem{proposizione}[teorema]{Proposition}
\newtheorem{lemma}[teorema]{Lemma}
\newtheorem{corollario}[teorema]{Corollary}
\newtheorem*{theorem*}{Theorem}

\theoremstyle{definition}
\newtheorem{definizione}{Definition}[section]

\theoremstyle{plain}
\newtheorem*{congettura}{Dentsity Problem}

\theoremstyle{plain}
\newtheorem{teorema2}{Theorem}

\theoremstyle{plain}
\newtheorem{proposizione2}{Proposition}

\theoremstyle{definition}

\theoremstyle{remark}
\newtheorem{Notation}{Notation}[section]

\theoremstyle{remark}
\newtheorem{osservazione}{Remark}[section]

\theoremstyle{remark}

\newcommand{\Tan}{\mathrm{Tan}}
\newcommand{\N}{\mathbb{N}}

\newcommand{\Q}{\mathbb{Q}}
\newcommand{\R}{\mathbb{R}}

\newcommand{\HH}{\mathbb{H}}

\newcommand{\G}{Gr}
\newcommand{\res}

\newcounter{const}
\newcommand{\newC}{\refstepcounter{const}\ensuremath{C_{\theconst}}}
\newcommand{\oldC}[1]{\ensuremath{C_{\ref{#1}}}}

\newcounter{eps}
\newcommand{\newep}{\refstepcounter{eps}\ensuremath{\varepsilon_{\theeps}}}
\newcommand{\oldep}[1]{\ensuremath{\varepsilon_{\ref{#1}}}}

\DeclareMathOperator*{\lip}{Lip_1^+}
\DeclareMathOperator*{\Card}{Card}
\DeclareMathOperator*{\supp}{supp}

\DeclareMathOperator*{\diam}{diam}

\DeclareMathOperator*{\dist}{dist}

\hypersetup{
colorlinks=true, breaklinks=true,bookmarksnumbered,
urlcolor=webbrown, linkcolor=RoyalBlue, citecolor=webgreen, 
pdftitle={}, 
pdfauthor={\textcopyright}, 
pdfsubject={}, 
pdfkeywords={}, 
pdfcreator={pdfLaTeX}, 
pdfproducer={LaTeX with hyperref and ClassicThesis} 
} 

\title{\normalfont\spacedallcaps{Marstrand-Mattila rectifiability criterion for $1$-codimensional measures in Carnot Groups}} 

\author{\spacedlowsmallcaps{Andrea Merlo\textsuperscript{*}}} 

\date{} 

\begin{document}


\renewcommand{\sectionmark}[1]{\markright{\spacedlowsmallcaps{#1}}} 
\lehead{\mbox{\llap{\small\thepage\kern1em\color{halfgray} \vline}\color{halfgray}\hspace{0.5em}\rightmark\hfil}} 

\pagestyle{scrheadings} 


\maketitle 

\setcounter{tocdepth}{2} 




\paragraph*{Abstract} 
This paper is devoted to show that the flatness of tangents of $1$-codimensional measures in Carnot Groups implies $C^1_\mathbb{G}$-rectifiability. As applications we prove a criterion for intrinsic Lipschitz rectifiability of finite perimeter sets in general Carnot groups and we show that measures with $(2n+1)$-density in the Heisenberg groups $\mathbb{H}^n$ are $C^1_{\mathbb{H}^n}$-rectifiable, providing the first non-Euclidean extension of Preiss's rectifiability theorem.



{\let\thefootnote\relax\footnotetext{* \textit{Universit\'a di Pisa, Largo Bruno Pontecorvo, 5, Pisa, Italy}}}

\section*{Introduction}\label{introduction}
In Euclidean spaces the following rectifiability criterion is available, known as Marstrand-Mattila rectifiability theorem. It was first proved by J. M. Marstrand in \cite{marstrand} for $m=2$ and $n=3$, later extended by P. Mattila to every $m\leq n$ in \cite{mattilarc} and eventually strengthened by D. Preiss in \cite{Preiss1987GeometryDensities}.

\begin{teorema2}\label{th:eu:mm}
Suppose $\phi$ is a Radon measure on $\R^n$ and let $m\in\{1,\ldots,n-1\}$. Then, the following are equivalent:
\begin{itemize}
    \item[(i)] $\phi$ is absolutely continuous with respect to the $m$-dimensional Hausdorff measure $\mathcal{H}^m$ and $\R^n$ can be covered $\phi$-almost all with countably many $m$-dimensional Lipschitz surfaces,
    \item[(ii)] $\phi$ satisfies the following two conditions for $\phi$-almost every $x\in\R^n$:
    \begin{itemize}
    \item[(a)]$0<\Theta_*^{m}(\phi,x)\leq \Theta^{m,*}(\phi,x)<\infty$,
    \item[(b)]$\Tan_m(\phi,x)\subseteq\{\lambda \mathcal{H}^{m}\llcorner V:\lambda>0, V\in\text{Gr}(n,m)\}$, where the set of tangent measures $\Tan_m(\phi,x)$ is introduced in Definition \ref{tangentsdef}.
    \end{itemize}
\end{itemize}
\end{teorema2}

The rectifiability of a measure, namely that (i) of Theorem \ref{th:eu:mm} holds, is a global property and as such it is usually very difficult to verify in applications. Rectifiability criteria serve the purpose of characterizing such global property with local ones, that are usually conditions on the \emph{density} and on the \emph{tangents} of the measure. Most of the more basic criteria impose condition (ii a) and the existence of an affine plane $V(x)$, depending only on the point $x$, on which at small scales the support of measure is squeezed on around $x$. The difference between these various elementary criteria relies on how one defines ``squeezed on'', for an example see Theorem 15.19 of \cite{Mattila1995GeometrySpaces}. However, the existence of just one plane approximating the measure at small scales may be still too difficult to prove in many applications and this is where 
Theorem \ref{th:eu:mm} comes into play. The Marstrand-Mattila rectifiability criterion says that even if we allow a priori the approximating plane to rotate at different scales, the density hypothesis (ii a) guarantees a posteriori this cannot happen almost everywhere.

It is well known that if a Carnot group $\mathbb{G}$ has Hausdorff dimension $\mathcal{Q}$, then it is $(\mathcal{Q}-1)$-purely unrectifiable in the sense of Federer, see for instance Theorem 1.2 of \cite{rigiditymagnani}. Despite this geometric irregularity, in the fundational paper \cite{Serapioni2001RectifiabilityGroup} B. Franchi, F. Serra Cassano and R. Serapioni introduced the new notion of $C^1_\mathbb{G}$-rectifiability in Carnot groups, see Definition \ref{regsur}. This definition allowed them to establish De Giorgi's rectifiability theorem for finite perimeter sets in the Heisenberg groups $\HH^n$:

\begin{teorema2}[Corollary 7.6, \cite{Serapioni2001RectifiabilityGroup}]\label{th:fps}
Suppose $\Omega\subseteq \HH^n$ is a finite perimeter set. Then its reduced boundary $\partial_\HH^* \Omega$ is $C^1_\mathbb{G}$-rectifiable.
\end{teorema2}

It is not hard to see that an open set with smooth boundary is of finite perimeter in $\HH^n$, but there are finite perimeter sets in $\HH^1$ whose boundary is a fractal from an Euclidean perspective, see for instance \cite{kirchhserra}. This means that the Euclidean and $C^1_\mathbb{G}$-rectifiability are not equivalent.

\medskip
The main goal of this paper is to establish a $1$-codimensional analogue of Theorem \ref{th:eu:mm} in Carnot groups:

\begin{teorema2}\label{main:2}
Suppose $\phi$ is a Radon measure on $\mathbb{G}$. The following are equivalent:
\begin{itemize}
    \item[(i)] $\phi$ is absolutely continuous with respect to the $(\mathcal{Q}-1)$-dimensional Hausdorff measure $\mathcal{H}^{\mathcal{Q}-1}$ and $\phi$-almost all $\mathbb{G}$ can be covered  by countably many $C^1_\mathbb{G}$-surfaces,
    \item[(ii)]$\phi$ satisfies the following two conditions for $\phi$-almost every $x\in\mathbb{G}$:
    \begin{itemize}
    \item[(a)]$0<\Theta_*^{\mathcal{Q}-1}(\phi,x)\leq \Theta^{\mathcal{Q}-1,*}(\phi,x)<\infty$,
    \item[(b)] $\Tan_{\mathcal{Q}-1}(\phi,x)$ is contained in $\mathfrak{M}$, the family of non-null Haar measures of the elements of $\text{Gr}(\mathcal{Q}-1)$, the $1$-codimensional homogeneous subgroups of $\mathbb{G}$.
    \end{itemize}
\end{itemize}
\end{teorema2}

While the fact that (i) implies (ii) follows for instance from Lemma 3.4 and Corollary 3.6 of \cite{VITTONE}, the viceversa is the subject of investigation of this work. Besides the already mentioned importance for the applications, Theorem \ref{th:eu:mm} is also relevant because it establishes that $C^1_\mathbb{G}$-rectifiability is characterized in the same way as the Euclidean one, and this is the main motivation behind the definition of $\mathscr{P}$-rectifiable measures, given in Definition \ref{def:rect}.
Our main application of Theorem \ref{main:2}, is the proof of the first extension of Preiss's rectifability theorem outside the Euclidean spaces:

\begin{teorema2}\label{main:preiss}
Suppose $\phi$ is a Radon measure on the Heisenberg group $\HH^n$ such that for $\phi$-almost every $x\in\HH^n$, we have:
$$0<\Theta^{2n+1}(\phi,x):=\lim_{r\to 0}\frac{\phi(B(x,r))}{r^{2n+1}}<\infty,$$
where $B(x,r)$ are the metric balls relative to the Koranyi metric.
Then $\phi$ is absolutely continuous with respect to $\mathcal{H}^{2n+1}$ and $\HH^n$ can be covered $\phi$-almost all with countably many $C^1_{\mathbb{\HH}^n}$-regular surfaces.
\end{teorema2}


Finally, an easy adaptation of the arguments used to prove Theorem \ref{main:2} also provides the following rectifiability criterion for finite perimeter sets in arbitrary Carnot groups. Theorem \ref{rett:perf} asserts that if the tangent measures to the boundary of a finite perimeter set are sufficiently close to vertical hyperplanes, then the boundary can be covered by countably many intrinsic Lipschitz graphs.

\begin{teorema2}\label{rett:perf}
There exists an $\varepsilon_\mathbb{G}>0$ such that if $\Omega\subseteq \mathbb{G}$ is a finite perimeter set for which:
    $$\limsup_{r\to 0} d_{x,r}(\lvert\partial \Omega\rvert_\mathbb{G},\mathfrak{M}):=\limsup_{r\to 0}\inf_{\nu\in \mathfrak{M}}\frac{W_1(\lvert\partial \Omega\rvert_\mathbb{G}\llcorner B(x,r),\nu\llcorner B(x,r))}{r^\mathcal{Q}}\leq \varepsilon_\mathbb{G},$$
for $\lvert\partial \Omega\rvert_{\mathbb{G}}$-almost every $x\in \mathbb{G}$ where $W_1$ is the $1$-Wasserstein distance, then $\mathbb{G}$ can be covered $\lvert\partial \Omega\rvert_\mathbb{G}$-almost all with countably many intrinsic Lipschitz graphs. 
\end{teorema2}

\medskip

We present here a survey on the strategy of the proof of our main result, Theorem \ref{main:2}. For the sake of discussion, let us put ourselves in a simplified situation. Assume $E$ is a compact subset of a Carnot group $\mathbb{G}=(\R^n,*)$ such that:
\begin{itemize}
    \item[\hypertarget{AD}{($\alpha$)}] there exists an $\eta_1\in\N$ such that $\eta_1^{-1}r^{\mathcal{Q}-1}\leq \mathcal{H}^{\mathcal{Q}-1}(E\cap B(x,r))\leq \eta_1 r^{\mathcal{Q}-1}$ for any $0<r<\diam(E)$ and any $x\in E$,
    \item[(\hypertarget{flat}{$\beta$})]the functions $x\mapsto d_{x,r}(\mathcal{H}^{\mathcal{Q}-1}\llcorner E,\mathfrak{M})$ converge uniformly to $0$ on $E$ as $r$ goes to $0$.
\end{itemize}

The cryptic condition (\hyperlink{flat}{$\beta$}) can be reformulated, thanks to Propositions \ref{prop1vsinfty} and \ref{prop:bil2} in the following more geometric way.
For any $\epsilon>0$ there is a $\mathfrak{r}(\epsilon)>0$ such that for $\mathcal{H}^{\mathcal{Q}-1}$-almost any $x\in E$ and any $0<\rho<\mathfrak{r}(\epsilon)$ there is a plane $V(x,\rho)\in\G(\mathcal{Q}-1)$, depending on both the point $x$ and the scale $\rho$, for which:
\begin{align} 
        E\cap B(x,\rho)\subseteq \{y\in\mathbb{G}:\dist(y,x*V(x,\rho))\leq \epsilon \rho\},\label{4.8}\\
        B(y, \epsilon \rho)\cap E\neq \emptyset \text{ for any }y\in B(x,\rho/2)\cap x*V(x,\rho).\label{4.9}
\end{align}
In Euclidean spaces if a Borel set $E$ satisfies \eqref{4.8} and \eqref{4.9} it is said \emph{weakly linear approximable}, see for instance Chapter 5 of  \cite{DeLellis2008RectifiableMeasures}. The condition \eqref{4.8} says that  at small scales $E$ is squeezed on the plane $x*V(x,\rho)$, while \eqref{4.9} implies that inside $B(x,\rho)$ any point of $x*V(x,\rho)$ is very close to $E$, see the picture below:

\begin{figure}[H]
\centering
    \includegraphics[scale=0.15
    ]{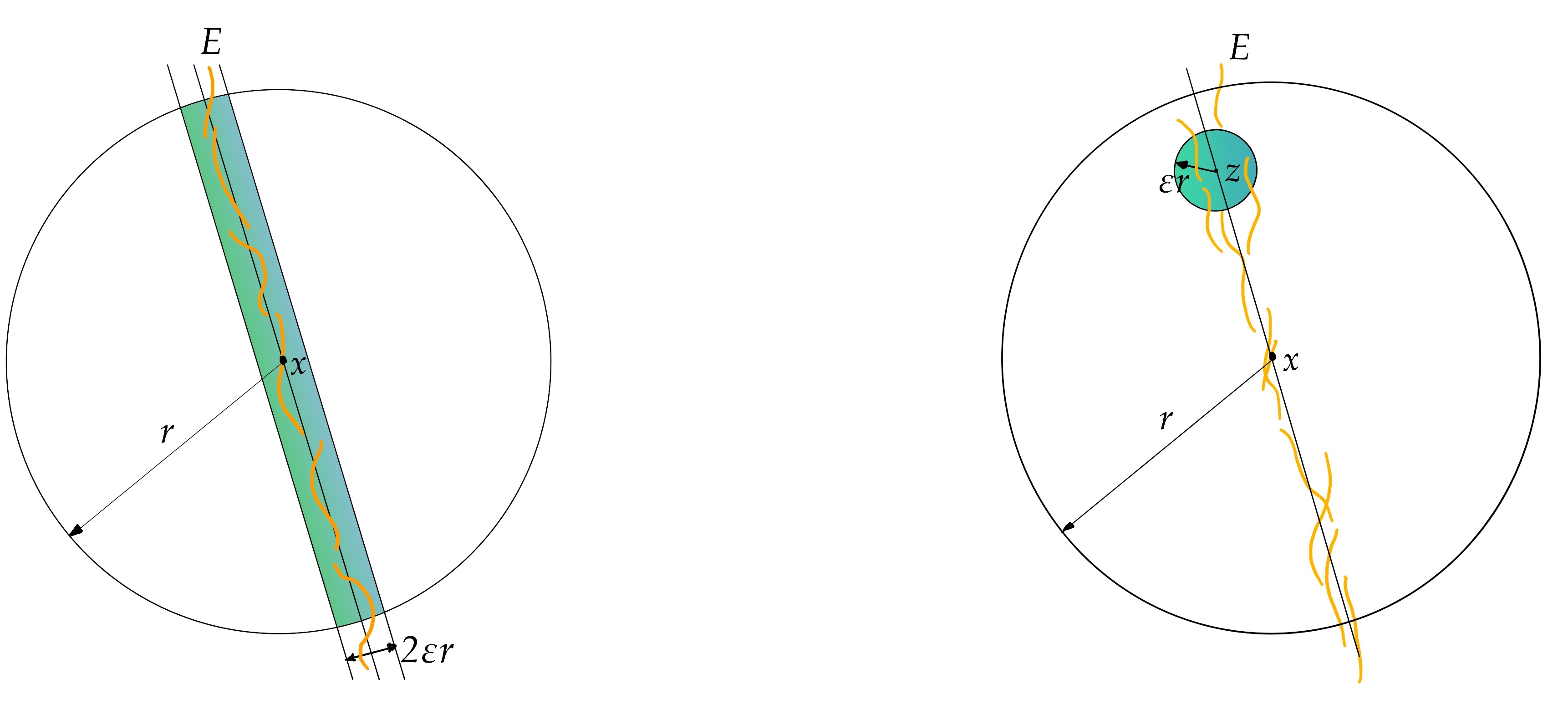}
\caption{On the left we see that \eqref{4.8} implies that at the scale $\rho$ the set $E$, in yellow, is contained in a narrow strip of size $2\epsilon \rho$ around $x*V(x,\rho)$. On the right we see that \eqref{4.9} implies that any ball centred on the plane $x*V(x,\rho)$ inside $B(x,\rho/2)$ and of radius $\epsilon \rho$ must meet $E$.}
\end{figure}
\medskip

The first step towards the proof of our main result is the following technical proposition that is a suitably reformulated version of Theorem 1.14 of \cite{DavidSemmes} and of Lemma 3.8 of \cite{CFO}. Proposition \ref{prop:uno} shows that if at some point $x$  the set $E$ has also big projections on some plane $W$, i.e. \eqref{proj:intr} holds, then around $x$ the set $E$ is almost a $W$-intrinsic Lipschitz surface. 

\begin{proposizione2}\label{prop:uno}
Let $k>10\eta_1^2$ and $\omega>0$. Suppose further that $x\in E$ and $\rho>0$ are such that:
\begin{itemize}
\item[(i)] $d_{x,k\rho}(\mathcal{H}^{\mathcal{Q}-1}\llcorner E,\mathfrak{M})\leq \omega$,
    \item[(ii)] there exists a plane $W\in\G(\mathcal{Q}-1)$  such that:
    \begin{equation}
        (\rho/k)^{\mathcal{Q}-1}\leq\mathcal{H}^{\mathcal{Q}-1}\llcorner W(P_W(B(x,\rho)\cap E)),
        \label{proj:intr}
    \end{equation}
    where $P_W$ is the splitting projection on $W$, see Definition \ref{def:splitproj}.
\end{itemize}
If $k$ is chosen suitably large and $\omega$ are suitably small, there exists an $\alpha=\alpha(\eta_1, k,\omega)>0$ with the following property.
For any $z\in E\cap B(x,\rho)$ and any $y\in B(x,k\rho/8)\cap E$ for which $10\omega \rho\leq d(z,y)\leq k\rho/2$, we have that
$y$ is contained in the cone $zC_{W}(\alpha)$, that is introduced in Definition \ref{deffi:cono}.
\end{proposizione2}

We remark that thanks to our assumption (\hyperlink{flat}{$\beta$}) on $E$, hypothesis (i) of the above proposition is satisfied almost everywhere on $E$ whenever $\rho<\tilde{\mathfrak{r}}(\omega)$, where $\tilde{\mathfrak{r}}(\omega)$ is suitably small and depends only on $\omega$.
Let us explain some of the ideas of the proof of Proposition \ref{prop:uno}. If the plane $W$ is almost orthogonal to $V(x,\rho)$, the element of $\G(\mathcal{Q}-1)$ for which \eqref{4.8} and \eqref{4.9} are satisfied by $E$ at $x$ at scale $\rho$, we would have that the projection  of $E$ on $W$ would be too small and in contradiction with \eqref{proj:intr}, see Figure \ref{fig1}.

\begin{figure}[H]
    \centering
    \includegraphics[scale=0.16
    ]{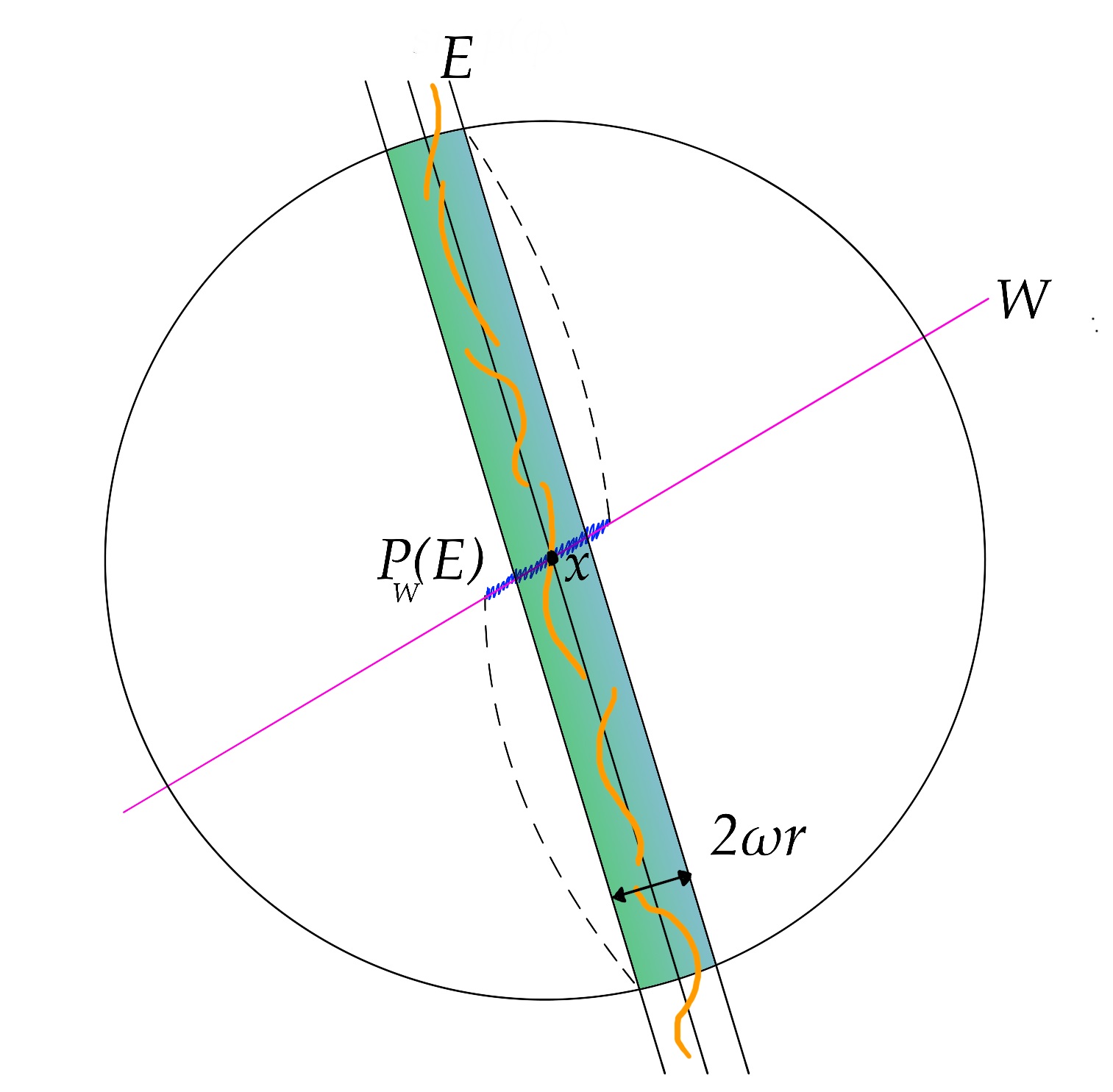}
     \caption{The weak linear approximability of $E$ implies that $E\cap B(x,\rho)$ is contained inside $V_{\omega \rho}$, an $\omega \rho$-neighbourhood of the plane $V(x,\rho)$. If $V(x,\rho)$ and $W$ are almost orthogonal, i.e. the Euclidean scalar product of their normals is very small, it can be shown that that the projection $P_W$ on $W$ of $V_{\omega \rho}\cap B(x,\rho)$ has $\mathcal{H}^{\mathcal{Q}-1}$-measure smaller than $(\omega \rho)^{\mathcal{Q}-1}$.}
     \label{fig1}
\end{figure}

It is possible to push the argument and prove that if the constants $k$ and $\omega$ are chosen suitably, the planes $V(x,\rho)$ and $W$ must be at a very small angle, and in particular inside $B(x,\rho)$ the plane $x*V(x,\rho)$ must be very close to $x*W$. So close in fact that it can be proved that $E\cap B(x,\rho)$ is contained in a $2\omega \rho$-neighbourhood $W_{2\omega \rho}$ of $W$. This implies that $z,y\in W_{2\omega \rho}$, and since $W$ and $V(x,\rho)$ are at a small angle, it is possible to show that $\dist(y,zW)\leq 4\omega \rho$. Furthermore, 
by assumption on $y,z$ we have $d(z,y)>10\omega \rho$ and thus we infer that  $\dist(y,zW)\leq 5 d(y,z)$. This implies in particular that $y\in zC_W(2/5)$.

\medskip

The second step towards the proof of the main result is to show that at any point $x$ of $E$ and for any $\rho>0$ sufficiently small there is a plane $W_{x,\rho}\in\G(\mathcal{Q}-1)$ on which $E$ has big projectons.

\begin{teorema2}\label{th:conoeu}
There is a $\eta_2\in\N$, such that for $\mathcal{H}^{\mathcal{Q}-1}$-almost every $x\in E$ and $\rho>0$ sufficiently small there is a plane $W_{x,\rho}\in\G(\mathcal{Q}-1)$ for which:
\begin{equation}
\mathcal{H}^{\mathcal{Q}-1}\big(P_{W_{x,\rho}}(E\cap B(x,\rho))\big)\geq \eta_2^{-1} \rho^{\mathcal{Q}-1}.\label{eq:bigprog2}
\end{equation}
\end{teorema2}

We now briefly explain the ideas
behind the proof of Theorem \ref{th:conoeu}, that are borrowed from Chapter 2, \S 2.1 and \S2.2 of \cite{DavidSemmes}. Fix two parameters $\eta_3\in\N$ and $\omega>0$ such that  $\omega<1/\eta_3^{\mathcal{Q}(\mathcal{Q}+1)}$ and for which:
\begin{align}
B_+:=B(\delta_{10\eta_3^{-1}}(v),\eta_3^{-1})\subseteq&  \{y\in B(0,1):\langle y, v\rangle> \omega \},\nonumber\\
B_-:=B_+*\delta_{20\eta_3^{-1}}(v^{-1})\subseteq& \{y\in B(0,1):\langle y,v\rangle< -\omega \},\nonumber
\end{align}
where $\delta_\lambda$ are the intrinsic dilations introduced in \eqref{eq:nummerigno} and $v\in V_1$ is an arbitrary vector with unitary Euclidean norm. Thanks to the assumption \eqref{4.8} on $E$, for any $0<\rho<\mathfrak{r}(\omega)$ we have that: $$E\cap B(x,\rho)\subseteq \{y\in B(x,\rho):\dist(y,x*V(x,\rho))\leq \omega \rho\}.$$
In particular thanks to the assumptions on $\eta_3$ and $\omega$ we infer that $E\cap x\delta_\rho B_+=\emptyset=E\cap x\delta_\rho B_-$. Let $W_{x,\rho}:=V(x,\rho)$ and for any $z\in x\delta_\rho B_+$ we define the curve:
$$\gamma_z(t):=z\delta_{20\eta_3^{-1}t}(\mathfrak{n}(W_{x,\rho})^{-1}),$$
as $t$ varies in $[0,1]$ and where $\mathfrak{n}(W_{x,\rho})$ is the unit Euclidean normal to $W_{x,\rho}$.
The curve $\gamma_z$ must intersect $W_{x,\rho}$ at the point $P_{W_{x,\rho}}(z)$ since $\gamma_z(1)\in x\delta_\rho B_-$ and as a consequence we have the inclusion $\gamma_{z}([0,1])\subseteq  P_{W_{x,\rho}}^{-1}(P_{W_{x,\rho}}(z))$.
Since conditions \eqref{4.8} and \eqref{4.9} heuristically say that $E$ almost coincides with the plane $x*W_{x,\rho}$ inside $B(x,\rho)$ and it has very few holes, most of the curves $\gamma_z$ should intersect the set $E$ too. 

More precisely, we prove that if some $\gamma_{z}$ does not intersect $E$, there is a small ball $U_{z}$ centred at some $q\in E$ such that $\gamma_{z}\cap U_z\neq \emptyset$. It is clear that defined the set:
\begin{equation}
    F:=E\cup \bigcup_{\substack{z\in x\delta_r B_+,\text{ } \gamma_z\cap E= \emptyset}} U_z,
    \nonumber
\end{equation}
we have $P_{W_{x,r}}(x\delta_r B_+)\subseteq P_{W_{x,r}}(F)$. 
So, intuitively speaking adding these balls $U_z$ allows us to close the holes of $E$. 
An easy computation proves that $\mathcal{H}^{\mathcal{Q}-1}(P_{W_{x,r}}(x\delta_r B_+))\geq r^{\mathcal{Q}-1}/\eta_3^{\mathcal{Q}-1}$ and thus in order to be able to conclude the proof of \eqref{eq:bigprog2} we should have some control over the size of the projection of the balls $U_z$. This control is achievable thanks to \eqref{4.9}, see Proposition \ref{lemma2.28} and Theorem \ref{TH:proiezioni}, and in particular we are able to show that: $$\mathcal{H}^{\mathcal{Q}-1}\bigg(P_{W_{x,r}}\Big(\bigcup_{\substack{z\in x\delta_r B_+\\ \gamma_z\cap E= \emptyset}} U_z\Big)\bigg)\leq \omega r^{\mathcal{Q}-1}.$$
This implies that $E$ satisfies the big projection properties, i.e. \eqref{eq:bigprog2} holds with $\eta_2:=2\eta_3^{\mathcal{Q}-1}$.  This part of the argument is rather delicate and technical. For the details we refer to the proof of Theorem \ref{TH:proiezioni}.

\medskip

The third step towards the proof of Theorem \ref{main:2} is achieved in Subsection \ref{intr:graph}, where we prove the following:

\begin{teorema2}\label{th:th:th:th}
There exists a intrinsic Lipschitz graph $\Gamma$ such that $\mathcal{H}^{\mathcal{Q}-1}(E\cap \Gamma)>0$.
\end{teorema2}

The strategy we employ to prove the above theorem is the following. We know that at $\mathcal{H}^{\mathcal{Q}-1}$-almost every point of $x\in E$ there exists a plane $W_{x,r}$ such that $\mathcal{H}^{\mathcal{Q}-1}(P_{W_{x,r}}(E\cap B(x,r)))\geq \eta_2^{-1}r^{\mathcal{Q}-1}$. For any $x\in E$ at which the previous inequality holds, we let $\mathscr{B}$ the points $y\in B(x,r)$ for which there is a scale $s\in(0,r)$ for which $W_{y,s}$ is almost orthogonal to $W_{x,r}$. Choosing the angle between  $W_{y,s}$ and $W_{x,r}$ sufficiently big it is possible to prove that the projection of $\mathscr{B}$ on $W_{x,r}$ is smaller than $\eta_2^{-1} r^{\mathcal{Q}-1}/2$. This follows from the intuitive idea that if $y\in\mathscr{B}$, the set $E\cap B(y,s)$ is contained in a narrow strip that is almost orthogonal to $W_{x,r}$ inside $B(y,s)$ and thus its projection on $W_{x,r}$ has very small $\mathcal{H}^{\mathcal{Q}-1}$-measure.
On the other hand, Proposition \ref{cor:2.2.19} tells us that $\mathcal{S}^{\mathcal{Q}-1}\llcorner V(P_{W_{x,r}}(E\cap B(x,r)\setminus \mathscr{B}))\leq 2c(V)\mathcal{S}^{\mathcal{Q}-1}(E\cap B(x,r)\setminus \mathscr{B})$, and this allows us to infer that there are many points $z\in B(x,r)\cap E$ for which $W_{z,s}$ is contained in a (potentially large) fixed cone with axis $W_{x,r}$ for any $0<s<r$. This uniformity on the scales allows us to infer thanks to Proposition \ref{prop:uno} that $E\cap B(x,r)\setminus \mathscr{B}$ is an intrinsic Lipschitz graph.

Since properties (\hyperlink{AD}{$\alpha$}) and (\hyperlink{flat}{$\beta$}) are stable for the restriction-to-a-subset operation, Theorem \ref{th:th:th:th} implies by means of a classical argument that $E$ can be covered $\mathcal{H}^{\mathcal{Q}-1}$-almost all with intrinsic Lipschitz graphs.

\medskip

Therefore, we are reduced to see how we can improve the regularity of the surfaces $\Gamma_i$ covering $E$ from intrinsic Lipschitz to $C^1_\mathbb{G}$. Since the blowups of $\mathcal{H}^{\mathcal{Q}-1}\llcorner E$ are almost everywhere flat, the locality of the tangents, i.e. Proposition \ref{prop:locality}, implies that the blowups of the measures $\mathcal{H}^{\mathcal{Q}-1}\llcorner\Gamma_i$ are flat as well, where we recall that a measure is said flat if it is the Haar measure of a $1$-codimensional homogeneous subgroup of $\mathbb{G}$.
Furthermore, since
intrinsic Lipschitz graphs can be extended to boundaries of sets of finite perimeter, see Theorem \ref{whole}, they have an associated normal vector field $\mathfrak{n}_{i}$.
Therefore, for $\mathcal{H}^{\mathcal{Q}-1}$-almost every $x\in\Gamma_i$ the elements of
$\Tan_{\mathcal{Q}-1}(\mathcal{H}^{\mathcal{Q}-1}\llcorner \Gamma_i,x)$ are also the perimeter measures of sets with constant horizontal normal $\mathfrak{n}_i(x)$, see Propositions \ref{prop:costnorm}, \ref{obs6.5}, and \ref{prop:tang}.
The above argument shows that on the one hand $\Tan_{\mathcal{Q}-1}(\mathcal{H}^{\mathcal{Q}-1}\llcorner \Gamma_i,x)$ are flat measures and on the other if seen as boundary of finite perimeter sets, they must have constant horizontal normal coinciding with $\mathfrak{n}_i(x)$ almost everywhere. Therefore, for $\mathcal{H}^{\mathcal{Q}-1}$-almost every $x\in E\cap \Gamma_i$ the set $\Tan_{\mathcal{Q}-1}(\mathcal{H}^{\mathcal{Q}-1}\llcorner \Gamma_i,x)$ must be contained in the family of Haar measures of the $1$-codimensional subgroup orthogonal to $\mathfrak{n}_i(x)$. The fact that $E\cap \Gamma_i$ is covered with countably many $C^1_\mathbb{G}$-surfaces follows by means of the rigidity of the tangents discussed above and a Whitney-type theorem, that is obtained in Appendix \ref{def:perimetro} with an adaptation of the arguments of \cite{MFSSC}.

\section*{Structure of the paper}

In Section \ref{preliminaries} we recall some well known facts about Carnot groups and Radon measures. Section \ref{sec:main} is divided in four parts. The main results of Subsection \ref{flat:1} are Propositions \ref{prop1vsinfty} and \ref{prop:bil2}, that allow us to interpret the flatness of tangents in a more geometric way. Subsection \ref{sub:project} is devoted to the proof of Proposition \ref{prop:cono}, that is roughly Theorem \ref{th:conoeu}. Subsection \ref{big:proj} is the technical core of this work and the main result proved in it is Theorem \ref{TH:proiezioni}, that codifies the fact that the flatness of tangents implies big projections on planes.
Finally, in Subsection \ref{intr:graph} we put together the results of the previous three subsections to prove Theorem \ref{th:intr:lipgraph} that asserts that for any Radon measure satisfying the hypothesis of Theorem \ref{main:2}, there is an intrinsic Lipschitz graph of positive $\phi$-measure. In Section \ref{sec:end} we prove  Theorem \ref{main} that is the main result of the paper and its consequences. In Appendix \ref{AppendiceA} we construct the dyadic cubes that are needed in Section \ref{sec:main} and in Appendix \ref{def:perimetro} we recall some well known facts about finite perimeter sets in Carnot groups and intrinsic Lipschitz graphs whose surface measures has flat tangents.

\section*{Notation}
We add below a list of frequently used notations, together with the page of their first appearance:
\medskip

\begin{longtable}{c p{0.7\textwidth} p{\textwidth}}
$\lvert\cdot\rvert$\label{euclide} & Euclidean norm, & \pageref{euclide}\\

$\lVert\cdot\rVert$ & smooth-box norm, & \pageref{smoothnorm}\\

$\langle \cdot,\cdot\rangle$ & scalar product in the Euclidean spaces, & \pageref{scalprod}\\

$V_i$ & layers of the stratification of the Lie algebra of $\mathbb{G}$, & \pageref{notiniz}\\

$n_i$ & dimension of the $i$-th layer of the Lie algebra of $\mathbb{G}$, &\pageref{notiniz}\\

$\pi_i(\cdot)$ & projections of $\R^n$ onto $V_i$, & \pageref{homog}\\

$h_i$ & the topological dimension of the vector space $V_1\oplus\ldots\oplus V_i$, &\pageref{notiniz}\\

$\mathcal{Q}$ & homogeneous dimension of the group $\mathbb{G}$, &\pageref{homog}\\

$\mathscr{Q}_i$ & coefficients of the coordinate representation of the group operation & \pageref{tran}\\

$\tau_x$ & left translation by $x$, & \pageref{tran}\\

$\delta_\lambda$ & intrinsic dilations, & \pageref{homog}\\

$U_{i}(x,r)$ &  open Euclidean ball in $V_i$ of radius $r>0$ and centre $x$, & \pageref{ballstrat}\\

$B(x,r)$ &  open ball of radius $r>0$ and centre $x$, & \pageref{smoothnorm}\\

$\overline{B(x,r)}$ &  closed ball of radius $r>0$ and centre $x$, & \pageref{smoothnorm}\\

$T_{x,r}\phi$ & dilated of a factor $r>0$ of the measure $\phi$ at the point $x\in\HH^n$, & \pageref{tang}\\

$\Tan_{m}(\phi,x)$ & set of $m$-dimensional tangent measures to the measure $\phi$ at $x$, & \pageref{tang}\\

  $\rightharpoonup$ & weak convergence of measures, &\pageref{lippi}
  \\

  $\G(m)$ & the $m$-dimensional Grassmanian, &\pageref{grass}\\

$\mathfrak{M}(m)$ & the set of the Haar measures of the elements of $\G(m)$, &\pageref{eq:flatmeasures}\\

  $\mathfrak{n}(V)$ & the normal of the plane $V\in\G(\mathcal{Q}-1)$, &\pageref{plano}\\
  
  $\mathfrak{N}(V)$ & the $1$-dimensional homogeneous subgroup generated by $\mathfrak{n}(V)$ &\pageref{scalprod}\\
  
    $\pi_{\mathfrak{N}(V)}$ & the orthogonal projection of $V_1$ onto $\mathscr{V}:=V\cap V_1$ where $V\in\G(\mathcal{Q}-1)$, &\pageref{scalprod}\\

   $\pi_{\mathscr{V}}$ & the orthogonal projection of $V_1$ onto $\mathfrak{N}(V)$, &\pageref{scalprod}\\

$P_V$ & splitting projection on the plane $V\in\G(\mathcal{Q}-1)$,& \pageref{def:splitproj}\\

$P_{\mathfrak{N}(V)}$ &splitting projection on the $1$-dimensional subgroup $\mathfrak{N}(V)$,& \pageref{def:splitproj}\\

$C_V(\alpha)$ &cone of amplitude $\alpha$ with axis $V$,& \pageref{deffi:cono}\\

$\Lambda(\alpha)$ & upper bound on $\lvert \pi_1(w)\rvert/\lVert w\rVert$ for any $x\in C_V(\alpha)$,& \pageref{prop:conifero}\\

$\lip(K)$ & non-negative $1$-Lipschitz functions with support contained in the compact set $K$. &\pageref{lippi}\\

$\mathcal{S}^{\alpha}$ & $\alpha$-dimensional spherical Hausdorff measure relative to the metric $\lVert\cdot\rVert$, & \pageref{hausdorrmeas}\\

$\mathcal{C}^{\alpha}$ & $\alpha$-dimensional centred spherical Hausdorff measure relative to the metric $\lVert\cdot\rVert$, & \pageref{hausdorrmeas}\\

  $\mathcal{H}^{k}_{eu}$ &  Euclidean $k$-dimensional Hausdorff measure, &\pageref{prop:rapp}\\
  
  $\Theta^m_*(\phi,x)$ & $m$-dimensional lower density of the measure $\phi$ at $x$, & \pageref{lippi}\\
  
  $\Theta^{m,*}(\phi,x)$ & $m$-dimensional upper density of the measure $\phi$ at $x$, & \pageref{lippi}\\
 
  $\mathfrak{c}(Q)$& centre of the cube $Q$ & \pageref{def:compactscube}\\
  
  $\alpha(Q)$& measure of the distance of the cube $Q$ from the planes $V\in\G(\mathcal{Q}-1)$& \pageref{def:compactscube}\\
  
  $d_{x,r}(\cdot,\mathfrak{M})$ & distance of the Radon measure $\phi$ inside the ball $B(x,r)$ from flat measures, &\pageref{def:metr}\\
  
  $d_H(\cdot,\cdot)$ & Hausdorff distance of sets& \pageref{def:Haus}\\
  
  $\Delta$ & dyadic cubes,&\pageref{notation1}
\end{longtable}

\medskip

Since dyadic cubes are used extensively throughout the paper we introduce some nomenclature for the relationships of two cubes. For any couple of dyadic cubes $Q_1,Q_2\in\Delta$:
\begin{itemize}
\item[(i)] if $Q_1\subseteq Q_2$, then $Q_2$ is said to be an \emph{ancestor} of $Q_1$ and $Q_2$ a \emph{sub-cube} of $Q_2$,
\item[(ii)] if $Q_2$ is the smallest cube for which $Q_1\subsetneq Q_2$, then $Q_2$ is said to be the \emph{parent} of $Q_1$ and $Q_1$ the \emph{child} of $Q_2$.
\end{itemize}

Finally, throughout the paper we should adopt the convention that the set of natural numbers $\N$ does not include the number $0$. The set of natural numbers including $0$ will be denoted by $\N_0$.

\section{Preliminaries}\label{preliminaries}

This preliminary section is divided into four subsections. In Subsections \ref{Carnotinizio} and \ref{conesplit} we introduce the setting, fix notations and prove some basic facts on splitting projections and intrinsic cones. In Subsection \ref{densityandtangents} we recall some well known facts on Radon measures and their blowups and finally in Subsection \ref{sub:rect} we introduce the two main notions of $1$-codimensional rectifiable sets available in Carnot groups.

\subsection{Carnot groups}\label{Carnotinizio}
In this subsection we briefly introduce some notations on Carnot groups that we will extensively use throughout the paper. For a detailed account on Carnot groups and sub-Riemannian geometry we refer to \cite{SCln}.

A Carnot group $\mathbb{G}$ of step $s$  is a connected and simply connected Lie group whose Lie algebra $\mathfrak g$ admits a stratification $\mathfrak g=V_1\,  \oplus \, V_2 \, \oplus \dots \oplus \, V_s$. The stratification has the further property  that the entire Lie algebra $\mathfrak g$ is generated by its first layer $V_1$, the so called \emph{horizontal layer}. We denote by $n$ the topological dimension of $\mathfrak g$, by $n_j$ the dimension of $V_j$ and by $h_j$ the number $\sum_{i=1}^j n_i$.  \label{notiniz}

Furthermore, we let $\pi_i:\mathbb{G}\to V_i$ be the projection maps on the $i$-th layer of the Lie algebra $V_i$ and denote by $U_i(a,r)$ the Euclidean ball of radius $r$ and centre $a$ inside the $i$-th layer $V_i$\label{ballstrat}. We shall remark that more often than not, we will shorten the notation to $v_i:=\pi_i v$.

The exponential map $\exp :\mathfrak g \to \mathbb{G}$ is a global diffeomorphism from $\mathfrak g$ to $\mathbb{G}$.
Hence, if we choose a basis $\{X_1,\dots , X_n\}$ of $\mathfrak g$,  any $p\in \mathbb{G}$ can be written in a unique way as $p=\exp (p_1X_1+\dots +p_nX_n)$. This means that we can identify $p\in \mathbb{G}$ with the $n$-tuple $(p_1,\dots , p_n)\in \R^n$ and the group $\mathbb{G}$ itself with $\R^n$ endowed with $*$, the operation determined by the Campbell-Hausdorff formula. From now on, we will always assume that $\mathbb{G}=(\R^n,*)$ and as a consequence, that the exponential map $\exp$ acts as the identity.

The stratificaton of $\mathfrak{g}$ carries with it a natural family of dilations $\delta_\lambda :\mathfrak{g}\to \mathfrak{g}$, that are Lie algebra automorphisms of $\mathfrak{g}$ and are defined by: 
\begin{equation}
     \delta_\lambda (v_1,\dots , v_s)=(\lambda v_1,\lambda^2 v_2,\dots , \lambda^s v_s),
     \label{eq:nummerigno}
\end{equation}
where $v_i\in V_i$. The stratification of the Lie algebra $\mathfrak{g}$  naturally induces a stratification on each of its Lie sub-algebras $\mathfrak{h}$, that is:
\begin{equation}
    \mathfrak{h}=V_1\cap \mathfrak{h}\oplus\ldots\oplus V_s\cap \mathfrak{h}.
    \label{eq:intr1}
\end{equation}
Furthermore, note that since the exponential map acts as the identity, the Lie algebra automorphisms $\delta_\lambda$ are also group automorphisms of $\mathbb{G}$.

\begin{definizione}[Homogeneous subgroups]\label{homsub}
A subgroup $V$ of $\mathbb{G}$ is said to be \emph{homogeneous} if it is a Lie subgroup of $\mathbb{G}$ that is invariant under the dilations $\delta_\lambda$ for any $\lambda>0$.
\end{definizione}

Thanks to Lie's theorem and the fact that $\exp$ acts as the identity map, homogeneous Lie subgroups of $\mathbb{G}$ are in bijective correspondence through $\exp$ with the Lie sub-algebras of $\mathfrak{g}$ that are invariant under the dilations $\delta_\lambda$. Therefore, homogeneous subgroups in $\mathbb{G}$ are identified with the Lie sub-algebras of $\mathfrak{g}$ (that in particular are vector sub-spaces of $\R^n$) that are invariant under the intrinsic dilations $\delta_\lambda$. 

For any nilpotent, Lie algebra $\mathfrak{h}$ with stratification $W_1\oplus\ldots\oplus W_{\overline{s}}$, we define its \emph{homogeneous dimension} as:
$$\text{dim}_{hom}(\mathfrak{h}):=\sum_{i=1}^{\overline{s}} i\cdot\text{dim}(W_i).$$
Thanks to \eqref{eq:intr1} we infer that, if $\mathfrak{h}$ is a Lie sub-algebra of $\mathfrak{g}$, we have $\text{dim}_{hom}(\mathfrak{h}):=\sum_{i=1}^{s} i\cdot\text{dim}(\mathfrak{h}\cap V_i)$. It is a classical fact that the Hausdorff dimension (for a definition of Hausdorff dimension see for instance Definition 4.8 in \cite{Mattila1995GeometrySpaces}) of a nilpotent, connected and simply connected Lie group coincides with the homogeneous dimension $\text{dim}_{hom}(\mathfrak{h})$ of its Lie algebra. Therefore, the above discussion implies that if $\mathfrak{h}$ is a vector space of $\R^n$ which is also an $\alpha$-dimesional homogeneuous subgroup of $\mathbb{G}$, we have:
\begin{equation}
    \alpha=\sum_{i=1}^{s} i\cdot\text{dim}(\mathfrak{h}\cap V_i)=\text{dim}_{hom} (\mathfrak{h}).\label{eq:intr2}
\end{equation}

\begin{definizione}
Let $\mathcal{Q}:=\text{dim}_{hom}(\mathfrak{g})$\label{homog} and for any $m\in\{1,\ldots,\mathcal{Q}-1\}$ we define the $m$-dimensional Grassmanian of $\mathbb{G}$, denoted by $\G(m)$, as the family of all homogeneous subgroups $V$ of $\mathbb{G}$ of Hausdorff dimension $m$.
\label{grass}
\end{definizione}
Furthermore, thanks to \eqref{eq:intr2} and some easy algebraic considerations that we omit, one deduces that for the elements of  $\G(\mathcal{Q}-1)$ the following identities hold:
\begin{equation}
    \dim(V\cap V_1)=n_1-1\qquad\text{and}\qquad\dim(V\cap V_i)=\dim(V_i)\text{ for any }i=2,\ldots,s.\label{eq:intr3}
\end{equation}
Thanks to \eqref{eq:intr3}, we infer that for any $V\in\G(\mathcal{Q}-1)$ there exists a $\mathfrak{n}(V)\in V_1$ such that:
$$V=\mathscr{V}\oplus V_2\oplus\ldots\oplus V_s,$$
where $\mathscr{V}:=\{w\in V_1:\langle \mathfrak{n}(V),w\rangle=0\}$.\label{scalprod}
In the following we will denote with $\mathfrak{N}(V)$ the $1$-dimensional homogeneous subgroup generated by the horizontal vector $\mathfrak{n}(V)$\label{plano}.
We shall remark that the above discussion implies that the elements of $\G(\mathcal{Q}-1)$ are hyperplanes in $\R^n$ whose normal lies in $V_1$. It is not hard to see that the viceversa holds too and that the elements of $\G(\mathcal{Q}-1)$ are normal subgroups of $\mathbb{G}$.

For any $p\in \mathbb{G}$, we define the left translation $\tau _p:\mathbb{G} \to \mathbb{G}$ as:\label{tran}
\begin{equation*}
q \mapsto \tau _p q := p* q.
\end{equation*}
As already remarked above, we assume without loss of generality that the group operation $*$ is determined by the Campbell-Hausdorff formula, and therefore it has the form:
\begin{equation*}
p* q= p+q+\mathscr{Q}(p,q) \quad \mbox{for all }\, p,q \in  \R^n,
\end{equation*} 
where $\mathscr{Q}=(\mathscr{Q}_1,\dots , \mathscr{Q}_s):\R^n\times \R^n \to V_1\oplus\ldots\oplus V_s$, and the $\mathscr{Q}_i$s have the following properties. For any $i=1,\ldots s$ and any $p,q\in \mathbb{G}$ we have:
\begin{itemize}
    \item[(i)]$\mathscr{Q}_i(\delta_\lambda p,\delta_\lambda q)=\lambda^i\mathscr{Q}_i(p,q)$,
    \item[(ii)] $\mathscr{Q}_i(p,q)=-\mathscr{Q}_i(-q,-p)$,
    \item[(iii)] $\mathscr{Q}_1=0$ and $\mathscr{Q}_i(p,q)=\mathscr{Q}_i(p_1,\ldots,p_{i-1},q_1,\ldots,q_{i-1})$.
\end{itemize}
Thus, we can represent the product $*$ more precisely as:
\begin{equation}\label{opgr}
p* q= (p_1+q_1,p_2+q_2+\mathscr{Q}_2(p_1,q_1),\dots ,p_s +q_s+\mathscr{Q}_s (p_1,\dots , p_{s-1} ,q_1,\dots ,q_{s-1})). 
\end{equation}

\begin{definizione}
A metric $d:\mathbb{G}\times \mathbb{G}\to \R$ is said to be homogeneous and left invariant if for any $x,y\in \mathbb{G}$ we have:
\begin{itemize}
    \item[(i)] $d(\delta_\lambda x,\delta_\lambda y)=\lambda d(x,y)$ for any $\lambda>0$,
    \item[(ii)] $d(\tau_z x,\tau_z y)=d(x,y)$ for any $z\in \mathbb{G}$.
\end{itemize}
\end{definizione}

Throughout the paper we will always endow, if not otherwise stated, the group $\mathbb{G}$ with the following homogeneous and left invariant metric:

\begin{definizione}\label{smoothnorm}
For any $g\in \mathbb{G}$, we let:
$$\lVert g\rVert:=\max\{\epsilon_1\lvert g_1\rvert,\epsilon_2\lvert g_2\rvert^{1/2},\ldots, \epsilon_s\lvert g_s\rvert^{1/s}\},$$
where $\epsilon_1=1$ and $\epsilon_2,\ldots \epsilon_s$ are suitably small parameters depending only on the group $\mathbb{G}$. For the proof that $\lVert\cdot\rVert$ is a left invariant, homogeneous norm on $\mathbb{G}$ for a suitable choice of $\epsilon_2,\ldots, \epsilon_s$, we refer to Section 5 of \cite{step2}. Furthermore, we define:
$$d(x,y):=\lVert x^{-1}*y\rVert,$$
and let $B(x,r):=\{z\in \mathbb{G}:d(x,z)<r\}$ be the open metric ball relative to the distance $d$ centred at $x$ at radius $r>0$.
\end{definizione}

\begin{osservazione}\label{box}
Fix an othonormal basis $\mathcal{E}:=\{e_1,\ldots,e_n\}$ of $\R^n$ such that:
\begin{equation}
    e_j\in V_i, \text{ whenever }h_{i}\leq j<h_{i+1}.
    \label{eq:orthbasis}
\end{equation}
From the definition of the metric $d$, it immediately follows that the ball $B(0,r)$ is contained in the box:
\begin{equation}
\text{Box}_{\mathcal{E}}(0,r):=\big\{p\in \R^n: \text{ for any }i=1,\ldots,s \text{ whenever }\lvert \langle p,e_j\rangle\rvert\leq r^i/\epsilon_i \text{ for any }h_i\leq j<h_{j+1}\big\}.
\nonumber
\end{equation}
\end{osservazione}

\begin{definizione}\label{hausdorrmeas}
Let $A\subseteq \HH^n$ be a Borel set. For any $0\leq \alpha\leq 2n+2$ and  $\delta>0$, define:
\begin{equation}
\begin{split}
    &\mathscr{C}^{\alpha}_{\delta}(A):=\inf\Bigg\{\sum_{j=1}^\infty r_j^\alpha:A\subseteq \bigcup_{j=1}^\infty \overline{B_{r_j}(x_j)},\text{ }r_j\leq\delta\text{ and }x_j\in A\Bigg\},\\
    &\mathscr{S}^{\alpha}_{\delta}(A):=\inf\Bigg\{\sum_{j=1}^\infty r_j^\alpha:A\subseteq \bigcup_{j=1}^\infty \overline{B_{r_j}(x_j)},\text{ }r_j\leq\delta\Bigg\},
    \nonumber
\end{split}    
\end{equation}
and $\mathscr{S}^{\alpha}_{\delta,E}(\emptyset):=0=:\mathscr{C}^{\alpha}_{\delta}(\emptyset)$. Eventually, we let:
\begin{equation}
    \begin{split}
  \mathcal{C}^{\alpha}(A):=\sup_{B\subseteq A}\sup_{\delta>0}\mathscr{C}^{\alpha}_{\delta}(B)&\qquad\text{be the centred spherical Hausdorff measure},\\
   \mathcal{S}^{\alpha}(A):=\sup_{\delta>0}\mathscr{S}^{\alpha}_{\delta}(A)&\qquad\text{be the spherical Hausdorff measure}.
   \nonumber
    \end{split}
\end{equation}
Both $\mathcal{C}^\alpha$ and $\mathcal{S}^\alpha$ are Borel regular measures, see \cite{Schechter-2000} and Section 2.2 of \cite{Federer1996GeometricTheory} respectively. 
\end{definizione}

In the following definition, we introduce a family of measures that will be of great relevance throughout the paper.

\begin{definizione}[Flat measures]\label{def:flatmeasures}
For any $m\in\{1,\ldots,\mathcal{Q}-1\}$ we define the family of $m$-dimensional flat measures $\mathfrak{M}(m)$ as:
\begin{equation}
    \mathfrak{M}(m):=\{\lambda\mathcal{S}^{m}\llcorner V:\text{ for some }\lambda>0 \text{ and }V\in\G(m)\}.
    \label{eq:flatmeasures}
\end{equation}
In order to simplify notation in the following we let $\mathfrak{M}:=\mathfrak{M}(\mathcal{Q}-1)$.
\end{definizione}

The following proposition gives a representation of $(\mathcal{Q}-1)$-flat measures, that will come in handy later on:

\begin{proposizione}\label{prop:rapp}
For any $V\in \G(\mathcal{Q}-1)$ we have:
$$\mathcal{S}^{\mathcal{Q}-1}\llcorner V=\beta^{-1}\mathcal{H}^{n-1}_{eu}\llcorner V,$$
where $\beta:=\mathcal{H}^{n-1}_{eu}(B(0,1)\cap V)$ and $\beta$ does not dopend on $V$.
\end{proposizione}

\begin{proof}
Let $E:=\{z\in\mathbb{G}:\langle z_1, \mathfrak{n}(V)\rangle<0\}$ and let $\partial E$ be the perimeter measure of $E$, see Definition \ref{def:norm}. Thanks to identity (2.8) in \cite{ambled}, it can be proven that:
$$\partial E=\mathfrak{n}(V)\mathcal{H}_{eu}^{n-1}\llcorner V.$$
On the other hand, since the reduced boundary $\partial^*E=V$ of $E$ is a $C^1_\mathbb{G}$-surface, see Definition \ref{regsur}, thanks to Theorem 4.1 of \cite{magnanishape} we conclude that:
$$\beta(\lVert\cdot\rVert, \mathfrak{n}(V)) \mathcal{S}^{\mathcal{Q}-1}\llcorner V=\lvert \partial E\rvert_{\mathbb{G}}=\mathcal{H}_{eu}^{n-1}\llcorner V,$$
where $\beta(\lVert\cdot\rVert, \mathfrak{n}(V)):=\max_{z\in B(0,1)} \mathcal{H}^{n-1}_{eu}(B(z,1)\cap V)$. Furthermore since $B(0,1)$ is convex, Theorem 5.2 of \cite{magnanishape} implies that:
$$\beta(\lVert\cdot\rVert,\mathfrak{n}(V))=\mathcal{H}^{n-1}_{eu}(B(0,1)\cap V).$$
Finally note that the right hand side of the above identity does not depend on $V$ since $B(0,1)$ is invariant under rotations of the first layer $V_1$.
\end{proof}

Proposition \ref{prop:rapp} has the following useful consequence:

\begin{proposizione}
\label{unif}
A function $\varphi:\mathbb{G}\to\R$ is said to be radially symmetric if there is a profile function $g:[0,\infty)\to\R$ such that $\varphi(x)=g(\lVert x\rVert)$.
For any $V\in \G(\mathcal{Q}-1)$ and any radially symmetric, positive function $\varphi$ we have:
$$\int \varphi d\mathcal{S}^{\mathcal{Q}-1}\llcorner V=(\mathcal{Q}-1)\int s^{\mathcal{Q}-2}g(s)ds.$$  
\end{proposizione}

\begin{proof}
It suffices to prove the proposition for positive simple functions, since the general result follows by Beppo Levi's convergence theorem. Thus, let $\varphi(z):=\sum_{i=1}^N a_i \chi_{B(0,r_i)}(z)$ and note that for any $V\in\G(\mathcal{Q}-1)$ we have:
\begin{equation}
\begin{split}
        \int \varphi(z)d\mathcal{S}^{\mathcal{Q}-1}\llcorner V=&\sum_{i=1}^N a_i\mathcal{S}^{\mathcal{Q}-1}\llcorner V(B(0,r_i))=\beta^{-1}\sum_{i=1}^N a_i\mathcal{H}_{eu}^{n-1}\llcorner V(B(0,r_i))=\beta^{-1}\mathcal{H}_{eu}^{n-1}\llcorner V(B(0,1))\sum_{i=1}^N a_ir_i^{\mathcal{Q}-1}\\
        =&(\mathcal{Q}-1)\sum_{i=1}^N a_i\int_0^{r_i}s^{\mathcal{Q}-2}ds=(\mathcal{Q}-1)\int\sum_{i=1}^N a_i s^{\mathcal{Q}-2}\chi_{[0,r_i]}(s)ds
        =(\mathcal{Q}-1)\int s^{\mathcal{Q}-2}g(s)ds.
        \nonumber
\end{split}
\end{equation}
\end{proof}

\subsection{Cones and splitting projections}
\label{conesplit}
For any $V\in \G(\mathcal{Q}-1)$, the group $\mathbb{G}$ can be written as a semi-direct product of $V$ and $\mathfrak{N}(V)$, i.e.: \begin{equation}
    \mathbb{G}=V\rtimes \mathfrak{N}(V).
    \label{eq:intr10}
\end{equation}
In this subsection we specialize some of the results on projections of Section 2.2 of \cite{MFSSC} to the case in which splitting of $\mathbb{G}$ is given by \eqref{eq:intr10}.

\begin{definizione}[Splitting projections]\label{def:splitproj} For any $g\in\mathbb{G}$, there are two unique elements $P_Vg\in V$ and $P_{\mathfrak{N}(V)}g\in \mathfrak{N}(V)$ such that:
$$g=P_Vg*P_{\mathfrak{N}(V)}g.$$
\end{definizione}

The following result is a particular case of Proposition 2.2.16 of \cite{MFSSC}.

\begin{proposizione}\label{prop:proiezioni}
For any $V\in\G(\mathcal{Q}-1)$, we let:
\begin{equation}
    \begin{split}
        A_2 g_2:=&g_2-\mathscr{Q}_2\big(\pi_\mathscr{V}g_1, \pi_{\mathfrak{n}(V)}g_1\big),\\
        A_i g_i:=&g_i-\mathscr{Q}_i\big(\pi_\mathscr{V}g_1,A_2g_2,\ldots,A_{i-1}g_{i-1},  \pi_{\mathfrak{n}(V)}g_1,0,\ldots,0\big)\text{ for any }i=3,\ldots,s,
    \end{split}
    \nonumber
\end{equation}
where $\pi_{\mathfrak{n}(V)}g_1:=\langle g_1,\mathfrak{n}(V)\rangle\mathfrak{n}(V)$ and $\pi_{\mathscr{V}}g_1=g_1-\pi_{\mathfrak{n}(V)}g_1$. With these definitions, the projections $P_V$ and $P_{\mathfrak{N}(V)}$ have the following expressions in coordinates:
\begin{equation}
\begin{split}
P_Vg=\Big(\pi_\mathscr{V}g_1,A_2g_2,\ldots,A_sg_s\Big),\qquad\text{and}\qquad
    P_{\mathfrak{N}(V)}g=\Big(\pi_{\mathfrak{n}(V)}g_1,0,\ldots,0\Big).
\end{split}
\nonumber
\end{equation}
Furthermore, for any $x,y\in \mathbb{G}$ the above representations and the fact that $V\in\G(\mathcal{Q}-1)$ is normal imply:
\begin{itemize}
    \item[(i)]$P_{V}(x*y)=x * P_Vy * P_{\mathfrak{N}(V)}x^{-1}$,
    \item[(ii)] $P_{\mathfrak{N}(V)}(x*y)=P_{\mathfrak{N}(V)}(x)*P_{\mathfrak{N}(V)}(y)=P_{\mathfrak{N}(V)}(x)+P_{\mathfrak{N}(V)}(y)$,
\end{itemize}
where here the symbol $+$ has to be intended as the sum of vectors.
\end{proposizione}

\begin{osservazione}
Throughout the paper the reader should always keep in mind that the projections $P_V$ are \emph{not} Lipschitz maps and that this is the major source of the technical problems we have to overcome in order to prove our main result, Theorem \ref{main}.
\end{osservazione}

The splitting projections allows us to give the following intrinsic notion of cone.

\begin{definizione}\label{deffi:cono}
For any  $\alpha>0$ and $V\in \G(\mathcal{Q}-1)$, we define the cone $C_V(\alpha)$ as:
$$C_V(\alpha):=\{w\in\mathbb{G}:\lVert P_{\mathfrak{N}(V)}(w)\rVert\leq \alpha\lVert P_{V}(w)\rVert\}.$$
\end{definizione}

The following result is very useful, since one of the major difficulties when dealing with geometric problems in Carnot groups is that $d(x,y)\approx \lvert x-y\rvert^{1/s}$ if $x$ and $y$ are not suitably chosen. However, Proposition \ref{prop:conifero} shows that if $y\not\in xC_V(\alpha)$, then $d(x,y)$ is bi-Lipschitz equivalent to the Euclidean distance $\lvert x-y\rvert$.

\begin{proposizione}\label{prop:conifero}
Suppose $x,y\in \mathbb{G}$ are such that $x^{-1}y\not\in C_V(\alpha)$ for some $\alpha>0$ and $V\in\G(\mathcal{Q}-1)$. Then, there is a decreasing function $\alpha\mapsto\Lambda(\alpha)$ for which:
$$d(x,y)\leq\Lambda(\alpha)\lvert \pi_1(x^{-1}y)\rvert.$$
\end{proposizione}

\begin{proof}
For the sake of notation, we define 
$v:=x^{-1}y$ and claim that:
\begin{equation}
    \lvert v_i\rvert\leq \Lambda_i\lvert v_1\rvert^i,\text{ for any }i=2,\ldots,s,
    \label{eq:205}
\end{equation}
for some constant $\Lambda_i$ depending only on $\alpha$ and $i$. Before proving \eqref{eq:205} by induction, note that thanks to Proposition \ref{prop:proiezioni}, the definition of the distance $d$ and the fact that $v\not\in C_V(\alpha)$ we have:
\begin{equation}
    \epsilon_{i}\lvert A_i v_i\rvert^{1/i}\leq \alpha^{-1}\lvert \langle v_1,\mathfrak{n}(V)\rangle\rvert, \text{ for any }i=2,\ldots,s.
    \label{eq:m1}
\end{equation}
If $i=2$ thanks to the omogeneity of $\mathscr{Q}_2(\cdot,\cdot)$ we have:
\begin{equation}
\begin{split}
    \lvert\mathscr{Q}_2\big(\pi_\mathscr{V}v_1, \pi_{\mathfrak{n}(W)}v_1\big)\rvert=&\lvert v_1\rvert^2\lvert\mathscr{Q}_2\big(\delta_{\lvert v_1\rvert^{-1}}(\pi_\mathscr{W}v_1), \delta_{\lvert v_1\rvert^{-1}}(\pi_{\mathfrak{n}(W)}v_1)\big)\rvert\\
    =&\lvert v_1\rvert^2\Big\lvert\mathscr{Q}_2\bigg(\frac{\pi_\mathscr{W}v_1}{\lvert v_1\rvert}, \frac{\pi_{\mathfrak{n}(W)}v_1}{\lvert v_1\rvert}\bigg)\Big\rvert\leq \lvert v_1\rvert^2\sup_{a_1,b_1\in U_{1}(0,1)}\mathscr{Q}(a_1,b_1)=: \lambda_2\lvert v_1\rvert^2.
    \label{eq:2006}
\end{split}
\end{equation}
Putting \eqref{eq:m1} and \eqref{eq:2006} together we have:
$$\epsilon_2^2(\lvert v_2\rvert-\lambda_2\lvert v_1\rvert^2)\leq\epsilon_2^2\lvert v_2-\mathscr{Q}_2\big(\pi_\mathscr{W}v_1, \pi_{\mathfrak{n}(W)}v_1\big)\rvert=\epsilon_2^2\lvert A_2v_2\rvert\leq \alpha^{-2}\lvert \langle v_1,\mathfrak{n}(V)\rangle\rvert^2\leq \alpha^{-2}\lvert v_1\rvert^2.$$
Therefore, for $i=2$, inequality \eqref{eq:205} holds with the choice $\Lambda_2(\alpha):=\lambda_2+(1/\alpha\epsilon_2)^2$. Note that the function $\Lambda_2$ is decreasing in $\alpha$.

Suppose now that for any $i=2,\ldots,k$ inequality \eqref{eq:205} holds and that the functions $\Lambda_i(\alpha)$ are decreasing in $\alpha$. Then, thanks to the homogeneity of $\mathscr{Q}_{k+1}$, we have:
\begin{equation}
    \begin{split}
        \lvert\mathscr{Q}_{k+1}\big(\pi_\mathscr{W}v_1,&A_2v_2,\ldots,A_{k}v_{k},  \pi_{\mathfrak{n}(W)}v_1,0,\ldots,0\big)\rvert\\
        =&\lvert v_1\rvert^{k+1} \lvert\mathscr{Q}_{k+1}\big(\delta_{\lvert v_1\rvert^{-1}}(\pi_\mathscr{W}v_1),\delta_{\lvert v_1\rvert^{-1}}(A_2v_2),\ldots,\delta_{\lvert v_1\rvert^{-1}}(A_{k}v_{k}),  \delta_{\lvert v_1\rvert^{-1}}(\pi_{\mathfrak{n}(W)}v_1),0,\ldots,0\big)\rvert\\
        \leq&\lvert v_1\rvert^{k+1}\sup_{\substack{a_1,b_1\in U_{1}(0,1)\\a_i\in U_{i}(0,\Lambda_i),~i=2,\ldots,k.}} \lvert\mathscr{Q}_i(a_1,\ldots,a_k,b_1,0,\ldots,0)\rvert=:\lambda_{k+1}\lvert v_1\rvert^{k+1}.
    \end{split}
    \label{eq:206}
\end{equation}
Note that $\alpha\mapsto\lambda_{k+1}$ is decreasing in $\alpha$ and that the bounds \eqref{eq:m1} and \eqref{eq:206} imply:
\begin{equation}
\begin{split}
\epsilon_{k+1}^{k+1}(\lvert v_{k+1}\rvert-\lambda_{k+1}\lvert v_1\rvert^{k+1})& \leq
 \epsilon_{k+1}^{k+1}(\lvert v_{k+1}\rvert-\lvert\mathscr{Q}_{k+1}\big(\pi_\mathscr{W}v_1,A_2v_2,\ldots,A_{k}v_{k},  \pi_{\mathfrak{n}(W)}v_1,0,\ldots,0\big)\rvert)\\&\leq\epsilon_{k+1}^{k+1}\lvert A_{k+1}v_{k+1}\rvert\leq \alpha^{-(k+1)}\lvert v_1\rvert^{k+1}.\nonumber
\end{split}
\end{equation}
Defined $\Lambda_{k+1}(\alpha):=\lambda_{k+1}+(1/\alpha\epsilon_{k+1})^{k+1}$, inequality \eqref{eq:205} is proved. Finally we infer that:
\begin{equation}
    d(x,y)=\max\{\lvert v_1\rvert,\epsilon_2\lvert v_2\rvert^{1/2},\ldots,\epsilon_{k+1}\lvert v_s\rvert^{1/s}\}\leq \sum_{i=1}^s\Lambda_i(\alpha)^{1/i}\lvert v_1\rvert=:\Lambda(\alpha) \lvert v_1\rvert=\Lambda(\alpha)\lvert \pi_1(x^{-1}y)\rvert.
    \nonumber
\end{equation}
The fact that $\Lambda(\alpha)$ is a decreasing function, is easily seen from the fact that the $\Lambda_i(\alpha)$s are decreasing.
\end{proof}

The following proposition allows us to exactly estimate the distance of a point $g\in\mathbb{G}$ from a plane $V\in\G(\mathcal{Q}-1)$ and gives us a bound on the Lipschitz constant of $P_V$ at the origin.

\begin{proposizione}\label{cor:SSCM2.2.14}
For any $V\in\G(\mathcal
{Q}-1)$ and any
$g\in\mathbb{G}$ we have:
$\dist(P_{\mathfrak{N}(V)} g,V)=\lvert \pi_{\mathfrak{n}(V)} g_1\rvert$,
and in particular $\dist(g,V)=\lvert \pi_{\mathfrak{n}(V)} g_1\rvert$. Furthermore, there is a constant $\newC\label{C:0}>0$ depending only on $\mathbb{G}$ such that:
$$\lVert P_V(g)\rVert\leq \oldC{C:0}\lVert g\rVert.$$
\end{proposizione}

\begin{proof}
For any $v\in V$ and $g\in\mathbb{G}$, thanks to Proposition \ref{prop:proiezioni} and identity \eqref{opgr}, we have:
$$P_{\mathfrak{N}(V)}(g)^{-1}*v=\big(v_1-\pi_{\mathfrak{n}(V)}g_1,v_2+\mathscr{Q}_2\big(-\pi_{\mathfrak{n}(V)}g_1,v_1\big),\ldots,v_s+\mathscr{Q}_s\big(-\pi_{\mathfrak{n}(V)}g_1,0,\ldots,0,v_1\ldots,v_s\big)\big).$$
Since by assumption $\langle v_1,\mathfrak{n}(V)\rangle=0$, in order to minimize $\lVert P_{\mathfrak{N}(V)}(g)^{-1}*v \rVert$ as $v$ varies in $V$, we choose:
\begin{equation}
\begin{split}
v_1^*:=&0,\\
    v_2^*:=&-\mathscr{Q}_2\big(-\pi_{\mathfrak{n}(V)}g_1,0\big),\\
v_3^*:=&-\mathscr{Q}_3\big(-\pi_{\mathfrak{n}(V)}g_1,0,0,v_2^*\big),\\
\vdots\\
v_s^*:=&-\mathscr{Q}_s\big(-\pi_{\mathfrak{n}(V)}g_1,0,\ldots,0,0,v_1^*,\ldots,v_{s-1}^*\big).
\end{split}
\end{equation}
Let us prove that $v^*$ is an element of $V$ and that it is of minimal distance for $P_{\mathfrak{N}(V)}(g)$ on $V$. In order to show this, let us note that thanks to the definition of $d$, we have:
\begin{equation}
    \dist(P_{\mathfrak{N}(V)}(g),V)=\inf_{v\in V}\lVert P_{\mathfrak{N}(V)}(g)^{-1}*v\rVert\geq \inf_{v\in V}\lvert-\pi_{\mathfrak{n}(V)}g_1+v_1\rvert=\lvert\pi_{\mathfrak{n}(V)}g_1\rvert,
    \label{eq:n100}
\end{equation}
Furthermore, thanks to the definition of the operation and few algebraic computations that we omit, we have:
\begin{equation}
    P_{\mathfrak{N}(V)}(g)^{-1}*v^*=(-\pi_{\mathfrak{n}(V)}g_1,0,\ldots,0).
    \label{eq:nummm101}
\end{equation}
Therefore, since $v^*\in V$ thanks to \eqref{eq:n100} and \eqref{eq:nummm101} the first part of the proposition is proved.
The second part follows since:
$$\dist(g,V)=\inf_{v\in V}d(g,v)=\inf_{v\in V}d(P_V g*P_{\mathfrak{N}(V)}g,v)=\inf_{v\in V}d(P_{\mathfrak{N}(V)}g,P_V g^{-1}*v)=\dist(P_{\mathfrak{N}(V)}g,V),$$
where the last identity comes from the fact that the translation by $P_Vg^{-1}$ is surjective on $V$.

We proceed with the proof of the last claim of the proposition, estimating the norm of each component of $P_V(g)$. By definition of $\pi_\mathscr{V}$ and $\pi_{\mathfrak{n}(V)}$ we have:
\begin{equation}
    \lvert\pi_{\mathfrak{n}(V)}(g_1)\rvert\leq \lvert g_1\rvert\leq\lVert g\rVert\qquad\text{and}\qquad
        \lvert\pi_\mathscr{V}(g_1)\rvert\leq \lvert g_1\rvert\leq \lVert g\rVert.
        \nonumber
\end{equation}
In order to estimate the norm of $\pi_2(P_V g),\ldots, \pi_s(P_Vg)$ we proceed by induction and make use of their representations yielded by Proposition \ref{prop:proiezioni}. The base case is the estimate of the norm of  $\pi_2(P_Vg)$:
\begin{equation}
    \begin{split}
        \lvert\pi_2 (P_V g)\rvert=\lvert A_2g_2\rvert&\leq \lvert g_2\rvert+\lvert\mathscr{Q}_2\big(\pi_\mathscr{V}g_1, \pi_{\mathfrak{n}(V)}g_1\big)\rvert\leq \lVert g\rVert^2\Big(\epsilon_2^{-1}+\Big\lvert\mathscr{Q}_2\Big(\frac{\pi_\mathscr{V}g_1}{\lVert g\rVert}, \frac{\pi_{\mathfrak{n}(V)}g_1}{\lVert g\rVert}\Big)\Big\rvert\Big)\leq (\epsilon_2^{-1}+\lambda_2)\lVert g\rVert^2,
    \end{split}
    \nonumber
\end{equation}
where $\lambda_2$ is the constant introduced in \eqref{eq:2006}. Note that the constant $\mathfrak{c}_2:=\epsilon_2^{-1}+\lambda_2$ depends only on $\mathbb{G}$.

Suppose now that for any $i=2,\ldots, k$, there is a constant $\mathfrak{c}_i$, depending only on $\mathbb{G}$, such that $\lvert \pi_i(P_V g)\rvert\leq \mathfrak{c}_i\lVert g\rVert^i$. This implies that:
\begin{equation}
\begin{split}
     \lvert \pi_{k+1}(P_V g)\rvert= &\lvert A_{k+1}g_{k+1}\rvert\leq \lvert g_{k+1}\rvert+\lvert \mathscr{Q}_k(\pi_\mathscr{V}g_1, A_2 g_2,\ldots, A_k g_k \pi_{\mathfrak{n}(V)}g_1,0,\ldots,0)\rvert\\
     \leq &\lVert g\rVert^{k+1}\Big(\epsilon_{k+1}^{-1}+\Big\lvert \mathscr{Q}_k\Big(\frac{\pi_\mathscr{V}g_1}{\lVert g\rVert}, \frac{A_2 g_2}{\lVert g\rVert^2},\ldots, \frac{A_k g_k}{\lVert g\rVert^k}, \frac{\pi_{\mathfrak{n}(V)}g_1}{\lVert g\rVert},0,\ldots,0\Big)\Big\rvert\Big)\\
     \leq&   \sup_{\substack{a_1,b_1\in U_{1}(0,1)\\a_i\in U_{i}(0,\mathfrak{c}_i),~i=2,\ldots,k.}} \lvert\mathscr{Q}_i(a_1,\ldots,a_k,b_1,0,\ldots,0)\rvert\cdot \lVert g\rVert^{k+1} =:\mathfrak{c}_{k+1}\lVert g\rVert^{k+1}.
\end{split}
\nonumber
\end{equation}
Note that since the $\mathfrak{c}_i$ depended only on $\mathbb{G}$, so does $\mathfrak{c}_{k+1}$.
Thanks to the definition of $\lVert\cdot\rVert$, we eventually deduce that:
$$\lVert P_Vg\rVert\leq \sum_{i=1}^s\epsilon_i\mathfrak{c}_i^{1/i}\lVert g\rVert=:\oldC{C:0}\lVert g\rVert,$$
which concludes the proof of the proposition.
\end{proof}

The following result is the analogue of Proposition 2.2.11 of \cite{MFSSC}, where $\mathbb{M}:=V$ and $\mathbb{H}:=\mathfrak{N}(V)$.

\begin{proposizione}\label{prop:2.2.11}
For any $V\in\G(\mathcal{Q}-1)$ and any $g\in\mathbb{G}$, defined $\newC\label{C:split}:=(\oldC{C:0}+1)^{-1}$ we have:
\begin{equation}
    \oldC{C:split}(\lVert P_{\mathfrak{N}(V)}g\rVert+\lVert P_{V}g\rVert)\leq \lVert g\rVert\leq \lVert P_{\mathfrak{N}(V)}g\rVert+\lVert P_{V}g\rVert.
    \label{eq:n1}
\end{equation}
\end{proposizione}

\begin{proof}
The right hand side of \eqref{eq:n1} follows directly from the triangular inequality. Furthermore,  thanks to Proposition \ref{cor:SSCM2.2.14} we deduce that $\lVert P_{\mathfrak{N}(V)}(g)\rVert=\lvert \pi_{\mathfrak{n}(V)}(g_1)\rvert\leq \lVert g\rVert$ and $\lVert P_{\mathfrak{N}(V)}g\rVert+\lVert P_{V}g\rVert\leq(\oldC{C:0}+1)\lVert g\rVert$.
\end{proof}

The following proposition allows us to estimate the distance of parallel $1$-codimensional  planes.

\begin{proposizione}\label{prop:dist-piani}
Let $x,y\in\mathbb{G}$ and $V\in\G(\mathcal{Q}-1)$. Defined:
\begin{equation}
    \dist(xV,yV):=\max\bigg\{\sup_{v\in V}\dist(xv,yV),\sup_{v\in V}\dist(yv,xV)\bigg\},
    \nonumber
\end{equation}
we have:
\begin{itemize}
    \item[(i)]$\dist(xV,yV)=\dist(x,yV)=\dist(y,xV)=\lvert \pi_{\mathfrak{n}(V)}(\pi_1(x^{-1}y))\rvert$,
    \item[(ii)]$    \dist(u,xV)\leq \dist(u,yV)+\dist(xV,yV), \text{ for any }u\in\mathbb{G}$.
\end{itemize}
\end{proposizione}

\begin{proof}
For any $v\in V$ we have:
\begin{equation}
    \dist(xv,yV)=\inf_{w\in V}\dist(xv,yw)=\inf_{w\in V} d(x,y (y^{-1}xv^{-1}x^{-1}y)w)=\inf_{w\in V} d(x,yw)=\dist(x,yV),
    \nonumber
\end{equation}
where the second last identity comes from the fact that $v^*:=y^{-1}xv^{-1}x^{-1}y\in V$ and the transitivity of the translation by $v^*$ on $V$. Therefore, we have
$\sup_{v\in V}\dist(xv,yV)=d(x,yV)$ and thus by Proposition \ref{cor:SSCM2.2.14} we infer:
\begin{equation}
    \begin{split}
        \dist(xV,yV)=&\max\big\{\dist(x,yV),\dist(y,xV)\big\}
        =\max\{\lvert\pi_{\mathfrak{n}(V)}(\pi_1(y^{-1}x))\rvert,\lvert \pi_{\mathfrak{n}(V)}(\pi_1(x^{-1}y))\rvert\}\\=&\lvert \pi_{\mathfrak{n}(V)}(\pi_1(x^{-1}y))\rvert=\dist(x,yV)=\dist(y,xV),
        \nonumber
    \end{split}
\end{equation}
where the last identity comes from the arbitrariness of $x$ and $y$.
In order to prove (ii), let $w^*$ be the element of $V$ for which $\dist(u,yV)=d(u,yw^*)$ and note that:
\begin{equation}
\begin{split}
    \dist(u,xV)=&\inf_{v\in V}d(u,xv)\leq d(u,yw^*)+\inf_{v\in V}d(yw^*,xv)=\dist(u,yV)+\inf_{v\in V}d(y w^*,xv)\\
    =&\dist(u,yV)+d(xw^*,yV)\leq \dist(u,yV)+\dist(xV,yV).
    \nonumber
\end{split}
\end{equation}
This concludes the proof of the proposition.
\end{proof}

The following result is a direct consequence of Proposition 2.2.19 of \cite{MFSSC}. The bound \eqref{eq:n520} is obtained with the same argument used by V. Chousionis, K. F\"assler and T. Orponen to prove Lemma 3.6 of \cite{CFO}.

\begin{proposizione}\label{cor:2.2.19}
For any $V\in \G(\mathcal{Q}-1)$ there is a constant $1\leq c(V)\leq\mathcal{S}^{\mathcal{Q}-1}(B(0,\oldC{C:0})\cap V)=:\newC\label{C:up}$ such that for any $p\in\mathbb{G}$ and any $r>0$ we have:
$$\mathcal{S}^{\mathcal{Q}-1}\llcorner V\big(P_V(B(p,r))\big)=c(V)r^{\mathcal{Q}-1}.$$
Furthermore, for any Borel set $A\subseteq \mathbb{G}$ for which $\mathcal{S}^{\mathcal{Q}-1}(A)<\infty$, we have:
\begin{equation}
    \mathcal{S}^{\mathcal{Q}-1}\llcorner V(P_V(A))\leq 2c(V)\mathcal{S}^{\mathcal{Q}-1}(A).
    \label{eq:n520}
\end{equation}
\end{proposizione}

\begin{proof}
The existence of the constant $c(V)$ is yielded by Proposition 2.2.19 of \cite{MFSSC}. Furthermore  Propositions \ref{prop:rapp},  \ref{cor:SSCM2.2.14} and the fact that $B(0,1)\cap V\subseteq P_V(B(0,1))$ for any $V\in\G(\mathcal{Q}-1)$, imply that:
\begin{equation}
    1=\beta^{-1}\mathcal{H}^{n-1}_{eu}(B(0,1)\cap V)=\mathcal{S}^{\mathcal{Q}-1}\llcorner V(B(0,1)\cap V)\leq \mathcal{S}^{\mathcal{Q}-1}\llcorner V\big(P_V(B(0,1))\big)=c(V),
    \nonumber
\end{equation}
and:
\begin{equation}
    \mathcal{S}^{\mathcal{Q}-1}\llcorner V(P_V(B(0,1)))\leq \mathcal{S}^{\mathcal{Q}-1}(B(0,\oldC{C:0})\cap V)=\oldC{C:up},
    \nonumber
\end{equation}
where the constant $\oldC{C:up}$ does not depend on $V$ thanks to Proposition \ref{prop:rapp} and since the ball $B(0,\oldC{C:0})$ is invariant under rotations of $V_1$.

Suppose now $A$ is a Borel set with finite $\mathcal{S}^h$-measure and let
$\{B(x_i,r_i)\}_{i\in\N}$ be a sequence of balls covering $A$ such that $\sum_{i\in\N} r_i^{\mathcal{Q}-1}\leq 2\mathcal{S}^{\mathcal{Q}-1}(A)$. Then:
\begin{equation}
    \mathcal{S}^{\mathcal{Q}-1}(P_V(A))\leq \mathcal{S}^{\mathcal{Q}-1}\Big(P_V\Big(\bigcup_{i\in\N}B(x_i,r_i)\Big)\Big)\leq c(V)\sum_{i\in\N} r_i^{\mathcal{Q}-1}\leq 2 c(V)\mathcal{S}^{\mathcal{Q}-1}(A).
    \nonumber
\end{equation}
This concludes the proof.
\end{proof}

\subsection{Densities and tangents of Radon measures}
\label{densityandtangents}
In this subsection we briefly recall some facts and notations about Radon measures on Carnot groups and their blowups.

\begin{definizione}
If $\phi$ is a Radon measure on $\mathbb{G}$, we define:
$$\Theta_*^{m}(\phi,x):=\liminf_{r\to 0} \frac{\phi(B(x,r))}{r^{m}}\qquad \text{and}\qquad \Theta^{m,*}(\phi,x):=\limsup_{r\to 0} \frac{\phi(B(x,r))}{r^{m}},$$
and say that $\Theta_*^{m}(\phi,x)$ and $\Theta^{m,*}(\phi,x)$ are respectively the lower and upper $m$-density of $\phi$ at the point $x\in\mathbb{G}$.
\end{definizione}

\begin{definizione}[weak convergence of measures]\label{lippi}
A sequence of Radon measures $\{\mu_i\}_{i\in\N}$ is said to be weakly converging in the sense of measures to some Radon measure $\nu$, if for any continuous functions with compact support $f\in \mathcal{C}_c$, we have:
$$\int fd\mu_i\to\int fd\nu.$$\label{compi}
Throughout the paper, we denote such convergence with $\mu_i\rightharpoonup \nu$.
\end{definizione}

\begin{definizione}
For any couple of Radon measures $\phi$ and $\psi$ and any compact set $K\subseteq \mathbb{G}$ we let:
\begin{equation}
    F_{K}(\phi,\psi):=\sup\Big\{\Big\lvert \int f d\phi-\int fd\psi\Big\rvert: f\in\lip(K)\Big\},
\label{eq:F}
\end{equation}
where $\lip(K)$ is the set of non-negative $1$-Lipschitz functions whose support is contained in $K$.
Furthermore, if $K=\overline{B(x,r)}$ we shorten the notation to $F_{x,r}(\phi,\psi):=F_{\overline{B(x,r)}}(\phi,\psi)$.
\end{definizione}

The next lemma is an elementary fact on Radon measures. We omit its proof.

\begin{lemma}\label{radii}
If $\phi$ is a Radon measure, for any $x\in \mathbb{G}$ there are at most countably many radii $R>0$ for which: $$\phi(\partial B(x,R))>0.$$
\end{lemma}

The following proposition allows us to characterise the weak convergence of measures by means of the convergence to $0$ of the functional $F_K$. 

\begin{proposizione}\label{prop:convF}
Assume that $\{\mu_i\}_{i\in\N}$ is a sequece of Radon measures such that $\limsup_{i\to \infty} \mu_i(\overline{B(0,R)})<\infty$ for every $R>0$ and let $\mu$ be a Radon measure on $\mathbb{G}$. The following are equivalent:
\begin{itemize}
    \item[(i)] $\mu_i\rightharpoonup \mu$,
    \item[(ii)] $\lim_{i\to\infty}F_K(\mu_i,\mu)=0$ for any compact set $K\subseteq \mathbb{G}$.
\end{itemize}
\end{proposizione}

\begin{proof}
As a first step, we prove that (i) implies (ii). If (i) holds, for any non-negative Lipschitz function $f$ with compact support we have:
$$\lim_{i\to\infty} \int f d\mu_i=\int f d\mu.$$
Fix a compact subset $K$ of $\mathbb{G}$ contained in $\overline{B(0,R)}$ for some $R>0$.
It is immediate to see that $\lip(K)$ is a compact metric space when endowed with the supremum distance $\lVert\cdot\rVert_\infty$ and thus it is totally bounded.
Therefore, for any $\epsilon>0$, there is a finite set $S\subseteq \lip(K)$ such that whenever $f\in \lip(K)$, we can find a $g\in S$ such that $\lVert f-g\rVert_\infty<\epsilon$. This implies that:
\begin{equation}
\begin{split}
    \Big\lvert \int f d\mu_i-\int fd\mu\Big\rvert\leq&\Big\lvert \int f d\mu_i-\int gd\mu_i\Big\rvert+\Big\lvert \int g d\mu_i-\int gd\mu\Big\rvert+\Big\lvert \int g d\mu-\int fd\mu\Big\rvert\\
\leq& \epsilon(\mu_i(K)+\mu(K))+\Big\lvert \int g d\mu_i-\int gd\mu\Big\rvert.
\end{split}
\nonumber
\end{equation}
Thanks to the arbitrariness of $f$, we infer that:
\begin{equation}
    F_K(\mu_i,\mu)\leq \epsilon(\mu_i(K)+\mu(K))+\sup_{g\in S}\Big\lvert \int g d\mu_i-\int gd\mu\Big\rvert.
    \label{eq:n6000}
\end{equation}
Taking the limit as $i$ goes at infinity on both sides of \eqref{eq:n6000}, thanks to (i) we have:
$$\limsup_{i\to\infty} F_K\big(\mu_i,\mu)\leq \epsilon (\mu(K)+\limsup_{i\to \infty}\mu_i(K))\leq  \epsilon (\mu(K)+\limsup_{i\to \infty}\mu_i(\overline{B(0,R)})\big).$$
The arbitrariness of $\epsilon$ eventually concludes the proof.

We are left to prove that (ii) implies (i). Let $g$ be a non-negative Lipschitz function with compact support. Since we assumed that (ii) holds, we infer that:
$$\lim_{i\to\infty}\Big\lvert \int g d\mu_i-\int g d\mu\Big\rvert=\text{Lip}(g)\lim_{i\to\infty}\Big\lvert \int \frac{g}{\text{Lip}(g)} d\mu_i-\int \frac{g}{\text{Lip}(g)} d\mu\Big\rvert\leq \text{Lip}(g)\lim_{i\to\infty}F_{\supp(g)}(\mu_i,\mu)=0.$$
Let $K$ be a compact set in $\mathbb{G}$. By the continuity of the measure from above, for any $\epsilon>0$ there exists an $s>0$ such that $\mu(B(K,s))\leq \mu(K)+\epsilon$, where $B(K,s):=\{x\in \mathbb{G}:\dist(x,K)\leq s\}$. Defined $f(z):=\min\{1,\dist(z,B(K,s)^c)/s\}$, we conclude that:
$$\mu(K)+\epsilon\geq \int f d\mu=\lim_{i\to\infty} \int fd\mu_i\geq \limsup_{i\to \infty}\mu_i(K).$$
The arbitrariness of $\epsilon$ finally implies that
$\limsup_{i\to\infty}\mu_i(K)\leq \mu(K)$ for any compact set $K\subseteq \mathbb{G}$.
Assume now $g$ is a continuous function whose support is contained in a ball $B(0,R)$ with $\mu$-null boundary. This choice can be done without loss of generality thanks to Lemma \ref{radii}. Since $g$ is continuous, it can be easily proved that for any $D\subseteq \R$ we have:
\begin{equation}
    \partial g^{-1}(D)\subseteq g^{-1}(\partial D).
    \label{eq:n10E6}
\end{equation}
Let $\epsilon>0$ and note that $A:=\{t\in\R:\mu(g^{-1}(\{t\}))>0\}$ is the set of atoms of the measure push forward $g_\#(\mu)$. Since $g$ is bounded, there is an $N\in\N$ and a finite sequence $\{t_i\}_{i=1,\ldots,N}$ such that:
\begin{itemize}
    \item[(a)]$t_1\leq -\lVert g\rVert_\infty<t_2<\ldots<t_{N-1}<\lVert g\rVert_\infty\leq t_N$,
    \item[(b)]$t_n\in \R\setminus A$ and $\lvert t_{n+1}-t_n\rvert<\epsilon$ for any $n=1,\ldots,N-1$.
\end{itemize}
Let $E_n:=B(0,R)\cap g^{-1}([t_n,t_{n+1}])$ for any $n=1,\ldots,N-1$ and note by \eqref{eq:n10E6} we have:
\begin{equation}
\mu(\partial E_n)\leq\mu( \partial B(0,R)\cup g^{-1}(\{t_n\}\cup\{t_{n+1}\}))\leq\mu(\partial B(0,R))+\mu(g^{-1}(\{t_n\}))+\mu(g^{-1}(\{t_{n+1}\}))=0.
\label{eq:e1}
\end{equation}
Thanks to the definition of the $E_i$s it is possible to show that $E_i\cap E_j\subseteq \partial E_i\cup \partial E_j$ and thus by \eqref{eq:e1} we infer that $\sum_{n=1}^{N-1}\mu(E_n)\leq \mu(B(0,R))$. This implies that:
\begin{equation}
\begin{split}
     \limsup_{i\to\infty}\int gd\mu\leq& \limsup_{i\to\infty}\sum_{n=1}^{N-1} t_{n+1} \mu_i(E_n)\leq \sum_{n=1}^{N-1}t_{n+1}\mu(E_n)\leq\sum_{n=1}^{N-1}(t_{n+1}-t_{n})\mu(E_n)+ \sum_{n=1}^{N-1}t_{n}\mu(E_n)\\
     \leq &\epsilon\sum_{n}^{N-1}\mu(E_n)+\int g d\mu\leq \epsilon\mu(B(0,R))+\int gd\mu.
\end{split}
    \nonumber
\end{equation}
The arbitrariness of $\epsilon$ implies that $\limsup_{i\to\infty}\int g d\mu_i\leq \int g d \mu$. Repeating the argument for $-g$ we deduce that $\limsup_{i\to \infty} \int -g d\mu_i\leq \int -g d\mu$, concluding the proof of the proposition.
\end{proof}

\begin{definizione}[Tangent measures]\label{tangentsdef}
Let $\phi$ be a Radon measure on $\mathbb{G}$. For any $x\in \mathbb{G}$ and any $r>0$, we define $T_{x,r}\phi$ to be the Radon measure for which:\label{tang}
$$T_{x,r}\phi(B)=\phi(x\delta_r(B))\text{ for any Borel set }B\subseteq \mathbb{G}.$$
We define $\Tan_{m}(\phi,x)$, the set of the $m$-dimensional tangent measures to $\phi$ at $x$, as the collection of Radon measures $\nu$ for which there is an infinitesimal sequence $\{r_i\}_{i\in\N}$ such that:
$$r_i^{-m}T_{x,r}\phi\rightharpoonup \nu.$$
\end{definizione}

\begin{proposizione}\label{propspt1}
Let $\phi$ be a Radon measure and $\nu\in\Tan_m(\phi,x)$, i.e., $r_i^{-m} T_{x,r_i}\phi\rightharpoonup \nu$ for some $r_i\to 0$. If $y\in\supp(\nu)$, there exists a sequence $\{z_i\}_{i\in\N}\subseteq\supp(\phi)$ such that $\delta_{1/r_i}(x^{-1}z_i)\to y$.
\end{proposizione}

\begin{proof}
A simple argument by contradiction yields the claim, the proof follows verbatim its Euclidean analogue, Proposition 3.4 in \cite{DeLellis2008RectifiableMeasures}.
\end{proof}

\begin{proposizione}\label{tg:nonempty}
Suppose $\phi$ is a Radon measure on $\mathbb{G}$ such that:
$$0<\Theta^{m}_*(\phi,x)\leq \Theta^{m,*}(\phi,x)<\infty,\qquad\text{for }\phi\text{-almost every }x\in\mathbb{G}.$$
Then $\Tan_m(\phi,x)\neq \emptyset$ for $\phi$-almost every $x\in\mathbb{G}$.
\end{proposizione}

\begin{proof}
    This is an immediate consequence of the local uniform boundness of the rescaled measures $T_{x,r}\phi$. The proof follows verbatim the Euclidean one, see for instance Lemma 3.4 of \cite{DeLellis2008RectifiableMeasures}.
\end{proof}

The following result is the analogue of Proposition 3.12 of \cite{DeLellis2008RectifiableMeasures} which establishes the locality of tangents in the Euclidean space. We give here a detailed proof, since we need to avoid to use the Besicovitch covering theorem, which may not hold for the distance $d$. This proposition is of capital importance since it will ensure us that restricting and multiplying by a density a measure with flat tangents will yield a measure still having flat tangents.

\begin{proposizione}[Locality of the tangents]\label{prop:locality}
Let $\phi$ be a Radon measure such that:
\begin{equation}
    0<\Theta^{m}_*(\phi,x)\leq \Theta^{m,*}(\phi,x)<\infty,\qquad\text{for }\phi\text{-almost every }x\in\mathbb{G}.
    \label{numero1}
\end{equation}
Then, for any $\rho\in L^1(\phi)$ we have
$\Tan_{m}(\rho\phi,x)=\rho(x)\Tan_{m}(\phi,x)$
for $\phi$-almost every $x\in \mathbb{G}$.
\end{proposizione}

\begin{proof}
First of all, let us prove that $\phi$ is asymptotically doubling:
\begin{equation}
    \limsup_{r\to 0}\frac{\phi(B(x,2r))}{\phi(B(x,r))}\leq\limsup_{r\to 0}\frac{\phi(B(x,2r))}{(2r)^m}\frac{2^mr^m}{\phi(B(x,r))}\leq \frac{2^m\Theta^{m,*}(\phi,x)}{\Theta^{m}_*(\phi,x)}<\infty, 
\end{equation}
for $\phi$-almost every $x\in\mathbb{G}$. 
Thanks to Theorem 3.4.3 and the \emph{Lebesgue differentiation Theorem} in \cite{MR3363168}, we have:
\begin{equation}
    \phi(B_1):=\phi\Big(\Big\{x\in\mathbb{G}:\limsup_{r\to 0}\fint_{B(x,r)}\lvert \rho(y)-\rho(x)\rvert d\phi(y)>0\Big\}\Big)=0.
    \label{eq:Leb1}
\end{equation}
Let $x\in\mathbb{G}\setminus B_1$, suppose $\nu\in\Tan_{m}(\phi,x)$, and let $r_i\to 0$ be an infinitesimal sequence such that:
$$\nu_i:=\frac{T_{x,r_i}\phi}{r_i^{m}}\rightharpoonup \nu.$$
Defined $\nu^\prime_i:=r_i^{-m}T_{x,r_i}(\rho \phi)$ for every ball $B(0,l)$, we have:
\begin{equation}
    \lvert \rho(x)\nu_i-\nu_i^\prime\rvert(B(0,l))\leq \frac{1}{r_i^{m}}\int _{B(x,l r_i)} \lvert \rho(y)-\rho(x)\rvert d\phi(y)=\frac{\phi(B_{\rho r_i}(x))}{r_i^{m}}\cdot\fint_{B(x,l r_i)}\lvert \rho(y)-\rho(x)\rvert d\phi(y)=:(I)_i\cdot(II)_i,
    \nonumber
    \end{equation}
    where $ \lvert \rho(x)\nu_i-\nu_i^\prime\rvert$ is the total variation of the measure $\rho(x)\nu_i-\nu_i^\prime$.
For $\phi$-almost every $x\in B_1$ thanks to \eqref{eq:Leb1} we have  $\Theta^{m,*}(\phi,x)<\infty$ and $\lim_{i\to\infty}(II)_i=0$. This implies that: $$\lim_{i\to\infty}  \lvert \rho(x)\nu_i-\nu_i^\prime\rvert(B(0,l))=0,$$
for every $l>0$. If in addition to this, we also assume that $\nu(\partial B(0,l))=0$, since $\nu_i\rightharpoonup\nu$, then by Proposition 2.7 of \cite{DeLellis2008RectifiableMeasures} we also have that:
\begin{equation}
    \lim_{i\to\infty}  \lvert \rho(x)\nu-\nu_i^\prime\rvert(B(0,l))=0.
    \label{eq:nummm2010}
\end{equation}
Finally, since by Proposition \ref{radii} we have $\nu(\partial B(0,l))=0$ for almost every $l>0$, thanks to  \eqref{eq:nummm2010} we conclude that $\nu_i^\prime\rightharpoonup \rho(x)\nu$.
\end{proof}

\begin{proposizione}\label{prop:cpt}
Assume $\phi$ is a Radon measure supported on a compact set $K$ such that for $\phi$-almost every $x\in\mathbb{G}$ we have:
$$0<\Theta^{\mathcal{Q}-1}_*(\phi,x)\leq \Theta^{\mathcal{Q}-1,*}(\phi,x)<\infty.$$
Then, for any $\vartheta,\gamma\in\N$ the set $  E^\phi(\vartheta,\gamma):=\big\{x\in K:\vartheta^{-1}r^{{\mathcal{Q}-1}}\leq \phi(B(x,r))\leq \vartheta r^{{\mathcal{Q}-1}}\text{ for any }0<r<1/\gamma\big\}$
is compact.
\end{proposizione}

\begin{proof}
Since $K$ is compact, in order to verify that $E^\phi(\vartheta,\gamma)$
is compact, it suffices to prove that it is closed. If $E^\phi(\vartheta,\gamma)$ is empty or finite, there is nothing to prove. So, suppose there is a sequence $\{x_i\}_{i\in\N}\subseteq E^\phi(\vartheta,\gamma)$ converging to some $x\in K$. Fix an $0<r<1/\gamma$ and assume that $\delta>0$ is so small that $r+\delta<1/\gamma$. Therefore, if $d(x,x_i)<\delta$ and $r-d(x,x_i)>0$, we have:
$$\vartheta^{-1}\big(r-d(x,x_i)\big)^{{\mathcal{Q}-1}}\leq \phi\big(B(x_i,r-d(x,x_i))\big)\leq \phi(B(x,r))\leq \phi\big(B(x_i,r+d(x,x_i))\big)\leq \vartheta (r+d(x,x_i))^{{\mathcal{Q}-1}},$$
Taking the limit as $i$ goes to $\infty$, we see that $x\in E(\vartheta,\gamma)$.
\end{proof}

\begin{proposizione}\label{prop:Ecorsivo}
Assume $\phi$ is a Radon measure supported on a compact set $K$ such that for $\phi$-almost every $x\in\mathbb{G}$ we have:
$$0<\Theta^{\mathcal{Q}-1}_*(\phi,x)\leq \Theta^{\mathcal{Q}-1,*}(\phi,x)<\infty.$$
Then, for any $\vartheta,\gamma,\mu,\nu\in\N$ the set:
$$\mathscr{E}_{\vartheta,\gamma}^\phi(\mu,\nu)=\{x\in E^\phi(\vartheta,\gamma): (1-1/\mu)\phi(B(x,r))\leq \phi(B(x,r)\cap E^\phi(\vartheta,\gamma))\text{ for any }0<r<1/\nu\},$$
is compact.
\end{proposizione}

\begin{proof}
If $\mathscr{E}^\phi_{\vartheta,\gamma}(\mu,\nu)$ is empty or finite, there is nothing to prove.
Furthermore, thanks to Proposition \ref{prop:cpt} the sets $E^\phi(\vartheta,\gamma)$ are compact and thus to prove our claim it is sufficient to show that $\mathscr{E}^\phi_{\vartheta,\gamma}(\mu,\nu)$ is closed in $E^\phi(\vartheta,\gamma)$. Take a sequence $\{y_i\}_{i\in\N}\subseteq \mathscr{E}^\phi_{\vartheta,\gamma}(\mu,\nu)$ converging to some $y\in E^\phi(\vartheta,\gamma)$. Fix an $0<r<1/\nu$ and a $\delta\in (0,1/4)$ and let $i_0(\delta)\in\N$ be such that for any $i\geq i_0(\delta)$ we have $d(y,y_i)<\delta r$. These choices imply:
$$(1-1/\mu)\phi(B(y_i,r-2d(y,y_i)))\leq\phi(B(y_i,r-2d(y,y_i))\cap E^\phi(\vartheta,\gamma))\leq \phi(B(y,r)\cap E^\phi(\vartheta,\gamma)).$$
Note that the sequence of functions $f_i(z):=\chi_{B(y_i,r-2d(y,y_i))}(z)$ converges pointwise $\phi$-almost everywhere to $\chi_{B(y,r)}(z)$. This is due to the fact that for any $i\geq i_0(\delta)$ on the one hand we have $\supp(f_i)\Subset B(y,r)$ and on the other the functions $f_i$ are equal to $1$ on $B(y,r(1-3\delta))$. Thus, the dominated convergence theorem implies:
$$(1-1/\mu)\phi(B(y,r))=\lim_{i\to\infty}(1-1/\mu)\phi(B(y_i,r-2d(y,y_i)))\leq\phi(B(y,r)\cap E(\vartheta,\gamma)).$$
Since $r$ was aribtrarily chosen in $(0,1/\nu)$, this shows that $y\in \mathscr{E}_{\vartheta,\gamma}(\mu,\nu)$ concluding the proof.
\end{proof}

\begin{proposizione}\label{prop:BIGGI}
Assume $\phi$ is a Radon measure supported on a compact set $K$ such that for $\phi$-almost every $x\in\mathbb{G}$ we have:
$$0<\Theta^{\mathcal{Q}-1}_*(\phi,x)\leq \Theta^{\mathcal{Q}-1,*}(\phi,x)<\infty.$$
Then, for any $0<\epsilon<1/10$ there are $\vartheta_0,\gamma_0\in\N$ such that for any $\vartheta\geq \vartheta_0$, $\gamma\geq \gamma_0$ and $\mu\in\N$ there is a $\nu=\nu(\vartheta,\gamma,\mu)\in\N$ such that:
\begin{equation}
    \phi(K/\mathscr{E}^\phi_{\vartheta,\gamma}(\mu,\nu))\leq \epsilon\phi(K).
    \label{eq:A20}
\end{equation}
\end{proposizione}

\begin{proof}
Assume at first that $\vartheta$ and $\gamma$ are fixed. It is easy to check that:
\begin{align}
    \mathscr{E}^\phi_{\vartheta,\gamma}(\mu_2,\nu)&\subseteq \mathscr{E}^\phi_{\vartheta,\gamma}(\mu_1,\nu),~\text{whenever }\mu_1\leq \mu_2,
    \label{eq:70}\\
    \mathscr{E}^\phi_{\vartheta,\gamma}(\mu,\nu_1)&\subseteq \mathscr{E}^\phi_{\vartheta,\gamma}(\mu,\nu_2),~\text{whenever }\nu_1\leq \nu_2.
    \label{eq:71}
\end{align}
Fix some $\mu\in \N$ and let $x\in E^\phi(\vartheta,\gamma)\setminus\bigcup_{\nu\in\N}\mathscr{E}^\phi_{\vartheta,\gamma}(\mu,\nu)$. For such an $x$, we can find a sequence $r_\nu$ such that $r_\nu\leq 1/\nu$ and:
\begin{equation}
\phi(B(x,r_\nu)\cap E^\phi(\vartheta,\gamma))<(1-1/\mu)\phi(B(x,r_\nu)).
\nonumber
\end{equation}
This would imply that $\liminf_{r\to 0}\frac{\phi(B(x,r)\cap E^\phi(\vartheta,\gamma))}{\phi(B(x,r))}\leq(1-1/\mu)$, showing thanks to \emph{Lebescue differentiability Theorem} of \cite{MR3363168} that: 
\begin{equation}
    \phi\bigg(E^\phi(\vartheta,\gamma)\setminus\bigcup_{\nu\in\N}\mathscr{E}^\phi_{\vartheta,\gamma}(\mu,\nu)\bigg)=0.
    \label{eq:n102}
\end{equation}
Thanks to inclusion \eqref{eq:71} and identity \eqref{eq:n102}, for any $\mu\in\N$  we can find some $\nu\in\N$ such that:
$$\phi(E^\phi(\vartheta,\gamma)\setminus\mathscr{E}^\phi_{\vartheta,\gamma}(\mu,\nu))\leq \epsilon\phi(E^\phi(\vartheta,\gamma))/2.$$
We finally prove that there are $\vartheta,\gamma\in\N$ such that $\phi(K\setminus E^\phi(\vartheta,\gamma))\leq \epsilon\phi(K)/2$. For any $x\in K\setminus \bigcup_{\vartheta,\gamma\in\N} E^\phi(\vartheta,\gamma)$, it is immediate to see that $\Theta_*^{\mathcal{Q}-1}(\phi,x)=0$ or $\Theta^{\mathcal{Q}-1,*}(\phi,x)=\infty$, which thanks to the choice of $\phi$, implies:
\begin{equation}
    \phi\bigg(K\setminus \bigcup_{\vartheta,\gamma\in\N} E^\phi(\vartheta,\gamma)\bigg)=0.
    \label{eq:n24}
\end{equation}
Thanks to the fact that $E^\phi(\vartheta,\gamma)\subseteq E^\phi(\vartheta^\prime,\gamma^\prime)$ whenever $\gamma\leq \gamma^\prime$ and $\vartheta\leq \vartheta^\prime$ and the continuity from above of the measure, this proves the proposition.
\end{proof}

The following result allows us to compare the measure $\phi$ when restricted to $E^\phi(\vartheta,\gamma)$ with the spherical Hausdorff measure. 

\begin{proposizione}\label{prop:dens}
Let $\psi$ be a Radon measure supported on a Borel set $E$. Suppose further that there are $0<\delta_1\leq\delta_2$ such that:
$$\delta_1\leq \Theta^{m}_*(\psi,x)\leq \Theta^{m,*}(\psi,x)\leq \delta_2,\qquad \text{for }\psi\text{-almost every }x\in\mathbb{G}.$$
Then, $\psi=\Theta^{m,*}(\psi,x)\mathcal{C}^{m}\llcorner E$ where $\mathcal{C}^{m}$ is the centred spherical Hausdorff measure introduced in Definition \ref{hausdorrmeas} and in particular:
     \begin{equation}
     \begin{split}
      \delta_1\mathcal{C}^{m}\llcorner E\leq\psi\llcorner E\leq\delta_2\mathcal{C}^{m}\llcorner E,\qquad \text{and}\qquad          \delta_1\mathcal{S}^{m}\llcorner E\leq\psi\llcorner E\leq \delta_22^{m}\mathcal{S}^{m}\llcorner E.
         \nonumber
     \end{split}
     \end{equation}
\end{proposizione}

\begin{proof}
Thanks to \emph{Lebesgue differentiability theorem} of \cite{MR3363168}, for any Borel set $A\subseteq \mathbb{G}$ and $\psi$-almost all $x\in A\cap E$ we have:
\begin{equation}
\begin{split}
    \limsup_{r\to 0}\frac{\psi(B(x,r)\cap E\cap A)}{r^{m}}
    \leq\limsup_{r\to 0}\frac{\psi(B(x,r)\cap E\cap A)}{\psi(B(x,r))}\frac{\psi(B(x,r))}{r^{m}}\leq \delta_2.
     \nonumber
\end{split}
\end{equation}
Thanks to Proposition 2.10.17 of \cite{Federer1996GeometricTheory}, we conclude that $\psi\leq \delta_2 \mathcal{S}^{m}\llcorner E$ and in particular $\psi$ is absolutely continuous with respect to $\mathcal{S}^{m}\llcorner E$. Furthermore, thanks to Theorem 3.1 of \cite{densita}, we infer that:
\begin{equation}
    \psi=\Theta^{m,*}(\psi,x)\mathcal{C}^{m}\llcorner E.
    \label{eq:n101}
\end{equation}
Finally, from \eqref{eq:n101} we deduce that:
\begin{equation}
     \delta_1 \mathcal{S}^{m}\llcorner E\leq \delta_1 \mathcal{C}^{m}\llcorner E\leq \psi\leq \delta_2 \mathcal{C}^{m}\llcorner E\leq \delta_2 2^{m}\mathcal{S}^{m}\llcorner E,
     \nonumber
\end{equation}
where the first and last inequality follow from the fact that the measures $\mathcal{C}^{m}$ and $\mathcal{S}^{m}$ are equivalent and satisfy the bounds $\mathcal{S}^{m}\leq \mathcal{C}^{m}\leq 2^{m}\mathcal{S}^{m}$, for a reference see \cite{densita}.
\end{proof}

An immediate consequence of the above proposition is the following:

\begin{corollario}\label{cor:cor1}
For any $\vartheta,\gamma\in\N$ we have
$\vartheta^{-1}\mathcal{S}^{\mathcal{Q}-1}\llcorner E^\phi(\vartheta,\gamma)\leq \phi\llcorner E^\phi(\vartheta,\gamma)\leq \vartheta 2^{\mathcal{Q}-1}\mathcal{S}^{\mathcal{Q}-1}\llcorner E^\phi(\vartheta,\gamma)$.
\end{corollario}

The following result will be used in the proof of the very important Proposition \ref{prop:flatty}. It establishes the natural request that if a sequence of planes $V_i$ in $\G(\mathcal{Q}-1)$ convergences in the Grassmanian to some plane $V\in\G(\mathcal{Q}-1)$ (i.e. the normals converge as vectors in $V_1$), then the surface measures on the $V_i$s converge weakly to the surface measure on $V$.

\begin{proposizione}\label{prop:pianconv}
Suppose that $\{V_i\}_{i\in\N}$ is a sequence of planes in $\G(\mathcal{Q}-1)$ such that $\mathfrak{n}(V_i)\to\mathfrak{n}$ for some $\mathfrak{n}\in V_1$. Then, there exists a $V\in \G(\mathcal{Q}-1)$ such that $\mathfrak{n}(V)=\mathfrak{n}$ and:
$$\mathcal{S}^{\mathcal{Q}-1}\llcorner V_i\rightharpoonup \mathcal{S}^{\mathcal{Q}-1}\llcorner V.$$
\end{proposizione}

\begin{proof}
For any continuous function of compact support $f\in \mathcal{C}_c$ we have:
\begin{equation}
   \lim_{i\to\infty} \int f d\mathcal{S}^{\mathcal{Q}-1}\llcorner V_i-\int f d\mathcal{S}^{\mathcal{Q}-1}\llcorner V= \lim_{i\to\infty}\beta^{-1}\bigg(\int f d\mathcal{H}^{n-1}_{eu}\llcorner V_i-\int f d\mathcal{H}^{n-1}_{eu}\llcorner V\bigg)=0,
\end{equation}
where the last identity comes from the fact that $\mathcal{H}^{n-1}_{eu}\llcorner V_i\rightharpoonup \mathcal{H}^{n-1}_{eu}\llcorner V$.
\end{proof}

\subsection{Rectifiabile sets in Carnot groups}\label{sub:rect}
In this subsection we recall the two main notions of rectifiability in Carnot groups that will be extensively used throughout the paper. 
First of all, let us recall the definition of horizontal vector fields and of horizontal distribution.

\begin{definizione}
For any $i=1,\ldots,n_1$ and any $x\in\mathbb{G}$ we let 
$X_i(x):=\partial_t(x*\delta_t(e_i))\rvert_{t=0}$,
and say that the map $X_i:\mathbb{G}\cong\R^n\to \R^n$ so defined is the $i$-th \emph{horizontal vector field}. Furthermore, we define the \emph{horizontal distribution }of $\mathbb{G}$ to be the following $n_1$-dimensional distribution of planes of $\R^n$:
$$H\mathbb{G}(x):=\text{span}\{X_1(x),\ldots,X_{n_1}(x)\}.$$
Finally, for any open set $\Omega$ in $\mathbb{G}$ we denote by $\mathcal{C}^1_0(\Omega,H\mathbb{G})$ the sections of $H\mathbb{G}$ of class $\mathcal{C}^1$ with support contained in $\Omega$.
\end{definizione}

The definition of regular surface we are about to give is reminiscent of the characterisation of smooth surfaces in the Euclidean spaces through the local inversion theorem. Heuristically speaking, a $C^1_\mathbb{G}$-surface is a set that is transverse to $H\mathbb{G}$ and whose sections with $H\mathbb{G}$ are $C^1$-surfaces.

\begin{definizione}[$C^1_\mathbb{G}$-surfaces]
\label{regsur}
We say that a closed set $C\subseteq \mathbb{G}$ is a $C^1_\mathbb{G}$-\emph{surface} if there exists a continuous function $f:\mathbb{G}\to\R$ such that $C=f^{-1}(0)$ and whose horizontal distributional gradient $\nabla_\mathbb{G} f:=(X_1f,\ldots,X_{n_1}f)$ can be represented by a continuous, never-vanishing section of $H\mathbb{G}$.
\end{definizione}

\begin{osservazione}
Thanks to Corollary 4.27 of \cite{SCln}, if $C$ is a $C^1_\mathbb{G}$-regular surface, then $\mathcal{S}^{\mathcal{Q}-1}\llcorner C$ is $\sigma$-finite. 
\end{osservazione}

The second notion of regular surface we give in this subsection is inspired by the characterisation of Lipschitz graphs through cones. 

\begin{definizione}[Intrinsic Lipschitz graphs]
\label{def:intrgraph}
Let $V\in \G(\mathcal{Q}-1)$ and $E$ be a Borel subset of $V$. A function $f:E\to \mathfrak{N}(V)$ is said to be \emph{intrinsic Lipschitz} if there exists an $\alpha>0$ such that for any $v\in E$ we have:
$$gr(f):=\{wf(w):w\in E\}\subseteq vf(v)C_V(\alpha).$$
A Borel set $A\subseteq \mathbb{G}$ is said to be an \emph{intrinsic Lipschitz graph} if there is an intrinsic Lipschitz function $f:E\subseteq V\to \mathfrak{N}(V)$ such that $A=f(E)$.
\end{definizione}

The following extension theorem is of capital importance for us:

\begin{teorema}[Theorem 3.4, \cite{MR3060706}]\label{whole}
Suppose $V\in\G(\mathcal{Q}-1)$ and let $f:E\to\mathfrak{N}(V)$ be an intrinsic Lipschitz function. Then there is an intrinsic Lipschitz function $\tilde{f}:V\to \mathfrak{N}(V)$ such that $f(v)=\tilde{f}(v)$ for any $v\in E$.
\end{teorema}

The following result is an immediate consequence of Theorem \ref{whole}:

\begin{proposizione}\label{intrlipsigfin}
If $f:E\subseteq V\to\mathfrak{N}(V)$ is an intrinsic Lipschitz function, then $\mathcal{S}^{\mathcal{Q}-1}\llcorner f(E)$ is $\sigma$-finite.
\end{proposizione}

\begin{proof}
Theorem \ref{whole} together with Theorem 3.2.1 of \cite{MFSSC} immediately implies that $\mathcal{S}^{\mathcal{Q}-1}(gr(f)\cap B(0,R))<\infty$ for any $R>0$.
\end{proof}

From the notions of $C^1_\mathbb{G}$-surface and of intrinsic Lipschitz surface rise the two following definitions of rectifiability:

\begin{definizione}\label{def:rectif}
A Borel set $A\subseteq \mathbb{G}$ of finite $\mathcal{S}^{\mathcal{Q}-1}$-measure is said to be:
\begin{itemize}
    \item[(i)]$C^1_\mathbb{G}$-\emph{rectifiable} if there are countably many $C^1_\mathbb{G}$-surfaces $\Gamma_i$ such that $\mathcal{S}^{\mathcal{Q}-1}(A\setminus \bigcup_{i\in\N}\Gamma_i)=0$.
        \item[(ii)]\emph{intrinsic rectifiable} if there are countably many intrinsic Lipschitz graphs $\Gamma_i$ such that $\mathcal{S}^{\mathcal{Q}-1}(A\setminus \bigcup_{i\in\N}\Gamma_i)=0$.
\end{itemize}
\end{definizione}

The  following proposition is an adaptation of the well known fact that Borel sets can be written in an essentially unique way, as the union of a rectifiable and a purely unrectifiable set. For the Euclidean statement and its proof, we refer to Theorem 5.7 of \cite{DeLellis2008RectifiableMeasures}.

\begin{proposizione}(Decomposition Theorem)\label{th:dec}
Suppose $\mathscr{F}$ is a family of Borel sets in $\mathbb{G}$ for which $\mathcal{S}^{\mathcal{Q}-1}\llcorner C$ is $\sigma$-finite for any $C\in\mathscr{F}$.
Then,  for any Borel set $E\subseteq \mathbb{G}$ such that $\mathcal{S}^{\mathcal{Q}-1}(E)<\infty$, there are two Borel sets $E^u,E^r\subseteq E$ such that:
\begin{itemize}
    \item[(i)] $E^u\cup E^r=E$,
    \item[(ii)] $E^r$ is contained in countable union of elements of $\mathscr{F}$,
    \item[(iii)] $\mathcal{S}^{\mathcal{Q}-1}(E^u\cap C)=0$ for any $C\in\mathscr{F}$.
\end{itemize}
Such decomposition is unique up to $\mathcal{S}^{\mathcal{Q}-1}$-null sets, i.e. if $F^u$ and $F^r$ are Borel sets satisfing the three properties listed above, we have:
$$\mathcal{S}^{\mathcal{Q}-1}(E^r\triangle F^r)=\mathcal{S}^{\mathcal{Q}-1}(E^u\triangle F^u)=0.$$
\end{proposizione}

\begin{proof}
Define $\mathcal{R}(E):=\{E^\prime\subseteq E:\text{ there are }\{\Gamma_i\}_{i\in\N}\subseteq \mathscr{F}\text{ such that }E^\prime\subseteq \bigcup_{i\in\N}\Gamma_i\}$ and: $$\alpha:=\sup_{E^\prime\in\mathcal{R}(E)}\mathcal{S}^{\mathcal{Q}-1}(E^\prime).$$ 
Suppose $\{E_i:i\in\N\}\subseteq \mathcal{R}(E)$ is a maximizing sequence, i.e. $\lim_{i\to\infty} \mathcal{S}^{\mathcal{Q}-1}(E_i)=\alpha$. Then we let $E^r:=\bigcup_{i\in\N} E_i$ and note that $E_r$ is covered countably many sets in $\mathscr{F}$ and $\mathcal{S}^{\mathcal{Q}-1}(E^r)=\alpha$. The set $E\setminus E^r$ is unrectifiable with respect to $\mathscr{F}$. Indeed, if there was an intrinsic Lipschitz graph $\Gamma$ such that $\mathcal{S}^{\mathcal{Q}-1}(\mathbb{G}\setminus E\cap \Gamma)>0$, we would infer that: $$\mathcal{S}^{\mathcal{Q}-1}(E^r\cup (\mathbb{G}\setminus E^r \cap \Gamma))>\alpha.$$ 
Since $E^r\cup (\mathbb{G}\setminus E^r \cap \Gamma)\in\mathcal{R}(E)$, this would contradict the maximality of $\alpha$. 
If $F^r$ and $F^u$ are as in the statement, we have:
$$\mathcal{S}^{\mathcal{Q}-1}(E^r\cap E^u)=\mathcal{S}^{\mathcal{Q}-1}(E^r\cap F^u)=\mathcal{S}^{\mathcal{Q}-1}(F^r\cap E^u)=\mathcal{S}^{\mathcal{Q}-1}(F^r\cap F^u)=0.$$
Since $E^r\cup E^u=E=F^r\cup F^u$ and the above chain of identities proves the uniqueness of the decomposition.
\end{proof}

\begin{corollario}\label{corrii}
For any Borel set $E\subseteq \mathbb{G}$ such that $\mathcal{S}^{\mathcal{Q}-1}(E)<\infty$, there are two Borel sets $E^u,E^r\subseteq E$ such that:
\begin{itemize}
    \item[(i)] $E^u\cup E^r=E$,
    \item[(ii)] there are countably many intrinsic Lipschitz functions $f_i:V_i\to\mathfrak{N}(V_i)$, where $V_i\in\G(\mathcal{Q}-1)$, whose graphs cover $\mathcal{S}^{\mathcal{Q}-1}$-almost all $E^r$, 
    \item[(iii)] $\mathcal{S}^{\mathcal{Q}-1}(E^u\cap C)=0$ for any $C$ intrinsic Lipschitz graph.
\end{itemize}
\end{corollario}

\begin{proof}
Thanks to Proposition \ref{intrlipsigfin} we know that every intrinsic Lipschitz graph is $\mathcal{S}^{\mathcal{Q}-1}$-$\sigma$-finite. If we choose $\mathscr{F}$ in the statement of Proposition \ref{th:dec} to be the family of all intrinsic Lipschitz graphs of $\mathbb{G}$, we get two sets $E^u$ and $E^r$ whose union is the whole $E$, such that $E^u$ has $\mathcal{S}^h$-null intersection with every intrinsic Lipschitz graph and $E_r$ can be covered by countably many graphs of intrinsic Lipschitz functions $f_i:E_i\subseteq V_i\to \mathfrak{N}(V_i)$. The conclusion follows from Theorem \ref{whole}.
\end{proof}


\section{The support of \texorpdfstring{$1$}{Lg}-codimensional measures with flat tangents is intrinsic rectifiable}\label{sec:main}

Throughout this section we assume $\phi$ to be a fixed Radon measure on $\mathbb{G}$ whose support is a compact set $K$ and such that for $\phi$-almost every $x\in \mathbb{G}$ we have:
\begin{itemize}
    \item[(i)] $
    0<\Theta_*^{\mathcal{Q}-1}(\phi,x)\leq \Theta^{\mathcal{Q}-1,*}(\phi,x)<\infty$,
    \item[(ii)] $\Tan_{\mathcal{Q}-1}(\phi,x)\subseteq \mathfrak{M}$, where $\mathfrak{M}$ is the family of $1$-codimensional flat measures introduced in Definition \ref{def:flatmeasures}.
\end{itemize}

The main goal of this section is to prove the following:

\begin{restatable}{teorema}{ueue}
\label{ueue}
There is an intrinsic Lipschitz graph $\Gamma$ such that $\phi(\Gamma)>0$.
\end{restatable}

The strategy we employ in order to prove Theorem \ref{ueue} is divided in four parts. First of all in Subsection \ref{flat:1} we show that the hypothesis (ii) on $\phi$ implies that for $\phi$-almost any $x\in K$ and $r>0$ sufficiently small, there is a plane $V_{x,r}$ for which $K$ as a set is very close in the Hausdorff distance to $V_{x,r}$. In Subsection \ref{sub:project} we prove that if $K\cap B(x,r)$ has big projection on some plane $W$, then $W$ is very close to $V_{x,r}$ and there exists an $\alpha>0$ such that for any $y,z\in B(x,r)$ for which $d_H(y,z)\geq \text{dist} (W,V_{x,r}) r$, we have $z\in yC_W(\alpha)$. Subsection \ref{big:proj} is the technical core of this section, and its main result Theorem \ref{TH:proiezioni}, shows that for $\phi$-almost any $x\in K$ we have that the set $B(x,r)\cap K$ has a big projection on $V_{x,r}$. Finally, in Subsection \ref{intr:graph} making use of the results of the previous subsections, we construct the wanted $\phi$-positive intrinsic Lipschitz graph.

\subsection{Geometric implications of flat tangents}\label{flat:1}

In this subsection we reformulate the hypothesis (ii) on $\phi$ in more geometric terms. Furthermore, in the following Definition \ref{def:metr}, we introduce two functionals on Radon measures that will be used in the following to estimate at small scales the distance of $\supp(\phi)$ from the  planes in $\G(\mathcal{Q}-1)$. These functionals can be considered the Carnot analogue of the \emph{bilateral beta numbers} of Chapter 2 of  \cite{DavidSemmes} or of the functional $d(\cdot,\mathfrak{M})$ of Section 2 of \cite{Preiss1987GeometryDensities}.

\begin{definizione}\label{def:metr}
For any $x\in\mathbb{G}$ and any $r>0$ we define the functionals:
\begin{equation}
    \begin{split}
        d_{x,r}(\phi,\mathfrak{M}):=&\inf_{\substack{\Theta>0,\\ V\in \G(\mathcal{Q}-1)}} \frac{F_{x,r}(\phi,\Theta \mathcal{S}^{\mathcal{Q}-1}\llcorner xV)}{r^{\mathcal{Q}}},\qquad \text{and}\qquad
        \tilde{d}_{x,r}(\phi,\mathfrak{M}):=\inf_{\substack{\Theta>0,~z\in\mathbb{G},\\ V\in \G(\mathcal{Q}-1)}}\frac{F_{x,r}(\phi,\Theta\mathcal{S}^{\mathcal{Q}-1}\llcorner zV)}{r^{\mathcal{Q}}},
        \nonumber
    \end{split}
\end{equation}
where $F_{x,r}$ was introduced in \eqref{eq:F}.
\end{definizione}

In the following proposition we summarize some useful properties of the functionals introduced above.

\begin{proposizione}\label{prop:stab}
The functionals $d_{x,r}(\cdot,\mathfrak{M})$ and $\tilde{d}_{x,r}(\cdot,\mathfrak{M})$ satisfy the following properties.
\begin{itemize}
    \item[(i)] for any $x\in\mathbb{G}$, $k>0$ and $r>0$ we have $d_{x,kr}(\phi,\mathfrak{M})=d_{0,k}(r^{-(\mathcal{Q}-1)}T_{x,r}\phi,\mathfrak{M})$,
    \item[(ii)]for any $r>0$ the function $x\mapsto d_{x,r}(\phi,\mathfrak{M})$ is continuous,
    \item[(iii)] for any $x,y\in\mathbb{G}$ and $r,s>0$ for which $B(y,s)\subseteq B(x,r)$, we have $(s/r)^{\mathcal{Q}}\tilde{d}_{y,s}(\phi,\mathfrak{M})\leq \tilde{d}_{x,r}(\phi,\mathfrak{M})$,
\item[(iv)] for any $x\in \mathbb{G}$ and any $s\leq r$, we have $(s/r)^{\mathcal{Q}}d_{x,s}(\phi,\mathfrak{M})\leq d_{x,r}(\phi,\mathfrak{M})$.
\end{itemize}
\end{proposizione}

\begin{proof}
It is immediate to see that $f$ belongs to $\lip(\overline{B(x,kr)})$ if and only if there is a $g\in \lip(\overline{B(0,k)})$ such that $f(z)=rg(\delta_{1/r}(x^{-1}z))$. This implies that:
\begin{equation}
    \begin{split}
        \frac{1}{(kr)^\mathcal{Q}}\bigg(\int f d\phi-\Theta\int fd\mathcal{S}^{\mathcal{Q}-1}\llcorner xV\bigg)=&\frac{1}{k^{\mathcal{Q}}r^{\mathcal{Q}-1}}\bigg(\int g(\delta_{1/r}(x^{-1}z)) d\phi(z)-\Theta\int g(\delta_{1/r}(x^{-1}z))d\mathcal{S}^{\mathcal{Q}-1}\llcorner xV\bigg)\\
        =&\frac{1}{k^{\mathcal{Q}}}\bigg(\int g(z) d\frac{T_{x,r}\phi}{r^{Q-1}}(z)-\Theta\int g(z)d\mathcal{S}^{\mathcal{Q}-1}\llcorner V\bigg),
    \end{split}
    \nonumber
\end{equation}
and this proves (i). To show that the map $x\mapsto d_{x,r}(\phi,\mathfrak{M})$ is continuous, we prove the following stronger fact. For any $x,y\in \mathbb{G}$ we have:
\begin{equation}
    \lvert d_{x,r}(\phi,\mathfrak{M})-\delta_{y,r}(\phi,\mathfrak{M})\rvert\leq \frac{d(x,y)}{r^{\mathcal{Q}}}\phi(B(x,r+d(x,y))).
    \label{eq:2020}
\end{equation}
In order to prove \eqref{eq:2020}, for any $\epsilon>0$ we let $\Theta^*>0$ and $V^*\in \G(\mathcal{Q}-1)$ be such that:
    $$\Big\lvert \int fd\frac{T_{y,r}\phi}{r^{\mathcal{Q}-1}}-\Theta^*\int f d\mathcal{S}^{\mathcal{Q}-1}\llcorner V^*\Big\rvert\leq d_{y,r}(\phi,\mathfrak{M})+\epsilon,\text{ for any }f\in \lip(\overline{B(0,1)}),$$
Furthermore, by definition of $d_{y,r}$ we can find an $f^*\in \lip(\overline{B(0,1)})$ such that:
    $$d_{x,r}(\phi,\mathfrak{M})-\epsilon\leq\Big\lvert \int f^*d\frac{T_{x,r}\phi}{r^{\mathcal{Q}-1}}-\Theta^*\int f^* d\mathcal{S}^{\mathcal{Q}-1}\llcorner V_*\Big\rvert.$$
This choice of $f^*$, $\Theta^*$ and $V^*$ implies:
\begin{equation}
\begin{split}
    d_{x,r}(\phi,\mathfrak{M})-d_{y,r}(\phi,\mathfrak{M})\leq \Big\lvert \int f^*d\frac{T_{x,r}\phi}{r^{\mathcal{Q}-1}}-\Theta^*&\int f^* d\mathcal{S}^{\mathcal{Q}-1}\llcorner V^*\Big\rvert-\Big\lvert \int f^*d\frac{T_{y,r}\phi}{r^{\mathcal{Q}-1}}-\Theta^*\int f^* d\mathcal{S}^{\mathcal{Q}-1}\llcorner V^*\Big\rvert+2\epsilon\\
    \leq  \Big\lvert \int f^*d\frac{T_{x,r}\phi}{r^{\mathcal{Q}-1}}- \int f^*d\frac{T_{y,r}\phi}{r^{\mathcal{Q}-1}}\Big\rvert+2\epsilon\leq &r^{-(\mathcal{Q}-1)}\int \lvert f^*(\delta_{1/r}(x^{-1}w))-f^*(\delta_{1/r}(y^{-1}w))\rvert d\phi(w)+2\epsilon\\
    \leq & \frac{d(x,y)}{r^{\mathcal{Q}}}\phi(B(x,r+d(x,y)))+2\epsilon.
    \nonumber
    \end{split}
\end{equation}
Interchanging $x$ and $y$, the bound \eqref{eq:2020} is proved thanks to the arbitrariness of $\epsilon$. Finally the statements (iii) and (iv) follow directly by the definitions.
\end{proof}

The following proposition allows us to rephrase the rather geometric condition on $\phi$ that is the flatness of the tangents, into a more malleable functional-analytic condition that is the $\phi$-almost everywhere convergence of the functions $x\mapsto d_{x,kr}(\phi,\mathfrak{M})$ to $0$.

\begin{proposizione}\label{prop:flatty}
The two following conditions are equivalent:
\begin{itemize}
    \item[(i)]  $\lim_{r\to 0} d_{x, kr}(\phi,\mathfrak{M})=0$ for $\phi$-almost every $x\in \mathbb{G}$ and any $k>0$,
    \item[(ii)]$\Tan(\phi,x)\subseteq \mathfrak{M}$ for $\phi$-almost every $x\in \mathbb{G}$.
\end{itemize}
\end{proposizione}

\begin{proof}
Let us prove that (i) implies (ii). Fix $k>0$ and suppose that $\{r_i\}_{i\in\N}$ is an infinitesimal sequence such that:
\begin{equation}
    r_i^{-(\mathcal{Q}-1)}T_{x,r_i}\phi\rightharpoonup \nu\in \Tan_{\mathcal{Q}-1}(\phi,x).
    \label{eq:1000}
\end{equation}
By definition of $d_{x,kr}(\cdot,\mathfrak{M})$ for any $i\in\N$ we can find $\Theta_i(k)>0$ and $V_i(k)\in\G(\mathcal{Q}-1)$ such that whenever $f\in\lip(\overline{B(0,k)})$, we have:
\begin{equation}
\begin{split}
      \bigg\lvert\int f d\frac{T_{x,r_i}\phi}{r_i^{\mathcal{Q}-1}}-\Theta_i(k)\int fd\mathcal{S}^{\mathcal{Q}-1}\llcorner V_i(k)\bigg\rvert\leq& F_{0,k}\Big(r_i^{-(\mathcal{Q}-1)}T_{x,r_i}\phi, \Theta_i(k)\mathcal{S}^{\mathcal{Q}-1}\llcorner x V_i(k)\Big)\\
      \leq& 2k^\mathcal{Q}d_{0,k}\Big(r_i^{-(\mathcal{Q}-1)}T_{x,r_i}\phi,\mathfrak{M}\Big)=2k^\mathcal{Q}d_{x,kr_i}(\phi,\mathfrak{M}),
    \label{eq:1001}
\end{split}
\end{equation}
where the last identity comes from Proposition \ref{prop:stab}(i). Putting \eqref{eq:1000} and \eqref{eq:1001} together with the fact that we are assuming that (i) holds, we infer:
\begin{equation}
    \lim_{i\to\infty} \bigg\lvert\int f d\nu-\Theta_i(k)\int fd\mathcal{S}^{\mathcal{Q}-1}\llcorner V_i(k)\bigg\rvert=0,\text{ for any }f\in\lip(\overline{B(0,k)}).
    \label{eq:1002}
\end{equation}
Defined $\varphi(z):=\max\{k-d(0,z),0\}$, we deduce that:
\begin{equation}
\begin{split}
    0=\limsup_{i\to\infty} \bigg\lvert\int \varphi d\nu-\Theta_i(k)\int \varphi d\mathcal{S}^{\mathcal{Q}-1}\llcorner V_i(k)\bigg\rvert\geq\frac{k^{\mathcal{Q}}}{\mathcal{Q}}\limsup_{i\to\infty} \Theta_i(k)-\int \varphi d\nu,
\end{split}
\label{eq:2010}
    \end{equation}
    where the last identity comes Proposition \ref{unif} and the fact that $\varphi$ is radial. The bound \eqref{eq:2010} implies in particular that: 
    $$0\leq \limsup_{i\to\infty} \Theta_i(k)\leq \frac{\mathcal{Q}}{k^\mathcal{Q}}\int \varphi d\nu.$$
  Therefore, there are a $\Theta(k)\geq 0$, an $\mathfrak{n}(k)\in V_1$ and (not relabeled) a subsequence for which:
  $$\lim_{i\to\infty} \Theta_i(k)=\Theta(k)\qquad\text{and}\qquad\lim_{i\to \infty} \mathfrak{n}(V_i(k))=\mathfrak{n}(k),$$ 
  where $\mathfrak{n}(k)=\mathfrak{n}(V(k))$ for some $V(k)\in\G(\mathcal{Q}-1)$. This implies that for any $f\in\lip(\overline{B(0,k)})$ we have:
\begin{equation}
    \begin{split}
        \bigg\lvert\int f d\nu-\Theta(k)\int f d\mathcal{S}^{\mathcal{Q}-1}&\llcorner V(k)\bigg\rvert\leq \bigg\lvert\int f d\nu-\Theta_i(k)\int f d\mathcal{S}^{\mathcal{Q}-1}\llcorner V_i(k)\bigg\rvert\\
        +&\Theta_i(k)\bigg\lvert\int f d\mathcal{S}^{\mathcal{Q}-1}\llcorner V-\int f d\mathcal{S}^{\mathcal{Q}-1}\llcorner V_i(k)\bigg\rvert
        +\lvert\Theta(k)-\Theta_i(k)\rvert\int f d\mathcal{S}^{\mathcal{Q}-1}\llcorner V(k).
        \label{eq:1011}
    \end{split}
\end{equation}
Since the sequence $\{\Theta_i(k)\}_{i\in\N}$ is bounded, thanks to Proposition \ref{prop:pianconv} and the fact that inequality \eqref{eq:1011} holds for any $i\in\N$, we infer that:
\begin{equation}
\bigg\lvert\int f d\nu-\Theta(k)\int f d\mathcal{S}^{\mathcal{Q}-1}\llcorner V(k)\bigg\rvert=0,\text{ for any }f\in\lip(\overline{B(0,k)}).
    \label{eq:num2011}
\end{equation}
Since Lipschitz functions with support contained in $B(0,k)$ are dense in $L^1(\overline{B(0,k)})$, thanks to \eqref{eq:num2011}  we conclude that $\nu\llcorner B(0,k)=\Theta(k)\mathcal{S}^{\mathcal{Q}-1}\llcorner V(k)\cap B(0,k)$.
Let $k_1\leq k_2$ and note that:
$$\Theta(k_2)\mathcal{S}^{\mathcal{Q}-1}\llcorner V(k_2)\cap B(0,k_1)=\nu\llcorner B(0,k_1)=\Theta(k_1)\mathcal{S}^{\mathcal{Q}-1}\llcorner V(k_1)\cap B(0,k_1).$$
The above identity yields $\Theta(k_1)=\Theta(k_2)$,  $V(k_1)=V(k_2)$ and in particular  $\nu=\Theta(1)\mathcal{S}^{\mathcal{Q}-1}\llcorner V(1)$.

We are left to prove the viceversa. Assume by contradiction that (i) does not hold. This implies that we can find a $k>0$, an $\epsilon>0$ and an infinitesimal sequence $\{r_i\}_{i\in\N}$ such that:
\begin{equation}
    \liminf_{r_i\to 0}d_{x,kr_i}(\phi,\mathfrak{M})>\epsilon.
    \label{eq:n103}
\end{equation}
Thanks to the weak-$*$ pre-compactness of the sequence $r_i^{-(\mathcal{Q}-1)}T_{x,r_i}\phi$ and the fact that we are assuming that (ii) holds, we can find a (non-relabeled) sequence $\{r_i\}_{i\in\N}$ a $\Theta>0$ and a $V\in \G(\mathcal{Q}-1)$ such that
$r_i^{-(\mathcal{Q}-1)}T_{x,r_i}\phi\rightharpoonup \Theta \mathcal{S}^{\mathcal{Q}-1}\llcorner V$. For any $i\in \N$ let $f^*_i\in\text{Lip}_1(\overline{B(0,k)})$ be such that:
$$d_{0,k}(r_i^{-(\mathcal{Q}-1)}T_{x,r_i}\phi,\mathfrak{M})\leq 2 \bigg\lvert\int f^*_i d\frac{T_{x,r_i}\phi}{r_i^{\mathcal{Q}-1}}-\Theta \int f^*_i d\mathcal{S}^{\mathcal{Q}-1}\llcorner V\bigg\rvert,$$
and note that since $\lip(\overline{B(0,k)})$ is compact when endowed with the supremum distance, we can assume without loss of generality that $f_i^*$ is uniformly converging to some $f^*\in\lip(\overline{B(0,k)})$.
Thanks to Proposition \ref{prop:stab} this implies that:
\begin{equation}
  \liminf_{r_j\to 0}d_{x,kr_i}(\phi,\mathfrak{M})= \liminf_{r_j\to 0} d_{0,k}(r_i^{-(\mathcal{Q}-1)}T_{x,r_i}\phi,\mathfrak{M})\leq 2\liminf_{r_j\to 0} \bigg\lvert\int f^*_i d\frac{T_{x,r_i}\phi}{r_i^{\mathcal{Q}-1}}-\Theta \int f^*_i d\mathcal{S}^{\mathcal{Q}-1}\llcorner V\bigg\rvert=0,
  \nonumber
\end{equation}
where the last identity follows from the assumption that $r_i^{-(\mathcal{Q}-1)}T_{x,r_i}\phi\rightharpoonup \Theta \mathcal{S}^{\mathcal{Q}-1}\llcorner V$, contradicting \eqref{eq:n103}
\end{proof}

\begin{proposizione}\label{loc:tgbis}
Let $\delta\in \N$ and assume $\phi$ is a Radon measure such that for $\phi$-almost any $x\in \mathbb{G}$ we have:
\begin{itemize}
 \item[(i)] 
       $\delta^{-1}\leq \Theta^{\mathcal{Q}-1}_*(\phi,x)\leq\Theta^{\mathcal{Q}-1,*}(\phi,x)\leq \delta$,
    \item[(ii)]$\limsup_{r\to 0}d_{x,kr}(\phi, \mathfrak{M})<(2^{\mathcal{Q}}\delta)^{-1}$.
\end{itemize}
Then for any $\phi$-almost all $x\in\mathbb{G}$ and any $k>0$ we have:
$$\limsup_{r\to 0}d_{x,kr}(\phi, \mathfrak{M})= \sup\{d_{0,k}(\nu,\mathfrak{M}):\nu\in \Tan_{\mathcal{Q}-1}(\phi,x)\}.$$
\end{proposizione}

\begin{proof}
Since $\phi$ is locally finite, we can assume without loss of generality that it is supported on a compact set $K$. For any $\epsilon>0$, we let $C_\epsilon$ be a compact subset of $K$ such that:
\begin{enumerate}
 \item $\phi(K\setminus C_\epsilon)<\epsilon\phi(K)$,
    \item for any $x\in C_\epsilon$ we have
$\delta^{-1}\leq \Theta^{\mathcal{Q}-1}_*(\phi,x)\leq\Theta^{\mathcal{Q}-1,*}(\phi,x)\leq \delta$.
\end{enumerate}
Fix a point $x\in C_\epsilon$ such that $\Tan_{\mathcal{Q}-1}(\phi\llcorner C_\epsilon,x)=\Tan_{\mathcal{Q}-1}(\phi,x)\neq \emptyset$ and recall that such a choice can be done without loss of generality thanks to Proposition \ref{tg:nonempty}.

Suppose $\{r_i\}_{i\in\N}$ is an infinitesimal sequence such that $\lim_{i\to \infty}d_{x,kr_i}(\phi, x)=\limsup_{r\to 0} d_{x,kr}(\phi, x)$ and assume without loss of generality $\nu$ is an element of $\Tan_{\mathcal{Q}-1}(\phi,x)$ such that:
$$r_i^{-(\mathcal{Q}-1)}T_{x,r_i}\phi\rightharpoonup \nu.$$
Let us prove that $\limsup_{r\to 0}d_{x,kr}(\phi,\mathfrak{M})\leq d_{0,k}(\nu,\mathfrak{M})$. For any $0<\eta<1$ let $\Theta\mathcal{S}^{\mathcal{Q}-1}\llcorner V$ be an element of $\mathfrak{M}$ such that $F_{0,k}(\nu, \Theta\mathcal{S}^{\mathcal{Q}-1}\llcorner V)/k^\mathcal{Q}\leq d_{0,k}(\nu,\mathfrak{M})+\eta$. With this choice, thanks to the triangular inequality we infer that:
\begin{equation}
    \begin{split}
        \limsup_{i\to\infty} d_{0,k}(r_i^{-(\mathcal{Q}-1)}T_{x,r_i}\phi,\mathfrak{M})\leq &\limsup_{i\to\infty}\frac{F_{0,k}(r_i^{-(\mathcal{Q}-1)}T_{x,r_i}\phi,\Theta\mathcal{S}^{\mathcal{Q}-1}\llcorner V)}{k^{\mathcal{Q}}}\\
        \leq & \limsup_{i\to\infty}\frac{F_{0,k}(r_i^{-(\mathcal{Q}-1)}T_{x,r_i}\phi,\nu)+F_{0,k}(\nu,\Theta\mathcal{S}^{\mathcal{Q}-1}\llcorner V)}{k^{\mathcal{Q}}}\leq d_{0,k}(\nu,\mathfrak{M})+\eta,
    \end{split}
\end{equation}
where the last inequality comes from the choice of $\Theta$ and $V$ and Proposition \ref{prop:convF}.
The arbitrariness of $\eta$ concludes the proof of the first claim.

As a second and final step of the proof, let us fix a $\nu\in\Tan_{\mathcal{Q}-1}(\phi,x)$ and show that $\limsup_{r\to 0}d_{x,kr}(\phi,\mathfrak{M})\geq d_{0,k}(\nu,\mathfrak{M})$. Since $\nu\in\Tan_{\mathcal{Q}-1}(\phi,x)$, we can find an infinitesimal sequence $\{r_i\}_{i\in\N}$ such that:
$$r_i^{-(\mathcal{Q}-1)}T_{x,r_i}\phi\rightharpoonup \nu.$$
Furthermore, for any $0<\eta<1$ and any $i\in\N$ there exists a $\Theta_i>0$ and a $V_i\in\G(\mathcal{Q}-1)$ such that: $$\frac{F_{0,k}(r_i^{-(\mathcal{Q}-1)}T_{x,r_i}\phi,\Theta_i \mathcal{S}^{\mathcal{Q}-1}\llcorner V_i)}{k^\mathcal{Q}}\leq d_{0,k}(r_i^{-(\mathcal{Q}-1)}T_{x,r_i}\phi,\mathfrak{M})+\eta=d_{x,kr}(\phi,\mathfrak{M})+\eta,$$
where the last identity above comes from Proposition \ref{prop:stab}(i). In addition to this, thanks to assumption (ii) and since the sequence $\{r_i\}_{i\in\N}$ is infinitesimal, there exists a $i_0\in\N$ such that for any $i\geq i_0$ we have $d_{x,kr_i}(\phi,\mathfrak{M})\leq (2^\mathcal{Q}\delta)^{-1}$. This implies for any $i\geq i_0$ that:
\begin{equation}
    \bigg\lvert\int g(w) d\frac{T_{x,r_i}\phi(w)}{r_i^{\mathcal{Q}-1}}-\Theta_i\int g(w) d\mathcal{S}^{\mathcal{Q}-1}\llcorner V_i(w)\bigg\rvert\leq F_{0,k}(r_i^{-(\mathcal{Q}-1)}T_{x,r_i}\phi,\Theta_i \mathcal{S}^{\mathcal{Q}-1}\llcorner V_i)\leq (2^{\mathcal{Q}}\delta)^{-1}k^\mathcal{Q}+\eta k^\mathcal{Q},
    \label{eq:nummm10}
\end{equation}
where $g(x):=\min\{\dist(x,B(0,k)^c),(1-\eta) k\}$. Thanks to the definition of $g$ and to \eqref{eq:nummm10} we infer:
\begin{equation}
\begin{split}
      \Theta_i ((1-\eta)k)^{\mathcal{Q}}-\frac{(1-\eta)k\phi(B(x,k r_i))}{r_i^{\mathcal{Q}-1}}\leq &\Theta_i\int_{B(0,(1-\eta)k)} g(w) d\mathcal{S}^{\mathcal{Q}-1}\llcorner V_i(w)-\int_{B(0,k)} (1-\eta)k~ d\frac{T_{x,r_i}\phi(w)}{r_i^{\mathcal{Q}-1}}\\
      \leq & \bigg\lvert\int g(w) d\frac{T_{x,r_i}\phi(w)}{r_i^{\mathcal{Q}-1}}-\Theta_i\int g(w) d\mathcal{S}^{\mathcal{Q}-1}\llcorner V_i(w)\bigg\rvert\leq (2^{\mathcal{Q}}\delta)^{-1}k^\mathcal{Q}+\eta k^\mathcal{Q},
\end{split}
\label{eq:nummm12}
\end{equation}
With a similar argument, one can also prove that:
\begin{equation}
\begin{split}
   \frac{(1-\eta)k\phi(B(x,(1-\eta)k r_i))}{r_i^{\mathcal{Q}-1}}-\Theta_i(1-\eta)k^{\mathcal{Q}}\leq (2^{\mathcal{Q}}\delta)^{-1}k^\mathcal{Q}+\eta k^\mathcal{Q}.
\end{split}
\label{eq:nummm13}
\end{equation}
Rearranging inequality \eqref{eq:nummm12} and dividing both sides by $k^{\mathcal{Q}}$, thanks to the choice of $x$ and to the arbitrariness of $i$, we have:
$$\limsup_{i\to\infty}  (1-\eta)^{\mathcal{Q}}\Theta_i\leq  \limsup_{i\to\infty} \frac{(1-\eta)\phi(B(x,kr_i))}{(kr_i)^{\mathcal{Q}-1}}+ (2^{\mathcal{Q}}\delta)^{-1}+\eta\leq \delta+(2^{\mathcal{Q}}\delta)^{-1}+\eta.$$
Thanks to the arbitrariness of $0<\eta<1$ we finally deduce that $\limsup_{i\to\infty}\Theta_i\leq \delta+(2^{\mathcal{Q}}\delta)^{-1}$.

Similarly, rearranging inequality \eqref{eq:nummm13} and dividing both sides by $k^{\mathcal{Q}}$, thanks to the arbitrariness of $i$ we infer that:
$$(1-\eta)^{\mathcal{Q}-1}\Theta_*^{\mathcal{Q}-1}(\phi,x)-(2^{\mathcal{Q}}\delta)^{-1}-\eta=\liminf_{i\to\infty}(1-\eta)^{\mathcal{Q}-1}\frac{\phi(B(x,(1-\eta) r_i))}{(1-\eta)^{\mathcal{Q}-1}(kr_i)^{\mathcal{Q}-1}}-(2^{\mathcal{Q}}\delta)^{-1}-\eta \leq (1-\eta)\liminf_{i\to\infty} \Theta_i.$$
The arbitrariness of $\eta$ and the choice of $x$ conclude that $\liminf_{i\to\infty} \Theta_i\geq (1-2^{-\mathcal{Q}})\delta^{-1}$.

Thus, up to subsequences we can assume that $\Theta_i$ converge to some $\Theta\in [(1-2^{-\mathcal{Q}})\delta^{-1},\delta+(2^{\mathcal{Q}}\delta)^{-1}]$ and that there exists a $V\in \G(\mathcal{Q}-1)$ such that $\mathfrak{n}(V_i)\to \mathfrak{n}(V)$. By Proposition \ref{prop:pianconv} this implies that $\Theta_i\mathcal{S}^{\mathcal{Q}-1}\llcorner V_i\rightharpoonup \Theta\mathcal{S}^{\mathcal{Q}-1}\llcorner V$. Therefore, thanks to the triangular inequality, this implies for any $i\in\N$ that:
\begin{equation}
\begin{split}
       d_{0,k}(\nu,\mathfrak{M})
       \leq& \frac{F_{0,k}(\nu,r_i^{-(\mathcal{Q}-1)}T_{x,r_i}\phi)+F_{0,k}(r_i^{-(\mathcal{Q}-1)}T_{x,r_i}\phi,\Theta_i\mathcal{S}^{\mathcal{Q}-1}\llcorner V_i)+F_{0,k}(\Theta_i\mathcal{S}^{\mathcal{Q}-1}\llcorner V_i,\Theta\mathcal{S}^{\mathcal{Q}-1}\llcorner V)}{k^\mathcal{Q}}\\
       \leq &\frac{F_{0,k}(\nu,r_i^{-(\mathcal{Q}-1)}T_{x,r_i}\phi)}{k^\mathcal{Q}}+d_{0,k}(r_i^{-(\mathcal{Q}-1)}T_{x,r_i}\phi,\mathfrak{M})+\eta+\frac{F_{0,k}(\Theta_i\mathcal{S}^{\mathcal{Q}-1}\llcorner V_i,\Theta\mathcal{S}^{\mathcal{Q}-1}\llcorner V)}{k^\mathcal{Q}}.
\end{split}
\nonumber
\end{equation}
Finally, thanks to the arbitrariness of $i$ and of $\eta$ and to Proposition \ref{prop:convF}, we conclude that:
$$d_{0,k}(\nu,\mathfrak{M})\leq \limsup_{i\to\infty} d_{0,k}(r_i^{-(\mathcal{Q}-1)}T_{x,r_i}\phi,\mathfrak{M}).$$
\end{proof}

\begin{Notation}\label{definizione:piano}
Throughout Section \ref{sec:main} we let $0<\newep\label{eps:1}<1/10$ be a fixed constant. Proposition \ref{prop:BIGGI} yields two natural numbers $\vartheta,\gamma\in\N$, that from now on we consider fixed, such that:
$$\phi(K\setminus E(\vartheta,\gamma))\leq \oldep{eps:1}\phi(K).$$
These $\vartheta$ and $\gamma$ have the further property, again thanks to Proposition \ref{prop:BIGGI}, that for any $\mu\geq 4 \vartheta$ there is a $\nu\in\N$ for which:
$$\phi(K\setminus \mathscr{E}_{\vartheta,\gamma}(\mu,\nu))\leq \oldep{eps:1}\phi(K).$$
Furthermore, we define $\eta:=1/\mathcal{Q}$ and let: $$\delta_\mathbb{G}:=\min\Bigg\{\frac{1}{2^{4(\mathcal{Q}+1)}\vartheta},\frac{\eta^{\mathcal{Q}+1}(1-\eta)^{\mathcal{Q}^2-1}}{(32\vartheta)^{\mathcal{Q}+1}}\Bigg\}.$$
Eventually, if $\tilde{d}_{x,r}(\phi,\mathfrak{M})\leq \delta$ for some $0<\delta<\delta_{\mathbb{G}}$, we define $\Pi(x,r)$ to be the subset of planes $V\in\G(\mathcal{Q}-1)$ for which there exists a $\Theta>0$ and a $z\in\mathbb{G}$ such that:
$$F_{x,r}(\phi, \Theta \mathcal{S}^{\mathcal{Q}-1}\llcorner zV)\leq 2\delta r^{\mathcal{Q}}.$$
\end{Notation}

The following two propositions are the main results of this subsection. They are so relevant since they give a more geometric interpretation of the condition \say{flatness of the tangents} and In particular tell us that $E(\vartheta,\gamma)$ is a weakly linearly approximable set. For a discussion on how this will play a role in the proof of the main result of this work, we refer to the Introduction.

\begin{proposizione}\label{prop1vsinfty}
Let $x\in E^\phi(\vartheta,\gamma)$ be such that $\tilde{d}_{x,r}(\phi,\mathfrak{M})\leq \delta$ for some $\delta<\delta_\mathbb{G}$ and $0<r<1/\gamma$. Then for every $V\in \Pi(x,r)$ we have:
$$\sup_{w\in E^\phi(\vartheta,\gamma)\cap B(x,r/4)} \frac{\dist\big(w,xV\big)}{r}\leq2^{2+3/\mathcal{Q}}\vartheta^{1/\mathcal{Q}}\delta^{1/\mathcal{Q}}=: \newC\label{C:b1}\delta^{1/\mathcal{Q}}.$$
\end{proposizione}

\begin{proof}
Let $V$ be any element of $\Pi(x,r)$ and suppose $z\in \mathbb{G}$, $\Theta>0$ are such that:
$$\bigg\lvert \int f d\phi-\Theta\int fd\mathcal{S}^{\mathcal{Q}-1}\llcorner zV\bigg\rvert\leq 2\delta r^{\mathcal{Q}},\text{ for any }f\in\lip(\overline{B(x,r)}).$$
Since the function $g(w):=\min\{\dist(w,B(x,r)^c),\dist(w,zV)\}$ belongs to $\lip(\overline{B(x,r)})$, we deduce that:
$$2\delta r^{\mathcal{Q}}\geq\int g(w)d\phi(w)-\Theta\int g(w)d\mathcal{S}^{\mathcal{Q}-1}\llcorner zV=\int g(w)d\phi(w)\geq \int_{B(x,r/2)}\min\{r/2,\dist(w,zV)\}d\phi(w).$$
Suppose that $y$ is a point in $\overline{B(x,r/4)}\cap E^\phi(\vartheta,\gamma)$ furthest from $zV$ and let $D=\dist(y,zV)$.
If $D\geq r/8$, this would imply that:
\begin{equation}
\begin{split}
 2\delta r^{\mathcal{Q}}\geq&\int_{B(x,r/2)}\min\{r/2,\dist(w,zV)\}d\phi(w)\geq \int_{B(y,r/16)}\min\{r/2,\dist(w,zV)\}d\phi(w)
 \geq\frac{r}{16}\phi(B(y,r/16))\geq \frac{r^\mathcal
Q}{\vartheta 16^\mathcal{Q}},
    \nonumber
\end{split}
\end{equation}
which is not possible thanks to the choice of $\delta$.
This implies that $D\leq r/8$ and as a consequence, we have:
\begin{equation}
\begin{split}
     2\delta r^{\mathcal{Q}}\geq \int_{B(x,r/2)}\min\{r/2,\dist(w,zV)\}d\phi(w)\geq& \int_{B(y,D/2)}\min\{r/2,\dist(w,zV)\}d\phi(w)\\
     \geq& \frac{D\phi(B(y,D/2))}{2}\geq \vartheta^{-1}\left(\frac{D}{2}\right)^{\mathcal{Q}},
\end{split}
    \label{eq:2023}
\end{equation}
where the second inequality comes from the fact that $B(y,D/2)\subseteq B(x,r/2)$. This implies thanks to \eqref{eq:2023}, that:
$$\sup_{w\in E(\vartheta,\gamma)\cap B(x,r/4)}\frac{\dist(w,zV)}{r}\leq\frac{D}{r}\leq 2^{1+3/\mathcal{Q}}\vartheta^{1/\mathcal{Q}}\delta^{1/\mathcal{Q}}=\oldC{C:b1}\delta^{1/\mathcal{Q}}/2.$$
In particular we infer that $\dist(x,zV)/r\leq \oldC{C:b1}\delta^{1/\mathcal{Q}}/2$. Therefore, thanks to Proposition \ref{prop:dist-piani}, we deduce:
$$\sup_{w\in E(\vartheta,\gamma)\cap B(x,r/4)}\frac{\dist(w,xV)}{r}\leq \sup_{w\in E(\vartheta,\gamma)\cap B(x,r/4)}\frac{\dist(w,zV)+\dist(xV,zV)}{r}\leq \oldC{C:b1}\delta^{1/\mathcal{Q}}.$$
\end{proof}

\begin{proposizione}\label{prop:bil2}
Let $x\in E^\phi(\vartheta,\gamma)$ and $0<r<1/\gamma$ be such that for some $0<\delta<\delta_\mathbb{G}$ we have: 
\begin{equation}
    d_{x,r}(\phi,\mathfrak{M})+d_{x,r}(\phi\llcorner E^\phi(\vartheta,\gamma),\mathfrak{M})\leq \delta.
    \label{eq:n3}
\end{equation}
Then for any $V\in\Pi(x,r)$ and any $w\in B(x,r/2)\cap xV$ we have
$E^\phi(\vartheta,\gamma)\cap B(w,\delta^\frac{1}{\mathcal{Q}+1} r)\neq \emptyset$.
\end{proposizione}

\begin{osservazione}
The set $\Pi(x,r)$ in the above statement is non-empty since Proposition \ref{prop:stab} insures that \eqref{eq:n3} implies $\tilde{d}_{x,r}(\phi,\mathfrak{M})\leq \delta$.
\end{osservazione}

\begin{proof}
Fix some $V\in\Pi(x,r)$ and suppose that $\Theta>0$ is such that:
$$\frac{F_{x,r}(\phi,\Theta\mathcal{S}^{\mathcal{Q}}\llcorner xV)+F_{x,r}(\phi\llcorner E^\phi(\vartheta,\gamma),\Theta\mathcal{S}^{\mathcal{Q}}\llcorner xV)}{r^{\mathcal{Q}}}\leq 2\delta.$$
Defined $g(x):=\min\{\dist(x,B(0,1)^c),\eta\}$, we deduce that:
\begin{equation}
    \begin{split}
        \vartheta^{-1}(1-\eta)^{\mathcal{Q}-1}\eta r^{\mathcal{Q}}- \Theta \eta r^{\mathcal{Q}}\leq&\eta r\phi\big(B(x,(1-\eta)r)\big)-\eta r\Theta\mathcal{S}^{\mathcal{Q}-1}\llcorner xV(B(x,r))\\
        \leq& \int r g(\delta_{1/r}(x^{-1}z)) d\phi(z)-\Theta\int r g(\delta_{1/r}(x^{-1}z)) d\mathcal{S}^{\mathcal{Q}-1}\llcorner xV\leq2\delta r^\mathcal{Q},
        \nonumber
    \end{split}
\end{equation}
since $rg(\delta_{1/r}(x^{-1}\cdot))\in\lip(\overline{B(x,r)})$.
Simplifying and rearranging the above chain of inequalities, we infer that:
$$\Theta\geq \vartheta^{-1}(1-\eta)^{\mathcal{Q}-1}-2\delta/\eta\geq(2\vartheta)^{-1}(1-\eta)^{\mathcal{Q}-1}=(2\vartheta)^{-1}(1-1/\mathcal{Q})^{\mathcal{Q}-1},$$
where the first inequality comes from the choice of $\delta$ and the last equality from the definition of $\eta$, see Definition \ref{definizione:piano}.
Since the function $\mathcal{Q}\mapsto(1-1/\mathcal{Q})^{\mathcal{Q}-1}$ is decreasing and $\lim_{\mathcal{Q}\to\infty}(1-1/\mathcal{Q})^{\mathcal{Q}-1}=1/e$, we deduce that $\Theta\geq 1/2\vartheta e$.
Suppose that $\delta^{1/(\mathcal{Q}+1)}<\lambda<1/2$ and assume that we can find a $w\in xV\cap B(x,r/2)$ such that $\phi \big(B(w,\lambda r)\cap E(\vartheta,\gamma)\big)=0$. This would imply that:
\begin{equation}
\begin{split}
    \Theta \eta(1-\eta)^{\mathcal{Q}-1}\lambda^{\mathcal{Q}}r^{\mathcal{Q}}=&\Theta\eta \lambda r \mathcal{S}^{\mathcal{Q}-1}\llcorner xV\big(B(w,(1-\eta)\lambda r)\big)\\
\leq& \Theta\int \lambda rg(\delta_{1/\lambda r}(w^{-1}z))d\mathcal{S}^{\mathcal{Q}-1}\llcorner xV(z)\\
=&\Theta\int \lambda rg(\delta_{1/\lambda r}(w^{-1}z))d\mathcal{S}^{\mathcal{Q}-1}\llcorner xV(z)-\int \lambda rg(\delta_{1/\lambda r}(w^{-1}z)) d\phi(z)
\leq2\delta r^{\mathcal{Q}},
\label{eq:n104}
\end{split}
\end{equation}
where the last inequality comes from the choice of $\Theta$, $V$ and the fact that $\lambda rg(\delta_{1/\lambda r}(w^{-1}\cdot))\in\lip(\overline{B(x,r)})$.
Thanks to \eqref{eq:n104}, the choice of $\lambda$ and the fact that $1/4e\vartheta<\Theta$, we have that:
\begin{equation}
  \frac{\delta^{\frac{\mathcal{Q}}{\mathcal{Q}+1}}}{4e\vartheta}\eta(1-\eta)^{\mathcal{Q}-1}<\Theta\lambda^{\mathcal{Q}} \eta(1-\eta)^{\mathcal{Q}-1}\leq2\delta.
  \nonumber
\end{equation}
However, few algebraic computations that we omit show that the above inequality chain is in contradiction with the choice of $\delta<\delta_\mathbb{G}$.
\end{proof}

\subsection{Construction of cones complementing \texorpdfstring{$\supp(\phi)$}{Lg} in case it has big projections on planes}\label{sub:project}

This subsection is devoted to the proof of Proposition \ref{prop:cono} that tells us that if the measure $\phi$ is well approximated inside a ball $B(x,r)$ by some plane $V$ and if there exists some \emph{other} plane $W$ on which the $\mathcal{S}^{\mathcal{Q}-1}$-measure of the projection $P_W(\supp(\phi)\cap B(x,r))$ is comparable with $r^{\mathcal{Q}-1}$, then at scales comparable with $r$ the set $\supp(\phi)$ is a $W$-intrinsic Lipschitz surface. In other words, we can find an $\alpha>0$ such that:
$$y\in zC_W(\alpha)\text{ whenever }y,z\in B(x,r)\text{ and }d(z,y)\gtrsim r.$$ 

Before proceeding with the statement and the proof of Proposition \ref{prop:cono}, we fix some notation that will be extensively used throughout the rest of the paper.

\begin{Notation}\label{notation1}
Throughout this paragraph we assume $\sigma\in\N$ to be a fixed positive natural number. First of all, let us define the following two numbers:
$$\zeta(\sigma):=2^{-50\mathcal{Q}}\sigma^{-2}\qquad \text{and }\qquad N(\sigma):=\lfloor-4\log(\zeta(\sigma))\rfloor+40.$$
Secondly, we let:
\begin{center}
    \begin{tabular}{c c c c}
         $\newC\label{C:1.0}(\sigma):=2^{20}(n_1-1)\oldC{C:b1}(\sigma)^2$, & $\newC\label{C:F}(\sigma):=2^{24\mathcal{Q}}\sigma$, &$\newC\label{C:VIB}(\sigma):=\oldC{C:F}(32\zeta(\sigma)^{-2})^m$, & $\newC\label{child}(\sigma):=2^{\frac{2\log \oldC{C:F}}{m}+N}\zeta(\sigma)^{-2}$.
    \end{tabular}
\end{center}
Finally, we introduce six further new constants that depend only on $\sigma$. Although we could avoid giving an explicit expression for such constants, we choose nonetheless to make them explicit. This is due to a couple of reasons. First of all, having their values helps keeping under control their interactions in proofs, getting more precise statements. Secondly, fixing these constants once and for all, we avoid the practise of choosing them ``large enough'' when necessary. In doing so we hope to help the reader not to get distracted with the problem of whether these choices were legitimate or not. 

For the sake of readibility, we choose not make the dependence on $\sigma$ of the numbers $N,\zeta$ and the constants $\oldC{C:0},\ldots,\oldC{child}$ explicit in the following. We let:
\begin{itemize}
    \item[(i)] $A_0(\sigma):=2 \max\bigg\{\oldC{child},\frac{7\log2-2\log \zeta}{N\log 2-2}\bigg\}$,
    \item[(ii)] $k(\sigma):=40^{N+8}\oldC{C:0}\zeta^{-2}A_0^4(1+e^{8NA_0^2})$ and $0<R<2^{-(N+11)}\zeta^2 k$,
    \item[(iii)] $\varepsilon_\mathbb{G}(\sigma):=\frac{2^{3\mathcal{Q}-n-20}\beta\prod_{j=2}^s \epsilon_i^{n_i}}{(A_0 k)^{\mathcal{Q}-1}\oldC{C:0}^{Q-2}\oldC{C:VIB}^2}$ where $\beta$ is the constant introduced in Proposition \ref{prop:rapp},
    \item[(iv)]
    $\newep\label{e}(\sigma):=\frac{1}{2}\min\bigg\{\delta_{\mathbb{G}},\frac{\varepsilon_{\mathbb{G}}^\mathcal{Q}}{(2^{20}\oldC{C:b1}^2\oldC{C:VIB}^2A_0k)^{\mathcal{Q}+1}\big(1+2kR^{-1}\Lambda(1)(\oldC{C:0}+1)^2\big)^\mathcal{Q}},\Big(\frac{k-20}{20k}\Big)^{\mathcal{Q}+1},\frac{1}{(2A_0^2\oldC{C:1.0}+2A_0k\oldC{C:b1}\oldC{C:F}e^{8NA_0^2})^{\mathcal{Q}(\mathcal{Q}+1)}}\bigg\}$,
    \item[(v)]$\newep\label{eps:3}(\sigma):=\frac{1}{2^{2\mathcal{Q}}\oldC{C:F}^2(A_0\oldC{child})^{\mathcal{Q}-1}}$,
\end{itemize}
Since in the rest of Section \ref{sec:main} we make an extensive use of the dyadic cubes constructed in Appendix \ref{AppendiceA}, we recall here some of the notation. For any Radon measure $\psi$ such that for $\psi$-almost every $x\in \mathbb{G}$ we have:
$$0<\Theta^{\mathcal{Q}-1}_*(\psi,x)\leq \Theta^{\mathcal{Q}-1,*}(\psi,x)<\infty,$$
and any $\xi,\tau\in\N$ we denote by $\Delta^\psi(\xi,\tau)$ the family of dyadic cubes relative to $\psi$ and to the parameters $\xi$ and $\tau$ yielded by Theorem \ref{evev}.
\label{def:compactscube}
Furthermore, for any compact subset  $\kappa$  of $E^\psi(\xi,\tau)$ and $l\in \N$ we let:
\begin{equation}
    \begin{split}
    \Delta^\psi(\kappa;\xi,\tau,l):=&\{Q\in\Delta^\psi(\xi,\tau):Q\cap \kappa\neq \emptyset \text{ and }Q\in\Delta_j^\phi(\xi,\tau)\text{ for some }j\geq l\},
\end{split}
    \nonumber
\end{equation}
where $\Delta_j^\phi(\xi,\tau)$ \label{stratta} is the $j$-th layer of cubes, see Theorem \ref{evev}. Finally, for any $Q\in \Delta^\psi(E^\psi(\xi,\tau);\xi,\tau,1)$, we define: $$\alpha(Q):=\tilde{d}_{\mathfrak{c}(Q),2k\diam{Q}}(\psi,\mathfrak{M})+\tilde{d}_{\mathfrak{c}(Q),2k\diam Q}(\psi\llcorner E^\psi(\xi,\tau),\mathfrak{M}),$$
where $\mathfrak{c}(Q)\in Q$ is the centre of the cube $Q$, see Theorem \ref{evev}.

    Eventually,  we recall for the reader's sake some nomenclature on dyadic cubes that was already introduced in the section Notation at the beginning of the paper. For any couple of dyadic cubes $Q_1,Q_2\in\Delta^\psi(\xi,\tau)$:
\begin{itemize}
\item[(i)] if $Q_1\subseteq Q_2$, then $Q_2$ is said to be an \emph{ancestor} of $Q_1$ and $Q_2$ a \emph{sub-cube} of $Q_2$,
\item[(ii)] if $Q_2$ is the smallest cube for which $Q_1\subsetneq Q_2$, then $Q_2$ is said to be the \emph{parent} of $Q_1$ and $Q_1$ the \emph{child} of $Q_2$.
\end{itemize}
\end{Notation}

\begin{Notation}
If not otherwise stated, in order to simplify notation throughout Section \ref{sec:main} we will always denote by $\Delta:=\Delta^\phi(\vartheta,\gamma)$ the family of dyadic cubes constructed in Theorem \ref{evev} relative to the the measure $\phi$ that was fixed at the beginning of this Section and to the parameters $\vartheta,\gamma$, fixed in Notation \ref{definizione:piano}. 
Furthermore, we let:
$$E(\vartheta,\gamma):=E^\phi(\vartheta,\gamma),\qquad\mathscr{E}(\mu,\nu):=\mathscr{E}_{\vartheta,\gamma}^\phi(\mu,\nu)\qquad\text{and}\qquad\Delta(\kappa,l):=\Delta^\phi(\kappa;\vartheta,\gamma,l).$$

Finally, if the dependence on $\sigma$ of the constants introduced above is \emph{not} specified, we will always assume that $\sigma=\vartheta$, where once again $\vartheta$ is the one natural number fixed in Notation \ref{definizione:piano}. 
\end{Notation}

\begin{osservazione}\label{oss:01}
For any compact set $\kappa$ of $E(\vartheta,\gamma)$, we let $\mathcal{M}(\kappa, l)$ be the set of maximal cubes of $\Delta(\kappa,l)$ ordered by inclusion. The elements of  $\mathcal{M}(\kappa, l)$ are pairwise disjoint and enjoy the following properties:
\begin{itemize}
\item[(i)]for any $Q\in\Delta(\kappa,l)$ there is a cube $Q_0\in\mathcal{M}(\kappa,l)$ such that $Q\subseteq Q_0$,
\item[(ii)] if $Q_0\in \mathcal{M}(\kappa,l)$ and there exists some $Q^\prime\in \Delta(\kappa,l)$ for which $Q_0\subseteq Q^\prime$, then $Q_0=Q^\prime$.
\end{itemize}
\end{osservazione}

The proof of the following proposition is inspired by the argument employed in proving Lemma 2.19 of \cite{DS2} and its counterpart in the first Heisenberg group $\mathbb{H}^1$, Lemma 3.8 of \cite{CFO}.

\begin{proposizione}\label{prop:cono}
Suppose that $Q$ is a cube in $\Delta(E(\vartheta,\gamma),\iota)$ for some $\iota\in\N$ satisfying the two following conditions:
\begin{itemize}
\item[(i)] $\tilde{d}_{\mathfrak{c}(Q),4k\diam Q}(\phi\llcorner E(\vartheta,\gamma),\mathfrak{M})\leq \oldep{e}$,
    \item[(ii)] there exists a plane $W\in\G(\mathcal{Q}-1)$  such that:
\begin{equation}
\frac{\diam Q^{\mathcal{Q}-1}}{4\oldC{C:VIB}^2A_0^{\mathcal{Q}-1}} \leq\mathcal{S}^{\mathcal{Q}-1}\llcorner W\Big(P_W\Big[\mathfrak{c}(Q)^{-1}(Q\cap E(\vartheta,\gamma))\Big]\Big).
\label{eq:60}
\end{equation}
\end{itemize}
Let $x\in E(\vartheta,\gamma)\cap Q$ and $y\in B(x,(k-1)\diam Q/8)\cap E(\vartheta,\gamma)$ be two points for which: 
\begin{equation}
    R \diam Q\leq d(x,y)\leq 2^{N+6}\zeta^{-2}R\diam Q.
    \label{eq:2055}
\end{equation}
Then, for any $\alpha>\Big(\frac{\zeta^2\varepsilon_{\mathbb{G}}}{2^{8+N}R^{-1}k\Lambda(1)(\oldC{C:0}+1)}\Big)^{-1}=:\alpha_0$ we have $y\in xC_{W}(\alpha)$.
\end{proposizione}

\begin{osservazione}
Thanks to the definition of $R$ and $k$, we have:
\begin{equation}
    \begin{split}
         2^{(N+6)}\zeta^{-2}R= 2^{(N+6)}\zeta^{-2}\cdot2^{-(N+11)}\zeta^{2}k=k/32<(k-1)/8.
         \nonumber
    \end{split}
\end{equation}
This implies that $B(x,2^{N+6}\zeta^{-2}R\diam Q)\Subset B(x,(k-1)\diam Q/8)$ and thus the request $d(x,y)\geq R\diam Q$ is compatible with the fact that $y$ is chosen in $ B(x,(k-1)\diam Q/8)$.
\end{osservazione}

\begin{proof}[Proof of Proposition \ref{prop:cono}]
Suppose by contradiction there are two points $x,y\in E(\vartheta,\gamma)$ satisfying the hypothesis of the proposition such that $y\not\in xC_W(\alpha)$ for some $\alpha>\alpha_0$. This implies, since the cone $C_W(\alpha)$ is closed by definition, that we have $\pi_1(x^{-1}y)\neq 0$. Furthermore, Proposition \ref{prop:conifero} together with \eqref{eq:2055} yield:
\begin{equation}
    \diam Q\leq R^{-1}\Lambda(\alpha) \lvert\pi_1(x^{-1}y)\rvert\leq R^{-1}\Lambda(1) \lvert\pi_1(x^{-1}y)\rvert,
    \label{eq:207}
\end{equation}
where the last inequality comes from the fact that $\Lambda$, the function yielded by Proposition \ref{prop:conifero}, is decreasing.
Let $\rho:=\diam(Q)$ and note that Proposition \ref{prop:stab} and the fact that $B(x,4(k-1)\rho)\subseteq B(\mathfrak{c}(Q),4k\rho)$, imply:
$$\tilde{d}_{x,4(k-1)\rho}(\phi\llcorner E(\vartheta,\gamma),\mathfrak{M})\leq (k/(k-1))^{\mathcal{Q}-1}\tilde{d}_{\mathfrak{c}(Q),4k \rho}(\phi\llcorner E(\vartheta,\gamma),\mathfrak{M})\leq  2^{\mathcal{Q}-1}\oldep{e},$$
Therefore, thanks to Proposition \ref{prop1vsinfty} and the choice of $\oldep{e}$ we infer that there exists a plane $V\in \G(\mathcal{Q}-1)$, that we consider fixed throughout the proof, such that:
\begin{equation}
    \sup_{w\in E(\vartheta,\gamma)\cap B(x,(k-1)\rho)}\frac{\dist(w,xV)}{4(k-1)\rho}\leq 2\oldC{C:b1}\oldep{e}^{1/\mathcal{Q}}.
    \label{eq:55}
\end{equation}
Since $y\in B(x,(k-1)\rho)$ we deduce from \eqref{eq:55}, that:
\begin{equation}
    \dist(y,xV)\leq8(k-1) \oldC{C:b1}\oldep{e}^{1/\mathcal{Q}}\rho.
    \label{eq:n10}
\end{equation}

In this paragraph we prove that if there exists a point $v\in V$ such that $v_1\neq 0$ and $\lvert \pi_1(P_W v)\rvert\leq \vartheta \lvert v_1\rvert$ for some $0<\vartheta<1$, then:
\begin{equation}
\lvert \langle \mathfrak{n}(V),\mathfrak{n}(W)\rangle\rvert\leq\vartheta/\sqrt{1-\vartheta^2}. 
\label{eq:claimy1}
\end{equation}
We note that the assumptions on $v_1$ imply that:
\begin{equation}
    \lvert v_1\rvert^2-\langle \mathfrak{n}(W),v_1\rangle^2=\lvert v_1-\langle \mathfrak{n}(W),v_1\rangle \mathfrak{n}(W)\rvert^2=\lvert \pi_\mathscr{W}v_1\rvert^2=\lvert\pi_1(P_Wv)\rvert^2\leq \vartheta^2 \lvert v_1\rvert^2.
    \label{eq:2030}
\end{equation}
By means of few omitted algebraic manipulations of \eqref{eq:2030}, we conclude that $\sqrt{1-\vartheta^2}\lvert v_1\rvert\leq \lvert\langle \mathfrak{n}(W),v_1\rangle\rvert$.
Finally, since $\langle \mathfrak{n}(V),v_1\rangle=0$, thanks to \eqref{eq:2030} and Cauchy-Schwartz inequality, we have:
\begin{equation}
\begin{split}
    \vartheta\lvert v_1\rvert\geq& \lvert\langle \pi_\mathscr{W}v_1,\mathfrak{n}(V)\rangle\rvert =  \lvert\langle v_1-\langle \mathfrak{n}(W),v_1\rangle\mathfrak{n}(W),\mathfrak{n}(V)\rangle\rvert\\
    =&\lvert\langle \mathfrak{n}(W),v_1\rangle\langle \mathfrak{n}(V),\mathfrak{n}(W)\rangle\rvert\geq \sqrt{1-\vartheta^2} \lvert v_1\rvert\lvert \langle \mathfrak{n}(V),\mathfrak{n}(W)\rangle\rvert.
    \label{eq:2051}
\end{split}
\end{equation}
It is immediate to see that \eqref{eq:2051} is equivalent to \eqref{eq:claimy1}, proving the claim.

Given $V,W\in\G(\mathcal{Q}-1)$ and $x,y\in E(\vartheta,\gamma)$ as above, let us construct a $v$ with $v_1\neq 0$ that satisfies $\lvert \pi_1(P_W v)\rvert\leq \vartheta \lvert v_1\rvert$ for a suitably small $\vartheta$. Since $y\not\in xC_W(\alpha)$, thanks to Proposition \ref{prop:proiezioni} we have:
\begin{equation}
    \begin{split}
        \lvert  \pi_1(P_W(x^{-1}y))\rvert \leq& \lVert P_W(x^{-1}y)\rVert< \alpha^{-1}\lVert P_{\mathfrak{n}(W)}(x^{-1}y)\rVert
        = \alpha^{-1} \lvert\langle \mathfrak{n}(W), \pi_1(x^{-1}y)\rangle\rvert
        \leq\alpha^{-1}\lvert \pi_1(x^{-1}y)\rvert.
        \nonumber
    \end{split}
\end{equation}
Defined $v$ to be the point of $V$ for which $d(y, xv)=\dist(y,xV)$, thanks to \eqref{eq:55} and the fact that $y\in B(x,(k-1)\rho/8)$ we have: 
\begin{equation}
\begin{split}
    \lVert v\rVert\leq d(xv,y)&+d(y,x)\leq \dist(y,xV)+(k-1)\rho/8\leq (8\oldC{C:b1}\oldep{e}^{1/\mathcal{Q}}+1/8)k\rho<(k-1)\rho,
    \end{split}
    \nonumber
\end{equation}
    where the last inequality comes from the choice of $\oldep{e}$. Furthermore, thanks to \eqref{eq:207} and \eqref{eq:n10} we have:
    \begin{equation}
\begin{split}
    R\Lambda(1)^{-1}\rho/2 \leq(R\Lambda(1)^{-1}-8\oldC{C:b1}k\oldep{e}^{1/\mathcal{Q}})\rho\leq&\lvert \pi_1(x^{-1}y)\rvert-d(y,xv)\leq\lvert\pi_1(x^{-1}y)\rvert-\lvert\pi_1(y^{-1}xv)\rvert\\
    \leq &\lvert\pi_1(x^{-1}y)-\pi_1(y^{-1}xv)\rvert=\lvert v_1\rvert,
\end{split}
    \label{eq:211a}
\end{equation}
and where the first inequality above, comes from the choice of $\oldep{e}$.
Let us prove that $v$ satisfies the inequality $\lvert \pi_1(P_W v)\rvert\leq 2R^{-1}\Lambda(1)k(8\oldC{C:0}\oldC{C:b1}\oldep{e}^{1/\mathcal{Q}}+2^{6+N}\zeta^{-2}\alpha^{-1})\lvert v_1\rvert$. Since $v\not\in C_W(\alpha)$, thanks to Proposition \ref{prop:proiezioni}:
\begin{equation}
    \begin{split}
        \lvert \pi_1(P_W(v))\rvert\leq& \lvert \pi_1(P_W(v))-\pi_1(P_W(x^{-1}y))\rvert+\lvert \pi_1(P_W(x^{-1}y))\rvert\\
       \leq&\lvert\pi_1(P_W(y^{-1}xv))\rvert+\lVert P_W(x^{-1}y)\rVert
        \leq \lvert\pi_1(P_W(y^{-1}xv))\rvert+\alpha^{-1} \lVert P_{\mathfrak{N}(W)}(x^{-1}y)\rVert\\
        \leq&\lVert P_W(y^{-1}xv)\rVert+\alpha^{-1} \lvert \pi_1(x^{-1}y)\rvert\leq \lVert P_W(y^{-1}xv)\rVert+2^{6+N}\zeta^{-2}R\alpha^{-1} \rho,
        \label{eq:2040}
    \end{split}
\end{equation}
where the last inequality of the last line above comes from \eqref{eq:2055}. Theorem \ref{cor:SSCM2.2.14} together with  \eqref{eq:207}, \eqref{eq:211a} and \eqref{eq:2040} implies that:
\begin{equation}
\begin{split}
    \lvert \pi_1(P_W(v))\rvert\leq& \lVert P_W(y^{-1}xv)\rVert+2^{6+N}\zeta^{-2}R\alpha^{-1}
    \leq \oldC{C:0}\lVert y^{-1}xv\rVert+2^{6+N}\zeta^{-2}R\alpha^{-1}\rho\\
    =&\oldC{C:0}d(y,xV)+2^{6+N}\zeta^{-2}R\alpha^{-1}\rho
    \leq (8\oldC{C:0}\oldC{C:b1}(k-1)\oldep{e}^{1/\mathcal{Q}}+2^{6+N}\zeta^{-2}R\alpha^{-1})\rho\\
    \leq&2R^{-1}\Lambda(1)k(8\oldC{C:0}\oldC{C:b1}\oldep{e}^{1/\mathcal{Q}}+2^{6+N}\zeta^{-2}\alpha^{-1})\lvert v_1\rvert=:\theta(\alpha,\oldep{e})\lvert v_1\rvert.
    \label{eq:2052}
\end{split}
\end{equation}
Thanks to the choice of the constants $\oldep{e}, R$ and $k$ together with some algebraic computations that we omit, it is possible to prove that $\sqrt{1-\theta(\alpha,\oldep{e})^2}\geq 1/2$.
Therefore, since $\lvert \pi_1(P_W(v))\rvert\leq \theta(\alpha,\oldep{e})\lvert \pi_1(v)\rvert$, we deduce thanks \eqref{eq:claimy1} that:
\begin{equation}
    \lvert\langle\mathfrak{n}(V),\mathfrak{n}(W)\rangle\rvert\leq \frac{\theta(\alpha,\oldep{e})}{\sqrt{1-\theta(\alpha,\oldep{e})^2}}\leq2\theta(\alpha,\oldep{e}).
    \label{eq:212}
\end{equation}
Let us prove that \eqref{eq:212} is in contradiction with \eqref{eq:60}.
Suppose that $z\in B(x,(k-1)\rho/8)\cap E(\vartheta,\gamma)$ and note that:
\begin{equation}
\begin{split}
\lvert\langle \mathfrak{n}(V),\pi_1(P_W(x^{-1}z))\rangle\rvert=&\lvert\langle \mathfrak{n}(V),\pi_\mathscr{W}(z_1-x_1)\rangle\rvert\leq \lvert\langle \mathfrak{n}(V),z_1-x_1\rangle\rvert+\lvert\langle \mathfrak{n}(V),\pi_{\mathfrak{n}(W)}(z_1-x_1)\rangle\rvert\\
    \leq& \lvert\langle \mathfrak{n}(V),z_1-x_1\rangle\rvert+\lvert\langle \mathfrak{n}(V),\mathfrak{n}(W)\rangle\rvert\lvert \langle z_1-x_1, \mathfrak{n}(W)\rangle\rvert\\
    \leq &   \lVert P_{\mathfrak{N}(V)}(x^{-1}z)\rVert+d(x,z)\lvert\langle \mathfrak{n}(V),\mathfrak{n}(W)\rangle\rvert= \text{dist}(z,xV)+d(x,z)\lvert\langle \mathfrak{n}(V),\mathfrak{n}(W)\rangle\rvert,
\end{split}
    \label{eq:220}
\end{equation}
where the last identity comes from Proposition \ref{cor:SSCM2.2.14}.
Inequalities \eqref{eq:55}, \eqref{eq:212}, \eqref{eq:220} and the choice of $z$ imply:
\begin{equation}
\begin{split}
    \lvert\langle \mathfrak{n}(V),\pi_1(P_W(x^{-1}z))\rangle\rvert\leq&d(z,xV)+d(x,z)\lvert\langle \mathfrak{n}(V),\mathfrak{n}(W)\rangle\rvert\\
    \leq &8 \oldC{C:b1}k\oldep{e}^{1/\mathcal{Q}}\rho+2\theta(\alpha,\oldep{e})d(x,z)
    \leq8 \oldC{C:b1}k\oldep{e}^{1/\mathcal{Q}}\rho+2\theta(\alpha,\oldep{e})k\rho.
\end{split}
    \label{eq:dist}
\end{equation}
Furthermore, defined $\mathfrak{n}:=\pi_\mathscr{W}(\mathfrak{n}(V))$, it is immediate to see from \eqref{eq:212} that  $\lvert\mathfrak{n}-\mathfrak{n}(V)\rvert\leq 2 \theta(\alpha,\oldep{e})$, which yields thanks to \eqref{eq:55}, \eqref{eq:dist}, the triangular inequality and Proposition \ref{prop:2.2.11} the following bound:
\begin{equation}
\begin{split}
    \lvert\langle\mathfrak{n},\pi_1(P_W(x^{-1}z))\rangle\rvert
    \leq &\lvert\langle\mathfrak{n}(V),\pi_1(P_W(x^{-1}z))\rangle\rvert+\lvert \mathfrak{n}-\mathfrak{n}(V)\rvert\lvert \pi_1(P_W(x^{-1}z))\rvert\\
        \leq &\lvert\langle\mathfrak{n}(V),\pi_1(P_W(x^{-1}z))\rangle\rvert+\lvert \mathfrak{n}-\mathfrak{n}(V)\rvert\lVert P_W(x^{-1}z)\rVert\\
    \leq&(8k \oldC{C:b1}\oldep{e}^{1/\mathcal{Q}}\rho+2\theta(\alpha,\oldep{e})k\rho)+ 2\theta(\alpha,\oldep{e})\oldC{C:0}k\rho=8(\oldC{C:b1}\oldep{e}^{1/\mathcal{Q}}+(\oldC{C:0}+1)\theta(\alpha,\oldep{e}))k\rho.
\end{split}
    \label{eq:importante}
\end{equation}
For the sake of notation, we introduce the following set:
$$S:=\{w\in W:\lvert\langle\mathfrak{n},w_1\rangle\rvert
    \leq8(\oldC{C:b1}\oldep{e}^{1/\mathcal{Q}}+(\oldC{C:0}+1)\theta(\alpha,\oldep{e}))k\rho\}.$$
The bound \eqref{eq:importante} implies that the projection of $x^{-1}E(\vartheta,\gamma)\cap B(0,(k-1)\rho/8)$ on $W$ is contained in $S$, which is a very narrow strip around $V\cap W$ inside $W$. Furthermore, we recall that thanks to Proposition \ref{cor:SSCM2.2.14} we have:
\begin{equation}
    P_W(B(0,(k-1)\rho/8))\subseteq B(0,\oldC{C:0}(k-1)\rho/8).
\label{eq:n11}
\end{equation}
Finally, putting together  \eqref{eq:importante} and \eqref{eq:n11}, we deduce that:
\begin{equation}
\begin{split}
        P_W\Big(x^{-1}E(\vartheta,\gamma)\cap B(0,(k-1)\rho/8)\Big)\subseteq P_W\big(x^{-1}E(\vartheta,\gamma)\big)\cap P_W\big(B(0,(k-1)\rho/8)\big)
        \subseteq S \cap B(0,\oldC{C:0}(k-1)\rho/8).
        \label{eq:n12}
\end{split}
\end{equation}
Completing $\{\mathfrak{n}(W),\mathfrak{n}\}$ to an orthonormal basis $\mathcal{E}:=\{\mathfrak{n}(W),\mathfrak{n}/\lvert\mathfrak{n}\rvert, e_3,\ldots, e_n\}$ of $\R^n$ satisfying \eqref{eq:orthbasis} thanks to Remark \ref{box}, we have:
\begin{equation}
    S \cap B(0,\oldC{C:0}(k-1)\rho/8)\subseteq S\cap \text{Box}_{\mathcal{E}}(0,\oldC{C:0}k\rho/8).
    \label{eq:n14}
\end{equation}
The above inclusion together with Tonelli's theorem yieds:
\begin{equation}
\begin{split}
        \mathcal{H}_{eu}^{n-1}\llcorner W\Big(P_W\Big(x^{-1}E(\vartheta,\gamma)\cap B(0,(k-1)\rho/8)\Big)\Big)\leq& \mathcal{H}_{eu}^{n-1}\llcorner W(S \cap B(0,\oldC{C:0}(k-1)\rho/8))\\\leq \mathcal{H}_{eu}^{n-1}\llcorner W(S\cap \text{Box}_{\mathcal{E}}(0,\oldC{C:0}k\rho/8))
        =&16\big(\oldC{C:b1}\oldep{e}^{1/\mathcal{Q}}+(\oldC{C:0}+1)\theta(\alpha,\oldep{e})\big)k\rho\cdot2^{n-2}\prod_{i=2}^s\epsilon_i^{-n_i}\bigg(\frac{\oldC{C:0}k\rho}{8}\bigg)^{\mathcal{Q}-2}\\
        =&2^{n-3\mathcal{Q}+10}\oldC{C:0}^{\mathcal{Q}-2}\prod_{i=2}^s\epsilon_i^{-n_i}\big(\oldC{C:b1}\oldep{e}^{1/\mathcal{Q}}+(\oldC{C:0}+1)\theta(\alpha,\oldep{e})\big)(k\rho)^{\mathcal{Q}-1}.
\end{split}
\label{eq:n15E}
\end{equation}
The inclusion \eqref{eq:n12}, the bound \eqref{eq:n15E}, Proposition \ref{prop:rapp} and the definition of $A_0$ finally imply that:
\begin{equation}
    \begin{split}
        &\mathcal{S}^{\mathcal{Q}-1}\llcorner W\Big(P_W\Big(x^{-1}E(\vartheta,\gamma)\cap B(0,(k-1)\rho/8)\Big)\Big)\leq
        \mathcal{S}^{\mathcal{Q}-1}\llcorner W\Big( S\cap B(0,\oldC{C:0}k\rho/8
)\Big)\\
=&\beta^{-1}\mathcal{H}_{eu}^{n-1}\llcorner W\Big( S\cap B(0,\oldC{C:0}k\rho/8
)\Big)
\leq \beta^{-1}2^{n-3\mathcal{Q}+10}\oldC{C:0}^{\mathcal{Q}-2}\prod_{i=2}^s\epsilon_i^{-n_i}\big(\oldC{C:b1}\oldep{e}^{1/\mathcal{Q}}+(\oldC{C:0}+1)\theta(\alpha,\oldep{e})\big)(k\rho)^{\mathcal{Q}-1}\\
= &2^{-10}\oldC{C:VIB}^{-2}\varepsilon_{\mathbb{G}}^{-1}A_0^{-(\mathcal{Q}-1)}\big(\oldC{C:b1}\oldep{e}^{1/\mathcal{Q}}+(\oldC{C:0}+1)\theta(\alpha,\oldep{e})\big)\rho^{\mathcal{Q}-1},
   \label{eq:2054}
    \end{split}
\end{equation}
where the last identity comes from the definition of $\varepsilon_\mathbb{G}$, see Notation \ref{notation1}.
Furthermore, since $\mathcal{S}^{\mathcal{Q}-1}\llcorner W(P_W(p*E)=\mathcal{S}^{\mathcal{Q}-1}\llcorner W(P_W(E))$ for any measurable set $E$ in $\mathbb{G}$, see for instance the proof of Proposition 2.2.19 in \cite{MFSSC}, we deduce that:
\begin{equation}
\begin{split}
    \mathcal{S}^{\mathcal{Q}-1}\llcorner W\Big(P_W\Big(x^{-1}E(\vartheta,\gamma)\cap B(0,(k-1)\rho/8)\Big)\Big)=\mathcal{S}^{\mathcal{Q}-1}\llcorner W\Big(P_W\Big(\mathfrak{c}(Q)^{-1}E(\vartheta,\gamma)\cap B(\mathfrak{c}(Q)^{-1}x,(k-1)\rho/8)\Big)\Big).
    \nonumber
    \end{split}
\end{equation}
    Thanks to the choice of $k$ and the fact that $x\in Q$, we infer that $B(0,\rho)\subseteq B(\mathfrak{c}(Q)^{-1}x,(k-1)\rho/8)$. Together with \eqref{eq:60}, this allows us to deduce that:
    \begin{equation}
\begin{split}
    \mathcal{S}^{\mathcal{Q}-1}\llcorner W\Big(P_W\Big(x^{-1}E(\vartheta,\gamma)\cap B(0,(k-1)\rho/8)\Big)\Big)
    \geq&\mathcal{S}^{\mathcal{Q}-1}\llcorner W\Big(P_W\Big(\mathfrak{c}(Q)^{-1}E(\vartheta,\gamma)\cap B(0,\rho)\Big)\Big)
    \\\geq &\mathcal{S}^{\mathcal{Q}-1}\llcorner W\Big(P_W\big(\mathfrak{c}(Q)^{-1}(E(\vartheta,\gamma)\cap Q)\big)\Big)\geq \frac{\rho^{\mathcal{Q}-1}}{4\oldC{C:VIB}^2A_0^{\mathcal{Q}-1}}.
\end{split}
    \label{eq:301}
\end{equation}
Putting together \eqref{eq:2054} and \eqref{eq:301} we conclude:
\begin{equation}
\begin{split}
    2^8\varepsilon_{\mathbb{G}}\leq&\big(\oldC{C:b1}\oldep{e}^{1/\mathcal{Q}}+(\oldC{C:0}+1)\theta(\alpha,\oldep{e})\big)
       =\oldC{C:b1}\oldep{e}^{1/\mathcal{Q}}+2R^{-1}k\Lambda(1)(\oldC{C:0}+1)(\oldC{C:0}\oldC{C:b1}\oldep{e}^{1/\mathcal{Q}}+2^{6+N}\zeta^{-2}\alpha^{-1}),
    \nonumber
\end{split}
\end{equation}
The choice of $\oldep{e}$ and $\alpha$ imply, with some algebraic computations that we omit, that the above inequality is false, showing that the assumption $y\not\in xC_W(\alpha)$ is false.
\end{proof}

\subsection{Flat tangents imply big projections}\label{big:proj}

This subsection is devoted to the proof of the following result, that asserts that the hypothesis (ii) of Proposition \ref{prop:cono} is satisfied by $\phi\llcorner E(\vartheta,\gamma)$. More precisely we construct a compact subset $C$ of $E(\vartheta,\gamma)$ having big measure inside $E(\vartheta,\gamma)$ such that:

\begin{teorema}\label{th:projji}
For any cube $Q$ of sufficient small diameter such that $(1-\oldep{eps:3})\phi(Q)\leq \phi(Q\cap C)$, we have:
$$\mathcal{S}^{\mathcal{Q}-1}(P_{\Pi(Q)}(Q\cap C))\geq \frac{\diam Q^{\mathcal{Q}-1}}{2A_0^{\mathcal{Q}-1}}.$$
\end{teorema}

In the following it will be useful to reduce to a compact subset $C$ of $E(\vartheta,\gamma)$ where the distance of $\phi$ from planes is uniformly small under a fixed scale.

\begin{proposizione}\label{proppers}
For any $\mu\geq 4\vartheta$, there exists a $\nu\in\N$, a compact subset $C$ of $\mathscr{E}(\mu,\nu)$ and a $\iota_0\in\N$ such that:
\begin{itemize}
\item[(i)]$\phi(K\setminus C)\leq 2\oldep{eps:1}\phi(K)$,
    \item[(ii)]$d_{x,4kr}(\phi,\mathfrak{M})+d_{x,4kr}(\phi\llcorner E(\vartheta,\gamma),\mathfrak{M})\leq 4^{-\mathcal{Q}} \oldep{e}$  for any $x\in C$  and any $0<r<2^{-\iota_0 N+5}/\gamma$.
\end{itemize}
\end{proposizione}

\begin{proof}
Since by assumption $\Tan_{\mathcal{Q}-1}(\phi,x)\subseteq \mathfrak{M}$ for $\phi$-almost every $x\in \mathbb{G}$, thanks to Proposition \ref{prop:flatty}  we infer that the functions $f_r(x):=d_{x,4kr}(\phi,\mathfrak{M})$ converge $\phi$-almost everywhere to $0$ on $K$ as $r$ goes to $0$. Thanks to Proposition \ref{prop:locality}, the same line of reasoning implies also that $f_r^{\vartheta,\gamma}(x):=d_{x,4kr}(\phi\llcorner E(\vartheta,\gamma),\mathfrak{M})$ converges $\phi$-almost everywhere to $0$ on $E(\vartheta,\gamma)$. 
Proposition \ref{prop:stab} and Severini-Egoroff's theorem yield a compact subset $C$ of $\mathscr{E}(\mu,\nu)$ such that $\phi(E(\vartheta,\gamma)\setminus C)\leq \oldep{eps:1}\phi(E(\vartheta,\gamma))$ and such that the sequence $d_{x,4kr}(\phi,\mathfrak{M})+d_{x,4kr}(\phi\llcorner E(\vartheta,\gamma),\mathfrak{M})$ converges uniformly to $0$ on $C$. This directly implies both (i) and (ii) thanks to the choice of $\vartheta$ and $\gamma$.
\end{proof}

\begin{Notation}
From now on, we assume $\mu\geq 4\vartheta$, the set $C$ and the natural numbers $\nu$ and $\iota_0$ yielded by Proposition \ref{proppers} to be fixed. Furthermore, we define $\iota:=\max\{\iota_0,\nu\}$.
\end{Notation}

The construction of a compact set satisfying the thesis of Proposition \ref{proppers} can be achieved in the following slightly different conditions on the measure:

\begin{proposizione}\label{rk:primo}
Suppose $\psi$ is a Radon on $\mathbb{G}$ supported on the compact set $K$ and satisfying the following assumptions: 
\begin{itemize}
    \item[(i)] there exists a $\delta\in\N$ such that $\delta^{-1}\leq \Theta^{\mathcal{Q}-1}_*(\psi,x)\leq \Theta^{\mathcal{Q}-1,*}(\psi,x)\leq \delta$ for $\psi$-almost every $x\in\mathbb{G}$,
        \item[(ii)] $ \limsup_{r\to 0}d_{x,4kr}(\psi,\mathfrak{M})\leq 4^{-(\mathcal{Q}+1)}\oldep{e}(\delta)$ for $\psi$-almost every $x\in \mathbb{G}$.
\end{itemize}
Then, there exist $\tilde{\iota}_0\in\N$ and $\gamma\in\N$ for which for any $\mu\geq 8\delta$ we can find a $\nu\in\N$ and a compact set $\tilde{C}\subseteq \mathscr{E}_{2\delta,\gamma}^\psi(\mu,\nu)$ such that:
\begin{itemize}
\item[(i)]$\psi(K\setminus \tilde{C})\leq 2\oldep{eps:1}\psi(K)$ where $\oldep{eps:1}$ was introduced in Notation \ref{definizione:piano},
    \item[(ii)]$d_{x,4kr}(\psi,\mathfrak{M})+d_{x,4kr}(\psi\llcorner E^\psi(2\delta,\gamma),\mathfrak{M})\leq 4^{-\mathcal{Q}} \oldep{e}(\delta)$  for any $x\in \tilde{C}$  and any $0<r<2^{-\tilde{\iota}_0 N(2\delta)+5}/\gamma$,
\end{itemize}
where $\oldep{eps:1}$ and $\oldep{e}(\delta)$ are the constants introduced in Notation \ref{definizione:piano} and \ref{notation1} respectively.
\end{proposizione}

\begin{proof}
First of all, thanks to Propositions \ref{prop:cpt} and \ref{prop:BIGGI} we can find a $\gamma\in\N$ such that $\psi(K\setminus\mathscr{E}^\psi_{\vartheta,\gamma}(\mu,\nu))\leq \oldep{eps:1}\psi(K)$. Let us now prove that:
\begin{equation}
    \limsup_{r\to 0}d_{x,4kr}(\psi\llcorner E^\psi(2\delta,\gamma),\mathfrak{M})\leq 4^{-(\mathcal{Q}+1)}\oldep{e}(\delta),\text{ for }\psi\text{-almost every }x\in E^\psi(2\delta,\gamma).
   \nonumber
\end{equation}
Recall that for $\psi$-almost every $x\in E^\psi(2\delta,\gamma)$ we have that $\Tan_{\mathcal{Q}-1}(\psi\llcorner E^\psi(2\delta,\gamma),x)=\Tan_{\mathcal{Q}-1}(\psi,x)$. Thanks to this and the fact that $\oldep{e}(\delta)\leq \delta^{-1}$,  Proposition \ref{loc:tgbis} yields:
$$\limsup_{r\to 0}d_{x,4kr}(\psi,\mathfrak{M})=\limsup_{r\to 0}d_{x,4kr}(\psi\llcorner E^\psi(2\delta,\gamma),\mathfrak{M})\leq 4^{-(\mathcal{Q}+1)}\oldep{e}(\delta).$$

Therefore, for $\psi$-almost every $x\in E^\psi(2\delta,\gamma)$, there exists an $r(x)>0$ such that $d_{x,4kr}(\psi,\mathfrak{M})+d_{x,4kr}(\psi\llcorner E^\psi(2\delta,\gamma),\mathfrak{M})\leq 4^{-\mathcal{Q}}\oldep{e}(\delta)$, for every $0<r<r(x)$.
For any $j\in\N$, define:
$$E_j:=\{x\in\mathscr{E}^\psi_{\vartheta,\gamma}(\mu,\nu):r(x)>1/j\}.$$
Let us show that the sets $E_j$ are Borel. Thanks to Proposition \ref{prop:stab}(ii) we know that the map $x\mapsto d_{x,r}(\psi,\mathfrak{M})$ is continuous and thus for any $r>0$ the set $\Omega_{r}:=\{y\in\mathbb{G}: d_{y,r}(\psi,\mathfrak{M})<4^{-\mathcal{Q}}\oldep{e}(\delta)\}$ is open. In particular if $x\in \Omega_r$ for any $r\in (0,1/j)\cap \Q$, then thanks to Proposition \ref{prop:stab}(iv) we have that $x\in E_j$, and obviously the viceversa holds as well. This shows that the sets $E_j$ are $G_\delta$ sets intersected with the compact set $\mathscr{E}^\psi_{\vartheta,\gamma}(\mu,\nu)$ and thus are Borel.
Furthermore, thanks to the assumption on the measure $\psi$ we infer that:
\begin{equation}
    \psi(E^\psi(2\delta,\gamma)\setminus \bigcup_{j\in\N}E_j)=0.
    \label{eq:coverall}
\end{equation}
Finally, \eqref{eq:coverall} together with the measurability of the sets $E_j$ are measurable, implies that we can find a big enough $j\in\N$ and a compact set $C$ contained in $E_j$ satisfying (i) and (ii) of Proposition \ref{proppers}.
\end{proof}

The following lemma rephrases Propositions \ref{prop1vsinfty} and \ref{prop:bil2} into the language of dyadic cubes. 

\begin{lemma}\label{lemmaduepsei}
For any cube $Q\in \Delta(C,\iota)$ we have $\alpha(Q)\leq \oldep{e}$. Furthermore, there is a plane $\Pi(Q)\in \G(\mathcal{Q}-1)$ for which:
\begin{itemize}
    \item[(i)]$$\sup_{w\in E(\vartheta,\gamma)\cap B(\mathfrak{c}(Q),k\diam Q/2)}\frac{\dist(w,\mathfrak{c}(Q)\Pi(Q))}{2k\diam Q}\leq \oldC{C:b1}\oldep{e}^{1/\mathcal{Q}},$$
    \item[(ii)]for any $w\in B(\mathfrak{c}(Q),k\diam Q/2)\cap \mathfrak{c}(Q)\Pi(Q)$ we have $E(\vartheta,\gamma)\cap B(w,3k\oldC{C:b1}\oldep{e}^{1/(\mathcal{Q}+1)}\diam Q)\neq \emptyset$. 
\end{itemize}
\end{lemma}

\begin{proof}
Let $Q\in\Delta(C,\iota)$, fix a $x\in Q\cap C$ and define $\rho:=\diam Q$. 
Thanks to Proposition \ref{proppers} we know that:
\begin{equation}
    d_{x,4kr}(\psi,\mathfrak{M})+d_{x,4kr}(\psi\llcorner E(\vartheta,\gamma),\mathfrak{M})\leq 4^{-(\mathcal{Q}+1)}\oldep{e},
    \label{eq:n1000}
\end{equation}
for any  $r<2^{-\iota N+5}/\gamma$.
Thanks to Theorem \ref{evev}(iv) we have that $\rho<2^{-\iota N+5}/\gamma$ and thus by Proposition \ref{prop:stab} we have:
\begin{equation}
\tilde{d}_{\mathfrak{c}(Q),2k\rho}(\psi,\mathfrak{M})\leq 2^{\mathcal{Q}}\tilde{d}_{x,4k\rho}(\psi,\mathfrak{M})\leq 2^{-\mathcal{Q}}\oldep{e} \quad\text{and}\quad\tilde{d}_{\mathfrak{c}(Q),2k\rho}(\psi\llcorner E(\vartheta,\gamma),\mathfrak{M})\leq 2^{\mathcal{Q}}\tilde{d}_{x,4k\rho}(\psi\llcorner E(\vartheta,\gamma),\mathfrak{M})\leq 2^{-\mathcal{Q}}\oldep{e} .
\label{eq:n1001}
\end{equation}
The bounds in \eqref{eq:n1001} together with Proposition \ref{prop1vsinfty} imply (i) and that $\alpha(Q)\leq \oldep{e}$.

The proof of the second part of the statement is a little more delicate. Since $C$ is a subset of $\mathscr{E}_{\vartheta,\gamma}(\mu,nu)$, thanks to the choice of $\iota$ and Proposition \ref{prop:centri} we have that $\mathfrak{c}(Q)\in E(\vartheta,\gamma)\cap B(x,\rho)$. Secondly, Proposition \ref{prop:dist-piani}(iii), Proposition \ref{prop1vsinfty} and \eqref{eq:n1001} yield for any $V\in\Pi(x,2k\rho)$ that:
\begin{equation}
    \dist(\mathfrak{c}(Q),xV)\leq \dist(\mathfrak{c}(Q),\mathfrak{c}(Q)V)+\dist(xV,\mathfrak{c}(Q)V)=\dist(x,\mathfrak{c}(Q)V)\leq 2^{-(\mathcal{Q}-2)}k\rho\oldC{C:b1}\oldep{e}^{1/\mathcal{Q}},
    \label{eq:n1004}
\end{equation}
For any $V\in \Pi(x,2k\rho)$ and any $w\in B(\mathfrak{c}(Q),k\rho/2)\cap \mathfrak{c}(Q)V$, we define:
$$w^*:=P_{\mathfrak{N}(V)}(\mathfrak{c}(Q)^{-1}x)^{-1}P_{V}(\mathfrak{c}(Q)^{-1}x)^{-1}wP_{\mathfrak{N}(V)}(\mathfrak{c}(Q)^{-1}x).$$
With few computations that we omit, it is not difficult to see that:
\begin{equation}
    d(\mathfrak{c}(Q) w,x w^*)=\lVert P_{\mathfrak{N}(V)}(\mathfrak{c}(Q)^{-1}x)=\dist(\mathfrak{c}(Q),x V)\leq  2^{-(\mathcal{Q}-2)}k\rho\oldC{C:b1}\oldep{e}^{1/\mathcal{Q}},
    \label{eq:n1005}
\end{equation}
where the second identity follows from Proposition \ref{prop:dist-piani} and the last one from inequality \eqref{eq:n1004}. Thanks to definition of $w^*$ the triangular inequality, Proposition \ref{cor:SSCM2.2.14} and the fact that $d(\mathfrak{c}(Q),x)\leq \rho$, the norm of $w^*$ can be estimated as follows:
\begin{equation}
    \lVert w^*\rVert\leq 2\lVert P_{\mathfrak{N}(V)}(\mathfrak{c}(Q)^{-1}x)\rVert+\lVert P_{V}(\mathfrak{c}(Q)^{-1}x)\rVert+\lVert w\rVert\leq 2\rho +\oldC{C:0}\rho+k\rho/2<k\rho.
    \label{eq:n1002}
\end{equation}
Thanks to inequalities \eqref{eq:n1000} and \eqref{eq:n1002} and  Proposition \ref{prop:bil2} we infer that $B(xw^*, 2k\rho \oldep{e}^{1/(\mathcal{Q}+1)})\cap E(\vartheta,\gamma)\neq \emptyset$. Finally, thanks to  \eqref{eq:n1005} and the fact that $2^{\mathcal{Q}-1}\oldC{C:b1}<1$, we conclude that:
$$E(\vartheta,\gamma)\cap B(\mathfrak{c}(Q)w,3k\rho\oldC{C:b1}\oldep{e}^{1/(\mathcal{Q}+1)})\supseteq E(\vartheta,\gamma)\cap B(xw^*,2k\rho\oldep{e}^{1/(\mathcal{Q}+1)})\neq \emptyset.$$
This concludes the proof of the proposition.
\end{proof}

As in the case of Proposition \ref{proppers}, one can impose slightly different conditions on the measure and obtain a family of cubes satisfying the same thesis as Proposition \ref{lemmaduepsei}.

\begin{proposizione}\label{rk:C}
Suppose $\psi$ is a Radon on $\mathbb{G}$ supported on the compact set $K$ satisfying the following assumptions: 
\begin{itemize}
    \item[(i)] there exists a $\delta\in\N$ such that $\delta^{-1}\leq \Theta^{\mathcal{Q}-1}_*(\psi,x)\leq \Theta^{\mathcal{Q}-1,*}(\psi,x)\leq \delta$ for $\psi$-almost every $x\in\mathbb{G}$,
        \item[(ii)] $ \limsup_{r\to 0}d_{x,4kr}(\psi,\mathfrak{M})\leq 4^{-(\mathcal{Q}+1)}\oldep{e}(\delta)$ for $\psi$-almost every $x\in \mathbb{G}$.
\end{itemize}
Then, fixed $\mu\geq 8 \delta$ if $\gamma,\tilde{\iota}_0,\nu\in\N$ and $\tilde{C}\Subset\mathscr{E}^\psi_{\vartheta,\gamma}(\mu,\nu)$ are the natural numbers and the compact set yielded by Proposition \ref{rk:primo}, defined $\tilde{\iota}:=\max\{\tilde{\iota}_0,\nu\}$ for any cube $Q\in \Delta^\psi(\tilde{C};2\delta,\gamma,\iota)$ we have that $\alpha(Q)\leq \oldep{e}(\delta)$ and for any such cube $Q$ there is a plane $\Pi(Q)\in \G(\mathcal{Q}-1)$ for which:
\begin{itemize}
    \item[(i)]$$\sup_{w\in E^\psi(2\delta,\gamma)\cap B(\mathfrak{c}(Q),k(2\delta)\diam Q/2)}\frac{\dist(w,\mathfrak{c}(Q)\Pi(Q))}{2k\diam Q}\leq \oldC{C:b1}(2\delta)\oldep{e}(2\delta)^{1/\mathcal{Q}},$$
    \item[(ii)]for any $w\in B(\mathfrak{c}(Q),k(2\delta)\diam Q/2)\cap \mathfrak{c}(Q)\Pi(Q)$ we have: $$E^\psi(2\delta,\gamma)\cap B(w,3k\oldC{C:b1}(2\delta)\oldep{e}(2\delta)^{1/(\mathcal{Q}+1)}\diam Q)\neq \emptyset.$$
\end{itemize}
\end{proposizione}

\begin{proof}
If $\psi$ is a Radon measure satisfying the hypothesis of Proposition \ref{rk:primo}, we can find a compact a $\gamma\in \N$ and a compact set $\tilde{C}$ contained in $E^\psi(2\delta,\gamma)$ such that:
\begin{itemize}
\item[(i)]$\psi(K\setminus \tilde{C})\leq 2\oldep{eps:1}\psi(K)$ where $\oldep{eps:1}$ was introduced in Notation \ref{definizione:piano},
    \item[(ii)]$d_{x,4kr}(\psi,\mathfrak{M})+d_{x,4kr}(\psi\llcorner E^\psi(2\delta,\gamma),\mathfrak{M})\leq 4^{-\mathcal{Q}} \oldep{e}(\delta)$  for any $x\in \tilde{C}$  and any $0<r<2^{-\tilde{\iota}_0 N(2\delta)+5}/\gamma$.
\end{itemize}
Thus, if $\Delta^\psi(2\delta,\gamma)$ is the family of dyadic cubes relative to the parameters $2\delta,\gamma$ and the measure $\psi$ yielded by Theorem \ref{evev}, one can prove that the cubes of $\Delta^\psi(\tilde{C};2\delta,\gamma,\iota)$ satisfy (i) and (ii) by using the verbatim argument we employed in the proof of Proposition \ref{lemmaduepsei}.
\end{proof}

The arguments we will use in the rest of the subsection to prove  Proposition \ref{parentneigh} through Theorem \ref{TH:proiezioni}, follow from an adaptation of the techniques that can be found in Chapter 2, \S2 of \cite{DavidSemmes}. 
The first of such adaptations is the following definition, that is a way of saying that two cubes are close both in metric and in size terms:

\begin{definizione}[Neighbour cubes]\label{def:neighbour}
Let $A:=4A_0^2$ and let $Q_j\in\Delta_{i_j}$ be two cubes with $j=1,2$. We say that, $Q_1$ and $Q_2$ are \emph{neighbours} if: $$\underline{\dist}(Q_1,Q_2):=\inf_{x\in Q_1,~y\in Q_2} d(x,y)\underset{\text{(I)}}{\leq} A(\diam Q_1+\diam Q_2)\qquad \text{and}\qquad \lvert j_1-j_2\rvert\underset{\text {(II)}}{\leq} A.$$
Furthermore, in the following for any $Q\in \Delta(C,\iota)$ we let for the sake of notation: $$\mathfrak{n}(Q):=\mathfrak{n}(\Pi(Q)),$$ 
where $\Pi(Q)\in \G(\mathcal{Q}-1)$ is the plane yielded by Lemma \ref{lemmaduepsei}.

Finally, two planes $V,W\in \G(\mathcal{Q}-1)$ are said to have  \emph{compatible orientations} if their normals $\mathfrak{n}(V),\mathfrak{n}(W)\in V_1$ are chosen in such a way that $\langle \mathfrak{n}(V),\mathfrak{n}(W)\rangle> 0$. By extension, we will say that two cubes $Q_1,Q_2\in \Delta(C,\iota)$ have \emph{compatible orientation} themselves if $\Pi(Q_1)$ and  $\Pi(Q_2)$ are chosen to have compatible orientation.
\end{definizione}

\begin{proposizione}\label{parentneigh}
Suppose that $Q_j\in \Delta_{i_j}$ for $j=1,2$. Then:
\begin{itemize}
    \item[(i)] if $Q_1$ is the parent of $Q_2$, then $Q_1$ and $Q_2$ are neighbors,
    \item[(ii)] if $Q_1$ and $Q_2$ are neighbors for any non- negative integer $k\leq \min\{i_1,i_2\}$ their ancestors $\tilde{Q}_1\in\Delta_{i_1-k}$ and $\tilde{Q}_2\in\Delta_{i_2-k}$ are neighbors,
    \item[(iii)] if $Q_1,Q_2\in\Delta(E(\vartheta,\gamma),1)$ are neighbours, then
   $\Big\lvert\log\frac{\diam Q_1}{\diam Q_2}\Big\rvert\leq 2AN$.
\end{itemize}
\end{proposizione}

\begin{proof}
Let us prove (i). Since $Q_2\subseteq Q_1$ then (I) of Definition \ref{def:neighbour} follows immediately. On the other hand, thanks to Proposition \ref{prop:stabcubi}, we infer:
$$\lvert j_1-j_2\rvert\leq \oldC{child}\leq 4 A_0^2=A,$$
where the second inequality comes from the choice of $A_0$, see Definition \ref{notation1}, and this proves (II) of Definition \ref{def:neighbour}.
In order to prove (ii), we first note that $\lvert (i_1-k)-(i_2-k)\rvert=\lvert i_1-i_2\rvert\leq A$ and secondly that:
$$\underline{\dist}(\tilde{Q}_1,\tilde{Q}_2)\leq \underline{\dist}(Q_1,Q_2)\leq A(\diam Q_1+\diam Q_2)\leq A(\diam \tilde{Q}_1+\diam\tilde{Q}_2).$$
In order to prove (iii), note that thanks to Theorem \ref{evev}(ii) and (v), we infer that:
$$\Big\lvert\log\frac{\diam Q_1}{\diam Q_2}\Big\rvert\leq\Big\lvert\log\frac{2^{-Nj_1+5}/\gamma}{\zeta^22^{-Nj_2-1}/\gamma}\Big\rvert=(N\lvert j_2-j_1\rvert+6)\log 2-2\log \zeta\leq 2AN,$$
where the last inequality comes from the choice of $A_0$.
\end{proof}

\begin{osservazione}\label{oss:centr}
If $Q\in\Delta(C,\iota)$ then $\mathfrak{c}(Q)\in E(\vartheta,\gamma)$ thanks to Proposition \ref{prop:centri} and the fact that we chose $\iota\geq 4\vartheta$.
\end{osservazione}

\begin{proposizione}\label{prop:nearnorm}
Suppose that $Q_1,Q_2\in\Delta(C,\iota)$ are two neighbour cubes. Then:
$$(1-\oldC{C:1.0}\oldep{e}^{\frac{2}{\mathcal{Q}+1}})^{1/2}=(1-2^{20}(n_1-1)\oldC{C:b1}^2\oldep{e}^{\frac{2}{\mathcal{Q}+1}})^{   1/2}\leq\lvert \langle \mathfrak{n}(Q_1),\mathfrak{n}(Q_2)\rangle\rvert.$$
\end{proposizione}

\begin{proof}
Thanks to the definition of $k$, we have:
$$A(\diam Q_1+\diam Q_2)\leq 2A\max\{\diam Q_1,\diam Q_2\}\leq (k/4)\max\{\diam Q_1,\diam Q_2\}.$$
Without loss of generality we can assume that $\diam Q_2\leq \diam Q_1$. Since the cubes $Q_1,Q_2$ are supposed to be neighbors, we deduce that:
\begin{equation}
    \underline{\dist}(Q_1,Q_2)\leq (k/4)\diam Q_1.
    \nonumber
\end{equation}
This implies that for any $z\in Q_1$, we have:
\begin{equation}
    \underline{\dist}(z,Q_2)\leq \diam Q_1 +\inf_{y\in Q_1}\underline{\dist}(y,Q_2)\leq\diam Q_1+\dist(Q_1,Q_2)\leq (k/4+1)\diam Q_1<(k/2)\diam Q_1.
    \label{eq:C01}
\end{equation}
Inequality \eqref{eq:C01} implies that for any $z\in Q_1$ we have $Q_2\subseteq B(z,k\diam Q_1/2)$. This, together with Lemma \ref{lemmaduepsei}(i), implies that for any $w\in E(\vartheta,\gamma)\cap Q_2$ we have:
\begin{equation}
    \dist(w,\mathfrak{c}(Q_1)\Pi(Q_1))\leq 2\oldC{C:b1} \oldep{e}^\frac{1}{\mathcal{Q}} k\diam Q_1,
    \label{eq:n15}
\end{equation}
We now claim that the set $B_2:=\{ u\in\mathbb{G}:\dist(u,Q_2)\leq (k/20)\diam Q_2\}$ is contained in the ball $B(\mathfrak{c}(Q_1),k\diam Q_1/2)$.
In order to prove such inclusion, let $u\in B_2$ and note that:
\begin{equation}
\begin{split}
      \dist(u,\mathfrak{c}(Q_1))\leq &\inf_{w\in Q_2}(d(u,w)+d(w,\mathfrak{c}(Q_1))
      \leq \inf_{w\in Q_2}d(u,w)+\diam Q_1+\underline{\dist}(Q_1,Q_2)+\diam Q_2\\
      \underset{u\in B_2}{\leq} &\frac{k}{20}\diam Q_2+\diam Q_1+\underline{\dist}(Q_1,Q_2)+\diam Q_2
      \leq \frac{3k+20}{10}\diam Q_1<\frac{k}{2}\diam Q_1,
\end{split}
  \label{eq:81}
\end{equation}
where the second last inequality comes from the assumption that $Q_1$ is the cube with the biggest diameter. The inequality \eqref{eq:81} concludes the proof of the inclusion $B_2\subseteq B(\mathfrak{c}(Q_1),k\diam Q_1/2)$. 
Proposition \ref{parentneigh}(iii) together with the fact that $Q_1,Q_2\in\Delta(E(\vartheta,\gamma),\iota)$ and inequality \eqref{eq:n15} imply that for any $u\in E(\vartheta,\gamma)\cap B_2$ we have:
\begin{equation}
    \dist(u,\mathfrak{c}(Q_1)\Pi(Q_1))\leq 2\oldC{C:b1}\oldep{e}^{\frac{1}{\mathcal{Q}}}k\diam Q_1\leq 2\oldC{C:b1}e^{2NA}\oldep{e}^{\frac{1}{\mathcal{Q}}}k\diam Q_2.
    \label{eq:numbero10k}
\end{equation}
Furthermore, thanks to Remark \ref{oss:centr} we have $\mathfrak{c}(Q_2)\in B_2\cap E(\vartheta,\gamma)$. This also implies, by Proposition \ref{prop:dist-piani} and \eqref{eq:numbero10k}, that for any $u\in B_2\cap E(\vartheta,\gamma)$ we have:
\begin{equation}
    \begin{split}
        \dist(u,\mathfrak{c}(Q_2)\Pi(Q_1))\leq& \dist(u,\mathfrak{c}(Q_1)\Pi(Q_1))+\dist(\mathfrak{c}(Q_2)\Pi(Q_1),\mathfrak{c}(Q_1)\Pi(Q_1))\\
        =&\dist(u,\mathfrak{c}(Q_1)\Pi(Q_1))+\dist(\mathfrak{c}(Q_2),\mathfrak{c}(Q_1)\Pi(Q_1))\leq4\oldC{C:b1}e^{2NA}\oldep{e}^{\frac{1}{\mathcal{Q}}}k\diam Q_2.
    \end{split}
    \label{eq:84}
\end{equation}
Thanks to Lemma \ref{lemmaduepsei}(ii), we deduce that for any $y\in B(\mathfrak{c}(Q_2),k\diam Q_2/40)\cap \mathfrak{c}(Q_2)\Pi(Q_2)$ there exists some $w(y)$ in $E(\vartheta,\gamma)\cap B(y, 3k\oldC{C:b1}\oldep{e}^{1/(\mathcal{Q}+1)}\diam Q_2)$.
Since by definition $\oldep{e}\leq ((k-20)/20k)^{\mathcal{Q}+1}$, we have:
\begin{equation}
\begin{split}
       \dist(w(y),Q_2)\leq &\inf_{p\in Q_2} d(w(y),y)+d(y,\mathfrak{c}(Q_2))+d(\mathfrak{c}(Q_2),p)\\
       \leq &3k\oldC{C:b1}\oldep{e}^{1/(\mathcal{Q}+1)}\diam Q_2+\frac{k}{40}\diam Q_2+\diam Q_2\leq \frac{k}{20}\diam Q_2,
\end{split}
    \label{eq:n160}
\end{equation}
where the last inequality comes from the choice of $k$.
Inequality \eqref{eq:n160} implies that $w(y)\in B_2$ and thanks to \eqref{eq:84} we infer that:
$$\dist(w(y),\mathfrak{c}(Q_2)\Pi(Q_1))\leq 4\oldC{C:b1}e^{2NA}\oldep{e}^\frac{1}{\mathcal{Q}}k\diam Q_2.$$
Summing up, for any $y\in B(\mathfrak{c}(Q_2),k\diam Q_2/40)\cap \mathfrak{c}(Q_2)\Pi(Q_2)$ we have:
\begin{equation}
\begin{split}
    \dist(y, \mathfrak{c}(Q_2)\Pi(Q_1))\leq& 3\oldC{C:b1}d(y,w(y))+\dist(w(y),\mathfrak{c}(Q_2)\Pi(Q_1))\leq 3\oldC{C:b1}\oldep{e}^{\frac{1}{\mathcal{Q}+1}}k\diam Q_2+4\oldC{C:b1} e^{2NA}\oldep{e}^{\frac{1}{\mathcal{Q}}}k\diam Q_2\\
    \leq &(3\oldC{C:b1}+4\oldC{C:b1}e^{2NA}\oldep{e}^{\frac{1}{\mathcal{Q}(\mathcal{Q}+1)}})\oldep{e}^{\frac{1}{\mathcal{Q}+1}} k \diam Q_2\leq 6\oldC{C:b1}\oldep{e}^{\frac{1}{\mathcal{Q}+1}}  k \diam Q_2,
\end{split}
\label{eq:85}
\end{equation}
where the last inequality comes from the choice of $\oldep{e}$ and few algebraic computations that we omit.
Furthermore, inequality \eqref{eq:85} and Proposition \ref{cor:SSCM2.2.14} imply that:
$$ \lvert \langle \pi_1 (y^{-1}\mathfrak{c}(Q_2)),\mathfrak{n}(Q_1)\rangle\rvert= \lVert P_{\mathfrak{n}(Q_1)}(\mathfrak{c}(Q_2)^{-1}y)\rVert= \dist(y,\mathfrak{c}(Q_2)\Pi(Q_1))\leq 6\oldC{C:b1}\oldep{e}^{\frac{1}{\mathcal{Q}+1}}  k \diam Q_2.$$
Suppose $\{v_i$, $i=1,\ldots,n_1-1\}$ are the unit vectors of the first layer $V_1$ spanning the orthogonal of $\mathfrak{n}(Q_2)$ inside $V_1$ and let $y_j:=\mathfrak{c}(Q_2)\delta_{k\diam Q_2/80}(v_j)$.
Then, thanks to inequality \eqref{eq:85}, we deduce that:
\begin{equation}
\begin{split}
     1=\lvert \langle \mathfrak{n}(Q_1),\mathfrak{n}(Q_2)\rangle\rvert^2+\sum_{j=1}^{n_1-1}\lvert \langle v_j,\mathfrak{n}(Q_1)\rangle\rvert^2=&\lvert \langle \mathfrak{n}(Q_1),\mathfrak{n}(Q_2)\rangle\rvert^2+\sum_{j=1}^{n_1-1}\frac{\lvert \langle \pi_1(\mathfrak{c}(Q_2)^{-1}y_j),\mathfrak{n}(Q_1)\rangle\rvert^2}{(k\diam Q_2/80)^2}.\\
     \leq &\lvert \langle \mathfrak{n}(Q_1),\mathfrak{n}(Q_2)\rangle\rvert^2+2^{20}(n_1-1)\oldC{C:b1}^2\oldep{e}^{\frac{2}{\mathcal{Q}+1}}.
     \nonumber
\end{split}
\end{equation}
This concludes the proof of the proposition.
\end{proof}

\begin{proposizione}\label{prop:cubbit}
Let $Q_1,Q_2\in\Delta(C,\iota)$ and suppose that $\Pi(Q_1)$ and $\Pi(Q_2)$, the planes yielded by Lemma \ref{lemmaduepsei}, are chosen with compatible orientations. Then:
$$\lvert \mathfrak{n}(Q_1)- \mathfrak{n}(Q_2)\rvert\leq 2\sqrt{\oldC{C:1.0}}\oldep{e}^{1/(\mathcal{Q}+1)}.$$
Furthermore, the planes $\Pi(Q_i)$ have compatible orientations if and only if the planes $\Pi(\tilde{Q}_i)$ relative to the parent cubes $\tilde{Q}_1,\tilde{Q}_2$ of $Q_1$ and $Q_2$ respectively, have the compatible orientations.   
\end{proposizione}
 
\begin{proof}

If $Q_1$ and $Q_2$ are neighbors and have the compatible orientations, then $\langle \mathfrak{n}(Q_1),\mathfrak{n}(Q_2)\rangle\geq 0$ and thanks to Proposition \ref{prop:nearnorm} we infer that:
$$\lvert \mathfrak{n}(Q_1)-\mathfrak{n}(Q_2)\rvert^2=2-2\langle \mathfrak{n}(Q_1),\mathfrak{n}(Q_2)\rangle\leq 2-2(1-\oldC{C:1.0}\oldep{e}^{\frac{2}{\mathcal{Q}+1}})^{1/2}\leq 2\sqrt{\oldC{C:1.0}}\oldep{e}^{\frac{1}{\mathcal{Q}+1}}.$$
If $Q_1$ and $Q_2$ are neighbors, Proposition \ref{parentneigh} implies that the couples $\tilde{Q}_1$ and $\tilde{Q}_2$,  $Q_1$ and $\tilde{Q}_1$,  $Q_2$ and $\tilde{Q}_2$ are neighbors. Therefore, Proposition \ref{prop:nearnorm} implies that:
\begin{equation}
    \begin{split}
        \langle \mathfrak{n}(\tilde{Q}_1),\mathfrak{n}(\tilde{Q}_2)\rangle=&\langle\mathfrak{n}(Q_1),\mathfrak{n}(Q_2)\rangle +\langle\mathfrak{n}(\tilde{Q}_1)-\mathfrak{n}(Q_1),\mathfrak{n}(Q_2)\rangle+\langle \mathfrak{n}(\tilde{Q}_1),\mathfrak{n}(\tilde{Q}_2)-\mathfrak{n}(Q_2)\rangle\\
        \geq&(1-\oldC{C:1.0}\oldep{e}^{\frac{2}{\mathcal{Q}+1}})^{1/2}-4\sqrt{\oldC{C:1.0}}\oldep{e}^{\frac{1}{\mathcal{Q}+1}}\geq 1/10.
        \nonumber
    \end{split}
\end{equation}
Viceversa, if $\Pi(\tilde{Q}_1)$ and $\Pi(\tilde{Q}_2)$ have the same orientation, the same line of reasoning yields that the planes $\Pi(Q_1)$ and $\Pi(Q_2)$ have compatible orientation as well. 
\end{proof}

\begin{proposizione}
It is possible to fix an orientation on the planes $\{\Pi(Q):Q\in\Delta(C,\iota)\}$, in such a way that:
$$\lvert\mathfrak{n}(Q_1)-\mathfrak{n}(Q_2)\rvert\leq 1/10,$$
whenever $Q_1,Q_2\in\Delta(C,\iota)$ are neighbors and contained in the same maximal cube $Q_0\in\mathcal{M}(C,\iota)$.
\end{proposizione}

    \begin{proof}
     Let $Q_i\in\Delta_{j_i}$ for $i=1,2$ and suppose without loss of generality that $j_1\leq j_2$.
    Fix the normal of the plane $\Pi(Q_0)$ and determine the  normals of all other planes $\Pi(Q)$ as $Q$ varies in $\Delta(C,\iota)$ by demanding that the orientation of the cube $Q$ is compatible with the one of $\Pi(\tilde{Q})$, where $\tilde{Q}$ is the parent of $Q$. 
    
    If $Q_1=Q_0$, let us consider the finite sequence $\{\tilde{Q}_i\}_{i=1,\ldots,M}$ of ancestors of $Q_2$ such that $\tilde{Q}_1=Q_2$, $\tilde{Q}_M=Q_0$ and such that $\tilde{Q}_{i+1}$ is the parent of $\tilde{Q}_i$. Then, the orientation of $\Pi(Q_0)$ and $\Pi(Q_2)$ fixed in the above paragraph must be compatible, indeed:
    \begin{equation}
        \langle \mathfrak{n}(Q_0),\mathfrak{n}(Q_2)\rangle\geq \langle \mathfrak{n}(\tilde{Q}_2),\mathfrak{n}(Q_2)\rangle-\sum_{i=2}^M\lvert \mathfrak{n}(\tilde{Q}_i)-\mathfrak{n}(\tilde{Q}_{i+1})\rvert\geq(1-\oldC{C:1.0}\oldep{e}^{\frac{2}{\mathcal{Q}+1}})-2\sqrt{\oldC{C:1.0}}M\oldep{e}^{1/(\mathcal{Q}+1)},
        \label{eq:n16}
    \end{equation}
    where the last inequality comes from Propositions \ref{prop:nearnorm}, \ref{prop:cubbit} and the fact that the orientation of $\tilde{Q}_i$ and $\tilde{Q}_{i+1}$ were chosen to be compatible. Since $Q_0$ and $Q_2$ were assumed to be neighbours, we deduce that $M\leq A$ and thus, thanks to \eqref{eq:n16} and the choice of $\oldep{e}$, we have:
    $$\langle \mathfrak{n}(Q_0),\mathfrak{n}(Q_2)\rangle\geq(1-\oldC{C:1.0}\oldep{e}^{\frac{2}{\mathcal{Q}+1}})-2\sqrt{\oldC{C:1.0}}A\oldep{e}^{1/(\mathcal{Q}+1)}>0.$$
    This proves the statement if one of the cubes is $Q_0$. The proof of the general case can be obtained with the following argument. Thanks to Proposition \ref{prop:cubbit}, we know that the orientation of the planes $\Pi(Q_1)$ and $\Pi(Q_2)$ is compatible if and only if the orientation of $\Pi(\tilde{Q}_1)$ and $\Pi(\tilde{Q}_2)$, the planes relative to their parent cubes $\tilde{Q}_1$ and $\tilde{Q}_2$, are compatible. Thus, taking the parents of the parents and so on, one can reduce to the case in which one of the cubes is $Q_0$.
\end{proof}

\begin{definizione}
For each cube $Q\in\Delta(C,\iota)$, we let:
\begin{equation}
\begin{split}
    G_\pm(Q):=&\mathfrak{c}(Q)\{u\in B(0,A_0\diam Q): \pm \langle \pi_1u,\mathfrak{n}(Q)\rangle\geq A_0^{-1}\diam Q\}\\
    =&\{u\in B(\mathfrak{c}(Q),A_0\diam Q):\pm \langle \pi_1u-\pi_1(\mathfrak{c}(Q)),\mathfrak{n}(Q)\rangle\geq A_0^{-1}\diam Q\}.
    \nonumber
    \end{split}
\end{equation}
and $G(Q)=G_{+}(Q)\cup G_{-}(Q)$. Furthermore, for any $\overline{Q}\in\mathcal{M}(C,\iota)$ we let:
\begin{equation}
\mathfrak{G}_{\pm}(\overline{Q}):=\bigcup_{\substack{Q\in\Delta(C,\iota)\\ Q\subseteq \overline{Q}}}G_{\pm}(Q)\qquad\text{and}\qquad \mathfrak{G}(\overline{Q}):=\bigcup_{\substack{Q\in\Delta(C,\iota)\\ Q\subseteq \overline{Q}}}G(Q).
    \nonumber
\end{equation}
\end{definizione}

\begin{lemma}\label{lemma2.26}
For any cube $Q$ of $\Delta(C,\iota)$ and any $x\in G(Q)$, we have: 
\begin{equation}
    \frac{A_0^{-1}}{2}\diam Q\underset{\text{(A)}}{\leq} \dist(x,E(\vartheta,\gamma))\underset{\text{(B)}}{\leq} A_0\diam Q.
    \label{eq:numbero11k}
\end{equation}
\end{lemma}

\begin{proof}
Since $A_0\leq k/4$, if we let $z\in E(\vartheta,\gamma)$ be the point realizing the minimum distance of $x$ from $E(\vartheta,\gamma)$, we deduce that:
\begin{equation}
    d(x,z)=\dist(x,E(\vartheta,\gamma))\leq  d(x,\mathfrak{c}(Q))\leq A_0\diam Q,
    \label{eq:n17}
\end{equation}
where the first inequality above comes from the fact that $\mathfrak{c}(Q)\in E(\vartheta,\gamma)$, see Remark \ref{oss:centr}, and the last inequality comes from the very definition of $G(Q)$.
Note that inequality \eqref{eq:n17} proves \eqref{eq:numbero11k}(B).
Furthermore, since $1+A_0<k/2$ the bound \eqref{eq:n17} also implies that $z\in B(\mathfrak{c}(Q),k\diam Q/2)\cap E(\vartheta,\gamma)$ and thus, thanks to Proposition \ref{lemmaduepsei} we deduce that:
\begin{equation}
    \dist(z,\mathfrak{c}(Q)\Pi(Q))\leq 2\oldC{C:b1}\oldep{e}^{1/\mathcal{Q}}k\diam Q.
    \label{eq:n18}
\end{equation}
Let $w$ be an element of $\Pi(Q)$ satisfying the identity $d(z,\mathfrak{c}(Q)w)=\dist(z,\mathfrak{c}(Q)\Pi(Q))$ and note that \eqref{eq:n18} implies that:
\begin{equation}
    \begin{split}
       \dist(x,E(\vartheta,&\gamma))=\dist(x,z)\geq d(x,\mathfrak{c}(Q)w)-d(\mathfrak{c}(Q)w,z)
       \geq \dist(\mathfrak{c}(Q)^{-1}x,\Pi(Q))-\dist(z,\mathfrak{c}(Q)\Pi(Q))\\
       \geq& \lvert\langle \mathfrak{n}(Q),\pi_1(\mathfrak{c}(Q)^{-1}x)\rangle\rvert-2\oldC{C:b1}\oldep{e}^{1/\mathcal{Q}}k\diam Q\geq A_0^{-1}\diam Q-2\oldC{C:b1}\oldep{e}^{1/\mathcal{Q}}k\diam Q\geq \frac{1}{2A_0}\diam Q,
    \end{split}
\end{equation}
where the last inequality comes from the choice of $\oldep{e}$ and $A_0$.
\end{proof}

\begin{lemma}\label{lemma:intersectionG}
For any $\overline{Q}\in\mathcal{M}(C,\iota)$ we have $\mathfrak{G}_+(\overline{Q})\cap \mathfrak{G}_-(\overline{Q})=\emptyset$. 
\end{lemma}

\begin{proof}
Suppose this is not the case, assume that we can find two cubes $Q_1, Q_2\in \Delta(C,\iota)$ such that $G_+(Q_1)\cap G_-(Q_2)\neq \emptyset$ and let $x$ be a point of intersection. Thanks to the definition of $G_\pm(Q)$, we immediately deduce that:
\begin{equation}
    B(\mathfrak{c}(Q_1),A_0 \diam Q_1)\cap B(\mathfrak{c}(Q_2),A_0 \diam Q_2)\neq \emptyset.
    \label{eq:n19}
\end{equation}
This in particular implies that
$\dist(Q_1,Q_2)\leq 2A_0(\diam Q_1+\diam Q_2)$. Therefore, since $2A_0\leq A$, $Q_1$ and $Q_2$ satisfy the condition (I) of Definition \ref{def:neighbour}. Furthermore, Lemma \ref{lemma2.26} implies that:
\begin{equation}
    \frac{\diam Q_1}{2A_0}\leq \dist(x,E(\vartheta,\gamma))\leq A_0 \diam Q_1,\qquad\text{and}\qquad
    \frac{\diam Q_2}{2A_0}\leq \dist(x,E(\vartheta,\gamma))\leq A_0 \diam Q_2.
    \label{eq:90}
\end{equation}
Putting together the bounds in \eqref{eq:90}, we infer that:
\begin{equation}
(2 A_0^2)^{-1}\leq \frac{\diam Q_1}{\diam Q_2}\leq 2A_0^2.
    \label{eq:n190}
\end{equation}
Thanks to \eqref{eq:n190} and Theorem \ref{evev}(iv),(vii) we have that:
\begin{equation}
    (2A_0^2)^{-1}\leq \frac{\diam Q_1}{\diam Q_2}\leq \frac{2^{-j_1N+5}/\gamma}{\zeta^22^{-j_2N-1}/\gamma}\qquad\text{and}\qquad\frac{\zeta^22^{-j_1N-1}/\gamma}{2^{-j_2+5}/\gamma}\leq \frac{\diam Q_1}{\diam Q_2}\leq 2A_0^2
    \label{eq:n105}
\end{equation}
Finally, thanks to the bounds in \eqref{eq:n105} together with some algebraic computations that we omit, we deduce:
\begin{equation}
    \lvert j_2-j_1\rvert\leq \frac{\log(2^7\zeta^{-2}A_0^2)}{N\log2}\leq \log A_0,
    \nonumber
\end{equation}
where the last inequality comes from the choice of $A_0$. Since $A_0\geq 2$, we infer that $\lvert j_2-j_1\rvert\leq A$, proving (II) of Definition \ref{def:neighbour}. This concludes the proof that $Q_1$ and $Q_2$ are neighbors. Since $Q_1$ and $Q_2$ are neighbors, \eqref{eq:n19} together with Proposition \ref{parentneigh}(iii) implies that:
$$d(\mathfrak{c}(Q_1),\mathfrak{c}(Q_2))\leq d(\mathfrak{c}(Q_1),x)+d(x,\mathfrak{c}(Q_2))\leq A_0(\diam Q_1+\diam Q_2)\leq A_0(1+e^{2NA})\diam Q_2< k\diam Q_2/2,$$
where the last inequality comes from the choice of $k$ and of $A$.
Since $\mathfrak{c}(Q_2)\in E(\vartheta,\gamma)\cap B(\mathfrak{c}(Q_1),k\diam Q_2/2)$, thanks to Lemma \ref{lemmaduepsei}(i), we deduce that:
\begin{equation}
    \dist(\mathfrak{c}(Q_2),\mathfrak{c}(Q_1)\Pi(Q_1))\leq 2\oldC{C:b1}k\oldep{e}^{1/\mathcal{Q}}\diam Q_1\leq 2\oldC{C:b1}k e^{2NA}\oldep{e}^{1/\mathcal{Q}}\diam Q_2.
    \label{eq:n20}
\end{equation}
Furthermore, since Proposition \ref{prop:cubbit} implies that
$\lvert\mathfrak{n}(Q_1)-\mathfrak{n}(Q_2)\rvert\leq 2\oldC{C:1.0}\oldep{e}^{1/(\mathcal{Q}+1)}$, we have that:
\begin{equation}
    \begin{split}
        \langle \pi_1(\mathfrak{c}(Q_1)^{-1}x),&\mathfrak{n}(Q_1)\rangle\\=&\langle \pi_1(\mathfrak{c}(Q_2)^{-1}x),\mathfrak{n}(Q_2)\rangle+\langle \pi_1(\mathfrak{c}(Q_2)^{-1}x), \mathfrak{n}(Q_1)-\mathfrak{n}(Q_2)\rangle
        +\langle \pi_1(\mathfrak{c}(Q_1)^{-1}\mathfrak{c}(Q_2)),\mathfrak{n}(Q_1)\rangle\\
        \leq &-A_0^{-1} \diam Q_2+\lvert \pi_1(\mathfrak{c}(Q_2)^{-1}x)\rvert\lvert \mathfrak{n}(Q_1)-\mathfrak{n}(Q_2)\rvert+\dist(\mathfrak{c}(Q_2),\mathfrak{c}(Q_1)\Pi(Q_1))\\
        \leq &-A_0^{-1}\diam Q_2+A_0\diam Q_2\cdot 2\oldC{C:1.0}\oldep{e}^{1/(\mathcal{Q}+1)}+2\oldC{C:b1}k e^{2NA}\oldep{e}^{1/\mathcal{Q}}\diam Q_2,
        \label{eq:n200}
    \end{split}
\end{equation}
where last inequality comes from \eqref{eq:n20} and the fact that $x\in G_{-}(Q_2)$. The chain of inequalities in \eqref{eq:n200} and the definition of $A$ imply:
\begin{equation}
    \langle \pi_1(\mathfrak{c}(Q_1)^{-1}x),\mathfrak{n}(Q_1)\rangle\leq (-A_0^{-1}+A_0\oldC{C:1.0}\oldep{e}^{1/(\mathcal{Q}+1)}+\oldC{C:b1}k e^{8NA_0^2}\oldep{e}^{1/\mathcal{Q}})\diam Q_2\leq 0,
\end{equation}
where the last inequality comes from the definition of $\oldep{e}$ and some algebraic computations that we omit. This contradicts the fact that $x\in G_+(Q_1)$, proving that the assumption that $\mathfrak{G}(\overline{Q})_{+}\cap \mathfrak{G}_{-}(\overline{Q})\neq \emptyset$ was absurd.
\end{proof}

\begin{proposizione}\label{lemma2.28}
For any cube $\overline{Q}$ in $\mathcal{M}(C,\iota)$ we define: $$I(\overline{Q}):=\bigcup_{\substack{Q\in\Delta(C,\iota)\\ Q\subseteq \overline{Q}}}B(\mathfrak{c}(Q),(A_0-2)\diam Q).$$
Furthermore, for any  $x\in I(\overline{Q})$ we let:
\begin{equation}
    d(x):=\inf_{\substack{Q\in \Delta(C,\iota)\\Q\subseteq \overline{Q}}}\dist(x,Q)+\diam Q.
    \label{eq:n35}
\end{equation}
Then $\dist(x,E(\vartheta,\gamma))\leq 4A_0^{-1}d(x)$ whenever $x\in I(\overline{Q})\setminus \mathfrak{G}(\overline{Q})$.
\end{proposizione}

\begin{proof}
Fix some $x\in I(\overline{Q}) \setminus \mathfrak{G}(\overline{Q})$ and let $Q\subseteq \overline{Q}$ be a cube of $ \Delta(C,\iota)$ such that:
\begin{equation}
    \dist(x,Q)+\diam Q\leq 4 d(x)/3.
    \label{eq:103a}
\end{equation}
Let $Q^\prime$ be an ancestor of $Q$ in $\Delta(C,\iota)$, possibly $Q$ itself. Since $x\not\in \mathfrak{G}(\overline{Q})$, then $x\not \in G(Q^\prime)$ and, thanks to Proposition \ref{cor:SSCM2.2.14}, we have: 
\begin{equation}
    \dist(x,\mathfrak{c}(Q^\prime)\Pi(Q^\prime))=\lvert\langle \pi_1(\mathfrak{c}(Q^\prime)^{-1}x),\mathfrak{n}(Q^\prime)\rangle\rvert\leq A_0^{-1}\diam Q^\prime,
    \label{eq:101}
\end{equation}
where the last inequality is true provided $\dist(x, \mathfrak{c}(Q^\prime))< A_0\diam Q^\prime$. Since $x\in I(\overline{Q})$, there must exist some $\tilde{Q}\in \Delta(C,\iota)$ such that $\tilde{Q}\subseteq \overline{Q}$ and $x\in B(\mathfrak{c}(\tilde{Q}),(A_0-2)\diam \tilde{Q})$. This implies that:
\begin{equation}
    \dist(x,\mathfrak{c}(\overline{Q}))\leq d(x,\mathfrak{c}(\tilde{Q}))+d(\mathfrak{c}(\tilde{Q}),\mathfrak{c}(\overline{Q}))\leq (A_0-2)\diam \tilde{Q}+\diam \overline {Q}< A_0 \diam \overline{Q}.
    \label{eq:201}
\end{equation}
Therefore the inequality $\dist(x, \mathfrak{c}(\overline{Q}))< A_0\diam \overline{Q}$ is verified and hence \eqref{eq:101} holds for $Q^\prime=\overline{Q}$. Let $Q\subseteq Q_0\subseteq \overline{Q}$ be the smallest cube in $\Delta(C,\iota)$ for which $\dist(x, \mathfrak{c}(Q_0))< A_0\diam Q_0$ holds.

Let $w\in \Pi(Q_0)$ be the point for which $d(x,\mathfrak{c}(Q_0)w)=\dist(x,\mathfrak{c}(Q_0)\Pi(Q_0))$, and note that the choice of $Q_0$ and the bound \eqref{eq:101} imply:
\begin{equation}
\begin{split}
     \lVert w\rVert=\dist(\mathfrak{c}(Q_0) w,\mathfrak{c}(Q_0))\leq& d(\mathfrak{c}(Q_0)w,x)+d(x,\mathfrak{c}(Q_0))\leq\dist(x,\mathfrak{c}(Q_0)\Pi(Q_0))+ A_0\diam Q_0\\
     \leq&A_0^{-1}\diam Q_0+ A_0\diam Q_0\leq 2A_0\diam Q_0<k\diam Q/2.
\end{split}
    \label{eq:102}
\end{equation}
Since $Q_0\in\Delta(C,\iota)$, thanks to inequality \eqref{eq:102}, Lemma \ref{lemmaduepsei}(ii) implies that $E(\vartheta,\gamma)\cap B(\mathfrak{c}(Q_0) w, 3k\oldC{C:b1}\oldep{e}^{1/(\mathcal{Q}+1)}\diam Q_0)\neq \emptyset$. Therefore, since by definition of $Q_0$ the bound \eqref{eq:101} holds with $Q^\prime=Q_0$, we have:
\begin{equation}
\begin{split}
    \dist(x,E(\vartheta,\gamma))\leq& d(x,\mathfrak{c}(Q_0)w)+\dist(\mathfrak{c}(Q_0)w,E(\vartheta,\gamma))=d(x,\mathfrak{c}(Q_0)\Pi(Q_0))+\dist(\mathfrak{c}(Q_0)w,E(\vartheta,\gamma))\\
    \leq& A_0^{-1}\diam Q_0+3k\oldC{C:b1}\oldep{e}^{1/(\mathcal{Q}+1)}\diam Q_0\leq 2A_0^{-1}\diam Q_0,
    \label{eq:n22}
\end{split}
\end{equation}
where the last inequality comes from the choice of $\oldep{e}$.

If $Q_0=Q$, then \eqref{eq:103a} implies that $\dist(x,E(\vartheta,\gamma))\leq 2A_0^{-1} \diam Q_0\leq 4A_0^{-1}d(x)$. Otherwise, let $Q_1$ be the child of $Q_0$, possibly $Q_0$ itself, that contains $Q$. Thanks to the choice of
$Q_1$ we have $\dist(x, \mathfrak{c}(Q_1))\geq A_0\diam Q_1$, and thus:
\begin{equation}
    \dist(x,Q_1)\geq d(x,\mathfrak{c}(Q_1))-\diam Q_1\geq(A_0-1)\diam Q_1\geq \frac{A_0-1}{\oldC{child}} \diam Q_0\geq \diam Q_0,
    \label{eq:n230}
\end{equation}
where the second last inequality above follows from Proposition \ref{prop:stabcubi} and the last one from the choice of $A_0$.
Eventually, thanks to \eqref{eq:103a}, \eqref{eq:n22}, \eqref{eq:n230} and the fact that $Q\subseteq Q_1$, we deduce that:
$$\dist(x,E(\vartheta,\gamma))\leq 2A_0^{-1}\diam Q_0\leq 2 A_0^{-1}\dist(x,Q_1)\leq 2 A_0^{-1}\dist(x,Q)\leq 4A_0^{-1}d(x),$$
concluding the proof of the proposition.
\end{proof}

We are now ready to prove the main result of this subsection. Theorem \ref{TH:proiezioni} asserts that the compact set $C$ has big projections on planes.

\begin{teorema}\label{TH:proiezioni}
For any cube $Q\in \Delta(C,\iota)$  such that $(1-\oldep{eps:3})\phi(Q)\leq \phi(Q\cap C)$, we have:
$$\mathcal{S}^{\mathcal{Q}-1}(P_{\Pi(Q)}(Q\cap C))\geq \frac{\diam Q^{\mathcal{Q}-1}}{2A_0^{\mathcal{Q}-1}}.$$
\end{teorema}

\begin{proof}
Let $Q_0\in \Delta(C,\iota)$ be such that $(1-\oldep{eps:3})\phi(Q_0)\leq \phi(Q_0\cap C)$ and define:
$$F(Q_0):=C\cap Q_0\cup \bigcup_{Q\in \mathscr{I}(Q_0)} B(\mathfrak{c}(Q),2\oldC{child}\diam Q),$$
where $\mathscr{I}(Q_0)$ is a family of maximal cubes $Q\in\Delta(E(\vartheta,\gamma),\iota)$ such that $Q\subseteq Q_0$ and $Q\not\in\Delta(C,\iota)$.
As a first step, we estimate the size of the projection of the balls $\bigcup_{Q\in \mathscr{I}(Q_0)} B(\mathfrak{c}(Q),\oldC{child}\diam Q)$. Thanks to Proposition \ref{cor:2.2.19} we deduce that:
\begin{equation}
    \begin{split}
       \mathcal{S}^{\mathcal{Q}-1}\bigg(P_{\Pi(Q)}\bigg(\bigcup_{Q\in \mathscr{I}(Q_0)} B(\mathfrak{c}(Q),2\oldC{child}\diam Q)\bigg)\bigg)\leq 2^{\mathcal{Q}-1}c(\Pi(Q_0)) \oldC{child}^{\mathcal{Q}-1}\sum_{Q\in\mathscr{I}(Q_0)}\diam Q^{\mathcal{Q}-1}.
        \label{schifo1}
    \end{split}
\end{equation}
We now need to estimate the sum in the right-hand side of \eqref{schifo1}. Since the cubes in $\mathscr{I}(Q_0)$ are disjoint and they are contained in $\Delta(E(\vartheta,\gamma),\iota)$, thanks to Remark \ref{rk:rephrase} and the fact that $(1-\oldep{eps:3})\phi(Q_0)\leq \phi(Q_0\cap C)$ we have:
\begin{equation}
    \begin{split}
       \oldC{C:VIB}^{-1}\sum_{Q\in\mathscr{I}(Q_0)}\diam Q^{\mathcal{Q}-1}\leq \sum_{Q\in\mathscr{I}(Q_0)}\phi(Q)=\phi\bigg(\bigcup_{Q\in\mathscr{I}(Q_0)}Q\bigg)\leq \phi(Q_0\setminus C)\leq \oldep{eps:3}\phi(Q_0)\leq\oldep{eps:3} \oldC{C:VIB}\diam Q_0^{\mathcal{Q}-1}.
        \label{Schifo2}
    \end{split}
\end{equation}
Putting together \eqref{schifo1} and \eqref{Schifo2}, we deduce that:
\begin{equation}
    \begin{split}
       \mathcal{S}^{\mathcal{Q}-1}\bigg(P_{\Pi(Q)}\bigg(\bigcup_{Q\in \mathscr{I}(Q_0)} B(\mathfrak{c}(Q),2\oldC{child}\diam Q)\bigg)\bigg)\leq& 2^{\mathcal{Q}-1}c(\Pi(Q_0))\oldC{C:VIB}^2\oldep{eps:3}\oldC{child}^{\mathcal{Q}-1}\diam Q_0^{\mathcal{Q}-1}\\
       \leq& \frac{c(\Pi(Q_0))}{2A_0^{\mathcal{Q}-1}}\diam Q_0^{\mathcal{Q}-1},
        \label{Schifo1}
    \end{split}
\end{equation}
where the last inequality comes from the choice of $\oldep{eps:3}$, see Notation \ref{notation1}.

The next step in the proof of the proposition is to show that:
\begin{equation}\mathcal{S}^{\mathcal{Q}-1}(P_{\Pi(Q_0)}(F(Q_0)))\geq \frac{c(\Pi(Q_0))\diam Q_0^{\mathcal{Q}-1}}{A_0^{\mathcal{Q}-1}}.
\label{eq:200}
\end{equation}
In order to ease notations in the following we let $x=\mathfrak{c}(Q_0)\delta_{10A_0^{-1}\diam Q_0}(\mathfrak{n}(Q_0))$ and define:
$$B_{+}:=B(x,A_0^{-1}\diam Q_0) \qquad \text{and}\qquad B_{-}:=B(x,A_0^{-1}\diam Q_0)\delta_{20A_0^{-1}\diam Q_0}(\mathfrak{n}(Q_0)^{-1}).$$
Before proceeding further with the proof of \eqref{eq:200}, we give a brief outline of what we are going to do, hoping to help the reader in keeping track of what purpose each of our computations serve. 
As a first step towards the proof of \eqref{eq:200}, we prove that $B_{\pm}\subseteq G_{\pm }(Q_0)$.
Note that this implies that each one of the $B_+$ and $B_-$ are on different sides of the plane $\Pi(Q_0)$. Let $\overline{Q}$ is the element of $\mathcal{M}(C,\iota)$ containing $Q_0$ and recall that by Lemma \ref{lemma:intersectionG} $\mathfrak{G}_+(\overline{Q})$ and $\mathfrak{G}_-(\overline{Q})$ are disjoint open sets. Therefore, for any horizontal line parallel to the normal of the plane $\Pi(Q_0)$, with starting point in $B_+$ and end point in $B_-$, we can find an $y$ in such segment belonging the complement of $\mathfrak{G}(Q_0)$. Hence, our final step in the proof of \eqref{eq:200} is to show that $y\in F(Q_0)$, thus proving the inclusion $P_{\Pi(Q_0)}(Q_0)\subseteq P_{\Pi(Q_0)}(F(Q_0))$ and in turn our claim. To achieve this we will show that either $y$ belongs to $E(\vartheta,\gamma)$ or that there there is a centre $\mathfrak{c}(\tilde{Q})$ of a cube $\tilde{Q}$ in the family $\mathscr{I}(Q_0)$ such that $d(y,\mathfrak{c}(Q))\leq 2\oldC{child}\diam \tilde{Q}$.

Let us proceed with the proof of \eqref{eq:200}. We will prove that $B_+\subseteq G_+(Q_0)$ and $B_-\subseteq G_-(Q_0)$ separately, since the computations differ.

Let us begin with the proof of the inclusion $B_+\subseteq G_+(Q_0)$. For any $\Delta\in \mathbb{G}$ such that $\lVert\Delta\rVert\leq A_0^{-1}\diam Q_0$, we have:
\begin{equation}
    d(\mathfrak{c}(Q_0),x\Delta)=\lVert \delta_{10A_0^{-1}\diam Q_0}(\mathfrak{n}(Q_0))\Delta\rVert\leq 11A_0^{-1}\diam Q_0\leq A_0\diam Q_0\label{eq:n30}
\end{equation}
Moreover, from the definition of $B_+$ we infer:
\begin{equation}
\begin{split}
    \langle \pi_1(\mathfrak{c}(Q_0)^{-1}x\Delta),\mathfrak{n}(Q_0)\rangle=&\langle \pi_1(\delta_{10A_0^{-1}\diam Q_0}(\mathfrak{n}(Q_0))\Delta),\mathfrak{n}(Q_0)\rangle\\
    =&10A_0^{-1}\diam Q_0+\langle \pi_1\Delta,\mathfrak{n}(Q_0)\rangle\geq 9A_0^{-1}\diam Q_0.
\end{split}
\label{eq:n31}
\end{equation}
Inequalities \eqref{eq:n30} and \eqref{eq:n31} finally imply that $B_+\subseteq G_+(Q_0)$. 

Let us prove that $B_-\subseteq G_-(Q_0)$. Similarly to the previous case, for any $\lVert\Delta\rVert\leq A_0^{-1}\diam Q_0$, we have:
\begin{equation}
\begin{split}
      d(\mathfrak{c}(Q_0),x\Delta \delta_{20A_0^{-1}\diam Q_0}(\mathfrak{n}(Q_0)^{-1}))=&\lVert \delta_{10A_0^{-1}\diam Q_0}(\mathfrak{n}(Q_0))\Delta \delta_{20A_0^{-1}\diam Q_0}(\mathfrak{n}(Q_0)^{-1})\rVert\\
      \leq & 31A_0^{-1}\diam Q_0\leq A_0\diam Q_0.
\end{split}
  \label{eq:n32}
\end{equation}
Finally, we deduce that:
\begin{equation}
\begin{split}
    \big\langle \pi_1(\mathfrak{c}(Q_0)^{-1}&x\Delta \delta_{20A_0^{-1}\diam Q_0}(\mathfrak{n}(Q_0)^{-1})),\mathfrak{n}(Q_0)\big\rangle\\
    =&\big\langle \pi_1(\delta_{10A_0^{-1}\diam Q_0}(\mathfrak{n}(Q_0))\Delta \delta_{20A_0^{-1}\diam Q_0}(\mathfrak{n}(Q_0)^{-1})),\mathfrak{n}(Q_0)\big\rangle\\
    =&-10A_0^{-1}\diam Q_0+\langle \pi_1\Delta,\mathfrak{n}(Q_0)\rangle\leq -9A_0^{-1}\diam Q_0.
\end{split}
\label{eq:n33}
\end{equation}
Inequalities \eqref{eq:n32} and \eqref{eq:n33} together finally imply that $B_-\subseteq G_-(Q_0)$. 

Suppose now that $\overline{Q}$ is the unique cube in $\mathcal{M}(C,\iota)$ containing $Q_0$. Thanks to Lemma \ref{lemma:intersectionG} we know that the sets $\mathfrak{G}_{+}(\overline{Q})$ and $\mathfrak{G}_-(\overline{Q})$ are open and disconnected. With this in mind, for any $a\in B_+$ we define the curve $\gamma_a:[0,1]\to\mathbb{G}$ as:
$$\gamma_a(t):=a\delta_{20A_0^{-1}\diam Q_0 t}(\mathfrak{n}(Q_0)^{-1}).$$
Note that by definition of $B_{-}$, we have that $\gamma_a(1)\in B_-$. On the other hand, since $\gamma_a(0)\in B_+$ we infer that $\gamma_a$ must meet the complement of $\mathfrak{G}(\overline{Q})$ at $y=\gamma_a(s)$ for some $s\in(0,1)$.
We can estimate the distance of $y$ from $\mathfrak{c}(Q_0)$ in the following way:
\begin{equation}
\begin{split}
    d(y,\mathfrak{c}(Q_0))\leq& d(a\delta_{20 A_0^{-1}\diam Q_0s}(\mathfrak{n}(Q_0)),\mathfrak{c}(Q_0))\leq d(a,\mathfrak{c}(Q_0))+20 A_0^{-1}\diam Q_0 s\\
    \leq& d(x,\mathfrak{c}(Q_0))+d(x,a)+20 A_0^{-1}\diam Q_0 s\\
    \leq&  10A_0^{-1}\diam Q_0+A_0^{-1}\diam Q_0+20 A_0^{-1}\diam Q_0 s\leq 40A_0^{-1}\diam Q_0<(A_0-2)\diam Q_0,
    \label{eq:31}
\end{split}
\end{equation}
where the first inequality in the last line above comes from the definition of $x$ and the fact that $a\in B(x,A_0^{-1}\diam Q_0)$. The above computation, implies that if $\overline{Q}$ is the cube of $\mathcal{M}(C,\iota)$ containing $Q_0$, then $y\in I(\overline{Q})$.
Furthermore, thanks to inequality \eqref{eq:31}, the choice of $A_0$ and Proposition \ref{prop:centri} we have:
\begin{equation}
\begin{split}
    \dist(y,E(\vartheta,\gamma)\setminus Q_0)\geq& \dist(\mathfrak{c}(Q_0),E(\vartheta,\gamma)\setminus Q_0)-d(y,\mathfrak{c}(Q_0))\\\geq& 64^{-1}\zeta^2\diam Q_0-40A_0^{-1}\diam Q_0\geq 100 A_0^{-1}\diam Q_0.\end{split}
    \label{eq:32}
\end{equation}
Thanks to \eqref{eq:31} and \eqref{eq:32} we deduce that:
\begin{equation}
    \dist(y,E(\vartheta,\gamma)\setminus Q_0)\geq100 A_0^{-1}\diam Q_0>d(y,\mathfrak{c}(Q_0))\geq \dist(y,Q_0).
    \label{eq:33}
\end{equation}
Therefore, if $z\in E(\vartheta,\gamma)$ is the point of minimal distance of $y$ from $E(\vartheta,\gamma)$, \eqref{eq:33} implies that $z\in Q_0\cap E(\vartheta,\gamma)$. Furthermore, since by assumption $y\not\in \mathfrak{G}(\overline{Q})$ and by \eqref{eq:31} we have $y\in I(\overline{Q})$, Proposition \ref{lemma2.28} implies:
\begin{equation}
    d(z,y)=\dist(y,E(\vartheta,\gamma))\leq4A_0^{-1}d(y)< d(y)/10,
    \label{eq:34}
\end{equation}
where the last inequality can be strict only if $d(y)>0$. In this case, the definition of the function $d$ (see \eqref{eq:n35}) implies that:
\begin{equation}
d(z)\geq d(y)-d(z,y)>\frac{9d(y)}{10},
    \label{eq:n1010}
\end{equation}
and thus $z$ cannot be contained in a cube $Q\in \Delta(C,\iota)$ with $\diam Q\leq9d(y)/10$, where last inequality is strict only if $d(y)>0$.

If $d(y)=0$, the bound \eqref{eq:34} implies that $d(y,z)=0$ and thus since $E(\vartheta,\gamma)$ is compact we have $y=z\in E(\vartheta,\gamma)$. Therefore:
$$y\in E(\vartheta,\gamma)\cap Q_0\subseteq C\cap Q_0\cup\bigcup_{Q\in\mathscr{I}(Q_0)} Q\subseteq C\cap Q_0\cup \bigcup_{Q\in \mathscr{I}(Q_0)} B(\mathfrak{c}(Q),2\oldC{child}\diam Q)= F(Q_0).$$

If on the other hand $d(y)>0$, we claim that there is a cube $Q_1\in\Delta(C,\iota)$, contained in $Q_0$ and possibly $Q_0$ itself, such that:
\begin{itemize}
\item[(a)] $z\in Q_1$ and for any cube $Q\in\Delta(C,\iota)$ contained in $Q_1$ we have $z\not\in Q$,
\item[(b)]$\diam Q_1\geq9d(y)/10$,
    \item[(c)] there exists a $\tilde{Q}\in\mathscr{I}(Q_0)$, that is a child of $Q_1$ and for which $z\in\tilde{Q}$,
\end{itemize}
Let us verify that such a cube $Q_1$ exists.
Since $z\in Q_0$, for any cube $Q\in\Delta(C,\iota)$ such that $Q\subseteq Q_0$ and $z\in Q$ we have: 
\begin{equation}
    9d(y)/10\leq d(z)\leq \dist(z,Q)+\diam Q=\diam Q,
    \label{eq:disproj}
\end{equation}
where the first inequality above comes from \eqref{eq:n1010} and the second one from the definition of $d$.
Let $Q_1$ be the smallest cube of $\Delta(C,\iota)$ containing $z$ and note that for any cube $Q\subseteq Q_1$ belonging to $\Delta(C,\iota)$ we have that $z\not\in Q$. This proves (a) and (b).
In order to prove (c), we note that any ancestor of $Q_1$ in $\Delta(E(\vartheta,\gamma),\iota)$ must be contained in $\Delta(C,\iota)$. Furthermore, since the condition $\diam Q_1\geq 9 d(y)/10$ implies that $z\in E(\vartheta,\gamma)\setminus C$, we infer that there must exist a cube $\tilde{Q}$ in $\mathscr{I}(Q_0)$ for which $z\in \tilde{Q}$. Such cube must be a child of $Q_1$ otherwise the maximality of $\tilde{Q}$ would be contradicted.

Let us use (a), (b) and (c) to conclude the proof of the theorem. Items (a), (b) and inequality \eqref{eq:34} imply:
\begin{equation}
    \dist(y,Q_1)\leq d(y,z)=\dist(y,E(\vartheta,\gamma))\leq d(y)/10\leq \diam Q_1/9.
    \label{eq:n36}
\end{equation}
Therefore, Proposition \ref{prop:stabcubi} together with \eqref{eq:n36} imply:
\begin{equation}
    d(\mathfrak{c}(\tilde{Q}),y)\leq d(\mathfrak{c}(\tilde{Q}),z)+d(z,y)\leq \diam \tilde{Q}+\diam Q_1/9\leq \diam \tilde{Q}+\oldC{child}\diam\tilde{Q}/9<2 \oldC{child}\diam \tilde{Q}.
    \label{eq:n37}
\end{equation}
The bound \eqref{eq:n37} finally proves that $y\in F(Q_0)$, and more precisely we have shown that for any $a\in B_+$ the curve $\gamma_a$ meets the set $F(Q_0)$.
In turn, this shows that $F(Q_0)$ has big projections, indeed thanks to Proposition \ref{cor:2.2.19} we infer:
\begin{equation}
    \mathcal{S}^{\mathcal{Q}-1}\Big(P_{\Pi(Q_0)}(F(Q_0))\Big)\geq \mathcal{S}^{\mathcal{Q}-1}\Big(P_{\Pi(Q_0)}(B(x,A_0^{-1}\diam Q_0)\Big)= c(\Pi(Q_0))A_0^{-(\mathcal{Q}-1)}\diam Q_0^{\mathcal{Q}-1}.
    \label{eq:202}
\end{equation}
This implies thanks to \eqref{Schifo1}, \eqref{eq:202} and Proposition \ref{cor:2.2.19} that:
\begin{equation}
\begin{split}
     \mathcal{S}^{\mathcal{Q}-1}(P_{\Pi(Q_0)}(Q_0\cap C))
    \geq&\mathcal{S}^{\mathcal{Q}-1}(P_{\Pi(Q_0)}(F(Q_0)))-\mathcal{S}^{\mathcal{Q}-1}\bigg(P_{\Pi(Q_0)}\Big(\bigcup_{Q\in \mathscr{I}(Q_0)} B(\mathfrak{c}(Q),\oldC{child}\diam Q_0)\Big)\bigg)\\
    \geq& \frac{c(\Pi(Q_0))}{2A_0^{\mathcal{Q}-1}}\diam Q_0^{\mathcal{Q}-1}\geq\frac{\diam Q_0^{\mathcal{Q}-1}}{2A_0^{\mathcal{Q}-1}}.
    \nonumber
\end{split}
\end{equation}
\end{proof}

\begin{osservazione}\label{rk:propasjn}
Suppose $\psi$ is a Radon on $\mathbb{G}$ supported on the compact set $K$ and satisfying the following assumptions: 
\begin{itemize}
    \item[(i)] there exists a $\delta\in\N$ such that $\delta^{-1}\leq \Theta^{\mathcal{Q}-1}_*(\psi,x)\leq \Theta^{\mathcal{Q}-1,*}(\psi,x)\leq \delta$ for $\psi$-almost every $x\in\mathbb{G}$,
        \item[(ii)] $ \limsup_{r\to 0}d_{x,4kr}(\psi,\mathfrak{M})\leq 4^{-(\mathcal{Q}+1)}\oldep{e}(\delta)$ for $\psi$-almost every $x\in \mathbb{G}$.
\end{itemize}
Then Propositions \ref{prop:nearnorm}, \ref{lemma2.26},  \ref{lemma:intersectionG}, \ref{lemma2.28} and Theorem \ref{TH:proiezioni} can proved repeating verbatim their proofs, where in this case the compact $C$ is substituted with $\tilde{C}$, the compact set yielded by Proposition \ref{rk:primo} and instead of using of Proposition \ref{lemmaduepsei} we always use Proposition \ref{rk:C}.
\end{osservazione}

\subsection{Construction of the \texorpdfstring{$\phi$}{Lg}-positive intrinsic Lipschitz graph}
\label{intr:graph}

This subsection is devoted to the proof of the main result of Section \ref{sec:main}, Theorem \ref{ueue} that we restate here for reader's convenience:

\ueue*

The proof of Theorem \ref{ueue} follows the following argument. Fixed a cube $Q\in\mathcal{M}(C,\iota)$, we prove that the family $\mathcal{B}(Q)$ of the maximal sub-cubes of $Q$ having small projection on $\Pi(Q)$, thanks to Theorem \ref{TH:proiezioni} is small in measure. Therefore, we can find a cube $Q^\prime\in \Delta(C,\iota)\setminus \mathcal{B}(Q)$, that is contained in $Q$, and for which any sub-cube $\tilde{Q}$ of $Q$ has big projections on $\Pi(Q)$. This independence on the scales, thanks to Proposition \ref{prop:cono} implies that $C\cap Q$ is a $\Pi(Q)$-intrinsic Lipschitz graph.

\begin{proposizione}\label{intrgraph}
Suppose $E$ is a Borel subset of $\mathbb{G}$ and assume there is a plane $W\in\G(\mathcal{Q}-1)$ and an $\alpha>0$ such that for any $w\in E$ we have: \begin{equation}
    E\subseteq wC_W(\alpha).
    \label{eq:n60}
\end{equation}
Then $E$ is contained in an intrinsic Lipschitz graph.
\end{proposizione}

\begin{proof}
Thanks to the assumption on $E$, for any $w_1,w_2\in E$ we have $w_1^{-1}w_2\in C_W(\alpha)$.
This implies that for any $v\in P_W(E)$, there exists a unique $w\in E$ such that $P_W(w)=v$, otherwise we would have $w^{-1}_1w_2\in \mathfrak{N}(W)$.

Let $f:P_W(E)\to \mathfrak{N}(V)$ be the map associating every $w\in P_W(E)$ to the only element in its preimage $P_W^{-1}(w)$.
With this definition we have that the set $\text{gr}(f):=\{vf(v):v\in P_W^{-1}(E)\}$ coincides with $E$ and thus it is an intrinsic Lipschitz graph since
$\text{gr}(f)\subseteq vC_W(\alpha)$,
for any $v\in E$.
\end{proof}

\begin{proposizione}\label{propCa}
Defined $\newep\label{eps:4}:=\min\{\oldep{eps:1},(32\vartheta \oldC{C:up}\oldC{C:VIB}A_0^{\mathcal{Q}-1})^{-1}\}$. There exists a compact set $C_1\subseteq C$ and a $\iota_1\in\N$ such that:
\begin{itemize}
    \item[(i)]$\phi(C\setminus C_1)\leq \oldep{eps:4}\phi(C)$,
    \item[(ii)] whenever $Q\in \Delta(C_1,\iota_1)$, we have
    $(1-\oldep{eps:3}/32)\phi(Q)\leq \phi(Q\cap C)$.
\end{itemize}
\end{proposizione}

\begin{proof}
First of all, we prove that the set $\Delta(C,\iota)$ is a $\phi\llcorner C$ Vitali relation. It is immediate to see that the family $\Delta(C,\iota)$ is a fine covering of $C$. Furthermore, let $E$ be a Borel set contained in $C$ and suppose $\mathcal{A}\subseteq \Delta(C,\iota)$ is a fine covering of $E$. Defined $\mathcal{A}^*:=\{Q\in\mathcal{A}:Q\text{ is maximal}\}$, it is immediate too see that:
$$\bigcup_{Q\in\mathcal{A}}Q=\bigcup_{Q\in\mathcal{A}^*}Q,$$ and thus the family $\mathcal{A}^*$ is still a covering of $E$. The maximality of the elements of $\mathcal{A}^*$ implies that they are pairwise disjoint and thus $\Delta(C,\iota)$ is a $\phi$-Vitali relation in the sense of section 2.8.16 of \cite{Federer1996GeometricTheory}. Therefore, thanks to Theorem 2.9.11 in \cite{Federer1996GeometricTheory}, we deduce that:
\begin{equation}
    \lim_{Q\to x}\frac{\phi(C\cap Q)}{\phi(Q)}=1,
        \label{eq:40}
\end{equation}
for $\phi$-almost every $x\in C$. For any $j\in\N$, define the functions $f_j(x):=\phi(C\cap Q_j(x))/\phi(Q_j(x))$, where $Q_j(x)$ is the unique cube of the generation $\Delta_j$ containing $x$. The identity \eqref{eq:40} implies that $\lim_{j\to\infty} f_j(x)=1$ for $\phi$-almost every $x\in C$. Therefore, Severini-Egoroff theorem implies that we can find a compact subset $C_1$ of $C$ such that $\phi(C\setminus C_1)\leq \oldep{eps:4}\phi(C)$ and $f_j(x)$ converges uniformly to $1$ on $C_1$. This proves (i) and (ii) at once.
\end{proof}

\begin{osservazione}\label{prop:2.24holds}
If $\psi$ is a Radon measure on $\mathbb{G}$ satisfying the hypothesis of Proposition \ref{rk:primo} and $\tilde{C}$ is the compact set yielded by Proposition \ref{rk:C} with the same argument we employed above we can construct a compact subset $\tilde{C}_1$ of $\tilde{C}$ and a $\tilde{\iota}_1\in\N$ satisfying (i) and (ii) of Proposition \ref{propCa} provided $\oldep{eps:4}$ is substituted with $\oldep{eps:4}(2\delta):=\min\{\oldep{eps:1},(64\delta \oldC{C:up}\oldC{C:VIB}(2\delta)A_0^{\mathcal{Q}-1}(2\delta))^{-1}\}$, $\oldep{eps:3}$ with $\oldep{eps:3}(2\delta)$ and $\Delta(C,\iota)$ with $\Delta^\psi(\tilde{C};2\delta,\gamma,\tilde{\iota}_1)$.
\end{osservazione}

\begin{teorema}\label{th:intr:lipgraph}
Let $C_1$ be the compact set yielded by Proposition \ref{propCa} then, there exists a cube $Q^\prime\in \Delta(C_1,2\iota_1)$ such that $Q^\prime\cap C_1$ is an intrinsic Lipschitz graph of positive $\phi$-measure.
\end{teorema}

\begin{proof}
For any $Q_0\in \mathcal{M}(C_1,2\iota_1)$, Theorem \ref{TH:proiezioni} and Proposition \ref{propCa} imply that:
\begin{equation}
    \mathcal{S}^{\mathcal{Q}-1}(P_{\Pi(Q_0)}(Q_0\cap C))\geq \frac{\diam Q_0^{\mathcal{Q}-1}}{2A_0^{\mathcal{Q}-1}}.
    \label{eq:n55}
\end{equation}
Therefore, for any $Q_0\in\mathcal{M}(C_1,2\iota_1)$ we let $\mathcal{B}(Q_0)$ be the family of the  maximal cubes $Q\in \Delta(C_1,2\iota_1)$, contained in $Q_0$ for which:
    \begin{equation}
     \mathcal{S}^{\mathcal{Q}-1}\big(P_{\Pi(Q_0)}(E(\vartheta,\gamma)\cap Q)\big)< \frac{\diam Q^{\mathcal{Q}-1}}{4\oldC{C:VIB}^2A_0^{\mathcal{Q}-1}},
       \label{eq:50}
       \end{equation}
and we define $\mathscr{B}(Q_0):=\bigcup_{Q\in \mathcal{B}(Q_0)}Q$.
The first step of the proof of the theorem is to show that:
\begin{equation}
    \phi(C\cap [Q_0\setminus \mathscr{B}(Q_0)])>\frac{\phi(Q_0)}{8\vartheta \oldC{C:up}\oldC{C:VIB}A_0^{\mathcal{Q}-1}},\qquad\text{for any }Q_0\in\mathcal{M}(C_1,2\iota_1).
    \label{eq:n53}
\end{equation}

Throughout this paragraph we shall assume that $Q_0\in\mathcal{M}(C_1,2\iota_1)$ is fixed. The maximality of the elements $Q$ of $\mathcal{B}(Q_0)$ implies that they are pairwise disjoint and since $Q\cap E(\vartheta,\gamma)\neq \emptyset$, Remark \ref{rk:rephrase} yields:
\begin{equation}
     \mathcal{S}^{\mathcal{Q}-1}\big(P_{\Pi(Q_0)}(E(\vartheta,\gamma)\cap Q)\big)< \frac{\diam Q^{\mathcal{Q}-1}}{4\oldC{C:VIB}^2A_0^{\mathcal{Q}-1}}\leq \frac{\phi(Q)}{4\oldC{C:VIB}A_0^{\mathcal{Q}-1}}.
     \label{eq:n51}
       \end{equation}
Thanks to Proposition \ref{cor:2.2.19} and Corollary \ref{cor:cor1}, we have:
\begin{equation}
    \phi(C\cap [Q_0\setminus \mathscr{B}(Q_0)])\geq  \frac{\mathcal{S}^{\mathcal{Q}-1}(C\cap [Q_0\setminus \mathscr{B}(Q_0)])}{\vartheta}\geq \frac{\mathcal{S}^{\mathcal{Q}-1}\big(P_{\Pi(Q_0)}(C\cap [Q_0\setminus \mathscr{B}(Q_0)])\big)}{2c(\Pi(Q_0))\vartheta}.
    \label{eq:numero1}
\end{equation}
On the other hand, thanks to \eqref{eq:n55} we infer that:
\begin{equation}
    \begin{split}
       \mathcal{S}^{\mathcal{Q}-1}\big(P_{\Pi(Q_0)}(C\cap [Q_0\setminus \mathscr{B}(Q_0)])\big)
        &\geq \mathcal{S}^{\mathcal{Q}-1}(P_{\Pi(Q_0)}(C\cap Q_0))-\mathcal{S}^{\mathcal{Q}-1}\Big( P_{\Pi(Q_0)}\Big(E(\vartheta,\gamma)\cap\mathscr{B}(Q_0)\Big)\\
        &\geq \frac{\diam Q_0^{\mathcal{Q}-1}}{2A_0^{\mathcal{Q}-1}}
        -\sum_{Q\in\mathcal{B}(Q_0)}\mathcal{S}^{\mathcal{Q}-1}( P_{\Pi(Q_0)}(E(\vartheta,\gamma)\cap Q))).
        \label{eq:410}
    \end{split}
\end{equation}
Since $Q_0\cap E(\vartheta,\gamma)\neq \emptyset$, Remark \ref{rk:rephrase}, \eqref{eq:n51}, \eqref{eq:410} and the fact that the elements in $\mathcal{B}(Q_0)$ are disjoint imply:
\begin{equation}
\begin{split}
    \mathcal{S}^{\mathcal{Q}-1}\big(P_{\Pi(Q_0)}(C\cap [Q_0\setminus \mathscr{B}(Q_0)])\big)\geq&\frac{\phi(Q_0)}{2\oldC{C:VIB}A_0^{\mathcal{Q}-1}}-\frac{1}{4\oldC{C:VIB}A_0^{\mathcal{Q}-1}}\sum_{Q\in\mathcal{B}(Q_0)}\phi(Q)\\=&\frac{\phi(Q_0)}{2\oldC{C:VIB}A_0^{\mathcal{Q}-1}}-\frac{1}{4\oldC{C:VIB}A_0^{\mathcal{Q}-1}}\phi(\mathscr{B}(Q_0)).
\end{split}
    \label{eq:n52}
\end{equation}
Thanks to the bounds \eqref{eq:numero1} and \eqref{eq:n52}, we eventually deduce that:
\begin{equation}
    \begin{split}
        2c(\Pi(Q_0))\vartheta\phi(C\cap [Q_0\setminus \mathscr{B}(Q_0)])\geq& \frac{\phi(Q_0)}{2\oldC{C:VIB}A_0^{\mathcal{Q}-1}}-\frac{1}{4\oldC{C:VIB}A_0^{\mathcal{Q}-1}}\phi(\mathscr{B}(Q_0))
        \\=&\frac{\phi(Q_0)}{4\oldC{C:VIB}A_0^{\mathcal{Q}-1}}+\frac{1}{4\oldC{C:VIB}A_0^{\mathcal{Q}-1}}\phi(Q_0\setminus \mathscr{B}(Q_0)),
        \label{eq:n57}
    \end{split}
\end{equation}
where the last inequality above follows since $\mathcal{B}(Q_0)\subseteq Q_0$.
Inequality \eqref{eq:n57} together with the bound from above on $c(\Pi(Q_0))$, see Proposition \ref{cor:2.2.19}, immediately imply \eqref{eq:n53}.

Now that \eqref{eq:n53} is proved, we want to construct a cube $Q^\prime\in \Delta(C_1,2\iota_1)$ disjoint from $\bigcup_{Q_0\in\mathcal{M}(C_1,2\iota_1)}\mathscr{B}(Q_0)$ and such that $\phi(C_1\cap Q^\prime)>0$.
Since the elements of $\mathcal{M}(C_1,2\iota_1)$ are pairwise disjoint and their union covers $C_1$, we infer that:
\begin{equation}
\begin{split}
     \phi\bigg(C_1\setminus \bigcup_{Q_0\in\mathcal{M}(C_1,2\iota_2)}\mathscr{B}(Q_0)\bigg)=&\phi\bigg(\bigcup_{Q_0\in\mathcal{M}(C_1,2\iota_1)} C_1\cap [Q_0\setminus \mathscr{B}(Q_0)]\bigg)=\sum_{Q_0\in \mathcal{M}(C_1,2\iota_1)}\phi(C_1\cap [Q_0\setminus \mathscr{B}(Q_0)])\\
     \geq&\sum_{Q_0\in \mathcal{M}(C_1,2\iota_1)}\phi(C\cap [Q_0\setminus \mathscr{B}(Q_0)])-\phi((C\setminus C_1)\cap Q_0)\\
     \geq&\sum_{Q_0\in \mathcal{M}(C_1,2\iota_1)}\frac{\phi(Q_0)}{8\vartheta \oldC{C:up}\oldC{C:VIB}A_0^{\mathcal{Q}-1}}-\phi(C\setminus C_1)\geq \frac{\phi(C_1)}{8\vartheta \oldC{C:up}\oldC{C:VIB}A_0^{\mathcal{Q}-1}}-\oldep{eps:4}\phi(C).
     \label{eq:n2000}
\end{split}
\end{equation}
where the first inequality of the last line above follows from \eqref{eq:n52}. Therefore, 
Proposition \ref{propCa} and \eqref{eq:n2000} imply:
\begin{equation}
\phi\bigg(C_1\setminus \bigcup_{Q_0\in\mathcal{M}(C_1,2\iota_2)}\mathscr{B}(Q_0)\bigg)\geq \frac{1-\oldep{eps:4}}{8\vartheta \oldC{C:up}\oldC{C:VIB}A_0^{\mathcal{Q}-1}}\phi(C)-\oldep{eps:4}\phi(C)\geq \frac{\phi(C)}{16\vartheta \oldC{C:up}\oldC{C:VIB}A_0^{\mathcal{Q}-1}}.    \label{eq:n2001}
\end{equation}
Inequality \eqref{eq:n2001} implies that there must exist a cube $Q_0^\prime\in\mathcal{M}(C_1,2\iota_1)$ such that $\phi(C_1\setminus \bigcup_{Q\in\mathcal{B}(Q_0^\prime)}Q)>0$. 
Defined $\mathscr{G}$ to be the set of maximal cubes in $\Delta(C_1,2\iota_1)\setminus \mathcal{B}(Q_0^\prime)$ contained in $Q_0^\prime$, we can find at least a cube $Q^\prime\in \mathscr{G}$ for which $\phi(C_1\cap Q^\prime)>0$. Furthermore, thanks to the maximality of the elements in $\mathcal{B}(Q_0^\prime)$ we also deduce that any sub-cube of $Q^\prime$ is not an element of $\mathcal{B}(Q_0^\prime)$.

We prove now that $C_1\cap Q^\prime$ is contained in an intrinsic Lipschitz graph. Indeed, we claim that: \begin{equation}
    x_1^{-1}x_2\in C_{\Pi(Q_0^\prime)}(2\alpha_0)\qquad\text{for any }x_1,x_2\in C_1\cap Q^\prime,
\label{eq:cono}
\end{equation}
where $\alpha_0$ was defined in Proposition \ref{prop:cono}.  Fix $x_1,x_2\in C_1\cap Q^\prime$ and note that there exists a unique $j\in\N$ such that:
$$R\gamma^{-1}2^{-jN+5}\leq d(x_1,x_2)\leq R\gamma^{-1}2^{-(j-1)N+5}.$$
For $i=1,2$ we let $Q_{x_i}$ be the unique cubes in the $j$-th layer of cubes $\Delta_{j}$ for which $x_i\in Q_{x_i}$. Suppose $Q^\prime\in \Delta_{\overline{j}}$ and note that Theorem \ref{evev}(iv) and the choice of $j$ imply:
\begin{equation}
R\gamma^{-1}2^{-jN+5}\leq d(x_1,x_2)\leq \diam Q^\prime \leq \gamma^{-1}2^{-\overline{j}N+5}.
\label{eq:n2010}
\end{equation}
The chain of inequalities \eqref{eq:n2010} implies that $\overline{j}\leq j$ and thus by Theorem \ref{evev}(iii) we infer that $Q_{x_i}\subseteq Q^\prime$ for $i=1,2$. Furthermore, thanks to Theorem \ref{evev}(ii), (v), for $i=1,2$ we have:
\begin{equation}
    R\diam Q_{x_i}\leq R\gamma^{-1}2^{-jN+5}\leq d(x_1,x_2)\leq R\gamma^{-1}2^{-(j-1)N+5}\leq 2^{N+6}\gamma^{-1}R2^{-jN-1}\leq 2^{N+6}\zeta^{-2}R\diam Q_{x_i},
    \label{eq:cono1}
\end{equation}
Since $Q_{x_i}\in \Delta(C_1,2\iota_1)$, Lemma \ref{lemmaduepsei} implies that $\alpha(Q_{x_i})\leq \oldep{e}$ for $i=1,2$. Furthermore, the construction of $Q^\prime$ insures that for any cube $Q\in \Delta(C_1,2\iota_1)$ contained in $Q^\prime$, we have:
\begin{equation}
    \mathcal{S}^{\mathcal{Q}-1}\big(P_{\Pi(Q_0^\prime)}(E(\vartheta,\gamma)\cap Q)\big)\geq \frac{\diam Q^{\mathcal{Q}-1}}{4\oldC{C:VIB}^2A_0^{\mathcal{Q}-1}}.
    \label{eq:cono3}
\end{equation}
This proves that the hypothesis of Proposition \ref{prop:cono} are satisfied and thus $x_1\in x_2 C_{\Pi(Q_0^\prime)}(2\alpha_0)$. Finally $C_1\cap Q^\prime$ is proved to be contained in an intrinsic Lipschitz graph by means of Proposition \ref{intrgraph}.
\end{proof}

We observe that is immediate to infer that Theorem \ref{th:intr:lipgraph}
directly implies Theorem \ref{ueue}. Finally, 
{\color{red}[...]}
Theorem \ref{th:quasifprett} of Section \ref{sec:end}.

\begin{proposizione}\label{prop:mild}
Suppose $\psi$ is a Radon measure on $\mathbb{G}$ supported on a compact set $K$ satisfying the two following assumptions:
\begin{itemize}
    \item[(i)] $\delta^{-1}\leq \Theta^{\mathcal{Q}-1}_*(\psi,x)\leq \Theta^{\mathcal{Q}-1,*}(\psi,x)\leq \delta$ for $\psi$-almost every $x\in\mathbb{G}$,
        \item[(ii)] $ \limsup_{r\to 0}d_{x,4kr}(\psi,\mathfrak{M})\leq 4^{-(\mathcal{Q}+1)}\oldep{e}(\delta)$ for $\psi$-almost every $x\in \mathbb{G}$.
\end{itemize}
Then if $\tilde{C}_1$ be the compact set yielded by Remark \ref{prop:2.24holds} then, there exists a cube $Q^\prime\in \Delta^\psi(\tilde{C}_1;2\delta,\gamma,2\tilde{\iota}_1)$ such that $Q^\prime\cap \tilde{C}_1$ is an intrinsic Lipschitz graph of positive $\psi$-measure.
\end{proposizione}

\begin{proof}
Thanks to  Propositions \ref{rk:primo}, \ref{rk:C} and Remarks \ref{rk:propasjn} and \ref{prop:2.24holds}, the verbatim argument we used to prove \ref{TH:proiezioni} can be applied here.
\end{proof}


\section{Conclusions and discussion of the results}\label{sec:end}

In this section we use the main result of Section \ref{sec:main}, i.e. Theorem \ref{ueue}, to deduce a number of consequences. First of all we prove Theorem \ref{main}, that is the main result of this paper, that is a $1$-codimensional extension of the Marstrand-Mattila rectifiability criterion to general Carnot groups. Secondly, we provide in Corollary \ref{fpsrig} a rigidity results for finite perimeter sets in Carnot groups: we are able to show that if locally a Caccioppoli set is not too far from its natural tangent plane, then its boundary is an \emph{intrinsic rectifiable set}, see Definition \ref{def:rectif}. Eventually, we use Theorem \ref{main} to prove that a $1$-codimensional version of Preiss's rectifiability theorem in the Heisenberg groups $\HH^n$.

\subsection{Main results}
In this subsection we finally conclude the proof of the main results of this work.

\begin{teorema}\label{main}
 Suppose $\phi$ is a Radon measure on $\mathbb{G}$ and let $\tilde{d}(\cdot,\cdot)$ be a left invariant, homogeneous distance on $\mathbb{G}$. Assume further that for $\phi$-almost all $x\in \mathbb{G}$ we have:
\begin{itemize}
    \item[(i)] $$0<\liminf_{r\to 0}\frac{\phi(\tilde{B}(x,r))}{r^{\mathcal{Q}-1}}\leq \limsup_{r\to 0}\frac{\phi(\tilde{B}(x,r))}{r^{\mathcal{Q}-1}}<\infty,$$
    where $\tilde{B}(x,r)$ is the ball relative to the metric $\tilde{d}$ centred at $x$ of radius $r>0$,
    \item[(ii)]$\Tan_{\mathcal{Q}-1}(\phi,x)\subseteq \mathfrak{M}$, where $\mathfrak{M}$ is the family of $1$-codimensional flat measures introduced in Definition \ref{def:flatmeasures}.
\end{itemize}
Then $\phi$ is absolutely continuous with respect to $\mathcal{S}^{\mathcal{Q}-1}$ and $\mathbb{G}$ can be covered $\phi$-almost all with countably many $C^1_\mathbb{G}$-surfaces.
\end{teorema}

\begin{proof}
Since $\tilde{d}$ is bi-Lipschitz equivalent to $d$, see for instance Corollary 5.15 in \cite{equivmetr}, the hypothesis (i) implies:
\begin{equation}
    0<\Theta^{\mathcal{Q}-1}_*(\phi,x)\leq\Theta^{\mathcal{Q}-1,*}(\phi,x)<\infty,
    \label{eq:n4000}
\end{equation}
for $\phi$-almost every $x\in\mathbb{G}$. For any $\vartheta,\gamma,R\in\N$ we define:
\begin{equation}
\begin{split}
    E(\vartheta,\gamma,R):=&\{x\in \overline{B(0,R)}:\vartheta^{-1}r^{\mathcal{Q}-1}\leq \phi(B(x,r))\leq \vartheta r^{\mathcal{Q}-1}\text{ for any }0<r<1/\gamma\}.
    \nonumber
\end{split}    
\end{equation}
It is possible to prove, with the same arguments used in the proof of Proposition \ref{prop:BIGGI}, that the $E(\vartheta,\gamma, R)$ are compact sets and: 
\begin{equation}
    \phi(\mathbb{G}\setminus \bigcup_{\vartheta,\gamma,R} E(\vartheta,\gamma,R))=0,
    \label{eq:n4030}
\end{equation}
Thus, if $A$ is an $\mathcal{S}^{\mathcal{Q}-1}$-null Borel set, Proposition \ref{prop:dens} yields:
\begin{equation}
    \phi(A)\leq \sum_{\vartheta,\gamma,R\in\N}\phi(A\cap E(\vartheta,\gamma,R))\leq \sum_{\vartheta,\gamma,R\in\N}\vartheta 2^{\mathcal{Q}-1}\mathcal{S}^{\mathcal{Q}-1}(A\cap E(\vartheta,\gamma,R))=0.
    \nonumber
\end{equation}
The above computation proves that $\phi$ is absolutely continuous with respect to $\mathcal{S}^{\mathcal{Q}-1}$ and just to fix notations we let $\rho\in L^1(\mathcal{S}^{\mathcal{Q}-1})$ be such that $\phi=\rho\mathcal{S}^{\mathcal{Q}-1}$.

As a second step, we show that $\mathbb{G}$ can be covered $\phi$-almost all with countably many intrinsic Lipschitz graphs. Assume by contradiction there are $\vartheta,\gamma, R\in\N$ for which we can find a subset of $E(\vartheta,\gamma,R)$ of positive $\phi$-measure that we denote, following the notations of Corollary \ref{corrii}, with $E(\vartheta,\gamma,R)^u$ and that has $\mathcal{S}^{\mathcal{Q}-1}$-null intersection with any intrinsic Lipschitz graph. Thanks to Corollary 2.9.11 of \cite{Federer1996GeometricTheory} it is immediate to see that:
$$\vartheta^{-1}\leq \Theta^{\mathcal{Q}-1}_*(\phi\llcorner E(\vartheta,\gamma,R)^u,x)\leq \Theta^{\mathcal{Q}-1,*}(\phi\llcorner E(\vartheta,\gamma,R)^u,x)\leq \vartheta,$$
for $\phi$-almost every $x\in E(\vartheta,\gamma,R)^u$. Furthermore, thanks to Proposition \ref{prop:locality}, we infer that $\Tan_{\mathcal{Q}-1}(\phi\llcorner E(\vartheta,\gamma,R)^u,x)\subseteq \mathfrak{M}$ for $\phi$-almost every $x\in E(\vartheta,\gamma,R)^u$. And since its hypothesis are satisfied, Theorem \ref{ueue}  implies that there exists an intrinsic Lipschitz graph $\Gamma$ such that $\phi(\Gamma\cap E(\vartheta,\gamma, R)^u)>0$. However, this is not possible since Proposition \ref{prop:dens} would yield:
$$0<\phi(\Gamma\cap E(\vartheta,\gamma, R)^u)\leq \vartheta 2^{\mathcal{Q}-1}\mathcal{S}^{\mathcal{Q}-1}(E(\vartheta,\gamma, R)^u\cap \Gamma),$$
and this contradicts the fact that $E(\vartheta,\gamma,R)$ intersects in a $\mathcal{S}^{\mathcal{Q}-1}$-null set every intrinsic Lipschitz graph.

This concludes the first part of the proof of the proposition. Up to this point we have shown that the sets $E(\vartheta,\gamma,R)$ for any choice of $\vartheta,\gamma,R$ are covered $\mathcal{S}^{\mathcal{Q}-1}$-almost all by countably many intrinsic Lipschitz graphs. Furthermore, since  $\phi\ll\mathcal{S}^{\mathcal{Q}-1}$, thanks to \eqref{eq:n4030} we conclude that $\phi$-almost all $\mathbb{G}$ can be covered by countably many intrinsic Lipschitz graphs.

In this paragraph, we assume that $\vartheta,\gamma,R\in\N$ are fixed. Thanks to Proposition \ref{prop:dens} we infer that $\mathcal{S}^{\mathcal{Q}-1}\llcorner E(\vartheta,\gamma,R)$  is mutually absolutely continuous with respect to $\phi\llcorner E(\vartheta,\gamma)$ and in particular:
$$\vartheta^{-1}\leq \rho(x)\leq \vartheta2^{\mathcal{Q}-1}\text{ for }\mathcal{S}^{\mathcal{Q}-1}\text{-almost every }x\in E(\vartheta,\gamma,R).$$
Let $\{\gamma_i\}_{i\in\N}$ be the sequence of intrinsic Lipschitz functions $\gamma_i:W_i\to\mathfrak{N}(W_i)$ for which $\phi(E(\vartheta,\gamma,R)\setminus \bigcup_{i\in\N} \text{gr}(\gamma_i))=0$ and let $E_i:=\text{epi}(\gamma_i)$ be the epigraph of the function $\gamma_i$, that was defined in \eqref{eq:n70}.
Thanks to Proposition \ref{prop:locality} we deduce that for $\phi$-almost every $x\in E(\vartheta,\gamma,R)\cap \text{gr}(\gamma_i)$ we have:
\begin{equation}
\begin{split}
    \Tan_{\mathcal{Q}-1}(\phi\llcorner E(\vartheta,\gamma,R)\cap \text{gr}(\gamma_i),x)=\rho(x)\Tan_{\mathcal{Q}-1}(\mathcal{S}^{\mathcal{Q}-1}\llcorner \text{gr}(\gamma_i),x)=\rho(x)\mathfrak{d}(x)\Tan_{\mathcal{Q}-1}(\lvert \partial E_i\rvert_{\mathbb{G}},x),
    \nonumber
\end{split}
\end{equation}
where $\mathfrak{d}$ is the density yielded by Remark \ref{rk:b1} and $\lvert\partial E_i\rvert_{\mathbb{G}}$ is the perimeter measure of $E_i$. Finally, Proposition \ref{prop:tang} implies that:
\begin{equation}
\begin{split}
    \Tan_{\mathcal{Q}-1}(\phi\llcorner E(\vartheta,\gamma,R)\cap \text{gr}(\gamma_i),x)\subseteq \rho(x)\mathfrak{d}(x)\{\lambda\mathcal{S}^{\mathcal{Q}-1}\llcorner V_i(x):\lambda\in [L_{\mathbb{G}}^{-1},l_{\mathbb{G}}^{-1}]\},
    \label{eq:n4020}
\end{split}
\end{equation}
where $V_i(x)\in \G(\mathcal{Q}-1)$ is the plane orthogonal to $\mathfrak{n}_{E_i}(x)$, the generalized inward normal introduced in Definition \ref{def:norm}. 

We now prove that \eqref{eq:n4020} implies that for $\mathcal{S}^{\mathcal{Q}-1}$-almost every $x\in \text{gr}(\gamma_i)$ and every $\alpha>0$ we have:
\begin{equation}
    \lim_{r\to 0}\frac{\mathcal{S}^{\mathcal{Q}-1}(\text{gr}(\gamma_i)\cap B(x,r)\setminus xX_{V_x}(\alpha))}{r^{\mathcal{Q}-1}}=0,
    \label{eq:n3000}
\end{equation}
where $X_{V_x}(\alpha):=\{w\in \mathbb{G}:\dist(w,V_x)\leq \alpha\lVert w\rVert\}$.
Thanks to \eqref{eq:n4020}, for $\mathcal{S}^{\mathcal{Q}-1}$-almost every $x\in \text{gr}(\gamma_i)$ and any sequence $r_i\to 0$, there exists a $\lambda>0$ for which:
\begin{equation}
    \frac{T_{x,r}\mathcal{S}^{\mathcal{Q}-1}\llcorner E(\vartheta,\gamma,R)}{r_i^{\mathcal{Q}-1}}\rightharpoonup \lambda \mathcal{S}^{\mathcal{Q}-1}\llcorner V_x.
    \label{eq:n27}
\end{equation}
The convergence in \eqref{eq:n27} implies that:
\begin{equation}
\begin{split}
    \lim_{i\to \infty}\frac{\mathcal{S}^{\mathcal{Q}-1}\llcorner \text{gr}(\gamma_i)(B(x,r_i)\setminus xX_{V_x}(\alpha))}{r_i^{\mathcal{Q}-1}}=&\lim_{i\to\infty}\frac{T_{x,r_i}(\mathcal{S}^{\mathcal{Q}-1}\llcorner \text{gr}(\gamma_i))\big(B(0,1)\setminus X_{V_x}(\alpha)\big)}{r_i^{\mathcal{Q}-1}}\\
    =&\lambda(\mathcal{S}^{\mathcal{Q}-1}\llcorner V_x)\big(B(0,1)\setminus X_{V_x}(\alpha)\big)=0,
    \label{eq:n28}
\end{split}
\end{equation}
where the second last identity above comes from the fact that $\mathcal{S}^{\mathcal{Q}-1} \big(V_x\cap \partial B(0,1)\setminus X_{V_x}(\alpha)\big)=0$ and Proposition 2.7 of \cite{DeLellis2008RectifiableMeasures}.

Proposition \ref{prop:C1H} and \eqref{eq:n3000} together imply that each one of the intrinsic Lipschitz graphs $\text{gr}(\gamma_i)$ can be covered $\mathcal{S}^{\mathcal{Q}-1}$-almost all with $C^1_\mathbb{G}$-surfaces. In particular this shows that for any $\vartheta,\gamma, R$ the set $E(\vartheta,\gamma,R)$ can be covered $\mathcal{S}^{\mathcal{Q}-1}$-almost all, and thus $\phi$-almost all, by countably many $C^1_\mathbb{G}$-surfaces.
This, thanks to the arbitrariness of $\vartheta,\gamma,R\in\N$ and \eqref{eq:n4030} concludes the proof of the theorem.
\end{proof} 

The following theorem trades off the regularity of tangents, that are assumed only to be close enough to flat measures, with a strengthened hypothesis on the $(\mathcal{Q}-1)$-density of $\phi$.

\begin{teorema}\label{th:quasifprett}
Suppose $\phi$ is a Radon measure on $\mathbb{G}$ and let $\tilde{d}(\cdot,\cdot)$ be a left invariant, homogeneous distance on $\mathbb{G}$. If there exists a $\delta\in\N$ such that:
\begin{equation}
    \delta^{-1}<\liminf_{r\to 0}\frac{\phi(\tilde{B}(x,r))}{r^{\mathcal{Q}-1}}\leq \limsup_{r\to 0}\frac{\phi(\tilde{B}(x,r))}{r^{\mathcal{Q}-1}}<\delta\qquad \text{for }\phi\text{-almost every }x\in\mathbb{G},
    \label{eq:nn1}
\end{equation}
    where $\tilde{B}(x,r)$ is the ball relative to the metric $\tilde{d}$ centred at $x$ of radius $r>0$, then we can find an $\varepsilon(\delta,\tilde{d})>0$ such that, if:
    $$\limsup_{r\to 0}d_{x,r}(\phi,\mathfrak{M})\leq \varepsilon(\delta,\tilde{d}) \qquad\text{for }\phi\text{-almost every }x\in \mathbb{G},$$
then $\phi$ is absolutely continuous with respect to $\mathcal{S}^{\mathcal{Q}-1}$ and $\mathbb{G}$ can be covered $\phi$-almost all with countably many intrinsic Lipschitz surfaces.
\end{teorema}

\begin{proof}
The first step in the proof is to note that since the metric $\tilde{d}$ and $d$ are bi-Lipschitz equivalent, there exists a constant $\mathfrak{c}>1$, that we can assume without loss of generality to be a natural number, such that:
 $$ (\mathfrak{c}\delta)^{-1}<\liminf_{r\to 0}\frac{\phi(B(x,r))}{r^{\mathcal{Q}-1}}\leq \limsup_{r\to 0}\frac{\phi(B(x,r))}{r^{\mathcal{Q}-1}}<\mathfrak{c}\delta\qquad \text{for }\phi\text{-almost every }x\in\mathbb{G}.$$
If we let $\varepsilon(\delta,\tilde{d}):=4^{-(\mathcal{Q}+1)}\oldep{e}(\mathfrak{c}\delta)$ then the verbatim repetition of the first part of the argument used to prove Theorem \ref{main}, where instead of Theorem \ref{ueue} we make use of Proposition \ref{prop:mild}, proves the claim.
\end{proof}

An immediate consequence of Theorem \ref{th:quasifprett} is the following:

\begin{corollario}\label{fpsrig}
Let $\vartheta_{\mathbb{G}}:=\max\{l_\mathbb{G}^{-1},L_{\mathbb{G}}\}$ where $l_\mathbb{G}$ and $L_\mathbb{G}$ are the constants yielded by Theorem \ref{th:4.16} and suppose $\Omega\subseteq \mathbb{G}$ is a finite perimeter set such that:
    $$\limsup_{r\to 0}d_{x,r}(\lvert\partial \Omega\rvert_\mathbb{G},\mathfrak{M})\leq \varepsilon(\vartheta_{\mathbb{G}},d) \qquad\text{for }\lvert\partial \Omega\rvert_{\mathbb{G}}\text{-almost every }x\in \mathbb{G},$$
where $\varepsilon(\vartheta_{G},d)$ is the constant yielded by Theorem \ref{th:quasifprett} and $d$ is the metric introduced in Definition \ref{smoothnorm}. Then $\mathbb{G}$ can be covered $\lvert\partial \Omega\rvert_\mathbb{G}$-almost all with countably many intrinsic Lipschitz surfaces. 
\end{corollario}

\begin{proof}
Theorem \ref{th:4.16} implies that
$l_\mathbb{G}<\Theta_*^{\mathcal{Q}-1}(\lvert\partial \Omega\rvert_\mathbb{G},x)\leq \Theta^{\mathcal{Q}-1,*}(\lvert\partial \Omega\rvert_\mathbb{G},x)<L_{\mathbb{G}}$ for $\phi$-almost every $x\in\mathbb{G}$. Theorem \ref{th:quasifprett} directly imply the statement.
\end{proof}

As mentioned at the beginning of this section, the main application of Theorem \ref{main} is an extension Preiss's rectifiability theorem to $1$-codimensional measures in $\HH^n$.

\begin{teorema}\label{preissHn}
Suppose $d$ is the Koranyi metric in $\HH^n$ and $\phi$ is a Radon measure on $\HH^n$ such that:
\begin{equation}
    0<\Theta^{2n+1}(\phi,x):=\lim_{r\to 0}\frac{\phi(B(x,r))}{r^{2n+1}}<\infty,\qquad \text{for }\phi\text{-almost every }x\in \mathbb{H}^n.
    \label{eq:limit}
\end{equation}
Then $\phi$ is absolutely continuous with respect to $\mathcal{S}^{\mathcal{Q}-1}$ and $\mathbb{H}^n$ can be covered $\phi$-almost all with $C^1_{\mathbb{H}^n}$-surfaces.
\end{teorema}

\begin{proof}
Thanks to Theorem 1.2 of \cite{merlo}, the almost sure existence of the limit in \eqref{eq:limit} implies that
$\Tan(\phi,x)\subseteq
\mathfrak{M}$,
for $\phi$-almost every $x\in\mathbb{G}$. Thanks to Theorem \ref{main}, this proves the claim.
\end{proof}

\subsection{Discussion of the results}

Theorem \ref{main} shows that $C^1_\mathbb{G}$-rectifiability in Carnot groups can be characterized by the same conditions on the densities and on the tangents as the Lipschitz rectifiability in Euclidean spaces. With this in mind we introduce the following two definitions:

\begin{definizione}[$\mathscr{P}$-rectifiable measures]\label{def:rect}
Suppose that $\phi$ is a Radon measure on some Carnot group $\mathbb{G}$ endowed with a left invariant and homogeneous metric $d$ and let $m$ be a positive integer. We say that $\phi$ is $\mathscr{P}_m$-rectifiable if:
\begin{itemize}
    \item[(i)]$0<\Theta_*^{m}(\phi,x)\leq \Theta^{m,*}(\phi,x)<\infty$ for $\phi$-almost every $x\in\mathbb{G}$,
    \item[(\hypertarget{unique}{ii})]$\Tan_m(\phi,x)\subseteq \{\lambda \mu_x:\lambda>0\}$, for $\phi$-almost every $x\in\mathbb{G}$ where $\mu_x$ is some Radon measure on $\mathbb{G}$.
    \end{itemize}
\end{definizione}

\begin{osservazione}
It was already remarked by P. Mattila in the last paragraph of \cite{Mattila2005MeasuresGroups} that Definition \ref{def:rect} may be considered the correct notion of rectifiability in $\HH^1$.
\end{osservazione}

\begin{osservazione}
 Instead of condition (\hyperlink{unique}{ii}) of Definition \ref{def:rect}, we can assume without loss of generality that $\mu_x=\mathcal{H}^{m}\llcorner V(x)$ for some $V(x)\in\G(m)$, where $\G(m)$ is the family of $m$-dimensional homogeneous subgroups of $\mathbb{G}$ introduced in Definition \ref{def:flatmeasures}. This is due to Theorem 3.2 of \cite{Mattila2005MeasuresGroups} and Theorem 3.6 of \cite{Cartan}: the former result tells us that $\mu_x$ must be the Haar measure of a closed, dilation-invariant subgroup of $\mathbb{G}$ and the latter that such subgroup is actually a Lie subgroup.
\end{osservazione}

\begin{definizione}[$\mathscr{P}^*$-rectifiable measures]\label{def:rect1}
Suppose that $\phi$ is a Radon measure on some Carnot group $\mathbb{G}$ endowed with a left invariant and homogeneous metric $d$ and let $m$ be a positive integer. We say that $\phi$ is $\mathscr{P}^*_m$-rectifiable if:
\begin{itemize}
    \item[(i)]$0<\Theta_*^{m}(\phi,x)\leq \Theta^{m,*}(\phi,x)<\infty$ for $\phi$-almost every $x\in\mathbb{G}$,
    \item[(ii)]$\Tan_m(\phi,x)\subseteq \mathfrak{M}(m)$, for $\phi$-almost every $x\in\mathbb{G}$.
    \end{itemize}
\end{definizione}

The difference between Definitions \ref{def:rect} and \ref{def:rect1} is that in the former the tangent to $\phi$ is the same plane at every scale, while in the latter the tangents are planes that may vary at different scales. Although there is no a priori reason for which these definition should be equivalent in general, we see that our main result Theorem \ref{main}, may be rewritten as:

\begin{teorema}
Suppose $\phi$ is a Radon measure on $\mathbb{G}$. The following are equivalent:
\begin{itemize}
    \item[(i)] $\phi$ is $\mathscr{P}_{\mathcal{Q}-1}$-rectifiable,
    \item[(ii)] $\phi$ is $\mathscr{P}^*_{\mathcal{Q}-1}$-rectifiable,
    \item[(iii)] $\phi$ is absolutely continuous with respect to $\mathcal{H}^{\mathcal{Q}-1}$ and $\mathbb{G}$ can be covered $\phi$-almost all with countably many $C^1_\mathbb{G}$-surfaces.
\end{itemize}
\end{teorema}

The notion of $\mathscr{P}$-rectifiable measures is also relevant since in different contests it appears to imply the right notion of rectifiability. This is summarized in the following theorem, that is an immediate consequence of the Eucidean Marstrand-Mattila rectifiability criterion and Theorem \ref{main}:

\begin{teorema}\label{teorema:defrect}
The following two statements hold:
\begin{itemize}
    \item[(i)] A Radon measure $\phi$ on $\R^n$ is $\mathscr{P}_m$-rectifiable if and only if it is Euclidean $m$-rectifiable;
    \item[(ii)] A Radon measure $\phi$ on $\mathbb{G}$ is $\mathscr{P}_{\mathcal{Q}-1}$-rectifiable if and only if it is     $C^1_\mathbb{G}$-rectifiable.
\end{itemize}
\end{teorema}

In \cite{MR2789472} P. Mattila, F. Serra Cassano and R. Serapioni proved in Theorems 3.14 and 3.15 that whenever a good notion of regular surface is available in the Heisenberg group, provided the tangents are selected carefully, see Definition 2.16 of \cite{MR2789472}, a $\mathscr{P}_m$-rectifiable measure is also rectifiable with respect of the family of regular surfaces of the right dimension.
However, because of the algebraic structure of the group $\HH^n$, there is not an a priori (known) good notion of regular surface that includes the vertical line $\mathcal{V}:=\{(0,0,t):t\in\R\}$. For this reason the uniform measure $\mathcal{S}^2\llcorner \mathcal{V}$ is considered to be \emph{non-rectifiable} from the standpoint of \cite{MR2789472}.
Up to this point Haar measures of  not complemented  homogeneous subgroups (like the vertical line $\mathcal{V}$ in $\HH^1$) were considered non-rectifiable and thus preventing a possible extension of Preiss's Theorem to low dimension even in $\HH^1$. This was already remarked in \cite{Chousionis2015MarstrandsGroup}. On the other hand, we have:

\begin{teorema}\label{teoremavertic}
Let $\phi$ be a Radon measure on $\HH^1$ such that for $\phi$-almost every $x\in \HH^1$, we have:
$$0<\Theta^2(\phi,x):=\lim_{r\to 0}\frac{\phi(B(x,r))}{r^2}<\infty,$$
where $B(x,r)$ are the metric balls with respect to the Koranyi metric.
Then $\phi$ is $\mathscr{P}_2$-rectifiable.
\end{teorema}

\begin{proof}
This follows from Proposition 2.2 of \cite{merlo} and Theorem 1.4 of \cite{ChousionisONGROUP}.
\end{proof}

As remarked in the previous paragraph, to our knowledge in literature there is not a good candidate of rectifiability in Carnot groups for which the density problem may have a positive answer. On the other hand, Theorems \ref{preissHn}, \ref{teorema:defrect} and \ref{teoremavertic} encourage us to state the  density problem in Carnot groups in the following way:

\begin{congettura}\label{conj1}
Suppose $\phi$ is a Radon measure on the Carnot group $\mathbb{G}$. There exists a left invariant distance $d$ on $\mathbb{G}$ such that the following are equivalent:
\begin{itemize}
    \item[(i)] there exists an $\alpha>0$ such that for $\phi$-almost every $x\in\mathbb{G}$ we have $0<\Theta^\alpha(\phi,x):=\lim_{r\to 0}\phi(B(x,r))/r^\alpha<\infty$,
    \item[(ii)] $\alpha\in\{0,\ldots,\mathcal{Q}\}$ and $\phi$ is $\mathscr{P}_{\alpha}$- rectifiable. 
\end{itemize}
\end{congettura}

Neither one of the implications of the formulation of the density problem is of easy solution. In \cite{MerloAntonelli} the author of the present work in collaboration with G. Antonelli prove the implication (ii)$\Rightarrow$(i) of the Density Problem when the tangents measures to $\phi$ are supported on complemented subgroups.

Furthermore, as already observed in \cite{merlo}, if $d$ is a left invariant distance coming from a polynomial norm on $\mathbb{G}$ with the same argument used in \cite{Kirchheim2002UniformilySpaces} and later on in \cite{Chousionis2015MarstrandsGroup}, it is possible to show that if (i) in the Density Problem holds, then $\alpha\in\N$. 
In $\R^n$ this implies thanks to Theorem 3.1 of \cite{polynorms}, that there is an open and dense set $\Omega$ in the space of norms  (with the distance induced by the Hausdorff distance of the unit balls) for which for any $\lVert\cdot\rVert\in\Omega$, Marstrand's theorem holds.


\appendix

\section{Construction of dyadic cubes}\label{AppendiceA}
Throughout this section we assume $\phi$ to be a fixed Radon measure on the Carnot group $\mathbb{G}$, supported on a compact set $K$, and such that:
$$0<\liminf_{r\to 0}\frac{\phi(B(x,r))}{r^m}\leq \limsup_{r\to 0}\frac{\phi(B(x,r))}{r^m}<\infty,\text{ for }\phi\text{-almost every }x\in\mathbb{G}.$$
In the following, we construct a family of \emph{dyadic cubes for the measure} $\phi$. There are many constructions in literature of such objects both in the Euclidean and in (rather general) metric spaces. Unfortunately, in the context of general metric spaces, dyadic cubes are only available for AD-regular measures, see for instance \cite{christ}. However, since the measure $\phi$ is not so regular, we need to provide a generalisation. More precisely, for any fixed $\xi,\tau\in\N$, we construct a family of partitions $\{\Delta_j^\phi(\xi,\tau)\}_{j\in\N}$ of $K$ such that:
\begin{itemize}
    \item[(i)]$\diam Q\leq 2^{-j}/\tau$ and $\phi(Q)\leq 2^{-jm}/\tau$ for any $Q\in\Delta_j^\phi(\xi,\tau)$,
    \item[(ii)] $2^{-j}\lesssim\diam Q$ and $2^{-jm}\lesssim_\xi\phi(Q)$ for those cubes $Q\in \Delta_j^\phi(\xi,\tau)$, for which $Q\cap E^\phi(\xi,\tau)\neq \emptyset$.
\end{itemize}
The collection $\Delta^\phi(\xi,\tau):=\{\Delta_j^\phi(\xi,\tau):j\in\N\}$ is said to be a family of dyadic cubes for $\phi$ relative to the parameters $\xi,\tau$. The strategy we employ to construct such partitions, is to adapt the construction given by G. David for AD-regular measures that can be found in Appendix I of \cite{DavidLN} to this less regular case.

In Subsection \ref{has:distance} we briefly recall the definition of the Hausdorff distance of compact sets and prove some technical facts used in Section \ref{sec:main} and Subsection \ref{sec:proofdyadic}. In Subsection \ref{sec:dyadiccubes}, we state Theorem \ref{evev} that is the main result of this appendix and prove some of its consequences. Finally, in Subsection \ref{sec:proofdyadic}, we prove Theorem \ref{evev}.

\subsection{Hausdorff distance of sets}
\label{has:distance}

In this subsection we recall some of the properties of the Hausdorff distance of sets and the Hausdorff convergence of compact sets.

\begin{definizione}[Hausdorff distance]\label{def:Haus}
For any couple of sets in $A,B\subseteq \mathbb{G}$, we define their \emph{Hausdorff distance} as:
$$d_H(A,B):=\max\Big\{\sup_{x\in A}\text{dist}(x,B),\sup_{y\in B}\text{dist}(A,y)\Big\}.$$
Furthermore, for any compact set $\kappa$ in $\mathbb{G}$, we define $\mathfrak{F}(\kappa):=\{A\subseteq \kappa: A\text{ is compact}\}$.
\end{definizione}

\begin{osservazione}
Recall that for any couple of sets $A,B\subseteq \mathbb{G}$ we have $d_H(A,B)=d_H(\overline{A},\overline{B})$.
\end{osservazione}

The following is a well known property of the  Hausdorff metric.

\begin{teorema}
For any compact set $\kappa$ of $\mathbb{G}$, we have that $(\mathfrak{F}(\kappa),d_H)$ is a compact metric space.
\end{teorema}

\begin{proof}
See for instance Theorem VI of \S 28 in \cite{Hausdorff}.
\end{proof}

The next proposition will be used in the proof of Lemma \ref{lemma:A2} and Proposition \ref{prop:A10}. It establishes the  stability of the Hausdorff convergence under finite unions.

\begin{proposizione}\label{prop:1}
Let $M\in \N$ and assume $\{D_i\}_{i=1,\ldots,M}$ and $\{D^\prime_i\}_{i=1,\ldots M}$ are finite families of subsets of $\mathbb{G}$. Then:
    \begin{equation}
         d_H\Big(\bigcup_{i=1}^MD_i,\bigcup_{i=1}^MD_i^\prime\Big)\leq \max_{i=1,\ldots,M}d_H(D_i,D_i^\prime).
         \label{eq:A9010}
    \end{equation}
Let $\kappa$ be a compact set in $\mathbb{G}$. Suppose that for any $j=1,\ldots, M$, the sequences $\{A^j_i\}_{i\in\N}\subseteq \mathfrak{F}(\kappa)$ converge in the Hausdorff metric to some $A^j\in\mathfrak{F}(\kappa)$. Then:
    \begin{equation}
    \lim_{i\to\infty} d_H\bigg(\bigcup_{j=1}^M A_i^j, \bigcup_{j=1}^M A^j\bigg)=0.
    \label{eq:A9011}
    \end{equation}
Finally, if the set $B$ is contained in $\bigcup_{j=1}^MA_i^j$ for any $i\in\N$, then $B\subseteq \bigcup_{j=1}^MA^j$.
\end{proposizione}

\begin{proof}
First of all, note that since identity \eqref{eq:A9011} is an immediate consequence of \eqref{eq:A9010}, we just need to prove the former. Thanks to the definition of the distance $d_H$, we have: 
\begin{equation}
\begin{split}
     d_H\Big(\bigcup_{i=1}^MD_i,\bigcup_{i=1}^MD_i^\prime\Big)=&\max\bigg\{\sup_{x\in \bigcup_{i=1}^M D_i} \dist\bigg(x, \bigcup_{i=1}^M D^\prime_i\bigg),\sup_{y\in \bigcup_{i=1}^M D_i^\prime} \dist\bigg(y, \bigcup_{i=1}^M D_i\bigg)\bigg\}\\
    \leq&\max\bigg\{\max_{i=1,\ldots,M}\sup_{x\in D_i}\dist\bigg(x, \bigcup_{i=1}^M D^\prime_i\bigg),\max_{i=1,\ldots,M}\sup_{x\in D^\prime_i}\dist\bigg(x, \bigcup_{i=1}^M D^\prime_i\bigg)\bigg\}\\
    \leq& \max\bigg\{\max_{i=1,\ldots,M}\sup_{x\in D_i}\dist(x,  D^\prime_i),\max_{i=1,\ldots,M}\sup_{x\in D^\prime_i}\dist(x, D^\prime_i)\bigg\}\\
    =&\max_{i=1,\ldots,M}\max\bigg\{\sup_{x\in D_i}\dist(x,  D^\prime_i),\sup_{x\in D^\prime_i}\dist(x, D^\prime_i)\bigg\}=\max_{i=1,\ldots,M}\dist(D_i,D_i^\prime).
\end{split}
    \nonumber
\end{equation}
This concludes the proof of \eqref{eq:A9010} and thus of \eqref{eq:A9011}. The proof of the last part of the proposition follows by the pidgeonhole principle. Indeed, for any $b\in B$ there exists a $j(b)\in\{1,\ldots,M\}$ such that $b\in A_{j(b)}^{i_k}$ for any $k\in\N$. In particular, we conclude that $\dist(b,A_{j(b)})=0$ and thus, since $A_{j(b)}$ is closed, we infer that $b\in A_{j(b)}$.
\end{proof}

\subsection{Dyadic cubes}
\label{sec:dyadiccubes}
In this subsection we state the main theorem of this appendix, Theorem \ref{evev} and prove, assuming its validity, a couple of consequences that will be used in Section \ref{sec:main}. Throughout the rest of Appendix \ref{AppendiceA}, we will always assume that $\xi$ and $\tau$ are two fixed natural numbers such that $\phi(E^\phi(\xi,\tau))>0$, where the set $E^\phi(\xi,\tau)$ was defined in Proposition \ref{prop:cpt}.

\begin{definizione}
For any subset $A$ of $\mathbb{G}$ and any $\delta>0$, we let:
$$\partial(A,\delta):=\{u\in A:\text{dist}(u,K\setminus A)\leq\delta\}\cup\{u\in K\setminus A:\text{dist}(u,A)\leq\delta\},$$
where we recall that $K$ is the compact set supporting the measure $\phi$.
\end{definizione}

Throughout the rest of this subsection, we simplify the expression of the constant introduced in Notation \ref{notation1} to:
$$N:=N(\xi),\qquad\zeta:=\zeta(\xi),\qquad\oldC{C:F}:=\oldC{C:F}(\xi),\qquad \oldC{C:VIB}:=\oldC{C:VIB}(\xi),\qquad\oldC{child}:=\oldC{child}(\xi).$$

\begin{teorema}\label{evev}
For any $j\in\N$ there are disjoint partitions $\Delta_j^\phi(\xi,\tau)$ of $K$ having the following properties:
\begin{itemize}
    \item[(i)] if $j\leq j^\prime$, $Q\in\Delta_j^\phi(\xi,\tau)$ and $Q^\prime\in\Delta_{j^\prime}^\phi(\xi,\tau)$, then either $Q$ contains $Q^\prime$ or $Q\cap Q^\prime=\emptyset$,
    \item[(ii)]if $Q\in\Delta_j^\phi(\xi,\tau)$ we have  $\diam(Q)\leq 2^{-Nj+5}/\tau$,
    \item[(iii)] if $Q\in \Delta_j^\phi(\xi,\tau)$ and $Q\cap E^\phi(\xi,\tau)\neq \emptyset$, then $\oldC{C:F}^{-1}\big(2^{-Nj}/\tau\big)^{m}\leq \phi(Q)\leq \oldC{C:F}\big(2^{-Nj}/\tau\big)^{m}$,
    \item[(iv)] if $Q\in\Delta_j^\phi(\xi,\tau)$, we have $ \phi\big(\partial(Q,\zeta^2 2^{-Nj}/\tau )\big)\leq \oldC{C:F} \zeta\big(2^{-Nj}/\tau)^{m}$,
    \item[(v)] if $Q\in\Delta_j^\phi(\xi,\tau)$ and $Q\cap E^\phi(\xi,\tau)\neq \emptyset$, there exists a $\mathfrak{c}(Q)\in K$ such that $B(\mathfrak{c}(Q),\zeta^2 2^{-Nj-1}/\tau)\subseteq Q_j(x)$.
\end{itemize}
We will denote with the symbol $\Delta^\phi(\xi,\tau)$ the family of all dyadic cubes, i.e. $\Delta^\phi(\xi,\tau)=\{Q\in \Delta^\phi(\xi,\tau):j\in\N\}$.
\end{teorema}

\begin{osservazione}\label{rk:rephrase}
Part (iii) of Theorem \ref{evev} can be rephrased in the following useful way. Recalling that $\oldC{C:VIB}(\xi)=\oldC{C:F}(32\zeta^{-2})^m$ and  putting together Theorem \ref{evev} (ii), (iii) and (v) we infer that:
\begin{itemize}
    \item[(iii)'] if $Q\cap E^\phi(\xi,\tau)\neq \emptyset$ then $\oldC{C:VIB}^{-1}\diam Q^{m}\leq\phi(Q)\leq\oldC{C:VIB} \diam Q^m$.
\end{itemize}
\end{osservazione}

The families of cubes yielded by Theorem \ref{evev} may have the annoying property that fixed a cube $Q\in \Delta_j^\phi(\xi,\tau)$, the only sub-cube of $Q$ in the layer $\Delta_{j+1}^\phi(\xi,\tau)$ contained in $Q$, is just $Q$ itself. The following proposition shows that this is not much of a problem for the cubes intersecting $E^\phi(\xi,\tau)$.  

\begin{proposizione}\label{prop:stabcubi}
Suppose that $Q^*\in\Delta_j^\phi(\xi,\tau)$ is parent of a cube $Q\in\Delta_{j+k}^\phi(\xi,\tau)$ such that $Q\cap E^\phi(\xi,\tau)\neq \emptyset$, i.e. $Q^*$ is the smallest cube in $\Delta^\phi(\xi,\tau)$ strictly containing $Q$. Then $k < \big\lfloor2\log \oldC{C:F}/Nm\big\rfloor+1$ and:
$$
\frac{\diam Q^*}{\diam Q}\leq\oldC{child}.$$
\end{proposizione}

\begin{proof}
Suppose $\tilde{Q}$ is the ancestor of the cube $Q$ contained in the layer $\Delta_{j^\prime}^\phi(\xi,\tau)$ for some $j^\prime$ for which $j^\prime-j\geq \big\lfloor2\log \oldC{C:F}/Nm\big\rfloor+1$. Then $\tilde{Q}\cap E^\phi(\xi,\tau)\neq \emptyset$ and thanks to Theorem \ref{evev}(i) and (iii), we infer:
\begin{equation}
\phi(\tilde{Q}\setminus Q)= \phi(\tilde{Q})-\phi(Q)\geq \oldC{C:F}^{-1}\bigg(\frac{2^{-jN}}{\tau}\bigg)^m-\oldC{C:F}\bigg(\frac{2^{-j^\prime N}}{\tau}\bigg)^m=\oldC{C:F}^{-2}\bigg(\frac{2^{-jN}}{\tau}\bigg)^m(1-\oldC{C:F}^22^{-(j^\prime-j)Nm})>0,    
\label{eq:A1000}
\end{equation}
where the last inequality above comes from the choice of $j^\prime-j$. It is immediate to see that inequality \eqref{eq:A1000} implies that $Q$ is strictly contained in $\tilde{Q}$. Therefore, the parent cube of $Q$ must be contained in some $\Delta_{j+k}^\phi(\xi,\tau)$ with $0\leq k< \lfloor 2\log \oldC{C:F}/mN\rfloor+1$. Hence, thanks to Theorem \ref{evev}(v) we infer that:
$$\diam Q^*\leq 2^{-Nj+5}/\tau=2^{Nk+6}\zeta^{-2}\cdot\zeta^22^{-N(j+k)-1}/\tau\leq 2^{Nk+6}\zeta^{-2}\diam Q\leq 2^{\frac{2\log \oldC{C:F}}{m}+N}\zeta^{-2}\diam Q=\oldC{child}\diam Q.$$
\end{proof}

The following result tells us that item (v) of Theorem \ref{evev} in some cases can be strengthened to assuming that the centre of the cube $\mathfrak{c}(Q)$ is contained in $E^\phi(\xi,\tau)$.

\begin{proposizione}\label{prop:centri}
Assume that $\mu\in\N$ is such that $\mu\geq 4\xi$. Then, for any cube $Q\in \Delta^\phi(\mathscr{E}^\phi_{\xi,\tau}(\mu,\nu);\xi,\tau,\nu)$ we can find a $\mathfrak{c}(Q)\in E^\phi(\xi,\tau)\cap Q$ such that: $$B(\mathfrak{c}(Q),\zeta^2\diam Q/64)\cap K\subseteq Q.$$
\end{proposizione}

\begin{osservazione}
Recall that the set $\mathscr{E}^\phi_{\xi,\tau}(\mu,\nu)$ was introduced in Proposition \ref{prop:Ecorsivo} and $\Delta^\phi(\kappa;\xi,\tau,\nu)$ in Notation \ref{notation1}.
\end{osservazione}

\begin{proof}
In order to prove the proposition it suffices to show that:
\begin{equation}
    E^\phi(\xi,\tau)\cap Q\setminus \partial(Q,\zeta^2 \diam Q/32)\neq \emptyset.
    \label{eq:72}
\end{equation}
In order to fix ideas, we let $Q\in\Delta_j^\phi(\xi,\tau)$ for some $j\geq \nu$ and note that since  $Q\cap E^\phi(\xi,\tau)\neq \emptyset$, thanks to Theorem \ref{evev}(ii), (iii) and (iv), we have:
\begin{equation}
    \begin{split}
        \phi( E^\phi(\xi,\tau)\cap Q\setminus&
 \partial(Q,\zeta^2\diam Q/32))\\
        \geq& \phi(E^\phi(\xi,\tau)\cap Q)-\phi(\partial(Q,\zeta^2\diam Q/32))
        \geq\phi(E^\phi(\xi,\tau)\cap Q)-\phi(\partial(Q,\zeta^22^{-jN}/\tau))\\
        \geq& \phi(E^\phi(\xi,\tau)\cap Q)-\oldC{C:F}\zeta(2^{-jN}/\tau)^{\mathcal{Q}-1}
        =\phi(Q)-\phi(Q\setminus E^\phi(\xi,\tau))-\oldC{C:F}\zeta(2^{-jN}/\tau)^{\mathcal{Q}-1}\\
    \geq&\phi(Q)-\phi(Q\setminus E^\phi(\xi,\tau))-\oldC{C:F}^2\zeta\phi(Q).
        \label{eq:74}
    \end{split}
\end{equation}
Since $Q\in\Delta^\phi(\mathscr{E}^\phi_{\xi,\tau}(\mu,\nu);\xi,\tau,\nu)$ we have $\diam{Q}\leq 2^{-N\nu+5}/\tau$ and there exists a $w\in \mathscr{E}^\phi_{\xi,\tau}(\mu,\nu)\cap Q$. Therefore, the definition of $\mathscr{E}^\phi_{\xi,\tau}(\mu,\nu)$ and Theorem \ref{evev}(vi), imply:
\begin{equation}
\begin{split}
    \phi(Q\setminus E^\phi(\xi,\tau))\leq \phi(B(w,2^{-jN+5}/\tau)\setminus E^\phi(\xi,\tau))\leq& \mu^{-1}\phi(B(w,2^{-jN+5}/\tau))\\
    \leq& \mu^{-1}\xi (2^{-jN+5}/\tau)^{\mathcal{Q}-1}\leq \mu^{-1}\xi\oldC{C:F}\phi(Q).
\end{split}
    \label{eq:75}
\end{equation}
Putting together \eqref{eq:74} and \eqref{eq:75}, we conclude that:
$$\phi( E^\phi(\xi,\tau)\cap Q\setminus \partial(Q,\zeta^2\diam Q))\geq (1-\mu^{-1}\xi-\oldC{C:F}^2\zeta)\phi(Q)\geq \phi(Q)/4.$$
where the last inequality follows from the fact that $\oldC{C:F}^2\zeta=2^{48\mathcal{Q}}\xi^2\cdot 2^{-50\mathcal{Q}}\xi^{-2}\leq 1/2$ and $\mu^{-1}\xi\leq 1/4$.
This proves \eqref{eq:72} and in turn the proposition.
\end{proof}

\subsection{Construction of the dyadic cubes}
\label{sec:proofdyadic}
For the rest of the section, we fix $\xi,\tau\in\N$ in such a way that $\phi(E^\phi(\xi,\tau))>0$ and we let $\eta:=\xi^{-2}2^{-40Q}$.

\begin{definizione}
For any $j\in \N$, we let:
\begin{itemize}
    \item[(i)]$\gimel_0(j)$ be a maximal set of points in $E^\phi(\xi,\tau)$ such that $d(x,x^\prime)\geq 2^{-j}/\tau$ for any couple of distinct $x,x^\prime\in\gimel_0(j)$,
\item[(ii)]$\Xi_0(j)$ be a maximal set of points in $K$, containing $\gimel_0(j)$ and such that $d(x,x^\prime)\geq 2^{-j}/\tau$ for any couple of distinct  $x,x^\prime\in\Xi_0(j)$,
\end{itemize}
\end{definizione}

\begin{lemma}\label{lemma1}
Let $x\in K$ be a point such that $\phi(B(x,r))\leq \xi r^m$ for any $0<r<1/\tau$.
Then for any $j\in\N$, there is a ball $B_j(x)$ centred at $x$ of radius $r\in(2^{-j}/\tau,(1+\eta)2^{-j}/\tau)$ such that:
$$
         \phi(\partial(B_j(x),\eta^2 2^{-j}/\tau))\leq2^m\xi \eta(2^{-j}/\tau)^{m}=:\newC\label{C:5} \eta(2^{-j}/\tau)^{m}.
$$
\end{lemma}

\begin{proof}
For any $p\in\{1,\ldots, \lfloor 1/\eta\rfloor-1\}$, let:
$$C(x,p,j):=B\Big(x,\Big(1+\frac{\eta p}{\lfloor 1/\eta\rfloor}+\eta^2\Big)2^{-j}/\tau\Big)\setminus B\Big(x,\Big(1+\frac{\eta p}{\lfloor 1/\eta\rfloor}-\eta^2\Big)2^{-j}/\tau\Big).$$
The coronas $C(x,p,j)$ are pairwise disjoint since $\eta$ was chosen small enough, and their union is contained in the ball $B(x,(1+\eta)2^{-j}/\tau)$.
This implies, thanks to the choice of $x$, that there exists a $\overline{p}\in\{1,\ldots,\lfloor1/\eta\rfloor-1\}$ such that:
$$\phi(C(x,\overline{p},j))\leq \xi \frac{(1+\eta)^{m}}{\lfloor1/\eta\rfloor}(2^{-j}/\tau)^{m},$$
and this concludes the proof of the lemma.
\end{proof}

\begin{lemma}\label{lemma2}
For any $x\in \Xi_0(j)$, we have
$\text{Card}\big(\Xi_0(j)\cap B(x,2^{-j+3}/\tau)\big)\leq 32^\mathcal{Q} $.
\end{lemma}

\begin{proof}
The balls $\{B(y,2^{-j-1}/\tau):y\in \Xi_0(j)\cap B(x,2^{-j+3}/\tau)\}$ are disjoint and contained in $B(x,2^{-j+4}/\tau)$. This implies that:
$$(2^{-j-1}/\tau)^\mathcal{Q}\text{Card}(\Xi_0(j)\cap B(x,2^{-j+3}/\tau))\leq \mathcal{S}^\mathcal{Q}(B(x,2^{-j+4}/\tau))=(2^{-j+4}/\tau)^\mathcal{Q}.$$
\end{proof}

We let $\prec$ be a total order on $\Xi_0(j)$ such that $x\prec y$ whenever $x\in \gimel_0(j)$ and $y\in\Xi_0(j)\setminus \gimel_0(j)$. For every $x\in\Xi_0(j)$ we define:
$$B^\prime_j(x):=B_j(x)\cap \Bigg[\bigcup_{y\prec x} B_j(y) \Bigg]^c.$$
Therefore, thanks to the definitions of the order $\prec$, for any $x\in\gimel_0(j)$ and any $w\in \Xi_0(j)\setminus \gimel(j)$, we have:
\begin{equation}
    B_j(x)\cap B_j^\prime (w)=\emptyset.
    \label{eq:A11}
\end{equation}
The sets $B_j^\prime$ are by construction a disjoint cover of $K$, however their measure could be very small. To avoid this problem, we will glue together some of the $B^\prime_j$s.
Let $\gimel(j)$ be the subset of those $x\in \gimel_0(j)$ such that:
$$\phi(B^\prime_j(x))\geq 32^{-\mathcal{Q}}\xi^{-1}(2^{-j}/\tau)^{m}.$$

\begin{lemma}\label{lemma:rev1}
For any $x\in \gimel_0(j)$ we have $\gimel(j)\cap B(x,2^{-j+2}/\tau)\neq \emptyset$.
\end{lemma}

\begin{proof}
Since the sets $B_j^\prime$ are a disjoint cover of $K$, for any $x\in \gimel(j)$ we have:
\begin{equation}
    K\cap B_j(x)\subseteq  \bigcup_{\substack{y\in\Xi(j)\\B_j(x)\cap B_j^\prime(y)\neq \emptyset}} K\cap B_j^\prime(y) =\bigcup_{\substack{y\in\gimel_0(j)\\B_j(x)\cap B_j^\prime(y)\neq \emptyset}} K\cap B_j^\prime(y),
    \label{eq:A12}
\end{equation}
where the last identity above comes from \eqref{eq:A11}. Furthermore, thanks to the inclusion \eqref{eq:A12}, we have:
\begin{equation}
    \sum_{\substack{y\in \gimel_0(j)\\ B_j(x)\cap B_j^\prime(y)\neq \emptyset}} \phi(B^\prime_j(y))\geq\phi(B_j(x))\geq \frac{1}{\xi}\bigg( \frac{2^{-j}}{\tau}\bigg)^{m},
    \label{eq:A1300}
\end{equation}
where the last inequality comes from the fact that $x\in E^\phi(\xi,\tau)$. Finally, Lemma \ref{lemma2} together with \eqref{eq:A1300} and the fact that if $B_j(x)\cap B_j^\prime(y)\neq \emptyset$ then $y\in B(x,2^{-j+3}/\tau)$, imply that:
$$\xi^{-1} (2^{-j}/\tau)\leq \phi(B_j(x))\leq \sum_{\substack{y\in \gimel_0(j)\\ B_j(x)\cap B_j^\prime(y)\neq \emptyset}} \phi(B^\prime_j(y))\leq 32^\mathcal{Q}\max_{\substack{y\in \gimel_0(j)\\ B_j(x)\cap B_j^\prime(y)\neq \emptyset}} \phi(B^\prime_j(y)).$$
This concludes the proof.
\end{proof}

Let $h$ be a map assigning each $y\in\gimel_0(j)\setminus \gimel(j)$ to an element of $\gimel(j)\cap B(y,2^{-j+2})$, that exists thanks to Lemma \ref{lemma:rev1} and define:
\begin{equation}
    D_j(y):= \begin{cases}
   K\cap \bigg[B^\prime_j(y)\cup\bigcup_{x\in h^{-1}(y)} B^\prime_j(x)\bigg]&\text{if }y\in \gimel(j),\\
   K\cap B^\prime_j(y)&\text{if }y\in \Xi_0(j)\setminus\gimel_0(j).
    \end{cases}
    \nonumber
\end{equation}
Finally we let $\Xi(j):=\gimel(j)\cup \Xi_0(j)\setminus \gimel_0(j)$.

\begin{proposizione}\label{propA1}
For any $j\in \N$ the sets $\{D_j(y):y\in \Xi(j)\}$ are a disjoint cover of $K$ and:
\begin{itemize}
    \item[(i)]$\diam(D_j(y))\leq 2^{-j+4}/\tau$ and $D_j(y)\subseteq B(y,2^{-j+3}/\tau)$,
    \item[(ii)] for any $j\in\N$ we have $E^\phi(\xi,\tau)\subseteq \bigcup_{y\in\gimel(j)} D_j(y)$,
    \item[(iii)] $\phi(D_j(y))\geq 32^{-\mathcal{Q}}\xi^{-1}(2^{-j}/\tau)^{m}$ provided $y\in \gimel(j)$ and $j\geq 3$,
    \item[(iv)]$\phi(\partial(D_j(y),2^{-j}\eta^2/\tau))\leq \newC\label{C:6} \eta(2^{-j}/\tau)^{m}$ where $\oldC{C:6}:=2\cdot64^{\mathcal{Q}}\oldC{C:5}$,
    \item[(v)] for any $y\in \gimel(j)$ there exists a $c\in D_j(y)$ such that $B(c,\eta^2 2^{-j-1}/\tau)\subseteq D_j(y)$.
\end{itemize}
\end{proposizione}

\begin{proof}
The fact that $\{D_j(y):y\in \Xi(j)\}$ is a partition of $K$, follows from the fact that the $B_j^\prime$s are. In order to prove (ii), it is sufficient to note that \eqref{eq:A11} implies that the union of the sets $\{B_j^\prime(y):y\in \gimel_0(j)\}$ covers $E^\phi(\xi,\tau)$. Thanks to the definition of the $D_j$s, this also yields that $\{D_j(y):y\in\gimel(j)\}$ cover $E^\phi(\xi,\tau)$ as well. 
The claim (iii) follows directly from the fact that $B^\prime_j(y)\subseteq D_j(y)$ for any $y\in \gimel(j)$.

In order to prove (i), we distinguish two cases.  If $y\in \Xi(j)\setminus \gimel(j)$, we have $D_j(y)=B^\prime_j(y)\cap K\subseteq B_j(y)\cap K$ and thus the claim follows directly from the definition of $B_j(y)$.
On the other hand, if $y\in \gimel(j)$,
for any $w_1,w_2\in D_j(y)$ there are $a_1,a_2\in \{y\}\cup h^{-1}(y)$ such that $w_i\in B_j^\prime(a_i)$ for $i=1,2$. 
Since by construction  we have $d (h(w),w)\leq 2^{-j+2}/\tau$ for any $w\in \gimel_0(j)$ and $B^\prime_j(w)\subseteq B_j(w)$ for any $w\in \Xi_0(j)$, we deduce that:
$$d(w_1,w_2)\leq d(w_1,a_1)+d(a_1,y)+d(y,a_2)+d(a_2,w_2)\leq 2\cdot2^{-j+1}/\tau+2\cdot2^{-j+2}/\tau\leq 2^{-j+4}/\tau,$$
proving that $\diam(D_j(y))\leq 2^{-j+4}/\tau$. The proof of the second part of (i) follows similarly.

In order to prove (iv), we first estimate $\phi(\partial(B^\prime_j(y),2^{-j}\eta^2/\tau))$. Since for any $y\in \Xi_0(j)$, we have: $$\partial(B_j^\prime(y),2^{-j}\eta^2/\tau)\subseteq \bigcup_{\substack{w\preceq y\\B_j(w)\cap B_j(y)\neq \emptyset}}\partial(B_j(y),2^{-j}\eta^2/\tau),$$
this implies that:
\begin{equation}
    \phi(\partial(B^\prime_j(y),2^{-j}\eta^2/\tau))\leq \sum_{\substack{w\preceq y\\B_j(w)\cap B_j(y)\neq \emptyset}}\phi(\partial(B_j(w),2^{-j}\eta^2/\tau))\leq 32^{\mathcal{Q}}\oldC{C:5}\eta(2^{-j}/\tau)^m,
    \label{eq:A13}
\end{equation}
where the last inequality comes from Lemmas \ref{lemma1} and \ref{lemma2}.
We immediately see that the bound in \eqref{eq:A13} proves (iv) if $y\in \Xi(j)\setminus\gimel(j)$. On the other hand, if $y\in\gimel(j)$ thank to the inclusion:
$$\partial(D_j(y),2^{-j}\eta^2/\tau)\subseteq \partial(B^\prime_j(y),2^{-j}\eta^2/\tau)\cup \bigcup_{w\in h^{-1}(y)} \partial(B^\prime_j(y),2^{-j}\eta^2/\tau),$$
we infer that:
\begin{equation}
\begin{split}
    \phi(\partial(D_j(y),2^{-j}\eta^2/\tau))\leq& \phi(\partial(B^\prime_j(y),2^{-j}\eta^2/\tau))+\sum_{z\in h^{-1}(y)}\phi(\partial(B^\prime_j(z),2^{-j}\eta/\tau))\\
    \leq & \sum_{\substack{w\preceq y\\B_j(w)\cap B_j(y)\neq \emptyset}}\phi(\partial(B_j(w),2^{-j}\eta^2/\tau))+\sum_{z\in h^{-1}(y)}\sum_{\substack{w\preceq z\\B_j(w)\cap B_j(z)\neq \emptyset}}\phi(\partial(B_j(w),2^{-j}\eta^2/\tau))\\
    \leq& 32^{\mathcal{Q}}\oldC{C:5}\eta(2^{-j}/\tau)^m+32^{\mathcal{Q}}\text{Card}(h^{-1}(y))\oldC{C:5}\eta(2^{-j}/\tau)^m\leq 2\cdot 64^\mathcal{Q} \oldC{C:5}\eta (2^{-j}/\tau)^m,
    \nonumber
\end{split}
\end{equation}
where the first inequality of the last line comes from \eqref{eq:A13} and the last one from the estimate on the cardinality of $h^{-1}(y)$, which can be bound with the one of $\gimel_0(j)\cap B(y,2^{-j+2}/\tau)$, see Lemma \ref{lemma2}. The bound in (iv) follows from the choice $\oldC{C:6}$.

In order to verify (v) we first observe that for any $y\in\gimel(j)$ we have $2\phi(\partial(D_j(y),2^{-j}\eta^2/\tau))\leq \phi(D_j(y))$. This implies that:
$$\phi(\{u\in D_j(y):\text{dist}(u,K\setminus D_j(y))>2^{-j}\eta^2/\tau\})>\phi(\partial(D_j(y),2^{-j}\eta^2/\tau))>0,$$
and thus we can find a $c\in K$ such that $K\cap B(c,\eta^2 2^{-j-1}/\tau)\subseteq D_j(y)$, concluding the proof of the proposition.
\end{proof}

We need to modify the sets $D_j$ in order to make each generation of cubes to interact in the intended way with the others. 

\begin{definizione}
Suppose $j\in \{Nl:l\in\N\text{ and }l\geq 2\}$ and let $y\in \Xi(j)$. Since the sets $D_j$ are a disjoint cover of $K$ for any $j$, we can define $\varphi(y)$ as the point $x\in \Xi(j-N)$ such that $y\in D_{j-N}(x)$. Furthermore, for any $d\in\N$ we let:
$$E_d(x):=\{y\in \Xi(j+Nd):\varphi^d(y)=x\},$$
where $\varphi^d=\underbrace{\varphi\circ\varphi\circ\ldots\circ\varphi}_d$.
\end{definizione}

\begin{osservazione}\label{rk:A3}
Since the sets $D_{j}$ are a partition of $K$, this implies that for any $w\in \Xi(j+Nd)$ there exists an $x\in \Xi(j)$ such that $w\in D_j(x)$, or more compactly:
\begin{equation}
    \bigcup_{x\in \Xi(j)}E_d(x)=\Xi(j+Nd),
    \label{eq:A3}
\end{equation}
We define $A(j)$ as the set of those $x\in \Xi(j)$ for which $E_d(x)$ is not definetely empty, i.e.: 
$$A(j):=\{x\in \Xi(j):\text{Card}(\{d\in \N:E_d(x)\neq \emptyset\})=\infty\}.$$
Since $\Xi(j)$ is finite, it is immediate to see (thanks to the pidgeonhole principle) that $A(j)$ must be non-empty.
\end{osservazione}

The following proposition is an elementary consequence of the definition of the sets $E_d$.

\begin{proposizione}\label{prop:A2}
For any $j\in N\N$, $x\in \Xi(j)$ and $d\in\N$, if $E_{d}(x)\neq \emptyset$, then:
\begin{equation}
    \label{eq:A9000}
E_{k}(x)\neq \emptyset\qquad \text{ for any }k\leq d.
\end{equation}
Furthermore, the following identity holds for any $l\in\N$:
\begin{equation}
E_{d+l}(x)=\bigcup_{w\in E_d(x)} E_{l}(w)=\bigcup_{\substack{w\in E_d(x)\\ E_l(w)\neq \emptyset}} E_{l}(w).
\label{eq:A8}
\end{equation}
\end{proposizione}

\begin{proof}
If $p$ belongs to the right hand side of \eqref{eq:A8}, there exists a $w\in E_d(x)$ such that $p\in E_l(w)$. This implies that $p\in E_{d+l}(x)$, since $\varphi^{d+l}(p)=\varphi^d(\varphi^l(p))=\varphi^d(w)=x$.
Viceversa, if $p\in E_{d+l}(x)$ then $\varphi^d(p)\in E_l(x)$ and $p\in E(\varphi^d(p))\neq\emptyset$, proving that $p\in \bigcup_{w\in E_d(x),E_l(w)\neq \emptyset} E_{l}(w)$. Eventually, identity \eqref{eq:A8} directly implies \eqref{eq:A9000}. Indeed, if there was some $0<l< k$ such that $E_l(w)=\emptyset$ for any $w\in E_l(w)$, we would have by \eqref{eq:A8} that:
$$E_k(x)=\bigcup_{\substack{w\in E_{l}(x)\\ E_{k-l}(w)\neq \emptyset}} E_{k-l}(w)=\emptyset,$$
since the union over an empty family is the empty set and this is a contradiction with the fact that by assumption $E_d(x)\neq\emptyset$.
\end{proof}

\begin{osservazione}
Thanks to
Proposition \ref{prop:A2}, we also have that if $x\in A(j)$, then $E_{d}(x)\neq \emptyset$ for any $d\in\N$.
This implies that, again thanks to the finiteness of $\Xi(j)$, that for any $j\in\N$ there exists an $M(j)\in\N$ such that $E_{d}(x)=\emptyset$ for any $x\in \Xi(j)\setminus A(j)$ whenever $d\geq M(j)$.
\end{osservazione}

\begin{proposizione}\label{inclgimel}
For any $j\in N\N$ we have $\gimel(j)\subseteq  A(j)$.
\end{proposizione}

\begin{proof}
Thanks to Proposition \ref{propA1}(iv), for any $x\in \gimel(j)$ there exists a $c\in K$ such that:
\begin{equation}
K\cap B(c,\eta^22^{-j-2}/\tau)\subseteq \{u\in D_j(x):\text{dist}(u,K\setminus D_j(x))>2^{-j}\eta^2/\tau\}.
\label{eq:9001}
\end{equation}
Since the sets $D_{j+Nd}$ cover $K$ for any $d\in\N$, there exists a $z\in \gimel(j+Nd)$ such that
$c\in D_{j+Nd}(z)$. If $d$ is big enough, then $D_{j+Nd}(z)$ must be contained in $D_j(x)$ thanks to \eqref{eq:9001} and Proposition \ref{propA1}(i). Thus $E_d(x)\neq \emptyset$ for any $d\geq \overline{d}$.
\end{proof}

\begin{definizione}
Let $j\in N\N$ and $d\in \N$. For any $x\in \Xi(j)$ for which $E_d(x)\neq \emptyset$, we define:
$$D_{j,d}(x):=\bigcup_{y\in E_{d}(x)} D_{j+Nd}(y).$$
For the sake of notation we also let $D_{j,0}(x):=D_j(x)$.
\end{definizione}

\begin{lemma}\label{lemma3}
For any $j\in N\N$ and any $x\in \Xi(j)$ for which there is a $d\geq 1$ such that $E_d(x)\neq \emptyset$, we have:
\begin{equation}
    d_H(D_{j,k-1}(x),D_{j,k}(x))\leq 2^{-j-N(k-1)+4}/\tau,
    \label{eq:A7}\qquad\text{for any }1\leq k\leq d.
\end{equation}
\end{lemma}

\begin{proof}
We prove the proposition proceeding by induction on $k$. 
As a first step, we verify \eqref{eq:A7} in the base case $k=1$. Since $E_d(x)\neq \emptyset$, thanks to Proposition \ref{prop:A2} we have that $E_1(x)\neq \emptyset$. Therefore, if $z\in D_{j,1}(x)$ there exists an $y\in E_1(x)$ such that $z\in D_{j+N}(y)$. Thanks to Proposition \ref{propA1}(i) we know that $d(z,y)\leq 2^{-j-N+4}/\tau$ and in addition to this, thanks to the choice of $y$ we also have that $y\in D_j(x)$ and thus:
\begin{equation}
\text{dist}(z,D_{j}(x))\leq d(z,y)\leq 2^{-j-N+4}/\tau, \text{ for any }z\in D_{j,1}(x).
\label{eq:A9002}
\end{equation}
On the other hand for any $w\in D_j(x)$, we have:
\begin{equation}
    \begin{split}
        \text{dist}(w,D_{j,1}(x))\leq &\min_{y\in E_1(x)} d(w,y)\leq d(w,x)+\min_{y\in E_1(x)} d(x,y)
        \leq 2^{-j+3}/\tau+2^{-j+3}/\tau=2^{-j+4}/\tau,
        \label{eq:A9003}
    \end{split}
\end{equation}
where the last inequality comes from Proposition \ref{propA1}(i).
Summing up, \eqref{eq:A9002} and \eqref{eq:A9003} imply that for any $j\in\N$ and any $x\in \Xi(j)$ for which $E_1(x)\neq \emptyset$ we have:
\begin{equation}
d_H(D_{j}(x),D_{j,1}(x))\leq 2^{-j+4}/\tau.
\label{eq:A9012}
\end{equation}
Suppose now that the bound \eqref{eq:A7} holds for $k-1$ for some $1\leq k\leq d$. The definition of $D_{j,k}(x)$ together with identity \eqref{eq:A8} implies:
\begin{equation}
D_{j,k}(x)=\bigcup_{y\in E_{k}(x)}D_{j+Nk}(y)=\bigcup_{\substack{w\in E_{k-1}(x)\\E_1(w)\neq \emptyset}}\bigcup_{y\in E_1(w)} D_{j+N(k-1)+N}(y)=\bigcup_{\substack{w\in E_{k-1}(x)\\E_1(w)\neq \emptyset}}D_{j+N(k-1),1}(w).
\label{eq:nummiri2}
\end{equation}
Identity \eqref{eq:nummiri2} thanks to inequalities \eqref{eq:A9010} and  \eqref{eq:A9012}, allows us to infer:
\begin{equation}
\begin{split}
    d_H(D_{j,k-1}(x),D_{j,k}(x))=& d_H\Bigg(\bigcup_{\substack{w\in E_{k-1}(x)\\E_1(w)\neq \emptyset}}D_{j+N(k-1)}(w),\bigcup_{\substack{w\in E_{k-1}(x)\\E_1(w)\neq \emptyset}}D_{j+N(k-1),1}(w)\Bigg)\\
    \leq&\max_{\substack{w\in E_{k-1}(x)\\E_1(w)\neq \emptyset}}d_H\big(D_{j+N(k-1)}(w),D_{j+N(k-1),1}(w)\big)\leq 2^{-j-N(k-1)+4}/\tau,
    \nonumber
\end{split}
\end{equation}
concluding the proof of the proposition.
\end{proof}

Since the sets $\overline{D_{j,d}(x)}$ are contained in the compact set $K$ and, thanks to Lemma \ref{lemma3}, they form a Cauchy sequence with respect to the Hausdorff metric $d_H$, for any $j\in\N$ and $x\in A(j)$ there is a compact set $R_j(x)$ contained in $K$ such that:
$$\lim_{d\to \infty}d_H(\overline{D_{j,d}(x)},R_j(x))=0.$$

\begin{lemma}\label{lemma:A4}
For any $j\in\N$ and any $x\in A(j)$ we have $R_j(x)\subseteq B(x,2^{-j+5}/\tau)$.
\end{lemma}

\begin{proof}
Thanks to Lemma \ref{lemma3} and the triangular inequality for $d_H$, we have:
\begin{equation}
    d_H(R_j(x),D_j(x))\leq \sum_{d=0}^\infty d_H(D_{j,d+1}(x),D_{j,d}(x))\leq 2^{-j+3}/\tau\sum_{j=0}^\infty 2^{-Nd}=\frac{2^{-j+3}/\tau}{1-2^{-N}}.
    \label{eq:A14}
\end{equation}
Thanks to Proposition \ref{propA1}(i) we know that $D_j(x)\subseteq B(x,2^{-j+3}/\tau)$ and thus \eqref{eq:A14} and the triangular inequality yield:
$$ \sup_{y\in R_j(x)}d(x,y)\leq d_H(\{x\},R_j(x))\leq d_H(\{x\},D_j(x))+d_H(D_j(x),R_j(x))\leq2^{-j+3}/\tau+\frac{2^{-j+4}/\tau}{1-2^{-N}}<2^{-j+5}/\tau,$$
where the last inequality comes from the fact that $N>2$.
The above computation immediately implies that $R_j(x)\subseteq B(x,2^{-j+5}/\tau)$.
\end{proof}

\begin{lemma}\label{lemma:A2}
The compact sets $\{R_j(x):x\in A(j)\}$ are a covering of $K$. Furthermore, for any $j\in N\N$, any $x\in A(j)$ and any $l\in\N$ we have:
\begin{equation}
    R_j(x)= \bigcup_{z\in E_l(x)\cap A(j+Nl)} R_{j+Nl}(z).
    \label{eq:A9}
\end{equation}
\end{lemma}

\begin{proof}
Thanks to the definition of $M(j)$, see Remark \ref{rk:A3}, for any $d\geq M(j)$ we have:
$$\bigcup_{x\in A(j)} E_d(x)=\Xi(j+Nd).$$
Therefore, since by Proposition \ref{propA1} the sets $\{D_{j+Nd}(y):y\in \Xi(j+Nd)\}$ for any $j,d\in\N$ are a partition of $K$, for any $d\geq M(j)$ we deduce that:
\begin{equation}
K=\bigcup_{x\in \Xi(j+Nd)}D_{j+Nd}(x)=\bigcup_{x\in A(j)}\bigcup_{w\in E_d(x)}D_{j+Nd}(w)=\bigcup_{x\in A(j)} D_{j,d}(x).
\label{eq:nummiri3}
\end{equation}
The first part of the statement, namely that the $R_j$'s are a cover of $K$, follows from \eqref{eq:nummiri3}, Proposition \ref{prop:1} and the fact that $\lim_{d\to\infty}d_H(\overline{D_{j,d}(x)},R_j(x))=0$ for any $x\in A(j)$.

Concerning the second part of the statement for any $j,l\in\N$, any $w\in \Xi(j+lN)$ and any $d\geq M(j+Nl)$ thanks to Remark \ref{rk:A3}, we have $E_{d}(w)\neq \emptyset$ if and only if $w\in A(j+Nl)$. Thanks to Proposition \ref{prop:A2}, and the above argument we infer that:
\begin{equation}
\begin{split}
    D_{j,d+l}(x)=&\bigcup_{w\in E_{d+l}(x)}D_{j+Nl+Nd}(w)=\bigcup_{\substack{w\in E_{l}(x)\\ E_{d}(w)\neq \emptyset}}\bigcup_{ z\in E_d(w)}D_{j+Nl+Nd}(z)\\
    =&\bigcup_{\substack{w\in E_{l}(x)\\ E_{d}(w)\neq \emptyset}} D_{j+Nl,d}(w)=\bigcup_{w\in E_l(x)\cap A(j+Nl)} D_{j+Nl,d}(w).
\end{split}
\nonumber
\end{equation}
Since $\overline{D_{j,l+d}(x)}$ converges to $R_j(x)$ and $\overline{D_{j+Nl,d}(w)}$ to $R_{j+Nl}(w)$ in the Hausdorff metric as $d$ goes to infinity, the uniqueness of limit and Proposition \ref{prop:1} imply identity \eqref{eq:A9}.
\end{proof}

\begin{proposizione}\label{prop:A10}
For any $j\in\N$ we have $E^\phi(\xi,\tau)\subseteq \bigcup_{x\in \gimel(j)} R_j(x)$.
\end{proposizione}

\begin{proof}
Thanks to Proposition \ref{propA1}(ii), for any $j,k\in\N$ we have $E^\phi(\xi,\tau)\subseteq \bigcup_{w\in\gimel(j+Nk)}D_{j+Nk}(w)$. This implies that for any $y\in\gimel(j+Nk)$ we can find a $x\in \gimel(j)$ such that $\varphi^k(y)=x$. Thanks to Proposition \ref{inclgimel}, for any $k\in\N$ we have:
\begin{equation}
E^\phi(\xi,\tau)\subseteq\bigcup_{y\in \gimel(j+Nk)} D_{j+Nk}(y)\subseteq \bigcup_{x\in\gimel(j)} \bigcup_{y\in E_k(x)}D_{j+Nk}(y)=\bigcup_{x\in\gimel(j)}D_{j,k}(x). 
\label{eq:inclDjk}
\end{equation}

If we prove that $\dist(w,\bigcup_{x\in\gimel(j)}R_j(x))=0$ for any $w\in E^\phi(\xi,\tau)$, since $\bigcup_{x\in\gimel(j)}R_j(x)$ is closed and contained in $K$, we have $w\in \bigcup_{x\in\gimel(j)}R_j(x)$ proving the proposition. So, fix some $w\in E^\phi(\xi,\tau)$ and note that thanks to \eqref{eq:inclDjk} there exists a sequence $\{y_k\}\subseteq \gimel(j)$ such that $w\in \overline{D_{j,k}(y_k)}$. Since $\gimel(j)$ is a finite set, thanks to the pidgeonhole principle we can assume without loss of generality that $y_k$ is a fixed element $y\in\gimel(j)$ that it is also contained in $A(j)$ thanks to Proposition \ref{inclgimel}. This implies that for any $k\in\N$ we have:
$$\dist(w,R_j(y))\leq \dist(w,\overline{D_{j,k}(y)})+d_H(\overline{D_{j,k}(y)},R_j(y))=d_H(\overline{D_{j,k}(y)},R_j(y)).$$
The arbitrariness of $k$ and the definition of $R_j(y)$ concludes $\dist(w,R_j(y))=0$, proving the proposition.
\end{proof}

Despite the fact that the $R_j$s cover $E$, they may not be disjoint. In order to correct this, we put a total order $\preceq_j$ on $A(j)$ in such a way that:
\begin{itemize}
    \item[(i)]$x\prec_j y$ for any $x\in \gimel(j)$ and any $y\in A(j)\cap\Xi(j)\setminus \gimel(j)$
    \item[(ii)] $x\prec_j y$ if $\varphi(x)\prec_{j-N} \varphi(y)$.
\end{itemize}

\begin{definizione}
For any $j\in N\N$ and $x\in A(j)$, we define:
$$Q_j(x):=R_j(x)\cap \Bigg[\bigcup_{w\prec_j x} R_j(w) \Bigg]^c.$$
Furthermore, we let $\Delta_j:=\{Q_{j}(x):x\in A(j)\}$.
\end{definizione}

\begin{osservazione}\label{rk:A300}
Thanks to the definition of $Q_j$ we have $R_j(x)\cap Q_j(w)=\emptyset$ for any $x\in A(j)\cap \gimel(j)$ and any $w\in \Xi(j)\setminus \gimel(j)\cap A(j)$.
\end{osservazione}

\begin{proposizione}\label{prop:A3}
Let $j\leq j^\prime$, $x\in A(j)$ and $y\in A(j^\prime)$. Then either $Q_{j^\prime}(y)\subseteq Q_j(x)$ or $Q_{j^\prime}(y)\cap Q_j(x)\neq \emptyset$.
\end{proposizione}

\begin{proof}
First of all we note that if $j=j^\prime$, since the sets $Q_j$ are pairwise disjoint, either $Q_j(x)=Q_j(y)$, and thus $x=y$, or $Q_j(x)\cap Q_{j^\prime}(y)=\emptyset$. From now on, we can assume without loss of generality that $j^\prime= j+Nl$ for some $l\geq 1$.

 We claim that if $y\in E_l(\overline{z})$ for some $x\prec_j \overline{z}$, then $Q_{j^\prime}(y)\cap Q_{j}(x)=\emptyset$. Indeed, for any $w\in E_l(x)$, we have: $$\varphi^l(w)=x\prec_j\overline{z}=\varphi^l(y).$$ 
 and thanks to the definition of $\prec_{j^\prime}$, we deduce that $w\prec_j y$. Therefore this implies thanks to Lemma \ref{lemma:A2} that:
\begin{equation}
    R_j(x)=\bigcup_{w\in E_l(x)\cap A(j+Nl)} R_{j^\prime}(w)\subseteq\bigcup_{w\prec_{j^\prime}y} R_{j^\prime}(w).
    \label{eq:A10000}
\end{equation}
Thanks to the definition of $Q_{j^\prime}(y)$ and \eqref{eq:A10000}, we conclude that:
$$Q_{j^\prime}(y)\cap Q_{j}(x)\subseteq Q_{j^\prime}(y)\cap R_{j}(x)\subseteq R_{j^\prime}(y)\cap\bigcup_{w\prec_{j^\prime}y}R_{j^\prime}(w)=\emptyset.$$
On the other hand, if $y\in E_l(\overline{z})$ for some $ \overline{z}\prec_j x$, since by assumption $y\in A(j^\prime)$, thanks to Lemma \ref{lemma:A2}, we infer:
\begin{equation}
    \begin{split}
        Q_{j^\prime}(y)\cap Q_{j}(x)\subseteq& R_{j^\prime}(y)\cap Q_{j}(x)= R_{j^\prime}(y)\cap R_{j}(x)\cap \Bigg[\bigcup_{w\prec_j x} R_j(w) \Bigg]^c\\
        =& R_{j^\prime}(y)\cap R_{j}(x)\cap \Bigg[\bigcup_{w\prec_j x} \bigg( \bigcup_{z\in E_l(w)\cap A(j^\prime)} R_{j^\prime}(z) \bigg)\Bigg]^c
        \subseteq R_{j^\prime}(y) \cap R_{j}(x)\cap R_{j^\prime}(y)^c=\emptyset.
        \nonumber
    \end{split}
\end{equation}
Eventually, if $y\in E_l(x)$ we deduce thanks to Lemma \ref{lemma:A2}, that $R_{j^\prime}(y)\subseteq R_{j}(x)$. Thanks to the definition of $\prec_j$ we know that if $w\in E_l(z)$ for some $z\prec_j x$, then $\varphi^l(w)=z\prec_j x=\varphi^l(y)$,
and thus $w\prec_{j^\prime} y$ thanks to the definition of the order relation $\prec_{j^\prime}$. This and Lemma \ref{lemma:A2} imply that:
$$\bigcup_{z\prec_j x}R_j(z)=\bigcup_{z\prec_j x}\bigcup_{w\in E_l(z)\cap A(j+Nl)}R_{j^\prime} (p)\subseteq \bigcup_{w\prec_{j^\prime}y} R_{j^\prime}(w).$$
The above inclusion, shows that $R_{j^\prime}(y)\cap \bigcup_{z\prec_j x}R_j(z)=\emptyset$ and thus, thanks to the definition of the sets $Q_j$, we deduce that  $Q_{j^\prime}(y)\subseteq Q_j(x)$.
\end{proof}

\begin{proposizione}\label{lemma:A5}
For any $j\in N\N$ and any $x\in A(j)$ we have:
\begin{equation}
\sup_{v\in R_j(x)}\dist(v,D_j(x))\leq 2^{-j-N+5}/\tau.
\label{eq:A10001}
\end{equation}
Furthermore, the following inclusions hold:
\begin{equation}
    D_j(x)\triangle R_j(x)\subseteq \partial( D_j(x),\eta^22^{-j-1}/\tau)\qquad\text{and}\qquad D_j(x)\triangle Q_j(x)\subseteq \partial( D_j(x),\eta^22^{-j-1}/\tau).
    \label{eq:A10002}
\end{equation}
\end{proposizione}

\begin{proof}
If $u\in\ D_{j,1}(x)$, there is a $y\in \Xi(j+N)\cap D_j(x)$ such that $u\in D_{j+N}(y)$ and thanks to Proposition \ref{propA1}(i) we have:
$$\dist(u,D_j(x))\leq \dist(u,y)\leq \diam (D_{j+N}(y))\leq 2^{-j-N+3}/\tau.$$ Therefore, for any $v\in R_j(x)$, we have:
\begin{equation}
\begin{split}
    \dist(v, D_j(x))\leq& \inf_{u\in D_{j,1}(x)} d(u,v)+\dist(u,D_j(x))\\
    \leq&\inf_{u\in D_{j,1}(x)}d(u,v) +2^{-j-N+3}/\tau= \dist(v,D_{j,1}(x))+2^{-j-N+3}/\tau.
    \label{eq:nummm2020}
\end{split}
\end{equation}
Since $\dist(v,D_{j,1}(x))\leq d_H(R_j(x),D_{j,1}(x))$, Lemma \ref{lemma3} together with \eqref{eq:nummm2020} implies that:
\begin{equation}
    \begin{split}
        \dist(v,D_j(x))\leq&d_H(D_{j,1}(x),R_j(x))+2^{-j-N+3}/\tau\leq \sum_{d=1}^\infty d_H(D_{j,d}(x),D_{j,d+1}(x))+2^{-j-N+3}/\tau\\
        \leq& \sum_{d=1}^\infty 2^{-j-Nd+4}/\tau+2^{-j-N+3}/\tau\leq 2^{-j-N+5}/\tau,
        \nonumber
    \end{split}
\end{equation}
for any $v\in R_j(x)$, where the last inequality follows since $N>2$. 

Let us move to the proof of \eqref{eq:A10002}. For any $u\in D_j(x)\cap Q_j(x)^c$, we can assume without loss of generality that there exists an $x(u)\in A(j)\setminus\{x\}$ such that $u\in R_j(x(u))$. This is due to the following argument. The fact that the sets $R_j$ are a cover of $K$ shows that we can always find a $y\in A(j)$ such that $y\in R_j(y)$. Furthermore, if $u\in R(x)\cap Q_j(x)^c$, then by definition of $Q_j(x)$ there must exists another $x(u)\prec_j x$ such that $u \in R_j(x(u))$, and this shows that we can always choose $x(u)$ different from $x$.

Therefore, since $D_j(x(u))\subseteq K\setminus D_j(x)$, thanks to \eqref{eq:A10001} we deduce that for any $u\in D_j(x)\cap Q_j(x)^c$, we have:
\begin{equation}
\text{dist}(u,K\setminus D_j(x))\leq \dist(u,D_j(x(u)))\leq\sup_{v\in R_j(x(u))}\dist(v,D_j(x(u)))\leq 2^{-j-N+6}/\tau.
\label{eq:A10003}
\end{equation}
Thanks to inequality \eqref{eq:A10003}, we infer that:
\begin{equation}
D_j(x)\cap R_j(x)^c\subseteq D_j(x)\cap Q_j(x)^c\subseteq \{w\in D_j(x):\dist(w,K\setminus D_j(x))\leq 2^{-j-N+6}/\tau\}.
    \label{eq:A10005}
\end{equation}
On the other hand, for any $u\in D_j(x)^c\cap R_j(x)$, inequality \eqref{eq:A10001} implies that:
\begin{equation}
    D_j(x)^c\cap Q_j(x)\subseteq D_j(x)^c\cap R_j(x)\subseteq \{w\in K\setminus D_j(x): \dist(w, D_j(x))\leq 2^{-j-N+6}/\tau\}.
    \label{eq:A10006}
\end{equation}
The inclusions \eqref{eq:A10005} and \eqref{eq:A10006} yield both of the inclusion of \eqref{eq:A10002}, and this concludes the proof, since thanks to the choice of $\eta$ we have $2^{-N+6}\leq \eta^2/2$.
\end{proof}

We are left to prove that the families of sets $\Delta_j$ satisfy the properties of Theorem \ref{evev}:

\begin{teorema}
Defined $\newC\label{C:2}:=2^{7\mathcal{Q}}\xi$ and $\newC\label{C3}:=2^{24\mathcal{Q}}\xi$, the families $\{\Delta_j\}_{j\in\N}$ have the following properties:
\begin{itemize}
    \item[(i)] $K= \bigcup_{Q\in\Delta_j}Q$ for any $j\in\N$ and the elements of $\Delta_j$ are pairwise disjoint,
    \item[(ii)] for any $j\in\N$ we have $E^\phi(\xi,\tau)\subseteq \bigcup_{x\in \gimel(j)} Q_j(x)$,
    \item[(iii)] if $j\leq j^\prime$, $Q\in\Delta_j$ and $Q^\prime\in\Delta_{j^\prime}$, then either $Q$ contains $Q^\prime$ or $Q\cap Q^\prime=\emptyset$,
    \item[(iv)]for any $Q\in\Delta_j$ we have  $\diam(Q)\leq 2^{-j+5}/\tau$,
    \item[(v)] if $x\in\gimel(j)$, then $\oldC{C:2}^{-1}\big(2^{-j}/\tau\big)^{m}\leq \phi(Q_j(x))\leq \oldC{C:2}\big(2^{-j}/\tau\big)^{m}$,
    \item[(vi)] if $Q\in\Delta_j$, we have $ \phi\big(\partial(Q,\eta^2 2^{-j}/\tau )\big)\leq \oldC{C3} \eta\big(2^{-j}/\tau)^{m}$,
    \item[(vii)] for any $x\in\gimel(j)$ there is a $c\in K$ such that $B(c,\eta^2 2^{-j-1}/\tau)\cap K\subseteq Q_j(x)$.
\end{itemize}
\end{teorema}

\begin{proof}
Point (i) is satisfied by the definition of the sets $Q_j$ and Lemma \ref{lemma:A2} and (ii) follows from Proposition \ref{prop:A10} and Remark \ref{rk:A300}. Furthermore, (iii) is implied by Proposition \ref{prop:A3} and (iv) by Lemma \ref{lemma:A4}.

\paragraph{Proof of (v)}If $x\in\gimel(j)$ the upper bound follows immediately thanks to the fact that $x\in E^\phi(\xi,\tau)$ and Proposition \ref{lemma:A4}:
$$\phi(Q_j(x))\leq \phi(B(x,2^{-j+5}/\tau))\leq\xi(2^{-j+5}/\tau)^m.$$
For the lower bound note that:
$$\phi(Q_j(x))=\phi(D_j(x))-\phi(D_j(x)\setminus Q_j(x))+\phi(Q_j(x)\setminus D_j(x))
    \geq\phi(D_j(y))-\phi(Q_j(x)\triangle D_j(x)).$$
Thanks to Proposition \ref{lemma:A5} we know that $Q_j(x)\triangle D_j(x)\subseteq \partial(D_j(x),\eta^2 2^{-j}/\tau)$.
Therefore, Proposition \ref{propA1}(iii) and (iv) imply that:
\begin{equation}
\begin{split}
    \phi(Q_j(x))\geq\phi(D_j(y))-\phi(\partial(D_j(x),\eta^22^{-j}/\tau))\geq (32^{-\mathcal{Q}}\xi^{-1}-\oldC{C:6}\eta)(2^{-j}/\tau)^m\geq 64^{-\mathcal{Q}}\xi^{-1}(2^{-j}/\tau)^m,
    \nonumber
\end{split}
\end{equation}
where the last inequality follows from the fact that we assumed $\eta\leq 2^{-13\mathcal{Q}}\xi^{-2}$, see the beginning of Subsection \ref{sec:proofdyadic}.
This concludes the proof of (v) thanks to the choice of $\oldC{C:2}$.

\paragraph{Proof of (vi)}For any $j\in N\N$ and any $x\in A_j(x)$ let us define the sets:
\begin{equation}
    \begin{split}
        \flat_1Q_j(x):=&\{u\in Q_j(x):\dist(u,K\setminus Q_j(x))\leq \eta^22^{-j-1}/\tau\},\\
        \flat_2Q_j(x):=&\{u\in K\setminus Q_j(x):\dist(u, Q_j(x))\leq \eta^22^{-j-1}/\tau\}.
        \nonumber
    \end{split}
\end{equation}
First of all for any $x\in A(j)$, we estimate the measure of the set $\flat_1Q_j(x)$.
If $u\in \flat_1Q_j(x)$, since the $Q_j$s are a partition of $K$, we can find $x\neq x^\prime\in A(j)$ such that $u\in \flat_2Q_j(x^\prime)$. Thanks Proposition \ref{lemma:A5} and the triangular inequality, on the one hand we deduce that:
$$\dist(u,D_j(x))\leq \min_{w\in Q_j(x)}d(u,w)+\dist(w,D_j(x))\leq 2^{-j-N+6}\leq 2^{-j-1}\eta^2/\tau.$$
where the last inequality follows from the choice of $\eta$. On the other:
$$\dist(u,D_j(x^\prime))\leq \inf_{w\in Q_j(x^\prime)} d(u,w)+\dist(w,D_j(x^\prime)) \leq  2^{-j-1}\eta^2/\tau+2^{-j-N+6}\leq \eta^2 2^{-j}/\tau,$$
where once again the last inequality follows from the choice of $\eta$. Therefore the two above bounds imply in particular that $u\in \partial (D_j(x^\prime),\eta^22^{-j}/\tau)$. Therefore, thanks to the arbitrariness of $u$, we conclude that:
\begin{equation}
\flat_1Q_j(x)\subseteq \bigcup_{\substack{x^\prime\in A(j)\\\partial (D_j(x^\prime),\eta^22^{-j}/\tau)\cap \flat_1Q_j(x)\neq \emptyset }} \partial (D_j(x^\prime),\eta^22^{-j}/\tau)
\label{eq:nummm1998}
\end{equation}
Thanks to
Proposition \ref{propA1}(i), Lemma \ref{lemma:A4} and using the same argument we employed to prove Lemma \ref{lemma2}, one can show that:
\begin{equation}
\begin{split}
    \Card(\{y\in A(j):Q_j(y)\cap \partial(D_j(x),&\eta^22^{-j}/\tau)\neq \emptyset\})\leq\Card(\{y\in A(j):R_j(y)\cap \partial(D_j(x),\eta^22^{-j}/\tau)\neq \emptyset\})\\
    \leq& \Card(\{y\in A(j):B(x,2^{-j+5}/\tau)\cap B(x,2^{-j+4}/\tau)\neq \emptyset\})\leq 2^{7\mathcal{Q}}.
\end{split}
    \label{eq:A16}
\end{equation}
Therefore, the inclusion \eqref{eq:nummm1998}, the bound \eqref{eq:A16} and Proposition \ref{propA1}(iv) imply that:
\begin{equation}
    \begin{split}
        \phi(\flat_1Q_j(x))\leq 2^{7\mathcal{Q}}\max_{y:Q_j(y)\cap D_j(x)\neq \emptyset} \phi(\partial(D_j(x),&\eta^22^{-j}/\tau))
        \leq 2^{7\mathcal{Q}}\oldC{C:6} \eta (2^{-j}/\tau)^m.
        \label{eq:A15}
    \end{split}
\end{equation}

We are left to estimate the measure of $\flat_2 Q_j(x)$. Since the $Q_j$s are a partition of $K$, we infer that:
\begin{equation}
        \phi(\flat_2Q_j(x))=\sum_{y\in A(j)\setminus\{x\}}\phi(\{u\in Q_j(y):\dist(u, Q_j(x))\leq \eta^22^{-j-1}\}).
  \label{eq:A17}
\end{equation}
The cardinality of those $y\in A(j)$ for which there exists $u\in Q_j(y)$ such that $\dist(u,K\setminus Q_j(x))\leq \eta^22^{-j-1}$ can be bounded by $2^{9\mathcal{Q}}$ and this can be shown with the same argument used for \eqref{eq:A16}. Thanks to the bounds \eqref{eq:A15} and \eqref{eq:A17} we infer:
\begin{equation}
        \phi(\flat_2Q_j(x))\leq 2^{9\mathcal{Q}}\cdot\max_{\substack{y\in A(j)\setminus\{x\}\\Q_j(y)\cap B(Q_j(x),\eta^22^{-j-1})\neq \emptyset}}\phi(\{u\in Q_j(y):\dist(u, Q_j(x))\leq \eta^22^{-j-1}\})
        \leq  2^{16\mathcal{Q}}\oldC{C:6}\eta(2^{-j}/\tau)^m.
  \label{eq:A18}
\end{equation}
Hence putting together \eqref{eq:A15} and \eqref{eq:A18}, we deduce that:
\begin{equation}
    \phi(\partial(Q_j(x),\eta^22^{-j-1}/\tau))=\phi(\flat_1Q_j(x))+\phi(\flat_2Q_j(x))\leq 2^{7\mathcal{Q}}\oldC{C:6} \eta (2^{-j}/\tau)^m+2^{16\mathcal{Q}}\oldC{C:6}\eta(2^{-j}/\tau)^m\leq 2^{17\mathcal{Q}}\oldC{C:6}\eta(2^{-j}/\tau)^m.
    \nonumber
\end{equation}
Thus, (v) follows with the choice $\oldC{C3}:=2^{17\mathcal{Q}}\oldC{C:6}$.
\paragraph{Proof of (vii)}Thanks for  (iv) and (v), for any $x\in \gimel(j)$ we have:
\begin{equation}
    \begin{split}
    \oldC{C:2}^{-1}(2^{-j}/\tau)^m\leq &\phi(Q_j(x))\leq \phi(\partial(Q_j(x),\eta^2 2^{-j-1}/\tau))+\phi(\{u\in Q_j(x):\dist(u,K\setminus Q_j(x))>\eta^2 2^{-j-1}/\tau\})\\
    \leq & \oldC{C3}\eta(2^{-j}/\tau)^m+\phi(\{u\in Q_j(x):\dist(u,K\setminus Q_j(x))>\eta^2 2^{-j-1}/\tau\})
        \nonumber
    \end{split}
\end{equation}
Since $\eta\leq \xi^{-2}2^{-26\mathcal{Q}}$, we deduce that
$\oldC{C:2}^{-1}-\oldC{C3}\eta\geq \xi^{-1}2^{-8\mathcal{Q}}$ and thus:
$$\xi^{-1}2^{-8\mathcal{Q}}(2^{-j}/\tau)^m\leq \phi(\{u\in Q_j(x):\dist(u,K\setminus Q_j(x))>\eta^2 2^{-j-1}/\tau\}).$$
This implies that there exists $c\in K$ such that $B(c,\eta^22^{-j-1}/\tau)\subseteq Q_j(x)$.
\end{proof}

\section{Finite perimeter sets in Carnot groups}\label{def:perimetro}

Throughout this second part of the appendix if not otherwise stated, we will always endow $\mathbb{G}$ with the smooth-box metric introduced in Definition \ref{smoothnorm}.

\subsection{Finite perimeter sets, intrinsic Lipschitz graphs and regular surfaces}

In this subsection we recall the definitions of functions of bounded variations and finite perimeter sets and we collect from various papers some results that will be useful in the following.

\begin{definizione}
We say that a function $f:\mathbb{G}\to \R$ is of local bounded variation if $f\in L^1_{loc}(\mathbb{G})$ and:
$$\lVert \nabla_{\mathbb{G}} f\rVert(\Omega):=\sup\Big\{ \int_{\Omega} f(x)\text{div}_{\mathbb{G}} \varphi(x) dx:\varphi\in \mathcal{C}^1_0(\Omega,H\mathbb{G}),\lvert \varphi(x)\rvert\leq 1\Big\}<\infty,$$
for any bounded open set $\Omega\subseteq \mathbb{G}$, where $\text{div}_\mathbb{G}\varphi:=\sum_{i=1}^{n_1} X_i\varphi_i$ and where  $X_1,\ldots,X_{n_1}$ are the vector fields introduced in Definition \ref{regsur}. We denote by $\text{BV}_{\mathbb{G},loc}(\mathbb{G})$ the set of all  functions of locally bounded variation. As usual a Borel set $E\subseteq \mathbb{G}$ is said to be of \emph{finite perimeter} if $\chi_E$ is of bounded variation. 
\end{definizione}

\begin{teorema}[Theorem 4.3.2, \cite{MFSSC}]
If $f$ is a function of bounded variation, then $\lVert \nabla_{\mathbb{G}} f\rVert$ is a Radon measure on $\mathbb{G}$. Moreover there exists a $\lVert \nabla_{\mathbb{G}} f\rVert$-measurable horizontal section $\sigma_f:\mathbb{G}\to H\mathbb{G}$ such that $\lvert \sigma_f(x)\rvert=1$ for $\lVert \nabla_{\mathbb{G}} f\rVert$-a.e. $x\in\mathbb{G}$, and for any open set $\Omega$ we have:
$$\int_\Omega f(x)\text{div}_{\mathbb{G}}\varphi(x)dx=\int_\Omega \langle \varphi,\sigma_f\rangle d\lVert \nabla_{\mathbb{G}} f\rVert,\qquad\text{for every }\varphi\in\mathcal{C}^1_0(\Omega,H\mathbb{G}).$$
\end{teorema}

As in the Euclidean spaces functions of bounded variation are compactly embedded in $L^1$:

\begin{teorema}[Theorem 2.16, \cite{step2}]\label{scalepE}
The set $\text{BV}_{\mathbb{G},loc}(\mathbb{G})$ is compactly embedded in $L^1_{loc}(\mathbb{G})$.
\end{teorema}

\begin{definizione}\label{def:norm}
If $E\subseteq \mathbb{G}$ is a Borel set of locally finite perimeter, we let $\lvert\partial E\rvert_\mathbb{G}:=\lVert \nabla_{\mathbb{G}} \chi_E\rVert$. Furthermore we call \emph{generalized horizontal inward} $\mathbb{G}$-\emph{normal to} $\partial E$ the horizontal vector $\mathfrak{n}_E(x):=\sigma_{\chi_E}(x)$. Finally, we define the \emph{reduced boundary} $\partial_{\mathbb{G}}^*E$ to be the set of those $x\in\mathbb{G}$ for which:
\begin{itemize}
    \item[(i)]$\lvert\partial E\rvert_\mathbb{G}(B(x,r))>0$ for any $r>0$,
    \item[(ii)] $\lim_{r\to 0}\fint_{B(x,r)}\mathfrak{n}_Ed\lvert \partial E\rvert_{\mathbb{G}}$ exists,
    \item[(iii)] $\lim_{r\to 0}\Big\lVert\fint_{B(x,r)}\mathfrak{n}_Ed\lvert \partial E\rvert_{\mathbb{G}}\Big\rVert =1$.
\end{itemize}
\end{definizione}

The following lemma on the scaling of the perimeter will come in handy later on.

\begin{lemma}
Assume $E$ is a set of finite perimeter in $\mathbb{G}$ and let $x\in\mathbb{G}$ and $r>0$. Then:
$$\lvert \partial (\delta_{1/r}(x^{-1}E))\rvert_\mathbb{G}=r^{-(\mathcal{Q}-1)}T_{x,r}\lvert\partial E\rvert_{\mathbb{G}}.$$
\end{lemma}

\begin{proof}
For any $\varphi\in C^1_0(\mathbb{G},H\mathbb{G})$, any $x\in\mathbb{G}$ and any $r>0$, defined $\tilde{\varphi}(z):=\varphi(\delta_{1/r}(x^{-1}z))$, the following identity holds: 
\begin{equation}
    \text{div}_{\mathbb{G}}\tilde{\varphi}(z)=r^{-1}\text{div}_{\mathbb{G}}\varphi(\delta_{1/r}(x^{-1}z)).
    \label{eq:nummm1992}
\end{equation}
This, indeed is due to the fact that:
$$X_j\tilde{\varphi}_j(z):=\lim_{h\to 0}\frac{\tilde{\varphi}_j(z\delta_h(e_j))-\tilde{\varphi}_j(z)}{h}=\lim_{h\to 0}\frac{\varphi_j(\delta_{1/r}(x^{-1}z\delta_h(e_j)))-\varphi_j(\delta_{1/r}(x^{-1}z))}{h}=r^{-1}X_j\varphi_j(\delta_{1/r}(x^{-1}z)).$$
Thanks to identity \eqref{eq:nummm1992} and the fact that the Lebesgue measure is a Haar measure for $\mathbb{G}$, we infer that: $$\int \chi_{\delta_{1/r}(x^{-1}E)}(y)\text{div}_\mathbb{G}\varphi(y) dy=r^{-\mathcal{Q}}\int \chi_E \text{div}_\mathbb{G}\varphi(\delta_{1/r}(x^{-1}y))dy=r^{-(\mathcal{Q}-1)}\int \chi_E(y) \text{div}_\mathbb{G}\tilde{\varphi}(y)dy.$$
It is not hard to see that $\varphi\in C^1_0(\Omega,H\mathbb{G})$ if and only if $\tilde{\varphi}\in C^1_0(x\delta_r\Omega,H\mathbb{G})$ and thus for any open set $\Omega$ we have:
$$\lvert(\partial \delta_{1/r}(x^{-1}E))\rvert_\mathbb{G}(\Omega)=r^{-(\mathcal{Q}-1)}\lvert\partial E\rvert_\mathbb{G}(x\delta_r\Omega)=r^{-(\mathcal{Q}-1)}T_{x,r}\lvert\partial E\rvert_\mathbb{G}(\Omega).$$
This concludes the proof.
\end{proof}

\begin{teorema}[Theorem 4.16, \cite{ambled}]\label{th:4.16}
Let $E\subseteq \mathbb{G}$ be a set of locally finite perimeter. Then $\lvert \partial E\rvert_{\mathbb{G}}$ is asymptotically doubling, and more precisely the following hold. For $\lvert \partial E\rvert_{\mathbb{G}}$-a.e. $x\in\mathbb{G}$ there exists an $\overline{r}(x)>0$ such that:
\begin{equation}
    l_{\mathbb{G}}r^{\mathcal{Q}-1}\leq\lvert \partial E\rvert_{\mathbb{G}}(B(x,r))\leq L_{\mathbb{G}}r^{\mathcal{Q}-1},\qquad \text{for any }r\in(0,\overline{r}(x)),
    \label{eq:adreggi}
\end{equation}
where the constants $l_{\mathbb{G}}$ and $L_{\mathbb{G}}$ depend only on $\mathbb{G}$ and the metric $d$. As a consequence $\lvert \partial E\rvert_{\mathbb{G}}$ is concentrated on $\partial^*_{\mathbb{G}} E$, i.e. $\lvert \partial E\rvert_{\mathbb{G}}(\mathbb{G}\setminus \partial^*_{\mathbb{G}} E)=0$.
\end{teorema}

\begin{osservazione} \label{rk:b1}
 Proposition \ref{prop:dens} and Theorem \ref{th:4.16}  imply that $l_{\mathbb{G}}\mathcal{S}^{\mathcal{Q}-1}\llcorner \partial^*_\mathbb{G}E\leq \lvert \partial E\rvert_{\mathbb{G}}\leq L_{\mathbb{G}}\mathcal{S}^{\mathcal{Q}-1}\llcorner \partial^*_\mathbb{G}E$. Therefore, the measures $\mathcal{S}^{\mathcal{Q}-1}\llcorner \partial^*_\mathbb{G}E$ and $\lvert \partial E\rvert_{\mathbb{G}}$ are mutually absolutely continuous. In particular there exists a $\mathfrak{d}\in L^1(\lvert \partial E\rvert_{\mathbb{G}})$ such that: 
 $$\mathcal{S}^{\mathcal{Q}-1}\llcorner \partial^*_\mathbb{G}E=\mathfrak{d}\lvert \partial E\rvert_{\mathbb{G}},$$
 and for $\lvert \partial E\rvert_{\mathbb{G}}$-almost every $x\in \mathbb{G}$ we have $L_{\mathbb{G}}^{-1}\leq\mathfrak{d}(x)\leq l_{\mathbb{G}}^{-1}$.
\end{osservazione}

\begin{teorema}[Theorem 4.3.9, \cite{MFSSC}]
If $f:V\to \mathfrak{N}(V)$ is an intrinsic Lipschitz map, the sub-graph:
\begin{equation}
    \text{epi}(f):=\Big\{v*\delta_t(\mathfrak{n}(V)):t< \langle  \pi_1 f(v),\mathfrak{n}(V)\rangle\Big\},
\label{eq:n70}
\end{equation}
is a set with locally finite $\mathbb{G}$-perimeter.
\end{teorema}

Since the topological boundary of $\text{epi}(f)$ coincides with $\text{gr}(f)$, thanks to Theorem 3.2.1
of \cite{MFSSC}, we infer that $\lvert \partial\text{epi}(f)\rvert_\mathbb{G}(\mathbb{G}\setminus \partial_\mathbb{G}^* \text{epi}(f))=\lvert \partial\text{epi}(f)\rvert_\mathbb{G}( \text{gr}(f)\setminus \partial_\mathbb{G}^* \text{epi}(f))=0$.
In particular, thanks to Remark \ref{rk:b1}, we deduce the following:

\begin{proposizione}\label{vf}
$\mathcal{S}^{\mathcal{Q}-1}(\text{gr}(f)\setminus \partial_\mathbb{G}^* \text{epi}(f))=0$.
\end{proposizione}

It is convenient to associate to every intrinsic Lipschitz function $f:V\to\mathfrak{N}(V)$ a normal vector field to its graph:

\begin{definizione}\label{fnormal}
For any $f:V\to \mathfrak{N}(V)$ intrinsic Lipschitz function, we denote by $\mathfrak{n}_f:\partial_\mathbb{G}^* \text{epi}(f)\to H\mathbb{G}$ the inward inner $\mathbb{G}$-normal of $\text{epi}(f)$.
\end{definizione}

\subsection{Tangents measures versus tangent sets to finite perimeter sets}

In this subsection we connect the notion of tangent sets to Caccioppoli sets, that is extensively used in the theory of finite perimeter sets, to the notion of tangent measures. This will help us to prove that if the perimeter measure of a Caccioppoli set has flat tangents, then it has a unique tangent, that coincides with the plane in $\G(\mathcal{Q}-1)$ orthogonal to the normal.

\begin{definizione}(Tangent sets)
Let $E\subseteq \mathbb{G}$ be a set of locally finite perimeter and assume $x\in \partial^*_{\mathbb{G}} E$. We denote by $\Tan(E,x)$ the limit points in the topology of the local convergence in measure of the sets $\{\delta_{1/r}(x^{-1}E)\}_{r>0}$ as $r\to 0$. 
\end{definizione}

For a proof of the following proposition, we refer to \cite{ambled} and in particular to Proposition 5.3.

\begin{proposizione}\label{prop:costnorm}
If $E$ is a set of finite perimeter, for $\mathcal{S}^{\mathcal{Q}-1}$-almost every $x\in\partial^*_\mathbb{G} E$ we have:
\begin{itemize}
    \item[(i)]$\Tan(E,x)\neq \emptyset$,
    \item[(ii)] the elements of $\Tan(E,x)$ are finite perimeter sets,
    \item[(iii)]for any $F\in\Tan(E,x)$ we have $\mathfrak{n}_F(y)=\mathfrak{n}_E(x)\text{ for }\lvert\partial F\rvert_{\mathbb{G}}\text{-almost every }y\in \mathbb{G}$.
\end{itemize}
\end{proposizione}

The following proposition is a characterisation of the tangent measures of perimeter measures:

\begin{proposizione}[observation (6.5), \cite{ambled}]\label{obs6.5}
If $E$ is a set of locally finite perimeter, for $\lvert\partial E\rvert_\mathbb{G}$-almost every $x\in \partial^*_{\mathbb{G}} E$ we have that:
\begin{itemize}
    \item[(i)] if $\{r_i\}_{i\in\N}$ is an infinitesimal sequence such that $\delta_{1/r_i}(x^{-1}E)$ converges locally in measure to some Borel set $L$, then $L$ is a finite perimeter set and
    $r_i^{-(\mathcal{Q}-1)}T_{x,r_i}\lvert\partial E\rvert_\mathbb{G}\rightharpoonup \lvert\partial L\rvert_\mathbb{G}$.
    In particular, if $L\in \Tan(E,x)$ then $\lvert \partial L\rvert_\mathbb{G}\in \Tan_{\mathcal{Q}-1}(\lvert\partial E\rvert_\mathbb{G},x)$,
    \item[(ii)] if $\nu\in \Tan_{\mathcal{Q}-1}(\lvert \partial E\rvert_\mathbb{G},x)$, then there is an $L\in \Tan(E,x)$ such that $\nu=\lvert \partial L\rvert_{\mathbb{G}}$.
\end{itemize}
\end{proposizione}

\begin{proof}
Let us first prove (i). Proposition \ref{obs6.5}(ii) implies that for $\lvert \partial E\rvert_\mathbb{G}$-almost all $x\in\partial^* E$, every limit (locally in measure) of a sequence of the form $\delta_{1/r}(x^{-1}E)$ is of finite  perimeter. Therefore, from now on we assume without loss of generality that $x$ is a fixed point where Proposition \ref{obs6.5}(ii) holds.

Fix now any open and bounded set $\Omega\subseteq \mathbb{G}$ and note that for any $\varphi\in C^1_0(\Omega,H\mathbb{G})$ a simple computation yields:
\begin{equation}
    \bigg\lvert \int_{\delta_{1/r_i}(x^{-1}E)}\text{div}_{\mathbb{G}}\varphi(z) dz-\int_L \text{div}_{\mathbb{G}}\varphi(z)dz\bigg\rvert\leq \lVert \text{div}_\mathbb{G} \varphi\rVert_\infty \mathcal{L}^{n}(\delta_{1/r_i}(x^{-1}E)\triangle L\cap \Omega).
    \label{eq:n62}
\end{equation}

Suppose $\tilde{\varphi}^1\in C^1_0(\Omega,H\mathbb{G})$ is a vector field such that $\int_E\text{div}_{\mathbb{G}} \tilde{\varphi}^1(w) dw\geq \lvert \partial E\rvert_\mathbb{G}(x\delta_{r_i}\Omega)-\epsilon$ and define $\varphi^1(z):=\tilde{\varphi}^1(x\delta_{r_i}(z))$. The bound \eqref{eq:n62} implies that:
\begin{equation}
\begin{split}
    &r_i^{-(\mathcal{Q}-1)}\lvert\partial E\rvert_\mathbb{G}(x\delta_{r_i}\Omega)-\lvert \partial L\rvert_\mathbb{G}(\Omega)-\epsilon\leq \frac{1}{r_i^{\mathcal{Q}-1}}\int_{E}\text{div}_{\mathbb{G}}\tilde{\varphi}^1(z) dz-\int_L \text{div}_{\mathbb{G}}\varphi^1(z)dz\\
    =& \int_{\delta_{1/r}(x^{-1}E)}\text{div}_{\mathbb{G}}\varphi^1(z) dz-\int_L \text{div}_{\mathbb{G}}\varphi^1(z)dz= \lVert \text{div}_\mathbb{G} \varphi^1\rVert_\infty \mathcal{L}^{n}(\delta_{1/r_i}(x^{-1}E)\triangle L\cap \Omega),
\end{split}
    \label{eq:n63}
\end{equation}
where the first inequality of the second line follows from \eqref{eq:nummm1992}.
On the other hand, if $\tilde{\varphi}^2$ is such that $\int_L\text{div}_{\mathbb{G}} \varphi^2(w) dw\geq \lvert \partial L\rvert_\mathbb{G}(\Omega)-\epsilon$ we let $\varphi^2(z):=\tilde{\varphi}^2(\delta_{1/r_i}(x^{-1}z))$. Eventually, the same argument we used in \eqref{eq:n63}, imply:
\begin{equation}
\begin{split}
    \lvert \partial L\rvert_\mathbb{G}(\Omega)-r_i^{-(\mathcal{Q}-1)}\lvert\partial E\rvert_\mathbb{G}(x\delta_{r_i}\Omega)-\epsilon\leq&\int_L \text{div}_{\mathbb{G}}\varphi^2(z)dz- \frac{1}{r_i^{\mathcal{Q}-1}}\int_{E}\text{div}_{\mathbb{G}}\tilde{\varphi}^2(z) dz\\
    \leq& \lVert \text{div}_\mathbb{G} \varphi^2\rVert_\infty \mathcal{L}^{n}(\delta_{1/r_i}(x^{-1}E)\triangle L\cap \Omega).
\end{split}
    \label{eq:n64}
\end{equation}
Putting together \eqref{eq:n63} and \eqref{eq:n64} we infer that:
\begin{equation}
    \begin{split}
        \lim_{i\to\infty}\Big\lvert \lvert \partial L\rvert_\mathbb{G}(\Omega)-&r_i^{-(\mathcal{Q}-1)}\lvert\partial E\rvert_\mathbb{G}(x\delta_{r_i}\Omega)\Big\rvert-\epsilon\\
        \leq& (\lVert \text{div}_\mathbb{G} \varphi^1\rVert_\infty+\lVert \text{div}_\mathbb{G} \varphi^2\rVert_\infty)\lim_{i\to\infty} \mathcal{L}^{n}(\delta_{1/r_i}(x^{-1}E)\triangle L\cap \Omega)=0.
        \nonumber
    \end{split}
\end{equation}
The above inequality together with the arbitrariness of $\epsilon>0$, implies that for any bounded open set $\Omega$ we have:
\begin{equation}
    \lim_{r_i\to 0}\bigg\lvert \lvert \partial L\rvert_\mathbb{G}(\Omega)-\frac{ T_{x,r_i}\lvert\partial E\rvert_\mathbb{G}}{r_i^{\mathcal{Q}-1}}(\Omega)\bigg\rvert=0.
    \label{eq:n65}
\end{equation}
Let $\nu_i:=r_i^{-(\mathcal{Q}-1)}T_{x,r_i}\lvert\partial E\rvert_\mathbb{G}$ and for any $\rho>0$ define $\mu_i:=\nu_i(B(0,\rho))^{-1}\nu_i\llcorner B(0,\rho)$ and $\mu:=\lvert\partial L\rvert_\mathbb{G}(B(0,\rho))^{-1}\lvert\partial L\rvert_\mathbb{G}\llcorner B(0,\rho)$. Thanks to identity \eqref{eq:n65}, we infer that for any bounded open set $\Omega$ we have:
\begin{equation}
    \lim_{i\to\infty} \mu_i(\Omega)=\lim_{i\to\infty}\frac{\nu_i(B(0,\rho)\cap \Omega)}{\nu_i(B(0,\rho))}= \frac{\lvert\partial L\rvert_\mathbb{G}(B(0,\rho)\cap \Omega)}{\lvert\partial L\rvert_\mathbb{G}(B(0,\rho))}=\mu(\Omega).
\end{equation}
Thanks to Theorem 13.16 of \cite{Klenke}, we deduce that $\mu_i\rightharpoonup\mu$. Therefore, for any $f\in\mathcal{C}_c(\mathbb{G})$ we have:
\begin{equation}
    \lim_{i\to\infty}\int f(x)d\nu_i(x)=\lim_{i\to\infty}\nu_i(B(0,\rho))\int f(x)d\mu_i(x)=\int f(x)d\lvert\partial L\rvert_\mathbb{G}(x),
    \nonumber
\end{equation}
proving that $r_i^{-(\mathcal{Q}-1)}T_{x,r_i}\lvert\partial E\rvert_\mathbb{G}\rightharpoonup \lvert\partial L\rvert_\mathbb{G}$. The second part of the statement of (i) follows immediately.

We now prove  (ii). We can assume without loss of generality that $x=0$ satisfy the thesis of Theorem \ref{th:4.16} and that $\{r_i\}$ is an infinitesimal sequence such that: $$r_i^{-(\mathcal{Q}-1)}T_{0,r_i}\lvert\partial E\rvert_\mathbb{G}\rightharpoonup \nu\in\Tan_{\mathcal{Q}-1}(\lvert\partial E\rvert_\mathbb{G},x),$$
Now let $E_i:=\delta_{1/r_i}(E)$, so that $\lvert\partial E_i\rvert_\mathbb{G}=r_i^{\mathcal{Q}-1} T_{0,r_i}\lvert\partial E\rvert_\mathbb{G}$. For any open and bounded open set $\Omega$ we can find an $R>0$ such that $\Omega\subseteq B(0,R)$. Therefore, thanks to Lemma \ref{scalepE} we have:
$$\lvert \partial (\delta_{1/r_i}(x^{-1}E))\rvert_\mathbb{G}(\Omega)\leq\lvert \partial (\delta_{1/r_i}(x^{-1}E))\rvert_\mathbb{G}(B(0,R)) =r_i^{-(\mathcal{Q}-1)}T_{x,r_i}\lvert\partial E\rvert_\mathbb{G}(B(0,R))=\frac{\lvert\partial E\rvert_\mathbb{G}(B(x,R r_i))}{r_i^{\mathcal{Q}-1}}.$$
This implies, thanks to Theorem \ref{th:4.16} that:
$$\limsup_{i\to\infty}\lvert \partial (\delta_{1/r}(x^{-1}E))\rvert_\mathbb{G}(\Omega)\leq \limsup_{i\to\infty} \frac{\lvert\partial E\rvert_\mathbb{G}(B(x,R r_i))}{r_i^{\mathcal{Q}-1}}\leq L_\mathbb{G}R^{\mathcal{Q}-1}.$$
Thus, thanks to Theorem \ref{scalepE} the sequence $\{\delta_{1/r_i}(x^{-1}E)\}_{i\in\N}$ is precompact in $L^1_{loc}(\mathbb{G})$ and since we assumed $\delta_{1/r_i}(x^{-1}E)$ converges locally in measure to $L$ then $\delta_{1/r_i}(x^{-1}E)$ converges in $L^1_{loc}(\mathbb{G})$ to $L$. In particular, thanks to Theorem 2.17 of \cite{step2}, we infer that $L$ is of local finite perimeter. Thus, by definition of the tangent sets we have $L\in\Tan(E,0)$ and thanks to item (i), we conclude that $r_i^{-(\mathcal{Q}-1)}T_{0,r_i}\lvert\partial E\rvert_\mathbb{G}\rightharpoonup \lvert\partial L\rvert_\mathbb{G}$. Thanks to the uniqueness of the limit we conclude that $\lvert\partial L\rvert_\mathbb{G}=\nu$.
\end{proof}

\begin{proposizione}\label{Opennn}
If $E$ is an open set of finite perimeter in $\mathbb{G}$, for $\mathcal{S}^{\mathcal{Q}-1}$-almost any $x\in \partial E$ and any $L\in\Tan(E,x)$ we have $\mathcal{L}^{n}(L\setminus \text{int}(L))=0$. In particular the measures $\lvert \partial L\rvert_\mathbb{G}$ and $\lvert \partial (\text{int} (L))\rvert_\mathbb{G}$ coincide on Borel sets. 
\end{proposizione}

\begin{proof}
This follows for instance from Proposition \ref{prop:costnorm} and Theorem 1.1 of \cite{bellettini2019sets}.
\end{proof}

\begin{osservazione}\label{oss:bastaa}
Let $V_\pm:=\{w\in\mathbb{G}:\pm \langle \mathfrak{n}(V),w\rangle>0\}$. Thanks to identity (2.8) in \cite{ambled}, it is immediate to see that $V_\pm$ are open sets of finite perimeter in $\mathbb{G}$ and that
$\partial V_{\pm}=\mp\mathfrak{n}(V)\mathcal{H}_{eu}^{n-1}\llcorner V$. This implies that the horizontal normal of each of the half spaces determined by $V$ coincides, up to a sign, $\lvert\partial V_\pm\rvert_\mathbb{G}$-almost everywhere with $\mathfrak{n}(V)$.
\end{osservazione}

\begin{proposizione}\label{prop:tang}
Let $V\in\G(\mathcal{Q}-1)$ and let $f:V\to \mathfrak{N}(V)$ be an intrinsic Lipschitz function. Suppose that:
$$\Tan_{\mathcal{Q}-1}(\lvert\partial \text{epi(f)}\rvert_{\mathbb{G}},x)\subseteq \mathfrak{M},\qquad\text{for }\lvert\partial \text{epi(f)}\rvert_{\mathbb{G}}\text{-almost every }x\in \mathbb{G}.$$
Then for $\lvert\partial \text{epi(f)}\rvert_{\mathbb{G}}$-almost every $x\in \mathbb{G}$, we have:
$$\Tan_{\mathcal{Q}-1}(\lvert\partial \text{epi(f)}\rvert_{\mathbb{G}},x)\subseteq \{\lambda\mathcal{S}^{\mathcal{Q}-1}\llcorner V(x):\lambda\in [L_{\mathbb{G}}^{-1},l_{\mathbb{G}}^{-1}]\},$$
where $V(x)\in \G(\mathcal{Q}-1)$ is the plane orthogonal to $\mathfrak{n}_f(x)$, that is the normal to $\text{gr}(f)$ introduced in Definition \ref{fnormal}.
\end{proposizione}

\begin{proof}
Proposition \ref{obs6.5} implies that for $\mathcal{S}^{\mathcal{Q}-1}$-almost every $x\in \partial^*_{\mathbb{G}}\text{epi}(f)$ and for every $L\in\Tan(\text{epi}(f),x)$, we have:
\begin{equation}
    \lvert \partial L\rvert_\mathbb{G}=\lambda \mathcal{S}^{\mathcal{Q}-1}\llcorner V_{L,x}\text{ for some $ V_{L,x}\in\G(\mathcal{Q}-1)$ and $\lambda>0$}. \label{eq:nummm1993}
\end{equation}
Furthermore Remark \ref{rk:b1}, Proposition \ref{prop:rapp} and a simple computation that we omit, imply that $\lambda\in[l_\mathbb{G},L_{\mathbb{G}}]$. 

Fix now an $x\in \partial^*\text{epi}(f)$ at which \eqref{eq:nummm1993} hold and that satisfies the thesis of Proposition \ref{prop:costnorm} and let $L\in\Tan(\text{epi}(f),x)$. Thanks to these choices, $L$ is a finite perimeter set with constant horizontal normal and Proposition \ref{vf} and \eqref{eq:nummm1993} tell us that its topological boundary must coincide up to $\mathcal{S}^{\mathcal{Q}-1}$-null sets with the plane $V_{L,x}$. Therefore, since by Proposition \ref{Opennn} we can assume without loss of generality that $L$ is an open set, we conclude that $L$ must coincide with one of the two halfspaces determined by $V_{L,x}$. This implies however, thanks to Remark \ref{oss:bastaa}, that:
\begin{equation}
    \mathfrak{n}(V_{L,x})=\mathfrak{n}_L(y)\text{ for }\mathcal{S}^{\mathcal{Q}-1}\text{-almost every }y\in \partial L.
    \label{eq:nummm1994}
\end{equation}
Furthermore, Proposition \ref{prop:costnorm}(iii) and \eqref{eq:nummm1994} imply that $\mathfrak{n}(V_{L,x})=\mathfrak{n}_L(y)=\mathfrak{n}_f(x)$ for $\mathcal{S}^{\mathcal{Q}-1}$-almost all $y\in\partial L$. This shows however that for $\mathcal{S}^{\mathcal{Q}-1}$-almost all $x\in \text{gr}(f)$, every element of $\Tan(\text{epi}(f),x)$ is an halfspace whose boundary is the plane orthogonal to $\mathfrak{n}_f(x)$ and Proposition \ref{obs6.5} concludes the proof.
\end{proof}

\begin{proposizione}\label{prop:C1H}
Suppose $\gamma:V\to \mathfrak{N}(V)$ is an intrinsic Lipschitz function such that for $\mathcal{S}^{\mathcal{Q}-1}$-almost every $x\in V$ there exists a plane $V(x)\in\G(\mathcal{Q}-1)$ for which:
\begin{equation}
    \lim_{r\to 0}\frac{\mathcal{S}^{\mathcal{Q}-1}\big(\text{gr}(\gamma)\cap B(x\gamma(x),r)\setminus x\gamma(x)X_{V_\gamma(x\gamma(x))}(\alpha)\big)}{r^{\mathcal{Q}-1}}=0,
    \label{eq:reg}
\end{equation}
whenever $\alpha>0$ and where $X_{V_\gamma(x\gamma(x))}(\alpha):=\{w\in \mathbb{G}:\dist(w,V(x\gamma(x)))\leq \alpha\lVert w\rVert\}$. Then, $\text{gr}(\gamma)$ can be covered with countably many $C^1_{\mathbb{G}}$-surfaces.
\end{proposizione}

\begin{proof}
For any $i\in\N$ we define:
$$A_i:=\{x\in \text{gr}(\gamma):\eqref{eq:reg}\text{ holds at }x \text{ and }\mathcal{S}^h(B(x,r)\cap \text{gr}(\gamma))\geq L_\mathbb{G}^{-1}l_\mathbb{G} r^{\mathcal{Q}-1}\text{ for any }0<r<1/i \}.$$
Let us prove that the sets $A_i$ are $\mathscr{S}^h\llcorner \text{gr}(\gamma)$-measurable. It is immediate to see that if we show that the set:
$$\tilde{A}_i:=\{x\in \text{gr}(\gamma):\mathcal{S}^h(B(x,r)\cap \text{gr}(\gamma))\geq L_\mathbb{G}^{-1}l_\mathbb{G} r^{\mathcal{Q}-1}\text{ for any }0<r<1/i \},$$
is closed the measurability of $A_i$ immediately follows since \eqref{eq:reg} holds on a set of full $\mathcal{S}^h\llcorner \text{gr}(\gamma)$-measure. Since $\text{gr}(\gamma)$ is closed, to prove the closedness of $\tilde{A}_i$ it is sufficient to show that if a sequence $\{x_j\}_{j\in\N}\subseteq \tilde{A}_i$ converges to some $x\in \text{gr}(\gamma)$, then $x\in \tilde{A}_i$. So, let $0<r<1/i$ and note that if $d(x,x_j)<r$ we have:
$$L_\mathbb{G}^{-1}l_\mathbb{G} (r-d(x,x_j))^{\mathcal{Q}-1}\leq \mathcal{S}^h\llcorner \text{gr}(\gamma)(B(x_j,r-d(x,x_j)))\leq\mathcal{S}^h\llcorner \text{gr}(\gamma)(B(x,r)). $$
The arbitrariness of $j$ implies that for any $0<r<1/i$ we have $\mathcal{S}^h\llcorner \text{gr}(\gamma)(B(x,r))\geq L_\mathbb{G}^{-1}l_\mathbb{G}r^{\mathcal{Q}-1}$, proving that $x\in\tilde{A}_i$.

We now prove that the sets $A_i$ cover $\mathcal{S}^h$-almost all $\text{gr}(\gamma)$. Recall that $\text{gr}(\gamma)$ is the boundary of the set of locally finite perimeter $\text{epi}(\gamma)$. Thanks to Theorem \ref{th:4.16}, this implies that for $\lvert \partial\text{epi}(\gamma)\rvert_\mathbb{G}$-almost every $x\in \mathbb{G}$ there exists a $\overline{r}(x)>0$ such that for any $0<r<\overline{r}(x)$ we have:
$$ L_\mathbb{G}\mathcal{S}^h\llcorner \text{gr}(\gamma)(B(x,r))\geq \lvert \partial\text{epi}(\gamma)\rvert_\mathbb{G}(B(x,r))\geq l_\mathbb{G}r^{\mathcal{Q}-1},$$
where the first inequality above comes from Remark \ref{rk:b1}. Therefore, if $\overline{r}(x)>1/i$ and \eqref{eq:reg} holds at $x$, then $x\in A_i$ and this concludes the proof of the fact that $\mathcal{S}^h(\text{gr}(\gamma)\setminus \bigcup_{i\in\N} A_i)=0$.

For any $i,j\in \N$ and any $x\in A_i$ we let:
\begin{equation}
    \rho_{i,j}(x):=\sup\bigg\{\frac{\lvert \langle\mathfrak{n}_\gamma(x),\pi_1(x^{-1}y)\rangle\rvert}{d(x,y)}:y\in A_i\text{ and }0<d(x,y)<1/j\bigg\}.
    \nonumber
\end{equation}
We want to prove that for any $i\in\N$ and any $x\in A_i$ we have:
\begin{equation}
    \lim_{j\to \infty}\rho_{i,j}(x)=0.
    \label{eq:nlim}
\end{equation}
Assume by contradiction this is not the case and that there exists a $i\in\N$ and a $z\in A_i$ for which \eqref{eq:nlim} fails. Then, there is a $0<\mathfrak{c}\leq 1$ and an increasing sequence of natural numbers $\{j_k\}_{k\in\N}$ such that for any $k\in\N$ there is a $y_k\in A_i$ for which $y\in B(z,1/j_k)$ and $\lvert \langle\mathfrak{n}_\gamma(z),\pi_1(z^{-1}y_k)\rangle\rvert>\mathfrak{c}d(z,y_k)$.
Thanks to Proposition \ref{cor:SSCM2.2.14} we infer that $y_i\not\in zX_{V_\gamma(z)}(\mathfrak{c}/2)$, indeed:
\begin{equation}
    \dist(V_\gamma(z),z^{-1}y_k)=\lvert\langle \mathfrak{n}_\gamma(z),\pi_1(z^{-1}y_k)\rangle\rvert>\mathfrak{c}d(z,y_k).
\end{equation}
We now claim that for any $k\in\N$ we have:
\begin{equation}
B(y_k,\mathfrak{c}d(z,y_k)/4)\subseteq B(z,2d(z,y_k))\setminus z X_{V_\gamma(z)}(\mathfrak{c}/4).    
    \label{eq:reffin}
\end{equation}
In order to prove the inclusion \eqref{eq:reffin} we fix a $k\in \N$ and let $w:=y_k v\in B(y_k,\mathfrak{c}d(z,y_k)/8)$. With these choices Proposition \ref{cor:SSCM2.2.14} and the triangular inequality imply:
\begin{equation}
\begin{split}
   \dist(V_\gamma(z),z^{-1}w)=&\lvert\langle \mathfrak{n}_\gamma(z),\pi_1(z^{-1} w)\rangle\rvert
   \geq \lvert \langle\mathfrak{n}_\gamma(z),\pi_1(z^{-1} y_k)\rangle\rvert-\lvert \langle\mathfrak{n}_\gamma(z),\pi_1( y_k^{-1}w)\rangle\rvert\\
   \geq & \mathfrak{c}d(z,y_k)-d(y_k,w)\geq \mathfrak{c}d(z,w)-(1+\mathfrak{c})d(y_k,w).
   \label{eq:40000}
\end{split}
\end{equation}
Furthermore, thanks to the choice of $w$ we have:
\begin{align}
    d(y_k,w)\leq \mathfrak{c}d(z,y_k)/4\leq \mathfrak{c}d(z,w)/4+\mathfrak{c}d(y_k,w)/4,\label{eq:equationesmath}\\
d(z,w)\leq d(z,y_k)+d(y_k,w)\leq(1+\mathfrak{c}/8)d(z,y_k)\leq2d(z,y_k).\label{eq:equanigdo2}
\end{align}
Putting together \eqref{eq:40000}, \eqref{eq:equationesmath} and some algebraic computations that we omit, we infer that:
$$\dist(V_\gamma(z),z^{-1}w) \geq \mathfrak{c}d(z,w)/4.$$
The inclusion \eqref{eq:reffin} follows immediately from the above bound and \eqref{eq:equanigdo2}. Therefore, \eqref{eq:reg} and \eqref{eq:reffin} imply:
\begin{equation}
\begin{split}
     \limsup_{r\to 0}\frac{\mathcal{S}^{\mathcal{Q}-1}\big(\text{gr}(\gamma)\cap B(z,r)\setminus z X_{V_{\gamma}(z)}(\mathfrak{c}/8)\big)}{r^{\mathcal{Q}-1}}\geq& \lim_{k\to \infty}\frac{\mathcal{S}^{\mathcal{Q}-1}\big(\text{gr}(\gamma)\cap B(z,2d(z,y_k))\setminus z X_{V_{\gamma}(z)}(\mathfrak{c}/8)\big)}{(2 d(z,y_k))^{\mathcal{Q}-1}}\\
     \geq\lim_{k\to \infty}\frac{\mathcal{S}^{\mathcal{Q}-1}\big(\text{gr}(\gamma)\cap B(y_k,\mathfrak{c}d(z,y_k)/8)\big)}{2^{\mathcal{Q}-1} d(z,y_k)^{\mathcal{Q}-1}}\geq&\lim_{k\to \infty}\frac{L_\mathbb{G}^{-1}l_{\mathbb{G}}(\mathfrak{c}d(z,y_k)/8)^{\mathcal{Q}-1}}{2^\mathcal{Q}d(z,y_k)}=l_\mathbb{G}L_\mathbb{G}^{-1}\bigg(\frac{\mathfrak{c}}{16}\bigg)^{\mathcal{Q}-1},
     \label{eq:n4014}
\end{split}
\end{equation}
where the second last inequality comes from Theorem \ref{th:4.16} and the fact that $y_k\in A_i$. However, since by construction \eqref{eq:reg} holds at any point of $A_i$, \eqref{eq:n4014} is in contradiction with \eqref{eq:reg} and thus \eqref{eq:nlim} must hold at any $x\in A_i$. Define $f_i$ to be the function identically $0$ on $A_i$ and for any $\iota\in\N$ we let $K_i(\iota)$ be a compact subset of $A_i$ for which:
\begin{itemize}
    \item[(i)]$\mathcal{S}^{\mathcal{Q}-1}(A_i\setminus K_i(\iota))\leq 1/\iota$,
    \item[(ii)] $\mathfrak{n}_\gamma$ is continuous on $K_i(\iota)$,
    \item[(iii)] $\rho_{i,j}$ converges uniformly to $0$ on $K_i(\iota)$.
\end{itemize}
The existence of $K_i(\iota)$ is implied by Lusin's theorem and Severini-Egoroff's theorem. Thanks to Whitney extension theorem, see for instance Theorem 5.2 in \cite{step2}, we infer that we can find a $C^1_{\mathbb{G}}$-function such that $f_{i,\iota}\vert_K=0$ and $\nabla_\HH f_{i,\iota}(x)=\mathfrak{n}_\gamma(x)$ for any $x\in K_i(\iota)$. This imples that $A_i$ and thus $\text{gr}(\gamma)$, can be covered $\mathcal{S}^{\mathcal{Q}-1}$-almost all with $C^1_\mathbb{G}$-surfaces. 
\end{proof}

\printbibliography

\section*{Acknowledgements}
I am deeply grateful to Roberto Monti, without whose guidance and constant encouragment I would have never been able to tackle the density problem in the Carnot setting. I would like to also thank Joan Verdera, who invited me in Barcelona during the February of 2019 allowing me to have many fruitful conversation with him and with Xavier Tolsa. I finally would like to thank Gioacchino Antonelli, Giacomo Del Nin and Davide Vittone for many valuable suggestions and correction to the early versions of this work.

%
\end{document}